\documentclass[a4paper,11pt,toc=bib,DIV=11]{scrbook}
\usepackage[utf8]{inputenc}
\usepackage[T1]{fontenc}
\usepackage{graphicx}
\usepackage{grffile}
\usepackage{longtable}
\usepackage{wrapfig}
\usepackage{rotating}
\usepackage[normalem]{ulem}
\usepackage{amsmath}
\usepackage{textcomp}
\usepackage{amssymb}
\usepackage{capt-of}
\usepackage{hyperref}
\usepackage{olbase}
\usepackage{listings}
\usepackage{olthesis,olthmc}
\setkomafont{disposition}{\bfseries}
\titlehead{\hfil\includegraphics[width=8cm]{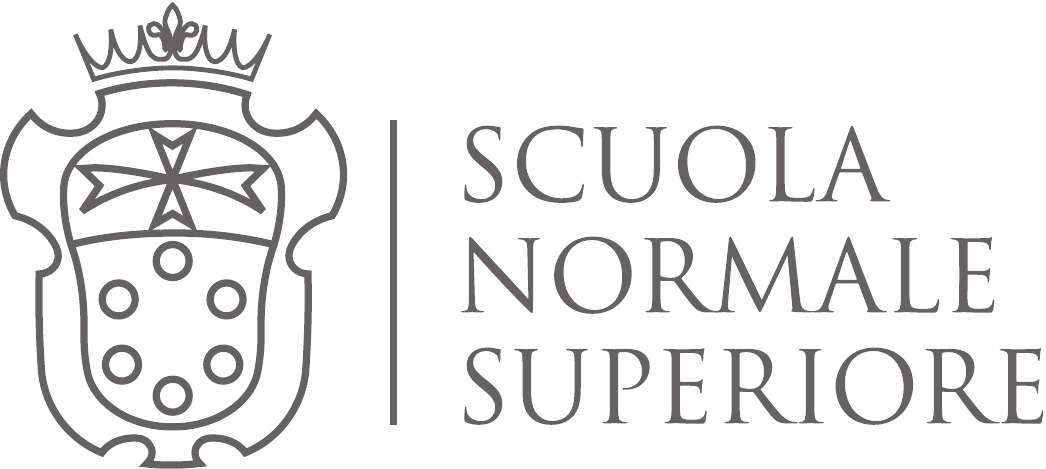}\hfil} 
\subject{Tesi di Perfezionamento in Matematica}
\publishers{Relatore: Prof. Umberto Zannier}
\dedication{To my friends, for all the time spent together.}
\newcommand{\pqmatrix}{\mathcal M}

\newcommand{\canondiv}{\di{K}_\CC}

\newcommand{\qeq}{\approx}

\newcommand{\PP}{\mathfrak{P}}

\newcommand{\EC}{\mathcal E}

\newcommand{\OOred}{\OO_{\mathrm{red}}}
\newcommand{\OOmultszar}{\mathcal V}

\newcommand{\Obb}{\mathbb{O}}

\newcommand{\posval}{positive valuation\xspace}
\newcommand{\negval}{negative valuation\xspace}

\renewcommand{\spann}[1]{\left(#1\right)}

\author{Olaf Merkert}
\date{Anno accademico 2015-2016}
\title{Reduction and specialization of hyperelliptic continued fractions}
\hypersetup{
 pdfauthor={Olaf Merkert},
 pdftitle={Reduction and specialization of hyperelliptic continued fractions},
 pdfkeywords={},
 pdfsubject={},
 pdfcreator={Emacs 25.2.1 (Org mode 9.0.8)}, 
 pdflang={English}}
\begin{document}

\maketitle
\tableofcontents


\chapter{Introduction}
\label{sec:org6b1216b}
This thesis investigates how prime factors arise in denominators of polynomial continued fractions, with a focus on continued fractions of the square root of a polynomial. This is strongly related to the problem of reducing polynomial continued fractions modulo a prime.

Continued fractions have a very long history -- those of rational numbers express the Euclidean Algorithm which was already known in ancient Greece. In modern times, mathematicians such as Lagrange and Galois studied continued fractions of irrational numbers, in particular quadratics (for example square roots). Even today, continued fractions of real numbers remain an important research topic in number theory and other branches of mathematics. 

We write a continued fraction as
\begin{equation*}
\alpha = [a_0, a_1, a_2, \dots] = a_0 + \dfrac{1}{a_1 + \dfrac{1}{a_2 + \ddots}}.
\end{equation*}
For the classical continued fractions with \(\alpha \in \R\), the \emph{partial quotients} \(a_n\) are integers, positive for \(n \geq 1\). 
Instead, one may also take the \(a_n \in \Q[X]\) to be polynomials, non-constant for \(n \geq 1\), to build the continued fraction of a Laurent series in \(\inv X\), i.e. \(\alpha \in \laurinx \Q\). The role of the nearest integer is then played by the polynomial part of the Laurent series.

We are interested for which \(n\) a given prime number \(\pp\) divides the denominator of the coefficients of the \(a_n\) (for brevity, we say the ``prime \(\pp\) appears in the denominator of \(a_n\)''). We are especially interested when it first appears and whether it can disappear again.

Of particular interest is the continued fraction of \(\sqrt{D}\), where \(D \in \Q[X]\) is a monic non-square polynomial of even degree \(2d\). It was first considered by Abel in 1826 \cite{abel-1826-ueber-integ-differ}, who used it to study the integration in elementary terms of certain algebraic functions. Abel showed that periodicity of this continued fraction is equivalent to the existence of a non-trivial solution \(p, q \in \Q[X]\), \(q \neq 0\) of the polynomial Pell equation \(p^2 - D \, q^2 = 1\) (see Chapter \ref{sec:org953fbfb} and Theorem \ref{thm-pellian-iff-cf-periodic}). We say that \(D\) is \emph{Pellian} if such a solution exists. Later, Chebyshev expanded upon these results \cite{chebyshev-1857-sur-integration}.

We call continued fractions of this type \emph{hyperelliptic} because they encode information about the (hyper)elliptic curve \(Y^2 = D(X)\), given that \(d \geq 1\) and \(D\) is also square-free. For example, if \(O_\pm\) are the two points at infinity in a smooth model, the class of \(\pd{O_+} - \pd{O_-}\) is torsion in the Jacobian of the curve \IFF \(D\) is Pellian, i.e. the continued fraction is periodic (see Theorem \ref{thm-pellian-iff-torsion}). 

Note that the polynomials of degree \(2d\), after some normalisation, form an affine variety of dimension \(2d-2\). The Pellian polynomials are then contained in a denumerable union of subvarieties of dimension at most \(d-1\) (see \cite{zannier-2014-pell-survey}, a survey focusing on the geometric aspects of the polynomial Pell equation).
This implies that unlike positive square-free integers which are always ``Pellian'', most polynomials \(D\) are not Pellian, and usually we do not expect a periodic continued fraction. But other results for the classical continued fractions have direct analogues for polynomial continued fractions, see for example \cite{schmidt-2000-continued-fractions-diophantine}.

\medskip

Let us also introduce the \emph{canonical convergents} which are defined via the recurrence relations
\begin{align*}
p_n &= a_n \, p_{n-1} + p_{n-2}, &  
q_n &= a_n \, q_{n-1} + q_{n-2}
\end{align*}
and \(p_0 = a_0, \; p_{-1} = q_0 = 1, \; q_{-1} = 0\).
These imply that \(p_n, q_n \in \Q[X]\) are coprime for any integer \(n \geq 0\), via the identity \(p_n \, q_{n-1} - q_n \, p_{n-1} = (-1)^{n+1}\). The canonical convergents arise by calculating the numerator and denominator of the finite continued fraction 
\begin{equation*}
\frac{p_n}{q_n} = [a_0, a_1, \dots, a_n] = a_{0} + \dfrac{1}{a_{1} + \dfrac{1}{ \ddots + \dfrac{1}{a_{n}}}}.
\end{equation*}
Note that they are usually not monic nor have content \(1\). This is related to prime numbers suddenly appearing in the denominators of the coefficients of the \(a_n\), something van der Poorten was already aware of (see \cite{poorten-2001-non-periodic-continued}).

This follows from the fact that \(D_\pp\), the reduction of \(D\) modulo \(\pp\), is Pellian unless it is a square (the Jacobian over \(\F_\pp\) is finite, all points on it are torsion), so the continued fraction of \(\sqrt{D_\pp}\) is automatically periodic. This leads to one of the main results of this thesis:
\begin{thm}
\label{thm-intro-infinite-poles-rationals}
Let \(\sqrt{D} = [a_0, a_1, a_2, \dots]\). If \(D \in \Q[X]\) is not Pellian, then for all prime numbers \(\pp\) except finitely many, \(\pp\) appears in infinitely many polynomials \(a_n\) in a denominator (of the coefficients).
\end{thm}
We prove this in Theorem \ref{thm-infinite-poles-number-field} more generally for arbitrary number fields.
Note that the formula for multiplying polynomial continued fractions with a constant,
\begin{equation}
\label{intro-eq-cf-mult-const}
\pp^e \, [b_0, b_1, b_2, \dots] = [\pp^e \, b_0, \pp^{-e} \, b_1, \pp^{e} \, b_2, \dots], \quad (e \in \Z),
\end{equation}
raises the question if -- at least for a fixed prime \(\pp\) -- the infinite occurrences in the denominators of the \(a_n\) arise in this rather trivial way. Indeed, this is not the case; we can show that for any \(e \in \Z\), the continued fraction of \(\sqrt{\pp^{-2e} D} = [b_0, b_1, b_2, \dots]\) enjoys the property that \(\pp\) appears in infinitely many \(b_n\) as a denominator.

The primes which are excluded in Theorem \ref{thm-intro-infinite-poles-rationals} are the prime \(2\), any primes appearing already in a denominator of \(D\) and those with \(D_\pp\) square. For technical reasons, we may also need to exclude further primes, depending on the first occurrence of an \(a_n\) with minimal degree. These primes can be determined effectively, too (see Remark \ref{minimal-an-degree-effective}). The prime \(2\) is of course excluded because we are taking square roots.

\begin{rem}
\label{intro-counterexample-infinite-good-reduction}
This result is true only for \(\sqrt{D}\), and does not apply to other elements of the hyperelliptic function field \(\Q(X, \sqrt{D})\). With an analogue of the fact that there are infinitely many primes \(\pp\) such that \(2^n \not\equiv 5 \mod \pp\) for all \(n\), we construct an example of type \(\alpha = \ifracBb{r + \sqrt{D}}{X}\) where there are infinitely many primes \(\pp\) that never appear in the denominators of the \(a_n\) (see Theorem \ref{thm-good-reduction-infinite-primes} in Section \ref{sec:org53fcc2b}). The proof relies on the \v{C}ebotarev density theorem, and represents a variant of the results of Schinzel \cite{schinzel-1960-the-congruence-a} and Corrales-Rodrigáñez-Schoof \cite{corrales-schoof-1997-support-problem-its}.
\end{rem}

For \(\deg D = 4\), another more explicit approach avoids the issue of excluding additional primes. This is described in the rather technical Theorem \ref{thm-genus1-zero-patterns} and Corollary \ref{cor-infinite-poles-deg4}. The former has another consequence for the Gauss norm of the convergents.

Recall that, given some valuation on a field \(K\), we may extend the valuation to polynomials. Define the valuation of a polynomial in \(K[X]\) as the minimum of the valuation on the coefficients (see Section \ref{sec:orge83be0b} for details). The corresponding absolute value is usually called a \emph{Gauss norm}. For \(D \in \Q[X]\), we naturally use the \(\pp\)-adic valuation \(\nu_\pp\). A negative \(\nu_\pp(f)\) then indicates that \(\pp\) appears in at least one denominator of the coefficients of the polynomial \(f\).

As a special case of Corollary \ref{cor-genus1-unbounded-gauss-norm}, we obtain:
\begin{thm}
\label{thm-intro-genus1-unbounded-gauss-norm}
Let \(D\) be a non-Pellian polynomial of degree \(4\), and let \(\pp\) an odd prime with \(D_\pp\) square-free and the class of \(\pd{O_+} - \pd{O_-}\) of \emph{even} torsion order \(m\) in the (finite) Jacobian of the elliptic curve \(Y^2 = D_\pp(X)\).
Then
\begin{align*}
(-1)^n \nu_\pp(a_{n}) &\geq 2 \floor{\ifracBb{n-1}{m}}_\Z + 2 \floor{\ifracBb{n+1}{m}}_\Z,\\
(-1)^n \nu_\pp(q_{n}) &\geq 2 \floor{\ifracBb{n+1}{m}}_\Z,
\end{align*}
where \(\floor{\cdot}_\Z\) denotes the floor function. In particular, the Gauss norms of the partial quotients and the convergents are unbounded both from above and below.
\end{thm}
In the case of \emph{odd} torsion order \(m\), the negative valuations are possibly cancelled out by positive valuations coming from phenomena as in \eqref{intro-eq-cf-mult-const}; this currently prevents any similar prediction (see Example \ref{ex-cfp1-zero-pattern-deg4}, in particular table \ref{cfr-mod19-valuations-table}). Moreover, the precise growth of these Gauss norms is not understood at all right now.
This is an even bigger issue for \(\deg D > 4\), where we have to keep track of further unknowns. This makes an exact estimation of the valuations for higher degrees much more difficult.

The Gauss norms are also related to the height of polynomials. However, we have no information on the archimedean place and the \(2\)-adic valuations, so we have to be careful if we want to compare with known results about the height of the convergents (see Section \ref{sec:orgba963e8}).

Indeed, the convergents \((p_n, q_n)\) are also Padé approximations of \(\sqrt{D}\), i.e. they satisfy
\begin{equation}
\label{intro-convergent-order-inequality}
\ord_\infty (p_n - \sqrt{D} \, q_n) > \deg q_n, 
\end{equation}
where \(\ord_\infty\) is the non-archimedean valuation with \(\ord_\infty X = -1\) and which makes \(\laurinx \Q\) the completion of \(\Q(X)\). In other words \(p_n - \sqrt{D} \, q_n\) has a zero of high order at infinity.

Then by a general result of Bombieri and Paula Cohen \cite{bombieri-cohen-1997-siegels-lemma-pade} on the height of Padé approximations, it follows in the non-periodic case that the logarithmic projective height of the convergents grows quadratically in \(n\). In this thesis, we have worked out the details of a simpler proof for the hyperelliptic case suggested by Zannier, see Theorem \ref{convergent-height-lower-bound} and Theorem \ref{convergents-upper-proj-height-bound} for lower respectively upper bounds. This leads to upper bounds for the projective height of the partial quotients as well (see Corollary \ref{partial-quotients-upper-proj-height-bound}). The corresponding lower bounds for the height of the partial quotients require different arguments, see \cite{zannier-2016-hyper-contin-fract}.

\medskip

The main approach to prove results like Theorem \ref{thm-intro-infinite-poles-rationals} and \ref{thm-intro-genus1-unbounded-gauss-norm} is to study reduction of continued fractions modulo primes. This is interesting in itself, as it gives an example of a map between two ``spaces'' of continued fractions. Chapter \ref{sec:orgd5f1900} contains a general exposition of reduction of continued fractions, using the theory of discrete valuation rings.

The idea is to compare the continued fractions of \(\sqrt{D}\) and \(\sqrt{D_\pp}\). Their partial quotients are contained in \(\Q[X]\) respectively in \(\F_p[X]\). A naive approach would be to try to reduce the partial quotients, but this does not capture the structure of the continued fraction sufficiently. Instead we have to try to reduce the complete quotients \(\alpha_n = [a_n, a_{n+1}, \dots]\) of \(\sqrt{D}\) which are Laurent series in \(\inv X\) over \(\Q\).

We say that a continued fraction has \emph{good reduction in \(\pp\)} if we can reduce the complete quotients of \(\sqrt{D}\) and obtain exactly the complete quotients of \(\sqrt{D_\pp}\). If this fails, we speak of \emph{bad reduction of the continued fraction}.  The latter is the usual situation for non-Pellian \(D\) over \(\Q\) -- and this is a key ingredient for the proof of Theorem \ref{thm-intro-infinite-poles-rationals}.
Other equivalent characterisations for good reduction of the continued fraction are given in Theorem \ref{cf-good-red-partial-quotients}. Note that this notion of good or bad reduction for the continued fraction of \(\sqrt{D}\) is very different from the good or bad reduction of the corresponding (hyper)elliptic curve.

If the continued fraction of \(\sqrt{D}\) is periodic, it trivially has good reduction at almost all primes \(\pp\). This implies that the period length of the continued fraction of \(\sqrt{D_\pp}\) is essentially independent of \(\pp\). This can also be stated and deduced directly in terms of reducing minimal solutions of the polynomial Pell equation, and has recently been used by Platonov \cite{platonov-2014-number-theoretic-properties}, also together with Benyash-Krivets \cite{benyash-platonov-2007-groups-s-units} and Petrunin \cite{platonov-petrunin-2012-the-torsion-problem}, to construct hyperelliptic curves over \(\Q\) of genus \(2\), where the Jacobian contains a torsion point of a specific order. These examples are relevant for the uniform boundedness conjecture for torsion points of abelian varieties.

Van der Poorten's approach to reduction of continued fractions deals primarily with reduction of the convergents: the inequality \eqref{intro-convergent-order-inequality} essentially characterises the convergents up to a common factor of small degree, constant if  \(p\) and \(q\) are coprime (see Corollary \ref{cf-convergent-classification}).
If we normalise \(p_n\) and \(q_n\) correctly, their reduction modulo \(\pp\) remains a convergent of \(\sqrt{D_\pp}\). Moreover, the following theorem holds (both for Pellian and non-Pellian \(D\)):
\begin{thm}[van der Poorten]
\label{thm-vdp-intro}
If the prime \(\pp\) does not appear in a denominator in \(D\), then the reductions modulo \(\pp\) of the normalised convergents \((\normal{p_n}, \normal{q_n})\) of \(\sqrt{D}\) yield \emph{all} the convergents of \(\sqrt{D_\pp}\).
\end{thm}
Unfortunately, the proofs given by van der Poorten (there are slightly different versions in  \cite{poorten-1998-formal-power-series}, \cite{poorten-1999-reduction-continued-fractions} and \cite{poorten-2001-non-periodic-continued}) do not appear to be complete. So one of the main goals of Chapter \ref{sec:orgd5f1900} is to give a more precise statement and a rigorous proof of van der Poorten's result (as in Theorem \ref{convergent-reduction-surjective}).

As might be expected, the reduction of the convergents is strongly related to the reduction of the continued fraction. For example, the bad reduction of the continued fraction is caused by two (or more) convergents of \(\sqrt{D}\) reducing to the same convergent modulo \(\pp\) -- see Proposition \ref{cf-good-reduction-lambda-bijective} and example \ref{ex-cfp2-zero-pattern-deg6}, in particular table \ref{cf2-mod3-degrees-table}.

Finally, we remark that periodicity of the continued fraction of \(\sqrt{D}\) is equivalent to \(\deg a_n = d\) for at least one \(n \geq 1\), where \(2d = \deg D\) (see Corollary \ref{cor-pq-degree-periodicity}). Bad reduction of the continued fraction is also determined by how these degrees increase under reduction (see the discussion in Section \ref{sec:org84e8497}) which connects periodicity of the continued fraction of \(\sqrt{D_\pp}\) and occurrences of \(\pp\) in the denominators. The interplay with the normalisation factors of the canonical convergents then allows us to exclude issues related to \eqref{intro-eq-cf-mult-const}, and leads to a proof of Theorem \ref{thm-intro-infinite-poles-rationals}.

\medskip

\noindent \textbf{On specialization}

The reduction theory for continued fractions of Chapter \ref{sec:orgd5f1900} applies also to specialization. Instead of reducing \(D \in \Q[X]\) modulo a prime, we take for example \(D \in \C(t)[X]\), and try to specialize \(t\) to some \(t_0 \in \C\). Searching for the values \(t_0\) of \(t\) that specialize to a periodic continued fraction of \(\sqrt{D_{t=t_0}}\) corresponds to a special case of the relative Manin-Mumford conjecture, which in turn is a consequence of Pink's conjecture. Recall that periodicity is equivalent to the class of \(\pd{O_+} - \pd{O_-}\) being torsion in the Jacobian of the curve \(Y^2 = D(X)\).

The periodicity of the reduction of the continued fraction was a crucial ingredient for the proof of Theorem \ref{thm-intro-infinite-poles-rationals}. It is therefore natural to ask for specialization analogues of this theorem. The answer depends on the geometry: 

For example Masser and Zannier showed that for \(D = X^6 + X + t\), the continued fraction of \(\sqrt{D}\) is non-periodic, the Jacobian of the curve \(Y^2 = D(X)\) is simple and there are only finitely many \(t_0 \in \C\) such that \(\sqrt{D_{t=t_0}}\) has a periodic continued fraction (see \cite{masser-zannier-2015-torsion-points-on}, here we have reformulated the results in the language of continued fractions). For these \(t_0\), all of them algebraic numbers, we can reuse the arguments from Theorem \ref{thm-intro-infinite-poles-rationals} and show that \(t-t_0\) appears in infinitely many \(a_n\) of the generic continued fraction as a denominator of a coefficient. 

However, from the results of Masser and Zannier follows also that there are infinitely many \(t_1 \in \closure{\Q}\) for which \(t-t_1\) appears at least once as a denominator of a coefficient of some \(a_n\). They might appear infinitely often, but we will show that this can happen only for the trivial reason that we excluded in Theorem \ref{thm-intro-infinite-poles-rationals}. More precisely we can find \(e \in \Z\) (perhaps not effectively), such that in
\begin{equation*}
(t-t_1)^e \, \sqrt{D} = [b_0, b_1, b_2, \dots], \qquad b_n \in \C(t)[X]
\end{equation*}
the ``prime'' \((t-t_1)\) appears only in finitely many \(b_n\) as a denominator. We will discuss this in more detail in Section \ref{sec:orgf9f9759}.

\section{Acknowledgements}
\label{sec:org9942b95}
First and foremost, I would like to thank my supervisor Prof. Umberto Zannier, for pointing me to interesting mathematical problems and sharing his mathematical insight. You have helped me to see number theory in a new light, and improved my understanding of various problems. This thesis would not exist without his input and support. Thank you for answering my many questions and teaching me not to give up and to be independent. I am indebted to you and Scuola Normale Superiore for offering me the chance to pursue a Perfezionamento.

I also would like to dearly thank Prof. David Masser for introducing me to the polynomial Pell equation, and sending me towards Pisa in the first place. I thank Prof. Vistoli for teaching me some algebraic geometry.

A very big ``thank you'' goes to Lars, for many discussions about mathematics and other more trivial topics, putting up with me as a flatmate, and actually reading a draft of this thesis.

Big thanks also to Laura, Fabrizio, Michele and Soli, for countless lunches, game nights and for working together. Thanks for all your help, and for listening to me, even if I made rather less sense. Special thanks to Laura for helping me from my first day in Italy, and to Michele for participating in many sometimes crazy activities.

I would like to thank Francesca for working together, and being a very diligent mathematician.

My referees I am indebted to for their suggestions and careful perusal of my thesis.

\medskip

I wish to thank all the wonderful and interesting people I met at Scuola Normale Superiore, for silly and serious conversations and reminding me that there are people in this world. Many of you I consider now my friends. 

Let me thank in particular Josefine for showing me Florence and the beach, Alex for early morning runs and literally talking to everybody, Sara for teaching me about real friendship, Alexey for extraordinary observations and highly entertaining discussions, Clélie for not being afraid to talk of anything, Błażej for making me a better table tennis player, and Giacomo for his delicious chinese cooking and strange questions. 

Thanks to Mario and Simone for explaining Italy, and Michele (the other one) for explaining biology with a passion. Thanks to Ilir, Marcello and Renata for being loyal hikers, to Umesh for playing table-tennis, to Adam for trying to take silly things seriously,to Elisa and Henry for chatting about fotography and to François and Max for reminding me that I am german.

Thanks also to all the people I spent time with at conferences, for interesting discussions and experiences from other places. Harry and Jung-Kyu, thanks for inviting me to visit the math department of Basel every once in a while.

\smallskip

To Aki, even if we have never met in real life, thank you for the countless hours in the skys of Georgia, Nevada and elsewhere, and in the woods of Chernarus, and for sharing your knowledge of aviation.
\enlargethispage{1cm}

Finally, I want to thank my parents, my brother Sven and my sister Heike, for your support (logistical and otherwise) and for \emph{always} believing that I could complete my PhD. It looks like you were right in the end.

\newpage
\section{Notation reference}
\label{sec:org69c4145}
\begin{center}
\begin{tabular}{ll}
Symbol & Description\\
\hline
\(\N\) & natural numbers: \(\{ 1, 2, 3, \dots \}\)\\
\(\N_0\) & natural numbers with \(0\): \(\{ 0, 1, 2, 3, \dots \}\)\\
\(\K\) & field of characteristic \(\neq 2\)\\
\(\Fr(R)\) & fraction field of integral domain \(R\)\\
\(\laurinx \K\) & Laurent series in \(\inv X\) with coefficients in \(\K\)\\
\(\ord(f)\) & zero-order at \(X = \infty\), sometimes denoted \(\ord_\infty\)\\
\(\LC(f)\) & leading coefficient of polynomial or Laurent series\\
\(\Batest{\K}\) & \(\{ (p, q) \in \K[X]^2 \mid q \neq 0 \}\)\\
\(\Coset{\alpha}{\K}\) & set of convergents of \(\alpha\)\\
\(\Baset{\alpha}{\K}\) & set of best-approximations of \(\alpha\)\\
\(D\), \(d\) & polynomial, with \(\deg D = 2d\) and \(\LC(D)\) a square\\
\(\sol{D}\) & solutions of polynomial Pell equation \eqref{pell}\\
\(\solu{D}\) & solutions of unit-norm equation \eqref{pellu}\\
\(\sigma\) & involution \(\sqrt{D} \to - \sqrt{D}\)\\
\hline
\(\O\), \(\mm = \spann{\pi}\) & discrete valuation ring and maximal ideal with uniformiser\\
\(K\), \(k\) & fraction field and residue field of \(\O\), of characteristic \(\neq 2\)\\
\(\nu\) & valuation, usually of \(\O\)\\
\(\laurinx K_\nu\) & Laurent series with coefficient valuations bounded from below\\
\(\normal{x}\) & normalisation of \(x \in \laurinx K_\nu\) to valuation \(\nu(\normal{x}) = 0\)\\
\(\Red{x} = \RedM{x}\) & reduction/specialization of \(x \in \laurinx \O\)\\
\(\Redn{x} = \RedM{\normal{x}}\) & reduction of normalisation\\
\(\pp\) & prime \emph{number} \(\pp\) (positive integer)\\
\(\PP\) & prime \emph{ideal} \(\PP\) (usually over \(\pp\))\\
\hline
\(\CF(\alpha)\) & continued fraction of \(\alpha\)\\
\(a_n\) & partial quotient of \(\alpha\)\\
\(\alpha_n\) & complete quotient of \(\alpha\)\\
\((p_n, q_n)\) & canonical convergent of \(\alpha\)\\
\(g_n\) & normalisation factor of canonical convergent, \(\nu(g_n) = \nu(q_n)\).\\
\(\vartheta_n\) & \(p_n - \alpha \, q_n\) normalised to \(\nu(\vartheta_n) = 0\)\\
\(c_n\) & partial quotient of \(\gamma = \Red{\alpha}\)\\
\(\gamma_n\) & complete quotient of \(\gamma\)\\
\((u_n, v_n)\) & canonical convergent of \(\gamma\)\\
\(h_n\) & correction factor (in \(k[X]\)) for reduced convergents\\
\(\lambda : \N_0 \to \N_0\) & \((\Redn{p_n}, \Redn{q_n}) = h_n \cdot (u_{\lambda(n)}, v_{\lambda(n)})\), see Corollary \ref{definition-convergent-reduction-map-lambda}\\
\hline
\(\pd{P}\) & point as divisor\\
\(\j{P}\) & divisor class of point\\
\(\di{D}\) & divisor (bold)\\
\(\jdi{D}\) & divisor class\\
\(\CCa, \CC\) & smooth affine and projective models of \(Y^2 = D(X)\)\\
\(O_\pm\), \(\OO\) & the two points of \(\CC\) at infinity; \(\OO = \pd{O_+} - \pd{O_-}\)\\
\(\sigma\) & conjugation of points, \(Y \to -Y\)\\
\hline
\end{tabular}
\end{center}

\chapter{Pell equation}
\label{sec:org953fbfb}
We begin by exploring some well-known basic properties of the Pell equation over polynomials, usually called the \emph{polynomial Pell equation}.
We also explain how to write square roots of polynomials in \(X\) as Laurent series in \(\inv X\), and use this to show that the group of solutions of the polynomial Pell equation has rank at most \(1\).

\medskip

Given a base field \(\K\) with \(\Char \K \neq 2\), let \(D \in \K[X]\) a \emph{non-constant} polynomial and consider the \emph{polynomial Pell equation}
\begin{equation}
\label{pell}
  p^2 - D \, q^2 = 1.
\end{equation}
Clearly, there \emph{always} exist the trivial solutions \((p,q) = (\pm 1,0)\), so naturally we ask if there exist other solutions \((p,q) \in \K[X]^2\) with \(q \neq 0\), which we call the \emph{non-trivial solutions}. If this is the case, we say \(D\) is \emph{Pellian}. If \(\K\) is finite, one may show as for the classical Pell equation over \(\Z\) that \(D\) is always Pellian. If \(\K\) is infinite, it is unlikely that \(D\) is Pellian -- because \(D\) Pellian is equivalent to a torsion condition on a point in the Jacobian of a (hyper)elliptic curve, see Chapter \ref{sec:org99a8e17} for details.

\begin{prop}
\label{pell-necessary-conditions}
Suppose \(D\) is Pellian. Then \(\deg D\) must be even, and the leading coefficient \(\LC(D)\) is a square in \(\K\). However \(D\) cannot be a square in \(\K[X]\).
\end{prop}
\begin{proof}
By the hypotheses \(D\) non-constant and \(q \neq 0\), we have \(\deg (D \, q^2) > 0\). Then \(p^2\) must cancel out the non-constant terms, hence \(\deg p^2 = \deg (D \, q^2)\) which implies \(\deg D = 2 (\deg p - \deg q)\) and that \(\LC(D) = \ifrac{\LC(p)^2}{\LC(q)^2}\) is a square.

Finally, we show that \(D\) is not a square in \(\K[X]\): It is obvious that for \(D = 1\), i.e. \(p^2 - q^2 = (p-q)(p+q) = 1\) there are only constant solutions because \(\units{\K[X]} = \units{\K}\). So if \(D = E^2\) with \(E \in \K[X] \setminus \K\), then for any solution \((p, q)\) we must have \(p, E \, q\) constant which implies \(q = 0\). 
\end{proof}
So these three conditions are necessary (but not sufficient) for the existence of non-trivial solutions.\footnote{The situation in characteristic \(2\) is however completely different, see Section \ref{sec:orgfe2a066} in the Appendix.}

\section{Multiplication law and unit-norm equation}
\label{sec:org6aa78e0}
We assume from now on that \(D\) has even degree \(2d\), is not a square, but \(\LC(D)\) is a square in \(K\) (for example \(1\) if \(D\) is monic).

The set of solutions \(\sol{D}\) (including trivial solutions) of \eqref{pell} carries an abelian group structure\footnote{This group is a twisted \(\G_m\). We can see \(D(X) \, Q^2 = P^2 - 1\) as a twist of \(Q^2 = P^2 - 1\) by the (hyper)elliptic curve \(Y^2 = D(X)\), via \((P, Q) \mapsto (P, Y \, Q)\). Of course \(Q^2 = P^2 - 1\) written as \(P^2 - Q^2 = 1\) is isomorphic to \(\G_m\). See \cite{hazama-1997-pell-equations-polynomials} for more details.} via the multiplication
\begin{equation*}
(p, q) * (p', q') = (p \, p' + D \, q \, q', p \, q' + p' \, q)
\end{equation*}
which comes from the map
\begin{equation*}
\sol{D} \longto \units{\HER}, \qquad (p,q) \mapsto p + q \, \sqrt{D}
\end{equation*}
which is an (injective) group homomorphism (see Section \ref{sec:orgc9b1f16} below).

Note that \((p, q) * (p, -q) = (p^2 - D \, q^2, 0) = (1,0)\) for any Pell solution, so \((1,0)\) is the neutral element, and \((p, -q)\) is the inverse of \((p, q)\).

Actually, we will not really work with \eqref{pell}. To study the structure of the solution set, it is far more convenient to relax to the unit-norm equation (see \cite{hellegouarch-mcquillan-1987-unites-de-certains} for a general treatment)
\begin{equation}
\label{pellu}
  p^2 - D \, q^2 = \omega \in \units{\K}
\end{equation}
where \(\omega\) is an arbitrary unit of \(\K\).
Clearly, any Pell solution satisfies also this equation. The converse does of course not hold, but from a non-trivial solution of \eqref{pellu} we can recover a non-trivial solutions of \eqref{pell}:
\begin{prop}
Suppose \eqref{pellu} has a non-trivial solution \((p,q) \in \K[X]^2\) (with \(q \neq 0\)). Then \(D\) is Pellian.
\end{prop}
\begin{proof}
The multiplication law from above generalises to \eqref{pellu}, with \((p, q) * (p, -q) = (\omega, 0)\), hence
\begin{equation*}
(p, q) * (p, q) * (p, -q) * (p, -q) = (\omega^2, 0).
\end{equation*}
Set
\begin{equation*}
(p', q') = (\inv\omega, 0) * (p, q) * (p, q) = \inv\omega \cdot (p^2 + D \, q^2, 2 \, p \, q),
\end{equation*}
so that \((p', q')\) remains in \(\K[X]\) and is clearly a solution of \eqref{pell}. As observed in the proof of Proposition \ref{pell-necessary-conditions}, \(q \neq 0\) implies \(p \neq 0\), hence \(\inv\omega \, 2 \, p \, q \neq 0\), so \((p', q')\) is a non-trivial Pell solution.
\end{proof}

From now on, we refer also to \eqref{pellu} as the \emph{Pell equation}, and mostly forget about \eqref{pell}. We denote by \(\solu{D}\) the set of all solutions of \eqref{pellu}. We will see that for the purposes of this thesis, it is more natural to work with the unit-norm equation.
\section{Units of hyperelliptic function fields}
\label{sec:orgc9b1f16}
The quadratic field extension \(\HEF\) of \(\K(X)\) is called a hyperelliptic function field -- specifically it is the function field of the hyperelliptic curve \(\CCa : Y^2 = D(X)\) which we will study in more detail in Chapter \ref{sec:org99a8e17}.
The subring \(\HER\) of \(\HEF\) is the integral closure of \(\K[X]\), describing the regular functions on the affine curve. For now, we show that the units of \(\HER\) correspond to solutions of the Pell equation \eqref{pellu}. See also \cite{hellegouarch-mcquillan-1987-unites-de-certains} for generalisations to other algebraic functions.

\begin{thm}
The map
\begin{equation*}
\pi: \solu{D} \longto \units{\HER}, \quad (p,q) \mapsto p + q\, \sqrt{D}
\end{equation*}
is bijective, and via the multiplication \(*\) on \(\solu{D}\) gives an isomorphism of abelian groups.
\end{thm}

Observe that there is a single non-trivial \(\K(X)\)-automorphism \(\sigma\) of \(\HEF\),  defined by \(\sigma(\sqrt{D}) = -\sqrt{D}\).

\begin{proof}
Actually, we defined \(*\) as the pullback under \(\pi\) of the multiplication on \(\HER\), so clearly
\begin{equation*}
\pi(\phi * \psi) = \pi(\phi) * \pi(\psi) \text{ for all } \phi, \psi \in \solu{D}.
\end{equation*}
And by the identity
\begin{equation*}
(p + q \, \sqrt{D} )(p - q \, \sqrt{D}) = p^2 - D \, q^2 = \omega \in \units \K
\end{equation*}
it follows that \(\im \pi \subset \units \HER\), so \(\pi\) is well defined.

Recall that we assume that \(D\) is not a square, so the ring \(\HER\) is a free rank 2 module over \(\K[X]\) with basis \((1, \sqrt{D})\): this implies that \(\pi\) is injective.

It remains to check that \(\pi\) is also surjective:
Let \(\phi= p + q \, \sqrt{D} \in \units{\HER}\) with \(p, q \in \K[X]\). Then we have
\begin{equation*}
\phi \cdot \sigma(\phi) = (p + q \, \sqrt{D})(p - q \, \sqrt{D}) = p^2 - D \, q^2 \in \K[X]
\end{equation*}
Applying the same argument to the inverse \(1/\phi\), we find \(p^2 - D \, q^2 \in \units{\K[X]} = \units \K\), so \((p,q)\) is a solution of \eqref{pellu}. This proves that \(\pi\) is surjective.
\end{proof}
\begin{rem}
Observe that the trivial solutions of \eqref{pellu} correspond precisely to the elements of \(\units{\K}\).
\end{rem}

\section{Laurent series and valuation}
\label{sec:org7391b36}
Define the field of Laurent series over \(\K\)
\begin{equation*}
\laurinx \K = \left\{\left. \lseries{N}{t}{n} \;\right|\; N \in \Z, t_n \in \K \right\}.
\end{equation*}
It contains \(\K[X]\) and its fraction field \(\K(X)\). Note that \(\laurinx \K\) is the completion of \(\K(X)\) with respect to the discrete valuation \(\ios = \ord_\infty\) (the zero-order at infinity), defined by
\begin{equation*}
\io{f} = \ord_\infty(f) = -N \text{ where } f = \lseries{N}{t}{n}, \; f_N \neq 0.
\end{equation*}

\begin{rem}
\label{rem-poly-no-poles}
For example if \(f \in \K[X]\), then \(\io{f} = - \deg f\). Moreover,
\begin{equation}
\label{poly-no-poles}
\io{f} > 0 \text{ and } \quad f \in \K[X] \text{ implies } f = 0.
\end{equation}
\end{rem}

There is a truncation operation which takes a Laurent series and returns a polynomial, essential for the continued fraction process:
\begin{defi}
\label{define-laurent-truncation}
For \(\alpha = \lseries{N}{t}{n} \in \laurinx \K\), we define the \emph{truncation} (or \emph{principal part})
\begin{equation*}
\gauss{\alpha} = \begin{cases}
0 & \text{ if } \io{\alpha} >0, \text{ i.e. } N < 0 \\
\poly{N}{t} & \text{ if } \io{\alpha} \leq 0, \text{ i.e. } N \geq 0
\end{cases}
\end{equation*}
as the \emph{polynomial part} of \(\alpha\).
\end{defi}

\begin{rem}
\label{truncation-unique}
We could also define \(\gauss{\alpha}\) as the \emph{unique} \(a \in \K[X]\) satisfying \(\io{\alpha - a} > 0\) -- unicity is a consequence of Remark \ref{rem-poly-no-poles}.
\end{rem}

\begin{rem}
\label{truncation-of-sum}
The preceding remark implies for \(\alpha, \beta \in \laurinx{\K}\) that \(\gauss{\alpha + \beta} = \gauss{\alpha} + \gauss{\beta}\).
\end{rem}

\begin{rem}
\label{truncation-of-rational}
Recall that \(\K[X]\) is Euclidean with respect to \(\deg\). So for \(p, q \in \K[X]\) with \(q \neq 0\) there exist \(a, r \in \K[X]\) satisfying \(p = a \, q + r\) and \(\deg r < \deg q\). Then
\begin{equation*}
\frac{p}{q} - a = \frac{r}{q} \text{ with } \io{\ifrac{r}{q}} > 0
\end{equation*}
implies \(\gauss{\ifrac{p}{q}} = a\), and moreover \(a, r\) are uniquely determined, again by Remark \ref{truncation-unique}.
\end{rem}

We now explain how to compute \(\sqrt{D}\) as a Laurent series in \(\inv X\):
\begin{prop}
\label{laurent-sqrt-d}
Let \(D \in \K[X]\) with \(\deg D = 2d\) and \(\LC(D) \in \K\) a square. Then \(\sqrt{D} \in \laurinx \K\), so \(D\) is a square in \(\laurinx \K\).
\end{prop}
\begin{proof}
Let \(D = \poly{2d}{d}\), where \(d_{2d}\) is a square in \(\K\). Hence we may reduce to the case \(d_{2d} = 1\), and write
\begin{equation*}
D = X^{2d} \, (1 + f(X)) \text{ where } f(X) = d_{2d-1} \, \inv X + \dots + d_0 \, X^{-2d}.
\end{equation*}
Of course \(X^{2d}\) is a square in \(\laurinx \K\), and because \(\io{f(X)} > 0\), we find that
\begin{equation*}
\sqrt{1+f(X)} = \sum_{n=0}^\infty \binom{1/2}{n} \, f(X)^n
\end{equation*}
converges in \(\laurinx{\K}\), so also \((1 + f(X))\) is a square.
\end{proof}

\begin{defi}
\label{choose-sqrt-d}
We choose once and for all one square root of \(D\), and denote it by \(\sqrt{D}\). We also set 
\(A = \gauss{\sqrt{D}}\). For example, if \(D\) is monic of degree \(2d\), then we choose \(\sqrt{D} = X^d + \dots\). 
\end{defi}
\pagebreak
\begin{prop}
\label{completion-of-square}
We have \(\deg A = \frac{1}{2} \deg D\), and \(\deg (D - A^2) < \deg A\).
\end{prop}
\begin{proof}
As \(\ios\) is a valuation, clearly \(- \deg D = \ios{D} = 2 \, \ios{\sqrt{D}} < 0\), hence \(-\deg A = \ios{A} = \ios{\gauss{\sqrt{D}}} = \ios{\sqrt{D}}\) which implies the first claim.

Moreover, we can write
\begin{equation}
\label{sqrt-d-A-plus-eps}
\sqrt{D} = A + \varepsilon \text{ with } \varepsilon \in \laurinx \K \text{ and } \io{\varepsilon} > 0.
\end{equation}
So
\begin{equation*}
D = A^2 + 2 \, A \, \varepsilon + \varepsilon^2
\end{equation*}
where of course
\begin{equation*}
\io{2 \, A \, \varepsilon + \varepsilon^2} = \min(\io{A}, \io{\varepsilon}) + \io{\varepsilon} = \io{A} + \io{\varepsilon} > \io{A}
\end{equation*}
implies the second claim.
\end{proof}

We can rephrase this as
\begin{lemma}[Completion of the square]
\label{completion-of-square-lemma}
There exist \(A, \Omega \in \K[X]\) with \(\deg \Omega < \deg A = \frac{1}{2} \, \deg D\) satisfying
\begin{equation*}
D = A^2 + \Omega
\end{equation*}
where \(A\) is unique up to a factor \(-1\).
\end{lemma}

\begin{rem}
Note that the lemma also holds if \(D\) is a square.
\end{rem}

\begin{rem}
\label{rem-q=1-pell-solution}
If \(\deg \Omega = 0\), then clearly \((A, 1)\) is a solution of the Pell equation \eqref{pellu}.
\end{rem}
\section{Group structure of Pell solutions}
\label{sec:org4f39766}
We apply the definitions of the previous section directly to study the structure of the Pell solutions. The group of solutions of \eqref{pellu} is essentially cyclic:

\begin{prop}
If \(D\) is not Pellian, then \(\sol{D} = \{ \pm 1 \}\) and  \(\solu{D} = \units{K}\). But if \(D\) is Pellian, then
\begin{equation*}
\sol{D} \iso \{\pm 1\} \times \Z \quad \text{ and } \quad \solu{D} \iso \units K \times \Z.
\end{equation*}
\end{prop}
\begin{proof}
We use that \(\solu{D} \iso \units{\HER}\). By Proposition \ref{laurent-sqrt-d}, we can embed \(\HER\) into \(\laurinx \K\), and define
\begin{equation*}
o(p,q) = \io{p + \sqrt{D} \, q} \text{ for } (p, q) \in \solu{D}.
\end{equation*}
This defines a group homomorphism \(o : \solu{D} \to \Z\). The kernel is made precisely of the trivial solutions:
\begin{equation*}
\io{p} = \io{p + \sqrt{D} \, q + p - \sqrt{D} \, q} \geq \min\left( \io{p + \sqrt{D} \, q}, \io{p - \sqrt{D} \,q}\right)
\end{equation*}
and
\begin{equation*}
\io{p + \sqrt{D} \, q} + \io{p - \sqrt{D} \, q} = 0
\end{equation*}
so \(\io{p + \sqrt{D}} = 0\) implies \(\deg p = - \io{p} \leq 0\), hence \(q = 0\).

If \(D\) is not Pellian, then the image of \(o\) is \(0\). But if \(D\) is Pellian, then the image of \(o\) is isomorphic to \(\Z\).

We can of course restrict \(o\) to \(\sol{D}\), and then the kernel becomes \(\{(\pm 1, 0) \} \iso \{\pm 1\}\).

The structure of \(\sol{D}\) and \(\solu{D}\) now follows from standard theorems about group homomorphisms.
\end{proof}

\medskip

We conclude our discussion of the polynomial Pell equation with the following observation:
\begin{cor}
\label{deg-2-always-pellian}
If \(\deg D = 2\) and the leading coefficient \(\LC(D)\) is a square, then \(D\) is \emph{always} Pellian (unless it is square).
\end{cor}
\begin{proof}
By Lemma \ref{completion-of-square-lemma}, in this case \(\deg \Omega < \deg A = 1\) so forcefully \(\deg \Omega = 0\), and Remark \ref{rem-q=1-pell-solution} says that \((A, 1)\) is a Pell solution.
\end{proof}
\chapter{Rational approximations}
\label{sec:orgbc894a6}
As mentioned before, the existence of non-trivial solutions is not guaranteed for the polynomial Pell equation. But one observes that the Pell solutions produce very good rational approximations for \(\sqrt{D}\) (as in the numerical case). This chapter introduces two notions of rational approximation: convergents and best-approximations. We will study in this chapter how they are related to each other and to the non-trivial Pell solutions. Their complete classification is however best understood with the help of continued fractions, to be discussed later in Section \ref{sec:orga882f91}.

For our purposes, it is convenient to keep track of common factors in the numerator and denominator of the rational approximation. \emph{Instead} of \(\K(X)\), our candidate set for rational approximations is the set of tuples representing quotients
\begin{equation*}
\Batest{\K} = \{ (p,q) \in \K[X]^2 \mid q \neq 0 \}.
\end{equation*}
We loosely refer to \(p\) as the \emph{numerator} and to \(q\) as the \emph{denominator}, in spirit of the obvious map \(\Batest{\K} \longto \K(X), \; (p,q) \mapsto p/q\).

For \(r, p , q \in \K[X]\) with \(r, q \neq 0\) we also write
\begin{equation*}
r \cdot (p, q) = (r \, p, r \, q).
\end{equation*}

With this terminology established, we can begin the study of different types of approximations. Of course, we are using the valuation \(\ios = \ord_\infty\) (the zero-order at infinity) introduced in Section \ref{sec:org7391b36} to measure how well we can approximate any Laurent series in \(\laurinx \K\).

\section{Convergents}
\label{sec:org32d755d}
A classical type of rational approximation is given by the convergents. They arise very naturally from the continued fraction expansion -- we will see details later in Chapter \ref{sec:org172eb73}. For now, we give a different characterisation in the spirit of the famous Dirichlet Lemma. This definition also shows immediately that the convergents are a special case of Padé approximations.

\begin{defi}
Let \(\alpha \in \laurinx \K\). A tuple \((p, q) \in \Batest{\K}\) is called a \emph{convergent} of \(\alpha\) over \(\K[X]\) if it satisfies
\begin{equation}
\label{convergent-condition}
\io{p - \alpha \, q} > \deg q.
\end{equation}
We denote the set of all convergents by \(\Coset{\alpha}{\K}\).
\end{defi}

\begin{rem}
\label{rem-convergents-existence}
It can easily be seen that convergents exist:
The condition \eqref{convergent-condition} is a linear condition on the coefficients of \(p\) and \(q\). Clearly \(p\) removes the coefficients of \(X^n\) for \(n \geq 0\) in \(\alpha \, q\); then only the coefficients of \(X^{-1}, \dots, X^{-\deg q}\) need to vanish, which can be accomplished by choosing the \(1 + \deg q\) coefficients of \(q\) appropriately. See Section \ref{sec:org542977d} for more details.
\end{rem}

\begin{rem}
\label{convergent-cancel-factor}
Suppose \(r, p, q \in \K[X]\). Then
\begin{equation*}
r \cdot (p, q) \in \Coset{\alpha}{\K} \implies (p, q) \in \Coset{\alpha}{\K}
\end{equation*}
because \(0 \geq \io{r}\) implies
\begin{equation*}
\io{p - \alpha \, q} \geq \io{r \, p - \alpha \, r \, q} > \deg(r \, q) \geq \deg q.
\end{equation*}
Note that the implication in the converse direction does not hold in general because multiplication with \(r\) decreases \(\ord\) and increases \(\deg\).
\end{rem}
In principle, one could for any convergent \((p, q)\) assume that \(p\) and \(q\) are coprime, and identify it with the fraction. This might improve the approximation quality, however it turns out that the common factors help to understand the reduction of convergents modulo a prime better (to be discussed in Chapter \ref{sec:orgd5f1900}).

Anyway the common factor usually has a small and controllable degree:
\begin{prop}
\label{convergent-common-factor-degree}
Let \((p, q) \in \Batest{\K}\) and \(r  \in \mino{\K[X]}\). Suppose
\begin{equation*}
\io{p - \alpha \, q} = \xi + \deg q.
\end{equation*}
Then \(r \cdot (p, q) \in \Coset{\alpha}{\K}\) is a convergent \IFF \(\deg r < \xi/2\).

In particular, suppose \(r' \in \mino{\K[X]}\) with \(\deg r \leq \deg r'\). Then \(r' \cdot (p, q) \in \Coset{\alpha}{\K}\) implies \(r \cdot (p, q) \in \Coset{\alpha}{\K}\).
\end{prop}
\begin{rem}
Note that the Proposition holds also when \(\xi = \infty\) -- but that happens only for \(\alpha \in \K(X)\).
\end{rem}
\begin{proof}
In order for \(r \cdot (p, q)\) to be a convergent, the following expression must be positive:
\begin{equation}
\label{convergent-factor-approx-quality}
\io{r\,p - \alpha\, r\,q} - \degb{r \, q} = \io{r} + \io{p - \alpha \, q} - \deg r - \deg q = \xi - 2 \, \deg r.
\end{equation}
The second part of the Proposition follows immediately.
\end{proof}
\begin{rem}
The above \eqref{convergent-factor-approx-quality} also suggests that for \((p, q)\) coprime we have the optimal relative approximation quality: higher is better.
\end{rem}

\begin{rem}
Let \(\alpha \in \laurinx \K\), set \(a = \gauss{\alpha}\). Then \((a, 1) \in \Coset{\alpha}{\K}\) because
\(\io{a - \alpha} > 0 = \deg 1\).
\end{rem}

\begin{prop}
\label{convergent-q-determines-p}
If \((p, q) \in \Coset{\alpha}{\K}\) is a convergent, then \(p\) is uniquely determined by \(q\) via \(p = \gauss{\alpha \, q}\).
\end{prop}
\begin{proof}
This follows immediately from \(\io{p - \alpha \, q} > \deg q \geq 0\), and Remark \ref{truncation-unique} characterising \(\gauss{\cdot}\).
\end{proof}
\section{Pell solutions are convergents}
\label{sec:org0d42630}
Let us for a moment return to the polynomial Pell equation, and show that the non-trivial Pell solutions (up to conjugate) are convergents of \(\sqrt{D}\). Obviously, not all convergents of \(\sqrt{D}\) need to be Pell solutions.

\begin{prop}
\label{weak-pell-solutions-are-convergents}
Let \((p, q) \in \Batest{\K}\) and \(p^2 - D \, q^2 = \Omega\). Then the inequality
\begin{equation}
\label{general-weak-pell-condition}
\deg \Omega < \tfrac{1}{2} \deg D
\end{equation}
holds \IFF either \((p, q) \in \Coset{\sqrt{D}}{\K}\) or \((p, -q) \in \Coset{\sqrt{D}}{\K}\) is a convergent of \(\sqrt{D}\).

In particular, if \((p, q) \in \solu{D}\) is a Pell solution with \(q \neq 0\), then one of \((p, q), (p, -q)\) is a convergent of \(\sqrt{D}\).
\end{prop}
\begin{proof}
Let us begin with some observation useful to both directions of implication. Note that
\begin{equation}
\label{omega-order-factors}
\io{\Omega} = \io{p^2 - D \, q^2} = \io{p + \sqrt{D} \, q} + \io{p - \sqrt{D} \, q}.
\end{equation}
And if \(\io{p - \sqrt{D} \, q} > 0\), the ultrametric inequality and \(\io{\sqrt{D} \, q} \leq 0\) imply
\begin{equation}
\label{cv-plus-bigger-minus}
\io{p + \sqrt{D} \, q} = \min\left(\io{2 \, \sqrt{D} \, q}, \io{p - \sqrt{D} \,q} \right) = \io{\sqrt{D} \, q} \leq 0.
\end{equation}

Now assume that \((p, q) \in \Coset{\sqrt{D}}{\K}\) is a convergent, hence \(\io{p - \sqrt{D} \,q} > \deg{q} \geq 0\). Then \eqref{omega-order-factors} and \eqref{cv-plus-bigger-minus} yield 
\begin{equation*}
\io{\Omega}  > \deg q + \io{\sqrt{D} \, q} = \io{\sqrt{D}}
\end{equation*}
which implies \(\deg \Omega < \tfrac{1}{2} \deg D\).

\medskip

For the other direction, assume that \((p, q)\) satisfies \eqref{general-weak-pell-condition}, hence \(\io{\Omega} > \io{\sqrt{D}} \geq \io{\sqrt{D} \, q}\). Without loss of generality, we may further assume \(\io{p - \sqrt{D} \, q} \geq  \io{p + \sqrt{D} \, q}\). It follows
\begin{equation*}
\io{\Omega} > \io{2 \, \sqrt{D} \,q} = \io{p + \sqrt{D} \, q - (p - \sqrt{D} \, q)} \geq \io{p + \sqrt{D} \, q} 
\end{equation*}
so by \eqref{omega-order-factors} \(\io{p - \sqrt{D} \, q} > 0\), which in turn implies \eqref{cv-plus-bigger-minus}. Using \eqref{omega-order-factors} again, we arrive at
\begin{multline*}
\io{p - \sqrt{D} \, q} = \io{\Omega} - \io{p + \sqrt{D} \, q} \\
= \io{\Omega} - \io{\sqrt{D} \,q} > -\io{q} = \deg q
\end{multline*}
as desired.
\end{proof}

\section{The universal property of best-approximation}
\label{sec:orgd9280fe}
The convergents have a useful universal property: they are in some sense the optimal approximations that we can find. For a discussion about where this particular universal property comes from, see \cite{khintchine-1956-kettenbruche}. See also \cite{cassels-1957-introduction-to-diophantine} where the continued fraction process for real numbers is defined using best-approximations.\footnote{The polynomial case is even simpler than the integer case treated there: because the absolute value (corresponding to the valuation \(\ios\)) is non-archimedean, there are no intermediate fractions to worry about.}

As we did with the convergents, we modify our definition so that it allows common factors; and we prefer a category theoretic style of universal property.

\begin{defi}
Let \(\alpha \in \laurinx \K\). A tuple \((p,q) \in \Batest{\K}\) is called a \emph{best-approximation} (of second type) in \(\K[X]\), if for every other tuple \((p',q') \in \Batest{\K}\) satisfying
\begin{equation}
\label{bestapproxcondition}
\io{p' - \alpha \, q'} \geq \io{p - \alpha \, q} \text{ and } \deg{q'} \leq \deg{q}
\end{equation}
we have \(\ifrac{p'}{q'} = \ifrac{p}{q}\).

We denote by \(\Baset{\alpha}{\K}\) the set of all best-approximations of \(\alpha\).
\end{defi}

\begin{rem}
\label{best-approx-cancel-factor}
If \((p, q) \in \Batest{\K}\) and \(r, r' \in \mino{\K[X]}\) with \(\deg r' \leq \deg r\) (for example \(r' = 1\)), then
\begin{equation*}
r \cdot (p, q) \in \Baset{\alpha}{\K} \implies r' \cdot (p, q) \in \Baset{\alpha}{\K}.
\end{equation*}
because
\begin{equation*}
\io{r' \, p - \alpha \, r' \,q} \geq \io{r \, p - \alpha \, r \, q} \text{ and } \degb{r' \, q} \leq \degb{r \, q}.
\end{equation*}
\end{rem}

So without loss of generality, one \emph{could} assume that for a best-approximation \((p, q)\), we have \(p\) and \(q\) coprime. This could also be enforced by changing the phrasing of the definition slightly, as is in fact usually done in the literature. However, in that case, \eqref{bestapproxcondition} becomes harder to satisfy because removing a common (non-constant) factor decreases \(\deg q\) and increases \(\io{p - \alpha \, q}\).

As alluded to before, when studying the reduction of convergents modulo a prime, it is useful to allow common factors. The notion of best-approximation presented here gives even more freedom for such common factors than our notion of convergent. We can indeed find best-approximations \((p, q)\) for arbitrary \(\deg q\), which may not be possible with convergents (see Section \ref{sec:orga882f91}). This simplifies their classification, and hence the classification of convergents.
\medskip

Before we investigate the relation between convergents and best-approximations, let us show that there are not so many best-approximations:
\begin{prop}
\label{best-approx-for-given-degree}
Let \((p, q) \in \Batest{\K}\) coprime and \(r \in \mino{\K[X]}\). Suppose \(r \cdot (p, q) \in \Baset{\alpha}{\K}\) is a best-approximation.

Then any (other) best-approximation \((p', q') \in \Baset{\alpha}{\K}\) with \(\deg{q'} = \degb{r \, q}\) has the shape
\begin{equation*}
(p', q') = r' \cdot (p, q) \text{ where } r' \in \K[X], \; \deg r = \deg r'.
\end{equation*}
\end{prop}
\begin{proof}
Because \((p', q'), r \cdot (p, q) \in \Baset{\alpha}{\K}\) with \(\deg q' = \degb{r \, q}\), at least one of
\begin{equation*}
\io{p' - \alpha \, q'} \geq \io{r\,p - \alpha \, r\,q} \text{ or } \io{p' - \alpha \, q'} \leq \io{r\,p - \alpha \, r\,q}
\end{equation*}
must be satisfied. Together with \(\deg q' = \degb{r \, q}\) this implies \(\frac{p'}{q'} = \frac{r\,p}{r\,q} = \frac{p}{q}\) by the best-approximation property of either \(r \cdot (p, q)\) or \((p', q')\).

Finally because we assume \(p, q\) are coprime, there exists \(r' \in \K[X]\) with \(q' = r' \, q\) and \(p' = r' \, p\).
\end{proof}
This proposition has two important consequences:
\begin{cor}
For any best-approximation \((p, q) \in \Baset{\alpha}{\K}\), the numerator \(p\) is uniquely determined by the denominator \(q\).
\end{cor}

\begin{cor}
\label{best-approx-coprime-unicity}
Given an integer \(n \geq 0\), there exists up to a constant factor at most one best-approximation \((p, q)\) with \(\deg {q} = n\) and \(p, q\) coprime.
\end{cor}

We proceed to show that best-approximations generalise the convergents.
\begin{prop}
\label{best-approx-common-factor-degree}
Let \((p, q) \in \Batest{\K}\) and \(r \in \mino{\K[X]}\). Suppose
\begin{equation*}
\io{p - \alpha \, q} = \xi + \deg q.
\end{equation*}
Then \(\deg r < \xi\) implies \(r \cdot (p, q) \in \Baset{\alpha}{\K}\) is a best-approximation.
\end{prop}
Putting  \(r = 1\) with \(\deg r = 0 < \xi\) by definition of convergents, we get:
\begin{cor}
\label{convergents-are-best-approx}
Every convergent is a best-approximation: \(\Coset{\alpha}{\K} \subset \Baset{\alpha}{\K}\).
\end{cor}

With Corollary \ref{weak-pell-solutions-are-convergents} this implies also:
\begin{cor}
\label{pell-solutions-are-bestapprox}
For every non-trivial solution \((p, q)\) of the Pell equation \eqref{pellu}, either \((p, q)\) or \((p, -q)\) is a best-approximation of \(\sqrt{D}\).
\end{cor}

\begin{proof}[Proof of Proposition \ref{best-approx-common-factor-degree}]
Let \((p', q') \in \Batest{\K}\) satisfy
\begin{equation*}
\deg{q'} \leq \degb{r \, q} \text{ and } \io{p' - \alpha \, q'} \geq \io{r} + \io{p - \alpha \, q}.
\end{equation*}
Now
\begin{equation*}
\det \mfour{p}{p'}{q}{q'} = \det \mfour{1}{-\alpha}{0}{1} \mfour{p}{p'}{q}{q'} = \det \mfour{p-\alpha\,q}{p'-\alpha \, q'}{q}{q'}
\end{equation*}
and taking the valuation \(\ios\) we get
\begin{multline*}
\io{p \, q' - p' \, q} \geq \min\left(\io{q'} + \io{p-\alpha\,q}, \io{q} + \io{p'-\alpha \, q'} \right) \\
\geq \io{r} + \io{q} + \io{p - \alpha \, q} = \xi - \deg r > 0.
\end{multline*}
But \(p \, q' - p' \, q \in \K[X]\), so it must be \(0\). This implies \(\ifrac{p'}{q'} = \ifrac{p}{q} = \ifrac{r \, p}{r \, q}\) as desired.
\end{proof}

\begin{rem}
Note that unlike Proposition \ref{convergent-common-factor-degree}, this is only a sufficient condition. It is not necessary: if we start with \((p, q)\) with \(\xi > 1\) (for example \(\xi = 2\)), then multiplying with \(r\) of maximal degree (for example \(\deg r = 1\)), we obtain a best-approximation \((p', q')= r \cdot (p, q)\) with \(\xi' \leq 0\) (in the example \(\xi' = 0\)). Then \(r' = 1\) does not satisfy \(\deg r' < \xi'\), even though \((p', q')\) is a best-approximation.
\end{rem}

\medskip

We conclude our study of best-approximations by investigating their ordering. Indeed we expect that increasing the ``height'' of the convergent (i.e. \(\deg q\)) should also increase the approximation quality:
\begin{prop}
\label{compare-best-approx-strict}
Let \((p, q), (p', q') \in \Baset{\alpha}{\K}\) different best-approximations, i.e. \(\ifrac{p}{q} \neq \ifrac{p'}{q'}\). Then
\begin{equation*}
\deg{q} < \deg{q'} \iff \io{p - \alpha \, q} < \io{p' - \alpha \, q'}.
\end{equation*}
\end{prop}
\begin{proof}
By the universal property, the statement
\begin{equation}
\label{best-approx-condition-reverse}
\io{p - \alpha \, q} \geq \io{p' - \alpha \, q'} \text { and } \deg{q} \leq \deg{q'}
\end{equation}
is false under the hypothesis of the fractions being different.

So if \(\deg q < \deg q'\), necessarily the first inequality must not hold, giving the \(\Rightarrow\) part.

Conversely, if \(\io{p - \alpha \, q} > \io{p' - \alpha \, q'}\), then the second inequality is false, i.e. \(\deg q > \deg q'\). But this is clearly the \(\Leftarrow\) part, with the roles of \((p, q)\) and \((p', q')\) swapped.
\end{proof}

If we restrict to coprime approximations, we don't even need strict inequalities:
\begin{prop}
\label{compare-best-approx-coprime}
Let \((p, q), (p', q') \in \Baset{\alpha}{\K}\) where \(p, q\) and \(p', q'\) respectively are coprime. Then
\begin{equation*}
\deg{q} \leq \deg{q'} \iff \io{p - \alpha \, q} \leq \io{p' - \alpha \, q'}.
\end{equation*}
\end{prop}
\begin{proof}
This is also covered by Proposition \ref{compare-best-approx-strict}, unless \(p/q = p'/q'\). But in this case, the best-approximations differ only by a constant factor, so both inequalities actually become equalities.
\end{proof}
\section{A linear system for computing convergents}
\label{sec:org542977d}
This thesis contains three different proofs for the existence of convergents of arbitrary approximation quality. There is a geometric argument to be explained in Chapter \ref{sec:org99a8e17}. The most elegant approach uses the continued fraction expansion, and yields a complete classification of convergents and best-approximations at the same time; it is one of the main goals of Chapter \ref{sec:org172eb73}. But here, we give an elementary proof which uses only some linear algebra and other results from this chapter.

We describe a linear system which allows to compute the convergents, alluded to already in Remark \ref{rem-convergents-existence}. This already demonstrates the existence of convergents. We will also use these results in Chapter \ref{sec:org56e7cb7} to produce estimates for the projective height of the convergents.

See also \cite{platonov-2014-number-theoretic-properties}, where a version of this linear system with additional conditions/rows is used to determine the existence of Pell solutions.

\medskip

From Proposition \ref{convergent-q-determines-p} we know \(p = \gauss{\alpha \, q}\) which gives a linear condition on the coefficients of \(p\). Moreover, from the Cauchy product formula, it is clear that every coefficient of \(\alpha \, q\) is a linear expression in the coefficients of \(q\). And \eqref{convergent-condition} requires just finitely many coefficients of \(p - \alpha \, q\) to vanish, so this produces a linear condition on the coefficients of \(q\) as well.

We make this more precise now, and start by fixing notation:

Write \(\alpha = \lseries{N}{A}{j}\) (with \(A_N \neq 0\), so \(\io{\alpha} = -N\)), and \(q = \poly{n}{Q}, \; p = \poly{n+N}{P}\). The \(A_n\) are given, and we are solving for \(P_{n+N}, \dots, P_0, Q_n, \dots, Q_0\), a total of \(N + 2n + 2\) unknowns. For simplicity, we assume \(N \geq 0\), but the argument works for negative \(N\) as well. We get
\begin{equation*}
  \begin{array}{lllllll}
    \alpha \, q -p = &  X^{n+N}   & (-P_{n+N}     &+ A_N \, Q_n) \\
                     &+ X^{n+N-1} & (-P_{n+N - 1} &+ A_{N-1} \, Q_n  &+A_N \, Q_{n-1}) \\
                     &            & \quad \vdots \\
                     &+ X^{n}     & (-P_n         &+ A_{0} \, Q_{n}  & \dots &+ A_n \, Q_0) \\
                     &+ X^0       & (-P_0         &+ A_{-n} \, Q_n   & \dots &+ A_{0} \, Q_0) \\
                     &            & \quad \vdots \\
                     &+ X^{-n}    & (             &+ A_{-2n} \, Q_n & \dots & + A_{-n} \, Q_0) \\
                     &+ \dots
  \end{array}
\end{equation*}
and the condition \(\io{\alpha \, q - p} > \deg q = n\) means that at the very least the coefficients of \(X^{n+N}, \dots, X^{-n}\) vanish. We count a total of \(N + 2n + 1\) conditions linear in the \(P_i\) and \(Q_i\).

So the matrix describing the linear system has \(N + 2n + 2\) columns and \(N + 2n + 1\) rows; the right part (and also the left) on its own has the shape of a Toeplitz matrix:\footnote{Or the shape of a Hankel matrix if we reverse the ordering of the columns.}
\begin{equation}
\label{convergents-condition-matrix}
\pqmatrix_n = \left( \begin{array}{ccccc|ccc}
-1 &       &        &        & 0  & A_N      &        & 0\\
   &\ddots &        &        &    & A_{N-1}  & \ddots \\
   &       & \ddots &        &    & \vdots   & \ddots & A_{N} \\
   &       &        & \ddots &    & \vdots   & \ddots & \vdots \\
0  &       &        &        & -1 & A_{-n}   & \dots  & A_{0} \\ \hline
   &       &        &        &    & A_{-n-1} & \dots  & A_{-1} \\
   &       & 0      &        &    & \vdots   & \ddots & \vdots \\
   &       &        &        &    & A_{-2n}  & \dots  & A_{-n}
\end{array} \right)
\end{equation}
\pagebreak
\begin{prop}
Every non-zero element of \(\ker \pqmatrix_n\) yields a convergent \((p, q)\). As always \(\ker \pqmatrix_n \neq 0\), this implies that for any \(\alpha \not \in \K(X)\) there exist convergents with arbitrarily high \(\io{p - \alpha \, q}\).
\end{prop}
\begin{proof}
From the discussion above, it is evident that an element of the kernel gives polynomials \((p, q)\) which are a convergent of \(\alpha\) as soon as \(q \neq 0\). But if an element of \(\ker \pqmatrix_n\) has all \(Q_i = 0\), then clearly it follows that also all \(P_i = 0\). So we only need to avoid the zero element. And elementary linear algebra tells as that \(\ker \pqmatrix_n \neq 0\) because there are more columns than rows.

If \(\alpha \in \K(X)\), then of course at some point \(\io{p - \alpha \, q} = \infty\), so the approximation quality can no longer be improved.
\end{proof}

Note that for a single \(\pqmatrix_n\), we do not get different convergents:
\begin{prop}
If \((p, q)\) and \((p', q')\) correspond to non-zero kernel elements, then \(p/q = p'/q'\).
\end{prop}
\begin{proof}
Let \((p_i, q_i)\) for \(i=1,\dots,r\) correspond to a basis of \(\ker \pqmatrix_n\). Then for any \((p, q)\) corresponding to a solution, we get
\begin{equation*}
(p, q) = \sum_{i=1}^r \eta_i \cdot (p_i, q_i) \text{ where } \eta_i \in \K
\end{equation*}
and hence
\begin{equation*}
\io{p - \alpha \, q} = \io{\sum_{i=1}^r \eta_i \, (p_i - \alpha \, q_i)} \geq \min_{i=1,\dots,r} \left(\io{p_i - \alpha \, q_i}\right)
\end{equation*}
so there exists \((p, q)\) in the kernel with \(\io{p - \alpha \, q}\) minimal. Write \(\io{p - \alpha \, q} = \xi + \deg q > n\). By Proposition \ref{best-approx-common-factor-degree} also \(X^{\xi - 1} \cdot (p, q)\) is a best-approximation.\footnote{Here we profit already from allowing common factors for best-approximations.} And by minimality of \(\io{p - \alpha \, q}\), we have for every \((p', q')\) in the kernel 
\begin{equation*}
\io{p' - \alpha \, q'} \geq \io{p - \alpha \, q} \geq \io{X^{\xi-1} \, (p - \alpha \, q)}
\end{equation*}
and moreover \(\deg q' \leq n \leq \degb{X^{\xi -1} \, q}\) which implies \(p'/q' = p/q\).
\end{proof}

We can also compute the dimension of the kernel (i.e. the rank of \(\pqmatrix_n\)):
\begin{prop}
\label{convergent-linear-matrix-full-rank}
There exists \((p, q)\) in the kernel with \(p\) and \(q\) coprime.

If \(\io{p - \alpha\, q} = \xi + \deg q\), then
\begin{equation*}
\dim \ker \pqmatrix_n = \min(1 + \floor{(\xi-1)/2}_\Z, 1 + n - \deg q, \xi + \deg q - n)
\end{equation*}
where \(\floor{\cdot}_\Z\) denotes the next lowest integer. So if \(\xi \leq 2\) or \(n = \deg q\), the matrix \(\pqmatrix_n\) has full rank.
\end{prop}
\begin{proof}
Removing a common factor decreases \(\deg q\) and increases \(\io{p - \alpha \, q}\), so the existence of any solutions implies the existence of a coprime solution. Of course, by the previous Proposition, we can produce all other solutions by adding back a common factor \(r\), with has to satisfy \(\deg r \leq n - \deg q\), \(\deg r < \xi/2\), and also
\begin{equation*}
\io{r} + \io{p - \alpha \, q} = \io{r} + \xi + \deg q > n
\end{equation*}
which is equivalent to \(\deg r < \xi + \deg q -n\).
\end{proof}

These results hold for any \(\alpha \in \laurinx \K\), even if \(\alpha \in \K(X)\).

With Cramer's rule we can compute an element of the kernel:
\begin{rem}
\label{rem-cramers-rule}
Denote by \(\det \pqmatrix_n(i)\) the \(i\)th minor obtained by striking the \(i\)th column. Then
\begin{multline}
\label{toeplitz-matrix-convergents}
\left(P_{n+N}, \dots, P_0, Q_n, \dots, Q_0\right) = \\
\left(\det \pqmatrix_n(1), - \det \pqmatrix_n(2), \det \pqmatrix_n(3), \dots \right. \\  \dots, (-1)^{N+2n} \det \pqmatrix_n(N+2n+1), \\ \left. (-1)^{N+2n+1} \det \pqmatrix_n(N+2n+2) \right)
\end{multline}
is an element of the kernel. If \(\pqmatrix_n\) has full rank, then it is clearly non-zero.
\end{rem}

These formulas present an alternative to computing convergents via the continued fraction, and we will later show that the convergents obtained in this way are actually optimally normalised (see Proposition \ref{prop-convergents-hankel-determinants-normalised}).

\chapter{A (hyper)elliptic curve}
\label{sec:org99a8e17}
In this chapter, we describe the (hyper)elliptic curve corresponding to a given polynomial Pell equation. We additionally assume that \(D\) is square-free, to avoid complications and so that we may work with the Jacobian of the curve.\footnote{If \(D\) is not square-free, we have to use generalised Jacobians instead. See \cite{zannier-2016-hyper-contin-fract} on how generalised Jacobians relate to the Pell equation and continued fractions, and \cite{serre-1988-algebraic-groups-class} for an introduction to generalised Jacobians.}

We also explain how the convergents give rise to principal divisors of particular shape (Lemma \ref{convergent-divisor-lemma}), and this gives rise to the torsion condition for \(D\) being Pellian (Theorem \ref{thm-pellian-iff-torsion}).

Most of the results of this chapter have long been known, probably already to Abel \cite{abel-1826-ueber-integ-differ} and Chebyshev \cite{chebyshev-1857-sur-integration}, albeit not in our modern mathematical language. More recent publications are \cite{adams-razar-1980-multiples-points-on} for elliptic curves, or \cite{berry-1990-periodicity-continued-fractions} for arbitrary genus.

As in the previous chapters, we assume that \(\K\) is a field of characteristic not \(2\).

\section{Defining the (hyper)elliptic curve}
\label{sec:orgf4393ef}
Let \(D \in \K[X] \setminus \K\) \emph{square-free} with even degree \(2(g+1)\) and \(\LC(D)\) a square in \(\K\). Then
\begin{equation*}
\CCa : Y^2 = D(X)
\end{equation*}
defines an affine (plane) curve over \(\K\) of genus \(g\).

\begin{prop}
\label{affine-curve-is-smooth-and-normal}
The curve \(\CCa\) is smooth and normal in \(\A^2_\K\).
\end{prop}
\begin{proof}
The curve is defined by the equation
\begin{equation*}
F = Y^2  - D(X).
\end{equation*}
Applying the Jacobian criterion we calculate
\begin{equation*}
\pdiff{X}{F} = - \pdif_{X} D(X) = D'(X) \qquad
\pdiff{Y}{F} = 2 \, Y
\end{equation*}
which are never simultaneously \(0\) because \(D\) square-free implies that \(D\) and \(D'\) are coprime.

For normality, we need to show that if \(p + Y \, q \in \Fr(\K[X,Y]/\spann{Y^2 - D(X)}) = \K(X)[Y]/\spann{Y^2 - D(X)}\) is integral, it is already contained in \(\K[X,Y]/\spann{Y^2 - D(X)}\), i.e. \(p, q \in \K[X]\) are polynomials. Recall that the integral closure is a subring of the fraction field, and \(p + Y \, q\) integral implies that the conjugate \(p - Y \, q\) is integral as well. It follows that \(2 \, p\) and \(p^2 - D \, q^2\) are integral. As we assume \(\Char \K \neq 2\), this implies \(p\) and also \(D \, q^2\) are integral over \(\K[X,Y]/\spann{Y^2 - D(X)}\), so in particular over the subring \(\K[X]\). As \(D\) is square-free, it follows that \(p, q \in \K[X]\) as desired, and \(\CCa\) is normal.
\end{proof}

\begin{rem}
If \(\deg D = 2\), then \(\closure{\CCa} \subset \P^2_\K\) remains smooth at infinity, so it is isomorphic to \(\P^1\) (see Proposition 7.4.1 in \cite{liu-2002-algebraic-geometry-arithmetic}).

But if \(\deg D > 2\), then \(\closure{\CCa} \subset \P^2_\K\) has a singularity at infinity (easily verified with the Jacobian criterion).
\end{rem}

We build a smooth projective model for \(\CCa\), as in Lemma III.1.7 of \cite{miranda-1995-algebraic-curves-riemann}:

Define the curve 
\begin{equation*}
\CCinf : V^2 = D^\flat(U) = U^{2(g+1)} D(1/U)
\end{equation*}
where \(D^\flat(U)\) is a polynomial of degree at most \(2(g+1)\) -- its coefficients are those of \(D\) in reverse order. Note that \(D^\flat(0) \neq 0\) because \(\deg D = 2(g+1)\), and by Proposition \ref{affine-curve-is-smooth-and-normal} the curve \(\CCinf\) is smooth in \(\A^2_\K\). 

The relations \(X \, U = 1\) and \(U^{g+1} \, Y = V\) (respectively \(X^{g+1} \, V = Y\)) describe a birational map between \(\CCa\) and \(\CCinf\) which is an isomorphism outside of \(U = 0\) and \(X = 0\). So we may glue \(\CCa\) and \(\CCinf\) together to obtain a curve \(\CC\). This simply adds two points \(O_\pm\) with \(U = 0\) to \(\CCa\), the points at infinity.

\begin{prop}
\label{projective-curve-is-smooth-and-normal}
The curve \(\CC\), glued together from \(\CCa\) and \(\CCinf\) is a normal smooth projective curve over \(\K\).
\end{prop}
\begin{proof}
Normality and smoothness of \(\CC\) are local conditions, hence they follow from Proposition \ref{affine-curve-is-smooth-and-normal} applied to \(\CCa\) and \(\CCinf\).

We get a finite morphism \(\CC \to \closure{\CCa} \subset \P^2_\K\), hence \(\CC\) is proper over \(\K\). As \(\CC\) is an algebraic variety, this implies by Remark 3.3.33 (1) in \cite{liu-2002-algebraic-geometry-arithmetic} that it is projective.
\end{proof}

There is an involution \(\sigma\) defined by \(X \mapsto X, \; Y \mapsto -Y\), or \(U \mapsto Y, \; V \mapsto -V\). By abuse of notation, we also consider it as an automorphism of the function field \(\K(X,Y)\). If we quotient \(\CC\) by the group \(\{1, \sigma\}\), we find that \(\CC\) is (hyper)elliptic (we use Definition 7.4.7 from \cite{liu-2002-algebraic-geometry-arithmetic} which is essentially the content of the following proposition):

\begin{prop}
\label{curve-is-hyper-elliptic}
There is finite morphism \(\pi : \CC \to \P^1\) of degree \(2\) defined by \((x,y) \mapsto (x:1)\) on \(\CCa\) and \(\pi(O_\pm) = (1:0)\). For \(g = 1\), the curve \(\CC\) is elliptic, and for \(g \geq 2\) it is hyperelliptic.
\end{prop}
\begin{proof}
The map \(\pi\) is defined on \(\CCa\) via \((x, y) \mapsto (x:1)\), and on \(\CCinf\) via \((u, v) \mapsto (1:u)\). Clearly the definitions are compatible on the intersection (because there we have \(x \, u = 1\)). It is also clear that \(\pi\) is a finite morphism of degree \(2\) which means that \(\CC\) is elliptic for \(g = 1\) and hyperelliptic for \(g \geq 2\).
\end{proof}

\section{Divisors and the Jacobian variety}
\label{sec:org40007f2}
We recall some basic notions about divisors and the Jacobian variety now. For more details, consult your favourite algebraic geometry book, for instance \cite{hartshorne-1977-algebraic-geometry}, \cite{goertz-wedhorn-2010-algebraic-geometry-i} or \cite{liu-2002-algebraic-geometry-arithmetic}. For the rest of the chapter, we work over the algebraic closure \(\closure \K\) to avoid complications.
\subsection{Divisors}
\label{sec:orgeef8812}
For any \(P \in \CC(\closure{\K})\), there is a discrete valuation
\begin{equation*}
\ord_P : \units{\closure{\K}(X,Y)} \to \Z,
\end{equation*}
the zero-order of \(P\) of a function on \(\CC\). In fact, all non-trivial discrete \(\closure\K\)-valuations (up to equivalence) on \(\closure\K(X,Y)\) arise in this way.

By the \emph{group of divisors} \(\DIV(\CC)\) we understand the free abelian group generated by all points of \(\CC(\closure\K)\) (we mark divisors in \textbf{bold}). For every divisor
\begin{equation*}
\di{D} = \divsum n_P \, \pd{P}, \text{ where } n_P \in \Z
\end{equation*}
we define the \emph{degree}
\begin{equation*}
\deg \di{D} = \divsum n_P.
\end{equation*}
A divisor is called \emph{effective} if \(n_P \geq 0\) for all \(P\).

For every element \(f \in \units{\closure\K(X,Y)}\), only finitely many \(\ord_P f\) are non-zero, so we can define the \emph{divisor of \(f\)} as
\begin{equation*}
\Div f = \divsum (\ord_P f) \, \pd{P}.
\end{equation*}
The divisors arising in this way are called \emph{principal divisors}, and they all have degree \(0\). So there is group homomorphism \(\Div : \units{\closure\K(X,Y)} \to \DIV^0(\CC)\) where \(\DIV^0(\CC)\) denotes the divisors of degree \(0\).

Recall that the Jacobian \(\Jac\) of \(\CC\) is an abelian variety of dimension \(g\). If \(g = 1\) (i.e. \(\deg D = 4\)), the curve \(\CC\) is an elliptic curve (Corollary 7.4.5 in \cite{liu-2002-algebraic-geometry-arithmetic}), and it is isomorphic to its Jacobian.

The \(\closure\K\)-rational points of the \emph{Jacobian} can be seen as the cokernel of the divisor map, more precisely the quotient
\begin{equation*}
\Jac = \Jac(\CC) = \DIV^0(\CC) / \im \Div,
\end{equation*}
with the projection \(\DIV^0(\CC) \to \Jac\). By abuse of language, we call both the algebraic variety and its set of \(\closure\K\)-rational points ``Jacobian''.

We write a divisor class in the Jacobian as
\begin{equation*}
\jdi{D} = \divsum n_P \, \j{P}.
\end{equation*}

\subsection{Order functions}
\label{sec:org85f4b43}
If we restrict \(\ord_{O_\pm}\) to \(\closure\K(X)\), it becomes exactly \(\ord_\infty = \ios\) from Section \ref{sec:org7391b36}. As mentioned before, there are precisely two embeddings of \(\closure\K(X,Y)\) into the completion \(\laurinx{\closure\K}\). To distinguish them properly, we set for \(p, q \in \closure\K(X)\)
\begin{equation*}
\ord_{O_+}(p + Y \, q) = \io{p + \sqrt{D} \, q}, \quad \ord_{O_-}(p + Y \, q) = \io{p - \sqrt{D} \, q}
\end{equation*}
for a fixed choice of \(\sqrt{D}\) (see Definition \ref{choose-sqrt-d}). Apart from this, the roles of \(O_+\) and \(O_-\) are essentially interchangeable (by the involution \(\sigma\)).

In a similar way, one may compute order functions for a finite point \(P = (x, y) \in \CCa\) by choosing an uniformiser. Sending \(X\) to \(T + x\) gives a homomorphism \(\closure\K(X) \to \laurent{\closure\K}{T}\), and if \(y \neq 0\), one may compute \(\sqrt{D(T+x)}\) in \(\laurent{\closure\K}{T}\) with the constant coefficient \(y\) determining the choice of square root. Sending \(Y\) to \(\sqrt{D(T+x)}\) then establishes a homomorphism \(\HEF \to \laurent{\closure\K}{T}\) with \(\ord_P\) corresponding to \(\ord_{T=0}\).

If \(y = 0\), one sends instead \(X\) to \(T^2 + x\), to ensure \(\sqrt{D(T^2+x)} \in \laurent{\closure\K}{T}\). Because the latter has odd \(\ord_{T=0}\), the choice of root does not matter, and one obtains as before the correspondence between \(\ord_P\) and \(\ord_{T=0}\). Note that in this case for \(f \in \closure\K(X)\) the zero-order \(\ord_P(f)\) is always \emph{even}.

\subsection{Embedding the curve in the Jacobian}
\label{sec:orgdaae3ac}
Choosing the base point \(O_+\) (a natural choice here, but any other point on \(\CC\) would do as well), define the map \(j :\CC \to \Jac\) via \(P \mapsto \j{P} - \j{O_+}\) which is an embedding for \(g \geq 1\) (see Theorem A8.1.1 in \cite{hindry-silverman-2000-diophantine-geometry}). Actually, for \(g = 1\), when \(\CC\) is an elliptic curve, it is an isomorphism of curves, determined uniquely by the choice of the base point.

Of course, we can extend \(j : \DIV(\CC) \to \Jac\) as a homomorphism of groups (using that \(\DIV(\CC)\) is a free group on \(\CC\)).

For each \(r \geq 0\), we may also define a subvariety of \(\Jac\)
\begin{equation*}
W_r = j(\CC) + \dots + j(\CC) \quad (r \text{ copies})
\end{equation*}
remarking \(W_g = \Jac\) (see again Theorem A8.1.1 in \cite{hindry-silverman-2000-diophantine-geometry}), while the Theta divisor \(\Theta = W_{g-1}\) forms a proper subvariety which depends on the embedding \(j\). We will use this divisor with the Weil height machine later, and likewise the canonical divisor.

\begin{prop}
The canonical divisor \(\canondiv\) on \(\CC\) is represented by
\begin{equation*}
\Div\left( \frac{\dif X}{Y} \right) = (g-1) \, ( \pd{O_+} + \pd{O_-} ). 
\end{equation*}
\end{prop}
\begin{proof}
From Riemann-Roch one deduces easily that \(\deg \canondiv = 2(g-1)\). It is also clear that we obtain the canonical divisor class by computing the divisor of any differential on \(\CC\).

Now outside of infinity, the sheaf of differentials is clearly generated by \(\dif X\) and \(\dif Y\) which enjoy the relation
\begin{equation*}
2 \, Y \, \dif Y = D'(X) \, \dif X
\end{equation*}
obtained by differentiating the equation of the curve. This tells us that outside of \(D(X) = 0\), the sheaf of differentials is generated by \(\dif X\), while outside of \(D'(X) = 0\) it is generated by \(\dif Y\). Moreover we see that \(\dif X\) vanishes only on \(Y = 0\) (i.e. \(D(X) = 0\)), while \(\dif Y\) vanishes only on \(D'(X) = 0\).

It follows that \(\frac{\dif X}{Y}\) has poles and zeroes only at infinity. As the divisor of this differential is invariant under the involution \(\sigma\) (which changes the differential only by a factor \(-1\)), and has to have degree \(2(g-1)\), we obtain the above formula.
\end{proof}

\section{Divisors of convergents}
\label{sec:org5685823}
Given a rational approximation \((p, q)\), it is very natural to build the function \(p - Y\, q\) and study its divisor. For the convergents, we will see that this divisor describes how the multiples of the divisor at infinity
\begin{equation}
\label{divisor-at-infinity}
\OO = \pd{O_+} - \pd{O_-} \in \DIV^0(\CC)
\end{equation}
are represented as sums of \(g\) points, i.e. as elements of \(W_g = \Jac\). Note that \(\OO\) is actually a \(\K\)-rational divisor, so \(\j{\OO}\) is a \(\K\)-rational point of \(\Jac\).

\begin{prop}
\label{order-finpt-polynomial}
Let \(p, q \in \closure\K(X)\) and \(\phi_\pm = p \pm Y \, q \neq 0\). \TFAE
\begin{enumerate}
\item \(p, q \in \closure\K[X]\)
\item For all \(P \neq O_\pm\) holds \(\ord_P \phi_+ \geq 0\)
\item For all \(P \neq O_\pm\) holds \(\ord_P \phi_- \geq 0\)
\end{enumerate}
\end{prop}
\begin{proof}
\begin{enumerate}
\item and 3. are clearly equivalent because \(\ord_P \phi_+ = \ord_{\sigma(P)} \phi_-\) for all \(P\).
\end{enumerate}

Together, they imply 1.:
\begin{align*}
\ord_P(p) = \ord_P(2\,p) &= \ord_P(\phi_+ + \phi_-) \geq \min( \ord_P(\phi_+), \ord_P(\phi_-)) \geq 0 \\
\ord_P(Y\,q) = \ord_P(2\,Y\,q) &= \ord_P(\phi_+ - \phi_-) \geq \min( \ord_P(\phi_+), \ord_P(\phi_-)) \geq 0
\end{align*}
so clearly \(p\) has no poles outside infinity, hence it is a polynomial. If \(\ord_P(Y) \neq 0\), then \(\ord_P(Y) = 1\) because \(D\) is square-free. But at the same time \(\ord_P(q)\) must be even (see Section \ref{sec:org85f4b43}). This shows that \(\ord_P(q) \geq 0\), and that \(q\) has no poles outside infinity which means it is a polynomial.

Conversely 1. implies 2.: if \(p, q \in \closure\K[X]\), then \(\ord_P(p) \geq 0\), \(\ord_P(q) \geq 0\) and of course \(\ord_P(Y) \geq 0\) for all \(P \neq O_\pm\). Hence
\begin{equation*}
\ord_P(\phi_\pm) = \ord_P(p \pm Y \, q) \geq \min(\ord_P(p), \ord_P(Y) + \ord_P(q)) \geq 0
\end{equation*}
as desired.
\end{proof}

\begin{lemma}
\label{convergent-divisor-lemma}
Let \(p, q \in \closure\K(X)\), and \(\phi_\pm = p \pm Y \, q \neq 0\). Set \(m = \deg p\). Then \((p, q) \in \Coset{\sqrt{D}}{\closure\K}\) (it is a convergent of \(\sqrt{D}\)) \IFF \(m > 0\) and there exists \(0 \leq r \leq \min(g, m)\) and \(P_1, \dots, P_r \in \CCa\) such that
\begin{equation}
\label{convergent-divisor-equation}
\Div \phi_- = -m \, \pd{O_-} + (m-r) \, \pd{O_+} + \pd{P_1} + \dots + \pd{P_r}.
\end{equation}
We call \(\Div \phi_-\) a \emph{convergent divisor}.
\end{lemma}
\begin{proof}
By Proposition \ref{order-finpt-polynomial} we can clearly restrict to the case \(p, q \in \closure\K[X]\) as the divisor in \eqref{convergent-divisor-equation} allows only poles at infinity, and convergents are always made of polynomials. The rest of the proof boils down to distinguishing the points at infinity and calculating \(r\).

Obviously \(\ord_{O_+} \phi_- = \io{p - \sqrt{D} \, q} \geq 0\) holds for both conditions and implies
\begin{multline*}
\ord_{O_-} \phi_- = \ord_{O_+} \phi_+ = \io{p + \sqrt{D} \, q} \\= \min(\io{p}, \io{p- \sqrt{D}\, q}) = \io{p} = -m.
\end{multline*}
Similarly, \(m = \deg q + g + 1\) (see also the proof of Proposition \ref{weak-pell-solutions-are-convergents}).

Now \(P_1, \dots, P_r\) are the finite zeroes of \(\phi_-\), accounted for with multiplicities. Of course \(\Div \phi_-\) has degree \(0\), hence
\begin{equation*}
\io{p - \sqrt{D} \, q} = \ord_{O_+} \phi_- = (m-r) = \deg q + g + 1 - r.
\end{equation*}
This is \(> \deg q\) (i.e. \((p, q) \in \Coset{\sqrt{D}}{\closure\K}\)) \IFF \(r \leq g\), so we have the desired equivalence.
\end{proof}

We will give a slight generalisation (extending to other elements of the function field) later in Section \ref{sec:org4e4485e}, to illustrate the connection with the continued fraction.

\begin{rem}
\label{rem-convergent-lemma-jacobian}
In the Jacobian, we can write this divisor relation as
\begin{equation*}
m \cdot j(O_-) = j(P_1) + \dots + j(P_r).
\end{equation*}
\end{rem}

\begin{rem}
\label{convergent-omega-degree}
With the notation from Proposition \ref{weak-pell-solutions-are-convergents}, we get \(\deg \Omega = r\) because
\begin{equation*}
\io{\Omega} = \ord_{O_+}(\phi_+ \cdot \phi_-) = \ord_{O_+}(\phi_+) + \ord_{O_+}(\phi_-) = -m + (m-r) = -r.
\end{equation*}
In the same proposition, the condition to obtain a convergent (up to sign of \(q\)) was \(r = \deg \Omega < \frac{1}{2} \deg D = g+1\) which matches the above lemma.
\end{rem}

\begin{prop}
For every \(n \in \N\) there exists \(\phi_n \in \mino{\closure\K(X,Y)}\) such that
\begin{equation*}
\Div \phi_n = -m \, \pd{O_-} + (m-r) \, \pd{O_+} + \pd{P_1} + \dots + \pd{P_r}
\end{equation*}
with \(m \geq n\), \(r \leq \min(g,m)\) and \(P_1, \dots, P_r \in \CCa\).
\end{prop}
\begin{proof}
For \(n \in \N\) define the divisor
\begin{equation*}
\di{D}_n = (n+g) \, \pd{O_-} - n \, \pd{O_+}.
\end{equation*}
which has degree \(\deg \di{D}_n = g\). Then the Riemann-Roch theorem (see Theorem IV.1.3 in \cite{hartshorne-1977-algebraic-geometry}) implies
\begin{equation*}
\dim \{ \phi \in \mino{\closure\K(X,Y)} \mid \Div \phi + \di{D}_n \geq 0 \} \geq \deg \di{D}_n - g + 1 = 1
\end{equation*}
so there exists \(\phi_n\) with \(\Div \phi_n \geq -\di{D}_n\). More precisely, we get
\begin{equation*}
\Div \phi_n = -(n+g) \, \pd{O_-} + n \pd{O_+} + \pd{P_1} + \dots + \pd{P_g}
\end{equation*}
where \(P_i \in \CC\) (possibly \(O_\pm\)). We can write this as
\begin{equation*}
\Div \phi_n = -m \pd{O_-} + (m-r) \pd{O_+} + \pd{P_{i_1}} + \dots + \pd{P_{i_r}},
\end{equation*}
cancelling out any \(O_-\) among the \(P_i\) (hence \(m \geq n + g - g = n\)) and absorbing any \(O_+\) from the \(P_i\) (hence \(m -r \geq n + g \geq 0\)). And of course \(r \leq g\).
\end{proof}

Via the above lemma, we now have another proof for the existence of convergents:
\begin{cor}
\(\sqrt{D}\) has convergents \((p, q)\) of arbitrarily high \(\deg p\) (or \(\deg q\)).
\end{cor}

\begin{thm}
\label{thm-pellian-iff-torsion}
The Pell equation \eqref{pellu} has a non-trivial solution \IFF \(\j{\OO}\) is a torsion point in the Jacobian \(\Jac\) of \(\CC\).
\end{thm}
\begin{proof}
From Proposition \ref{weak-pell-solutions-are-convergents} we know that the Pell solutions (up to conjugation) form a subset of the convergents. By Remark \ref{convergent-omega-degree} it is precisely the non-trivial Pell solutions for which we have \(r = 0\) in Lemma \ref{convergent-divisor-lemma}. 

By Remark \ref{rem-convergent-lemma-jacobian}, this implies \(m \, \j{\OO} = m \, j(O_-) = 0\) with \(m > 0\). In other words, \(\j{\OO}\) is a torsion point in the Jacobian \(\Jac\).

Conversely, if \(\j{\OO}\) is torsion, then there exists some function \(\phi\) with divisor \(\Div \phi = m \, \left( \pd{O_+} - \pd{O_-}\right)\) and \(m > 0\). By Lemma \ref{convergent-divisor-lemma}, we have \(\phi = p - Y \, q\) where \((p, q)\) is a convergent (actually a Pell solution because \(r = 0\)).
\end{proof}
\begin{rem}
Recall that the genus \(g\) corresponds to the dimension of the Jacobian.
For \(g = 0\), the Jacobian is the trivial group, hence \(\j{\OO}\) is trivially torsion. Hence \(D\) is always Pellian as observed before in Corollary \ref{deg-2-always-pellian}.
\end{rem}

\begin{rem}
\label{finite-field-torsion-bound}
If the base field \(\K\) is finite, i.e. \(\K = \F_q\) with \(q\) some prime power, the \(\K\)-rational points of the Jacobian form a finite group.
The Hasse-Weil interval (conjectured by E. Artin in his thesis, then proved by Hasse for elliptic curves \cite{hasse-1936-zur-theorie-1,hasse-1936-zur-theorie-2,hasse-1936-zur-theorie-3}, and generalised by Weil to higher genus curves in \cite{weil-1949-numbers-solutions-equations}) then provides the following bounds for the number of elements of the Jacobian:
\begin{equation*}
\ord( \j{\OO} ) \in [(\sqrt{q} - 1)^{2g}, (\sqrt{q} + 1)^{2g}].
\end{equation*}
\end{rem}
Note that \(\Jac(\F_q)\) can be cyclic (for elliptic curves, see for example \cite{gupta-murty-1990-cyclicity-generation-points}), so we cannot hope to improve this bound for the point \(\j{\OO}\).

\begin{rem}
If \(\K\) is not finite, then as mentioned in the introduction, this torsion condition allows to demonstrate the scarcity of Pellian polynomials. The polynomials of degree \(2d\), after some normalisation, form an affine variety of dimension \(2d-2\). The Pellian polynomials are then contained in a denumerable union of subvarieties of dimension at most \(d-1\), corresponding to the possible torsion orders. See Section 12.2.2 in \cite{zannier-2014-pell-survey} for details.
\end{rem}

\chapter{Continued fractions}
\label{sec:org172eb73}
In this chapter, we develop the theory of polynomial continued fractions, to build a solid foundation for the specialization questions that form the main results of this thesis. Beginning with formal continued fractions, moving on to convergence questions in \(\laurinx \K\) and a classification of the best-approximations, we conclude with a discussion of periodic continued fractions and reducedness which is relevant mostly for hyperelliptic continued fractions.

The first sections reiterate well-known facts about continued fractions in modern language. Already Abel \cite{abel-1826-ueber-integ-differ} and Chebyshev \cite{chebyshev-1857-sur-integration} worked with this type of polynomial continued fractions which they adapted from the numerical continued fraction expansion for square roots. Indeed there are not many differences with the theory of continued fractions for real numbers.

Most results in this chapter may already be found the literature, albeit presented differently. The formal definitions of continued fractions can be found in classical books on continued fractions, for example \cite{perron-1954-lehre-von-den}, \cite{perron-1957-lehre-von-den}, \cite{khintchine-1956-kettenbruche} and others. For polynomial continued fractions, see \cite{abel-1826-ueber-integ-differ}, \cite{berry-1990-periodicity-continued-fractions}, \cite{schmidt-2000-continued-fractions-diophantine} or the survey paper \cite{poorten-tran-2000-quasi-elliptic-integrals}.

As before \(\K\) is a field of characteristic not \(2\).

\section{Finite continued fractions}
\label{sec:orgb81a872}
For our formal continued fractions, we begin by using a double index notation, as this should make some calculations much clearer and precise. We will drop the first index once we no longer need it.
\enlargethispage{1cm}

\begin{defi}
\label{def-finite-continued-fraction}
Let \(m, n \in \Z, n \geq 0\). The expression
\begin{equation*}
\alpha_{m,n} = [a_m, a_{m+1}, \dots, a_{m+n}] = a_{m} + \dfrac{1}{a_{m+1} + \dfrac{1}{ \ddots + \dfrac{1}{a_{m+n}}}}
\end{equation*}
where we consider the \(a_i\) as free variables is called a \emph{finite continued fraction}. We define it recursively by
\begin{equation*}
\alpha_{m,0} = a_m \quad \text{ and } \quad \alpha_{m,n} = a_m + \frac{1}{\alpha_{m+1,n-1}} \text{ for } n \geq 1
\end{equation*}
respectively
\begin{equation*}
[a_m] = a_m \quad \text{ and } \quad [a_m, a_{m+1}, \dots, a_{m+n}] = a_m + \frac{1}{[a_{m+1}, \dots, a_{m+n}]} \text{ for } n \geq 1
\end{equation*}
in the square bracket notation.
\end{defi}

\begin{rem}
\label{cf-concatenation-rule}
By induction one obtains also for \(l \geq 1\)
\begin{equation*}
\alpha_{m,n} = [a_m, a_{m+1}, \dots, a_{m+l-1}, \alpha_{m+l,n-l}],
\end{equation*}
so the concatenation of \([a_m, a_{m+1}, \dots, a_{m+l-1}]\) and \([a_{m+l}, \dots, a_{m+n}]\) is the same as inserting the second at the end of the first finite continued fraction:
\begin{equation*}
[a_m, a_{m+1}, \dots, a_{m+n}] = [a_m, a_{m+1}, \dots, a_{m+l-1}, [a_{m+l}, \dots, a_{m+n}]].
\end{equation*}
\end{rem}

\section{Continued fractions and matrix products}
\label{sec:org932a99d}
Clearly a continued fraction \(\alpha_{m,n}\) can be seen as an element of \(\P^1(\Zamn)\), where the empty continued fraction corresponds to \([\,] = \frac{1}{0} \in \P^1\). This motivates the following viewpoint:

We can think of a finite continued fractions as a map on \(\P^1\), via
\begin{equation*}
x \in \P^1 \mapsto [a_m, \dots, a_{m+n}, x] \in \P^1.
\end{equation*}

We can relate such a map to the natural (left) action of \(\GL{2}{\Zai}\) on \(\P^1\) via Moebius transformations:
\begin{equation*}
x \mapsto \frac{a \, x + b}{c \, x + d} \corresponds \mfour{a}{b}{c}{d}.
\end{equation*}
Then clearly
\begin{equation*}
x \mapsto [a_m, x] = a_m + \frac{1}{x} \corresponds \mcf{a_m}.
\end{equation*}
As concatenation is the same as composition, this extends to
\begin{equation*}
x \mapsto [a_m, \dots, a_{m+n}, x] \corresponds \mcf{a_m} \cdots \mcf{a_{m+n}}.
\end{equation*}

By multiplying out these matrices, we can canonically compute the numerator and denominator of the fraction represented by a finite continued fraction.
\begin{prop}
\label{definition-convergents-matrix}
For every \(m, n \in \Z, n \geq -1\), there exist polynomials \(p_{m,n}, q_{m,n} \in \Zamn\) such that
\begin{equation}
\label{convergents-matrix}
\mcf{a_m} \cdots \mcf{a_{m+n}} = \mfour{p_{m,n}}{p_{m,n-1}}{q_{m,n}}{q_{m,n-1}},
\end{equation}
satisfying \(\ifrac{p_{m,n}}{q_{m,n}} = \alpha_{m,n}\).
\end{prop}
\begin{proof}
Take \(x = [\,] = \frac{1}{0}\) the empty continued fraction, then we define
\begin{equation*}
\frac{p_{m,n}}{q_{m,n}} := \mcf{a_m} \cdots \mcf{a_{m+n}} \frac{1}{0} = [a_m, \dots, a_{m+n}, [\,]] = \alpha_{m,n}
\end{equation*}
where clearly \(p_{m,n}, q_{m,n} \in \Zamn\) because the matrix entries are in that ring. Also note that \(\mcf{a_{m+n}} \dfrac{0}{1} = \dfrac{1}{0}\), so
\begin{equation*}
\mcf{a_m} \cdots \mcf{a_{m+n}} \frac{0}{1} = \mcf{a_m} \cdots \mcf{a_{m+n-1}} \frac{1}{0} = \frac{p_{m,n-1}}{q_{m,n-1}}.
\end{equation*}
\end{proof}

\begin{rem}
Transposing the matrix \(\mfour{p_{m,n}}{p_{m,n-1}}{q_{m,n}}{q_{m,n-1}}\) corresponds to reversing the ordering of the variables \(a_m, \dots, a_{m+n}\); and \(p_{m,n-1}\) depends only on \(a_m, \dots, a_{m+n-1}\), so one easily deduces that \(q_{m,n}\) is independent of \(a_m\): it follows \(q_{m,n} \in \Z[a_{m+1}, \dots, a_{m+n}]\).
\end{rem}

By taking the determinants of the matrix product, we get
\begin{cor}
\label{canonical-convergent-coprime}
For fixed \(m\) and \(n\), we have the relation
\begin{equation}
\label{canonical-convergent-determinant}
p_{m,n} \, q_{m,n-1} - q_{m,n} \, p_{m,n-1} = (-1)^{n+1}.
\end{equation}
Consequently, the \(p_{m,n}\) and \(q_{m,n}\) are coprime.
\end{cor}
This holds even if we assign values to the \(a_i\).
The sequences in \(n\) of the \(p_{m,n}\) and \(q_{m,n}\) may also be computed independently:
\begin{cor}
The \(p_{m,n}\) and \(q_{m,n}\) satisfy the recursion relations
\begin{equation}
\label{canonical-convergent-recursion}
\begin{aligned}
p_{m,n} &= a_{m+n} \, p_{m,n-1} + p_{m,n-2} \text{ for } n \geq 0, &p_{m,-1} &= 1, & p_{m,-2} &= 0, \\
q_{m,n} &= a_{m+n} \, q_{m,n-1} + q_{m,n-2} \text{ for } n \geq 1, &q_{m,0}  &= 1, & q_{m,-1} &= 0.
\end{aligned}
\end{equation}
\end{cor}

\section{Infinite continued fractions}
\label{sec:org5c8f70a}
To give sense to infinite continued fraction, we need some topology. In our case, we use \(\K[X]\) with the previously defined (non-archimedean) absolute valuation \(\ios = \ord_\infty\) (see Section \ref{sec:org7391b36}). We assume that all \(a_n \in \K[X]\). Then the \(\alpha_{m,n}\) are contained in \(\K(X)\), and we can hope to find a limit in the completion \(\laurinx \K\).

\begin{defi}
We define the \emph{infinite continued fraction}
\begin{equation*}
\alpha_m = \alpha_{m,\infty} = [a_m, a_{m+1}, \dots] = \limn \alpha_{m,n}
\end{equation*}
if the limit exists.
\end{defi}

From now on, we assume that all \(a_n \in \K[X]\), and search for a sufficient condition for the convergence of  \(\folge{\alpha_{m,n}}{n}\).

\begin{prop}
\label{canonical-convergents-degree-and-lc}
If \(\deg a_n \geq 1\) holds for all \(n \geq m+1\), then 
\begin{align}
\label{}
\label{convergent-deg-growth}
\deg{p_{m,n}} &= \sum_{j=0}^n \deg{a_{m+j}}, & \deg{q_{m,n}} &= \sum_{j=1}^n \deg{a_{m+j}}, \\
\label{convergents-leading-coeff}
\LC(p_{m,n}) &= \prod_{j=0}^n \LC(a_{m+j}), & \LC(q_{m,n}) &= \prod_{j=1}^n \LC(a_{m+j}).
\end{align}
\end{prop}

The proposition is a consequence of the following lemma:
\begin{lemma}
Let \(\folge{a_n}{n}\) a sequence in \(\K[X]\), with \(\deg{a_n} \geq 1\) for all \(n \geq 1\). Define a sequence \(\left(b_n\right)_{n \geq -1}\) via
\begin{equation}
\label{canon-convergent-generalised-recursion-relation}
b_{-1} = 0, \quad b_0 = 1, \quad b_{n} = a_n \, b_{n-1} + b_{n-2} \text{ for } n \geq 1
\end{equation}
Then \(\deg{b_n}\) is strictly increasing and for \(n \geq 0\)
\begin{align*}
\deg{b_n} &= \sum_{j=1}^n \deg{a_j}, & \LC(b_n) &= \prod_{j=1}^n \LC(a_j).
\end{align*}
\end{lemma}
\begin{proof}
We prove the statement by induction on \(n\). For \(n = 0\) we clearly have \(\deg{b_{-1}} < \deg{b_0} = 0\) and \(\LC(b_0) = 1\). For the induction step, note that by hypothesis
\(\deg{b_{n-2}} < \deg{b_{n-1}} < \degb{a_n \, b_{n-1}}\), so \eqref{canon-convergent-generalised-recursion-relation} implies
\begin{equation*}
\deg{b_n} = \degb{a_n \, b_{n-1} + b_{n-2}} = \deg{a_n} + \deg{b_{n-1}} = \deg{a_n} + \sum_{j=1}^{n-1} \deg{a_j} 
\end{equation*}
as desired, and clearly \(\deg{b_{n-1}} < \deg{b_{n}}\). It follows
\begin{equation*}
\LC(b_n) = \LC(a_{n} \, b_{n-1}) = \LC(a_n) \, \prod_{j=1}^{n-1} \LC(a_j).
\end{equation*}
\end{proof}

We can now answer the question about the convergence of infinite continued fractions:
\begin{prop}
\label{cf-cauchy-convergence}
Suppose \(\deg a_n \geq 1\) holds for \(n \geq m+1\). Then \(\folge{\alpha_{m,n}}{n}\) is a Cauchy sequence and converges in \(\laurinx{\K}\). We denote the limit by \(\alpha_m = \alpha_{m,\infty} = [a_m, a_{m+1}, a_{m+2}, \dots]\).
\end{prop}
\begin{proof}
Dividing \eqref{canonical-convergent-determinant} by \(q_{m,n-1} \cdot q_{m,n}\) implies
\begin{equation*}
\alpha_{m,n} - \alpha_{m,n-1} = \frac{p_{m,n}}{q_{m,n}} - \frac{p_{m,n-1}}{q_{m,n-1}} = \frac{(-1)^{n+1}}{q_{m,n-1} \cdot q_{m,n}},
\end{equation*}
hence
\begin{equation}
\label{convergent-fraction-difference}
\io{\alpha_{m,n} - \alpha_{m,n-1}} = \deg q_{m,n} + \deg q_{m,n-1} \geq 2n-1
\end{equation}
by Proposition \ref{canonical-convergents-degree-and-lc}. This means the ``distance'' between \(\alpha_{m,n}\) and \(\alpha_{m,n-1}\) converges to \(0\) as \(n \to \infty\). Because we are working with a non-archimedean valuation, this already implies that \(\folge{\alpha_{m,n}}{n}\) is a Cauchy sequence.
\end{proof}

So a continued fraction (with non-constant coefficients \(a_n \in \K[X]\)) produces an element of \(\laurinx \K\) (actually a sequence of elements of \(\laurinx \K\)). In the next section, we will reverse the process and produce a continued fraction for every element of \(\laurinx \K\), thus establishing a bijection between \(\laurinx \K\) and (a subset of) continued fractions over \(\K[X]\).
\section{Continued fraction process}
\label{sec:org34409a6}
We now define a process which produces a continued fraction for elements of \(\laurinx \K\), using the truncation \(\gauss{\cdot}\) from Definition \ref{define-laurent-truncation}. This is mostly analogous to classical continued fractions over \(\Z\), but slightly nicer because here we have a unique truncation operation, and we avoid ambiguity as for example with \([2] = 2 = 1 + \frac{1}{1} = [1, 1]\) in the integer case.

\begin{defi}
Let \(\alpha \in \laurinx \K\). We define the \emph{complete quotients} of \(\alpha\) as the (possibly finite) sequence
\begin{equation}
\label{cf-complete-quotients}
\alpha_0 = \alpha,  \quad \alpha_{n+1} = \frac{1}{\alpha_n - \gauss{\alpha_n}} \quad \text{ for } n \geq 0 \text{ and } \alpha_n \not\in \K[X].
\end{equation}
One defines also the \emph{partial quotients} \(a_n = \gauss{\alpha_n}\) whenever the corresponding complete quotient is defined. As \(\alpha_n = a_n + \inv{\alpha_{n+1}}\), this clearly gives rise to a (finite or infinite) continued fraction
\begin{equation*}
\CF(\alpha) = [a_0, a_1, \dots].
\end{equation*}
\end{defi}

\begin{rem}
\label{cf-absolute-values}
By definition of \(\gauss{\cdot}\) we have always \(\io{\alpha_n - \gauss{\alpha_n}} > 0\) which implies \(\io{\alpha_{n+1}} < 0\) whenever \(\alpha_{n+1}\) is defined. Then \(\io{a_{n+1}} = \io{\alpha_{n+1}} < 0\) which means \(\deg a_{n+1} \geq 1\). So if \(\CF(\alpha)\) is an infinite continued fraction, it converges by Proposition \ref{cf-cauchy-convergence}.
\end{rem}

The Euclidean algorithm works also in the ring \(\K[X]\), establishing a complete correspondence between finite continued fraction and rational functions.
\begin{prop}
\label{cf-euclidean-algorithm}
The continued fraction \(\CF(\alpha)\) is finite \IFF \(\alpha \in \K(X)\).
\end{prop}
\begin{proof}
If \(\CF(\alpha)\) is finite, it produces an element of \(\K(X)\), and obviously \(\alpha = \CF(\alpha)\).

Conversely, assume \(\alpha \in \K(X)\). Write \(\alpha = \frac{r_0}{r_1}\) with \(r_0, r_1 \in \K[X]\) and \(r_1 \neq 0\). In fact, we can write \(\alpha_n = \frac{r_n}{r_{n+1}}\) whenever defined, with \(r_n, r_{n+1} \in \K[X]\).

By Remark \ref{truncation-of-rational} we write \(r_n = a_n \, r_{n+1} + r_{n+2}\) where \(\deg{r_{n+2}} < \deg{r_{n+1}}\) because \(a_n = \gauss{\ifrac{r_n}{r_{n+1}}}\). Hence
\begin{equation*}
\alpha_n = \frac{r_n}{r_{n+1}} = a_n + \frac{r_{n+2}}{r_{n+1}} = a_n + \frac{1}{\alpha_{n+1}}.
\end{equation*}
So in this case, the continued fraction process corresponds to the Euclidean algorithm which is well known to terminate in a finite number of steps; so eventually \(r_{n+1} = 0\) for some \(n\) which means that \(\alpha_n \in \K[X]\) and that consequently \(\CF(\alpha)\) is finite.
\end{proof}
\section{Canonical convergents and classification of best-approximations}
\label{sec:orga882f91}
\begin{defi}
The sequence of \emph{canonical convergents} of \(\alpha\) is defined by
\begin{equation*}
(p_n, q_n) = (p_{0,n}, q_{0,n}) \in \Batest{\K} \quad \text{ for } n \geq -1
\end{equation*}
where we plug the partial quotients into the formulas from Section \ref{sec:orgb81a872}. If \(\CF(\alpha)\) is finite, this sequence is also finite.
\end{defi}
Note that Corollary \ref{canonical-convergent-coprime} implies that \(p_n\) and \(q_n\) are coprime for a given \(n\).
And the canonical convergents are in fact convergents, and we have precise information about their approximation quality:
\begin{prop}
\label{cf-expansion-yields-convergents}
Let \(n \geq 0\). Then unless \(\alpha = \ifrac{p_n}{q_n}\),
\begin{equation*}
\io{p_n - \alpha \, q_n} = \deg{q_{n+1}} = \deg a_{n+1} + \deg q_n > \deg q_n,
\end{equation*}
so \((p_n, q_n) \in \Coset{\alpha}{\K}\).
\end{prop}
\begin{rem}
If \(\alpha = \ifrac{p_n}{q_n}\), then clearly \(\io{p_n - \alpha \, q_n} = \infty > \deg q_n\) and obviously \((p_n, q_n) \in \Coset{\alpha}{\K}\).
\end{rem}
\begin{proof}
Unless \(\CF(\alpha) = [a_0, \dots, a_n]\) is finite of length exactly \(n+1\) which directly implies \(p_n - \alpha \, q_n = 0\) by the Proposition \ref{cf-euclidean-algorithm}, we have
\begin{equation*}
\alpha = [a_0, \dots, a_{n}, \alpha_{n+1}]
\quad \text{ i.e. }
\alpha = \mfour{p_n}{p_{n-1}}{q_n}{q_{n-1}} \cdot \alpha_{n+1}.
\end{equation*}
Multiplying with the inverse matrix, we get the important formula
\begin{equation}
\label{complete-quotient-convergent-quotient}
\alpha_{n+1} = (-1)^{n+1} \, \mfour{q_{n-1}}{-p_{n-1}}{-q_n}{p_n} \cdot \alpha= -\frac{p_{n-1} - \alpha \, q_{n-1}}{p_n - \alpha \, q_n}.
\end{equation}
Recall that \(p_{-1} = 1, q_{-1} = 0\), so a telescoping product yields
\begin{equation*}
(-1)^{n+1} \prod_{j=0}^{n} \alpha_{j+1} = \prod_{j=0}^{n} \frac{p_{j-1} - \alpha \, q_{j-1}}{p_j - \alpha \, q_j} = \frac{1}{p_n - \alpha \, q_n}.
\end{equation*}
Taking valuations, note that \(\io{\alpha_j} = \io{a_j} = - \deg a_j\) for \(j \geq 1\), hence
\begin{equation*}
\io{p_n - \alpha \, q_n} = -\sum_{j=1}^{n+1} \io{\alpha_j} = \sum_{j=1}^{n+1} \deg a_j = \deg a_{n+1} + \deg q_n = \deg q_{n+1},
\end{equation*}
the last two equalities being a consequence of Proposition \ref{canonical-convergents-degree-and-lc}.
\end{proof}

\pagebreak
\begin{prop}
The continued fraction of \(\alpha\) represents \(\alpha\) as an element of \(\laurinx \K\), i.e.
\begin{equation*}
\alpha = \CF(\alpha) \text{ in } \laurinx \K.
\end{equation*}
\end{prop}
\begin{proof}
For \(\alpha \in \K(X)\), this is was mentioned in the proof of Proposition \ref{cf-euclidean-algorithm}. Otherwise, \(\alpha \not \in \K(X)\), and from Proposition \ref{cf-expansion-yields-convergents} we conclude \(\alpha = \limn \frac{p_n}{q_n} = \CF(\alpha)\) as \(\limn \deg{q_n} = \infty\).
\end{proof}

With this information about the approximation quality of the canonical convergents, we can now give a complete classification of the best-approximations.

\begin{prop}[Classification of best-approximations]
\label{cf-best-approx-classification}
Let \(\alpha \in \laurinx \K \setminus \K(X)\), and \((p, q) \in \Baset{\alpha}{\K}\) a best-approximation. Then there exist a unique \(n \in \N_0\) and \(r \in \mino{\K[X]}\) with \(\deg r < \deg a_{n+1}\) such that
\begin{equation*}
(p, q) = r \cdot (p_n, q_n).
\end{equation*}
In particular, if \(p\) and \(q\) are coprime, then \(r \in \units \K\).

Moreover, if \((p', q') = r' \cdot (p_{n'}, q_{n'}) \in \Baset{\alpha}{\K}\) is another best-approximation with \(\deg q < \deg q'\), then \(n \leq n'\).
\end{prop}
\begin{proof}
With the sufficient condition for a best-approximation from Proposition \ref{best-approx-common-factor-degree} applied to \((p_n, q_n)\) and \(\xi = \deg a_{n+1}\), we see that for every possible \(\deg q\) we can produce a best-approximation of the shape \(r \cdot (p_n, q_n)\), with any \(r \in \K[X]\) satisfying \(0 \leq \deg r < \deg a_{n+1}\). Then by Proposition \ref{best-approx-for-given-degree}, all best-approximations have this shape. Because \(p_n\) and \(q_n\) are always coprime, and \(\deg q_n\) is strictly increasing in \(n\), no canonical convergent can be written as a multiple of another, so \(n\) must be unique.

Finally, the monotony result is obvious from \(\deg{q_n} \leq \degb{r \, q_n} < \deg q_{n+1}\).
\end{proof}
\begin{rem}
If \(\alpha \in \K(X)\), this argument works just as well, except for the last canonical convergent. However, if we put ``\(\deg a_{n+1} = \infty\)'', the statement trivially holds even for the last canonical convergent.
\end{rem}

For completeness, we also give the analogue for convergents (applying Proposition \ref{convergent-common-factor-degree} instead of Proposition \ref{best-approx-common-factor-degree}):
\begin{cor}
\label{cf-convergent-classification}
Let \(\alpha \in \laurinx \K \setminus \K(X)\), and \((p, q) \in \Coset{\alpha}{\K}\) a convergent. Then there exist \(n \in \N_0\) and \(r \in \mino{\K[X]}\) with \(\deg r < \frac{1}{2} \deg a_{n+1}\) such that
\begin{equation*}
(p, q) = r \cdot (p_n, q_n),
\end{equation*}
and if \(p\) and \(q\) are coprime, then \(r \in \units \K\).
\end{cor}
\section{Multiplication of a continued fraction by a constant}
\label{sec:org976a8fd}
One nice feature of polynomial continued fractions is that it is possible to multiply them with a constant factor. In \cite{schmidt-2000-continued-fractions-diophantine}, there is even a generalisation of this identity which holds also for non-constant factors. We limit ourselves to constants, however.

\begin{prop}
\label{cf-scalar-multiplication}
Let \(\mu \in \units \K\). Then
\begin{equation*}
\mu \, [a_0, a_1, a_2, a_3, \dots] = [\mu \, a_0, \inv \mu \, a_1, \mu \, a_2, \inv \mu \, a_3, \dots ].
\end{equation*}
\end{prop}
\begin{proof}
Again, it is convenient to think of the continued fraction as a product of matrices:
\begin{align*}
\mdi{\mu}{1} \mcf{a} &= \mfour{\mu \, a}{\mu}{1}{} = \mcf{\mu \, a} \mdi{1}{\mu}, \\
\mdi{1}{\mu} \mcf{a} &= \mfour{a}{1}{\mu}{} = \mcf{\invfrac{a}{\mu}} \mdi{\mu}{1}.
\end{align*}
As multiplication by \(\mu\) corresponds to \(\mdi{\mu}{1}\) and division by \(\mu\) corresponds to \(\mdi{1}{\mu}\), we obtain for \(n\) even
\begin{equation*}
\mdi{\mu}{1} \mcf{a_0}  \cdots  \mcf{a_n} = \mcf{\mu \, a_0} \mcf{\invfrac{a_1}{\mu}} \cdots \mcf{\mu \, a_n} \mdi{1}{\mu}
\end{equation*}
and for \(n\) odd
\begin{equation*}
\mdi{\mu}{1} \mcf{a_0} \cdots  \mcf{a_n} = \mcf{\mu \, a_0} \mcf{\invfrac{a_1}{\mu}} \cdots \mcf{\invfrac{a_n}{\mu}} \mdi{\mu}{1}
\end{equation*}
so the corresponding map would be for \(n\) even
\begin{equation*}
x \mapsto [\mu \, a_0, \invfrac{a_1}{\mu}, \dots, \mu \, a_n, \invfrac{x}{\mu}]
\end{equation*}
and for \(n\) odd
\begin{equation*}
x \mapsto [\mu \, a_0, \invfrac{a_1}{\mu}, \dots, \invfrac{a_n}{\mu}, \mu \, x]
\end{equation*}
as desired -- because for the empty continued fraction, we have \(\mu \cdot [\,] = \frac{\mu}{0} = \frac{1}{0} = [\,]\).
\end{proof}
\section{Periodic continued fractions}
\label{sec:org3afbb4f}
For classical continued fractions, it is a well-known result that continued fractions of quadratics are always periodic. As in the real case, a periodic polynomial continued fraction must be quadratic. However, a continued fraction of a quadratic need not be periodic in the polynomial case. For \(\sqrt{D}\) this in fact happens \IFF \(D\) is Pellian which we will prove in Section \ref{sec:org8d34125}.

Indeed periodicity gives a solution of \eqref{pellu} with \(\omega = \pm 1\) (this follows from \eqref{cf-sn-pell-eq}). But if the base field \(\K\) is very small, allowing arbitrary \(\omega\) may give a solution with smaller \(\deg q\). So one should not merely study periodicity, but periodicity up to a constant factor. We call this \emph{quasi-periodicity} (sometimes it is also called pseudo-periodicity in the literature). For the continued fraction of \(\sqrt{D}\), the period and the quasi-period are tightly linked, and one induces the other.

Later in Chapter \ref{sec:org5f9d2ce}, we will also see that quasi-periodicity is the more relevant notion for studying reductions of the continued fraction modulo a prime.

\subsection{Periods}
\label{sec:org8b4d61a}
\begin{defi}
\label{cf-periodicity-partial-quotients}
The (infinite) continued fraction \(\alpha_m = \alpha_{m,\infty} = [a_m, a_{m+1}, \dots]\) is said to be \emph{periodic} if for some \(m' \geq m\) there exists \(l \in \N\) (the minimal such \(l\) is called the \emph{period length}) such that
\begin{equation*}
\forall n \geq m': \; a_{n} = a_{n+l}
\end{equation*}
If \(m' = m\), the continued fraction is called \emph{pure periodic}, i.e. there is no preperiod. For compact notation, we usually write ``\(\CF(\alpha_m)\) is (pure) periodic''.
\end{defi}

From a computational view, this definition is somewhat problematic because there is an infinite number of conditions to check. Fortunately, this can be reduced to a single condition on the complete quotients.

\begin{prop}
\label{cf-periodicity-complete-quotients}
\TFAE
\begin{enumerate}
\item The continued fraction \(\CF(\alpha_m)\) is periodic.
\item There exist \(m' \geq m\) and \(l \in \N\) such that \(\alpha_{m'} = \alpha_{m'+l}\).
\item There exist \(m' \geq m\) and \(l \in \N\) such that for all \(n \geq m': \; \alpha_n = \alpha_{n+l}\).
\end{enumerate}
\end{prop}
\begin{proof}
Because \(a_n = \gauss{\alpha_n}\), 3. directly implies 1.

On the other hand, \(\alpha_{m'}\) is uniquely determined by \(a_{m'}, a_{{m'}+1}, \dots\), and by periodicity of the \(a_n\) one obtains
\begin{equation*}
\alpha_{{m'}+l} = [a_{{m'}+l}, a_{{m'}+l+1}, \dots] = [a_{m'}, a_{{m'}+1}, \dots] = \alpha_{m'}.
\end{equation*}
so 1. implies 2.

But through the continued fraction process, \(\alpha_{n+1}\) is uniquely determined by \(\alpha_n\) for every \(n\), so
\begin{equation*}
\alpha_n = \alpha_{n+l} \implies \alpha_{n+1} = \alpha_{n+l+1}.
\end{equation*}
and by the induction principle, 2. implies 3.
\end{proof}

\subsection{Quasi-periods}
\label{sec:org7617d8f}
We now generalise periodicity to quasi-periodicity which is essentially periodicity up to a unit factor. For cleaner notation, we first define
\begin{equation*}
\parity{n} = (-1)^n = \begin{cases}
1 & \text{ if } n \text{ is even,}\\
-1 & \text{ if } n \text{ is odd}
\end{cases}.
\end{equation*}

\begin{defi}
\label{cf-quasi-periodicity-partial-quotients}
The (infinite) continued fraction \(\alpha_{m}\) is called \emph{quasi-periodic}, if there exists \(m' \geq m\), \(\mu \in \units \K\) and \(l > 0\) (if minimal, called the \emph{quasi-period length}) such that
\begin{equation*}
\forall n \geq m': \; a_{n} = \mu^{\parity{n}} \, a_{n+l}.
\end{equation*}
If \(m' = m\), then it is called \emph{pure quasi-periodic}.
\end{defi}
\begin{rem}
Any periodic continued fraction is also quasi-periodic, with \(\mu = 1\). See below for a partial converse.
\end{rem}

\begin{rem}
\label{quasi-period-length-ideal}
It should be obvious that the \(l \in \Z\) such that \(\alpha_n / \alpha_{n+l} \in \units\K\) form an ideal, and the (quasi-)period length is the positive generator of it.

In particular, the period length must be a multiple of the quasi-period length.
\end{rem}

We also have a complete analogue to Proposition \ref{cf-periodicity-complete-quotients}:
\begin{prop}
\label{cf-quasi-periodicity-complete-quotients}
\TFAE
\begin{enumerate}
\item The continued fraction \(\CF(\alpha_m)\) is quasi-periodic.
\item There exist \(m' \geq m\), \(\mu \in \units \K\) and \(l > 0\) such that \(\alpha_{m'} = \mu^{\parity{m'}} \, \alpha_{m'+l}\).
\item There exist \(m' \geq m\), \(\mu \in \units \K\) and \(l > 0\) such that for all \(n \geq m' : \; \alpha_{n} = \mu^{\parity{n}} \, \alpha_{n+l}\).
\end{enumerate}
\end{prop}
\begin{proof}
Using Proposition \ref{cf-scalar-multiplication}, and \(\gauss{\mu \, \alpha} = \mu \, \gauss{\alpha}\) for \(\mu \in \units \K\), and
\begin{equation*}
\alpha_n = \mu \, \alpha_{n+l} \implies \alpha_{n+1} = \inv \mu \,\alpha_{n+l+1},
\end{equation*}
the proof is completely analogous to the one of Proposition \ref{cf-periodicity-complete-quotients}.
\end{proof}

\begin{prop}
\label{odd-quasi-period-implies-periodic}
If \(\CF(\alpha_m)\) is quasi-periodic with \emph{odd} quasi-period length \(l\) and \(\mu \neq 1\), then \(\CF(\alpha_m)\) is also periodic with period length \(2 \, l\).
\end{prop}
\begin{proof}
For all \(n \geq m'\), we have \(a_n = \mu^{\parity{n}} \, a_{n+l}\) and \(a_{n+l} = \mu^{\parity{n+l}} \, a_{n+2 l}\). As \(l\) is odd, we have \(\parity{n+l} = - \parity{n}\) so \(a_n = \mu^{\parity{n}+\parity{n+l}} \; a_{n+2 l} = a_{n+2 l}\).
\end{proof}

\begin{rem}
\label{quasi-period-length-shift-invariant}
The (quasi-)period length was above defined as the minimal \(l\), and does not depend on where the (quasi-)period starts, so two complete quotients \(\alpha_{m_1}\) and \(\alpha_{m_2}\) have the same (quasi-)period length.
\end{rem}

\begin{prop}
If \(\CF(\alpha)\) is quasi-periodic, then \(\alpha \in \laurinx \K\) is quadratic over \(\K(X)\) (it cannot be in \(\K(X)\) because it has an infinite continued fraction).
\end{prop}
\begin{proof}
From Section \ref{sec:org932a99d} we know
\begin{equation*}
\alpha_{m} = \mfour{p_{m,n}}{p_{m,n-1}}{q_{m,n}}{q_{m,n-1}} \cdot \alpha_{m+n+1} = \frac{p_{m,n} \, \alpha_{m+n+1} + p_{m,n-1}}{q_{m,n} \, \alpha_{m+n+1} + q_{m,n-1}}
\end{equation*}
so it suffices to treat the case where \(\CF(\alpha)\) is pure quasi-periodic, i.e. \(\alpha_l = \mu \, \alpha\). Then putting \(m = 0\) and \(n = l-1\) the above becomes
\begin{equation*}
\alpha = \mfour{p_{l-1}}{p_{l-2}}{q_{l-1}}{q_{l-2}} \cdot \mu \, \alpha = \frac{p_{l-1} \, \mu \, \alpha + p_{l-2}}{q_{l-1} \, \mu \, \alpha + q_{l-2}}.
\end{equation*}
Multiplying with the denominator, we then obtain
\begin{equation*}
q_{l-1} \, \mu \, \alpha^2 + (q_{l-2} - p_{l-1} \, \mu) \, \alpha - p_{l-2} = 0
\end{equation*}
and of course \(q_{l-1} \neq 0\) so \(\alpha\) is quadratic over \(\K(X)\).
\end{proof}
\section{Reduced complete quotients}
\label{sec:org6896ed2}
Our next goal is to understand the continued fraction expansion of \(\sqrt{D}\) better. We will explain how we can usually go backwards in this continued fraction. This means we can only have very short preperiods (here just \(a_0\) belongs to the preperiod), and allows to show that for \(\CF(\sqrt{D})\), quasi-periodicity is equivalent to periodicity, also in the case of \emph{even} quasi-period length.

In this case, the complete quotients are contained in the quadratic extension \(\HEF\) of \(\K(X)\) contained in \(\laurinx{\K}\). It has precisely one non-trivial \(\K(X)\)-automorphism \(\sigma\) which sends \(\sqrt{D}\) to \(-\sqrt{D}\). As we have chosen \(\sqrt{D} \in \laurinx \K\), we have an embedding of \(\HEF\) into \(\laurinx \K\).

\begin{defi}
\(\alpha \in \HEF\) is said to be \emph{\(\sigma\)-reduced} (with \(\sigma\) as above), if
\begin{equation*}
\io{\sigma(\alpha)} > 0 > \io{\alpha}.
\end{equation*}
\end{defi}

\begin{rem}
All elements of \(\K(X)\) are invariant under \(\sigma\), so none of them is \(\sigma\)-reduced.
\end{rem}

\begin{prop}
\label{sigma-reduced-no-translations}
Let \(\alpha \in \HEF \setminus \K(X)\). Then there exists at most one \(a \in \K[X]\) such that \(a+\alpha\) is \(\sigma\)-reduced.
\end{prop}
\begin{proof}
Clearly \(\sigma(a+\alpha) = a+ \sigma(\alpha)\). Assume \(\io{\sigma(a + \alpha)} = \io{a + \sigma(\alpha)} > 0\), then by Remark \ref{truncation-unique} \(a = - \gauss{\sigma(\alpha)}\) so there is at most one choice for \(a\).
\end{proof}
Note that this choice of \(a\) does not yet guarantee \(\io{a + \alpha} < 0\).

\begin{prop}
\label{sigma-reduced-all-complete-quotients}
If the complete quotient \(\alpha_m\) is \(\sigma\)-reduced, then so is \(\alpha_{m'}\) for all \(m' \geq m\).
\end{prop}
\begin{proof}
Using the induction principle, it suffices to treat the case \(m' = m+1\). By Remark \ref{cf-absolute-values}, we automatically have \(\io{\alpha_{m+1}} < 0\). Moreover,
\begin{equation*}
\sigma(\alpha_{m+1}) = \frac{1}{\sigma(\alpha_m) - a_m}
\end{equation*}
and \(\io{a_m} = \io{\alpha_m} < \io{\sigma(\alpha_m)}\) implies \(\io{\sigma(\alpha_{m+1})} = -\io{\alpha_m} > 0\) as desired.
\end{proof}

\pagebreak
\begin{lemma}
\label{sigma-reduced-cfsb}
\(\alpha\) is \(\sigma\)-reduced \IFF \(\cfsb{\alpha}\) is \(\sigma\)-reduced.
\end{lemma}
\begin{proof}
This is an immediate consequence of
\begin{align*}
\io{\alpha} &= - \io{\sigma\left(\cfsb{\alpha}\right)}, &
\io{\sigma(\alpha)} &= - \io{\cfsb{\alpha}}.
\end{align*}
\end{proof}

The most useful property of \(\sigma\)-reduced complete quotients is however that we may go backwards in the continued fraction expansion in a unique way:
\begin{prop}
\label{cf-reduced-backwards}
Suppose \(\alpha_1 \in \HEF\) is \(\sigma\)-reduced. Then there exists a unique \(\alpha_0 \in \HEF\) which is \(\sigma\)-reduced and satisfies
\begin{equation*}
\alpha_1 = \frac{1}{\alpha_0 - \gauss{\alpha_0}}.
\end{equation*}
\end{prop}
\begin{proof}
By Proposition \ref{sigma-reduced-no-translations}, there exists at most one \(a_0 \in \K[X]\) such that \(\alpha_0 = a_0 + \frac{1}{\alpha_1}\) is \(\sigma\)-reduced, namely \(a_0 = \gauss{\cfsb{\alpha_1}}\). Rewriting this to
\begin{equation*}
\cfsb{\alpha_0} = \frac{1}{\cfsb{\alpha_1} - a_0},
\end{equation*}
we see that \(\alpha_0\) is \(\sigma\)-reduced by applying twice Lemma \ref{sigma-reduced-cfsb} and once Proposition \ref{sigma-reduced-all-complete-quotients}.

Finally, as \(\io{\alpha_1} < 0\) it is also clear that \(a_0 = \gauss{\alpha_0}\).
\end{proof}
\begin{rem}
\label{cf-reduced-backwards-n}
Generally, for any \(n\) and \(\alpha_n\) \(\sigma\)-reduced, we have
\begin{equation*}
\cfsb{\alpha_n} = \frac{1}{\cfsb{\alpha_{n+1}} - \gauss{\cfsb{\alpha_{n+1}}}}
\end{equation*}
so also the \(\cfsb{\alpha_n}\) are the complete quotients of some continued fraction expansion, albeit with \(n\) decreasing.
\end{rem}

\begin{lemma}
\label{sigma-reduced-pure-period}
Suppose \(\alpha_m\) is \(\sigma\)-reduced and \(\CF(\alpha_m)\) is (quasi-)periodic, then \(\CF(\alpha_m)\) is pure (quasi-)periodic.
\end{lemma}
\begin{proof}
We use Propositions \ref{cf-periodicity-complete-quotients} and \ref{cf-quasi-periodicity-complete-quotients} here.

Suppose \(n > m, l \in \N\) and \(\mu \in \units{\K}\) (where \(\mu = 1\) in the case of periodicity) with
\(\alpha_{n} = \mu^{\parity{n}} \, \alpha_{n+l}\). By Proposition \ref{sigma-reduced-all-complete-quotients}, \(\alpha_{n-1}, \alpha_n, \alpha_{n+l-1}, \alpha_{n+l}\) are all \(\sigma\)-reduced, and we have
\begin{multline*}
\alpha_n = \frac{1}{\alpha_{n-1} - a_{n-1}} = \mu^{\parity{n}} \, \alpha_{n+l} \\ = \mu^{\parity{n}} \, \frac{1}{\alpha_{n+l-1} - a_{n+l-1}} = \frac{1}{\mu^{\parity{n-1}} \, \alpha_{n+l-1} - \mu^{\parity{n-1}} \, a_{n+l-1}}.
\end{multline*}
With \(\gauss{\mu^{\parity{n-1}} \, \alpha_{n+l-1}} = \mu^{\parity{n-1}} \, a_{n+l-1}\), Proposition \ref{cf-reduced-backwards} implies \(\alpha_{n-1} = \mu^{\parity{n-1}} \, \alpha_{n+l-1}\) as desired, and we may repeat this argument until we arrive at \(\alpha_{m} = \mu^{\parity{m}} \, \alpha_{m+l}\).
\end{proof}

\begin{thm}
\label{cf-berrys-thm}
Suppose \(\alpha \in \HEF\) is \(\sigma\)-reduced and has polynomial trace \(\alpha + \sigma(\alpha) \in \K[X]\). If \(\CF(\alpha)\) is quasi-periodic, it is even pure (quasi-)periodic.
\end{thm}
\begin{proof}
Lemma \ref{sigma-reduced-pure-period} already implies that \(\CF(\alpha)\) is pure quasi-periodic, and once we prove it is periodic, it is automatically pure periodic. For odd quasi-period length, the general Proposition \ref{odd-quasi-period-implies-periodic} already yields periodicity. For even quasi-period length, a bit more work is required.

From \(\gauss{f} = f\) for \(f \in \K[X]\) and \(\io{\sigma(\alpha)} > 0\) we obtain
\begin{equation*}
\alpha + \sigma(\alpha) = \gauss{\alpha + \sigma(\alpha)} = \gauss{\alpha} = a_0
\end{equation*}
so \(\alpha - a_0 = - \sigma(\alpha)\) which implies
\begin{equation*}
\alpha_1 = \cfsb{\alpha_0} \text{ and thus } \alpha_0 = \cfsb{\alpha_1}.
\end{equation*}
In the light of Remark \ref{cf-reduced-backwards-n}, the \(\cfsb{\alpha_n}\), going backwards, are complete quotients of some continued fraction expansion and actually extend \(\CF(\alpha)\) for negative \(n\):
\begin{equation*}
\begin{array}{cccccccc}
\dots & \cfsb{\alpha_3} & \cfsb{\alpha_2} & \cfsb{\alpha_1} & \cfsb{\alpha_0} \\
      &                 &                 & \alpha_0 & \alpha_1 & \alpha_2 & \alpha_3 & \dots
\end{array}
\end{equation*}
So we can define \(\alpha_{n} = \cfsb{\alpha_{1-n}}\) for \(n \leq 1\), with all \(\alpha_n\) \(\sigma\)-reduced, and by Lemma \ref{sigma-reduced-pure-period} the quasi-periodicity extends towards \(-\infty\) as well.

Denote by \(\QPL\) the quasi-period length of \(\CF(\alpha)\), so we may write
\begin{align*}
\alpha_0 &= \mu \, \alpha_\QPL, &
\alpha_\QPL &= \mu^{\parity{\QPL}} \, \alpha_{2\QPL}, &
\alpha_{1-\QPL} &= \mu^{\parity{1-\QPL}} \, \alpha_1.
\end{align*}
It follows
\begin{equation*}
\alpha_\QPL = \cfsb{\alpha_{1-\QPL}} = \frac{1}{\mu^{\parity{1-\QPL}}} \, \cfsb{\alpha_1} = \mu^{\parity{\QPL}} \, \alpha_0
\end{equation*}
and further \(\alpha_0 = \mu \, \mu^{\parity{\QPL}} \, \alpha_0\). Hence \(\mu \, \mu^{\parity{\QPL}} = 1\) (if \(\QPL\) is even, this means \(\mu = \pm 1\)), and then \(\alpha_0 = \mu \, \alpha_\QPL = \mu \, \mu^{\parity{\QPL}} \, \alpha_{2\QPL} = \alpha_{2\QPL}\), so \(\CF(\alpha)\) is periodic (with period length \(\QPL\) or \(2\QPL\)).
\end{proof}
\begin{rem}
This shows that the involution \(x \mapsto \cfsb{x}\) acts as a reflection with centre \(1/2\) on the \(\Z\)-series of \(\alpha_n\) (\(n \mapsto 1-n\) on the indices).
\end{rem}

\begin{rem}
\label{period-of-sqrt-d}
Obviously \(\sqrt{D}\) is not \(\sigma\)-reduced. However \(\alpha = A + \sqrt{D}\) (recall that \(A = \gauss{\sqrt{D}}\)) is \(\sigma\)-reduced, and \(\sqrt{D} - \gauss{\sqrt{D}} = \alpha - \gauss{\alpha}\), so
\begin{equation*}
\CF(\sqrt{D}) = [A, a_1, a_2, \dots]
\end{equation*}
differs from \(\CF(\alpha)\) only in the first complete (and partial) quotient. This means that if \(\CF(\sqrt{D})\) is quasi-periodic, it is almost pure periodic, and the preperiod has length \(1\) and consists just of \(A\).
\end{rem}

This reversibility of the continued fraction process also implies that the period must be a palindrome:
\begin{prop}
\label{palindromic-period}
Let \(\alpha \in \HEF\) \(\sigma\)-reduced with \(\alpha + \sigma(\alpha) \in \K[X]\). Let \(\QPL\) the quasi-period length.
\begin{itemize}
\item If \(\QPL\) is even, then \(\CF(\alpha)\) has actually period length \(\QPL\), and the period is palindromic, i.e.
\begin{equation*}
\CF(\alpha) = \left[\overline{a_0, a_1, \dots, a_{\QPL/2}, \dots, a_1}\right]
\end{equation*}
\item If \(\QPL\) is odd, then \(\CF(\alpha)\) has a ``quasi-palindromic'' quasi-period, i.e.
\begin{equation*}
\CF(\alpha) = \left[\overline{a_0, a_1, \dots, a_{(\QPL-1)/2}, \mu^{\pm 1} \, a_{(\QPL-1)/2}, \mu^{\mp 1} \, a_{(\QPL-3)/2}, \dots, \mu \, a_1}\right]
\end{equation*}
\end{itemize}
\end{prop}
\begin{rem}
In the second case, either \(\mu = 1\), or the period length \(2 \QPL\) is even. Then we can apply the first case for the period instead of the quasi-period to get a palindromic period.
\end{rem}
\begin{proof}
Recall how we defined the negative complete quotients, hence for any \(n \in \Z\)
\begin{equation*}
\alpha_{n} = \cfsb{\alpha_{1-n}} = \sigma\left(- \frac{1}{\alpha_{1-n}}\right) = \sigma\left(\alpha_{-n} - a_{-n}\right)
= a_{-n} + \frac{1}{\cfsb{\alpha_{-n}}} = a_{-n} + \frac{1}{\alpha_{n+1}},
\end{equation*}
the crux of which is \(a_{n} = \gauss{\alpha_{n}} = a_{-n}\).

Using quasi-periodicity, we then obtain
\begin{equation*}
a_n = a_{-n} = \mu^{\parity{-n}} \, a_{\QPL-n} = \mu^{\parity{n}} \, a_{\QPL-n}
\end{equation*}
and developing this for \(n \leq \QPL/2\) we obtain
\begin{equation*}
a_0 = \mu \, a_\QPL, \quad
a_1 = \inv\mu \, a_{\QPL-1}, \quad
a_2 = \mu \, a_{\QPL-2}, \quad \dots
\end{equation*}
until for \(\QPL\) odd we arrive at \(a_{(\QPL-1)/2} = \mu^{\parity{(\QPL-1)/2}} \, a_{(\QPL+1)/2}\) and for \(\QPL\) even we arrive at \(a_{\QPL/2} = \mu^{\parity{\QPL/2}} \, a_{\QPL/2}\) which also implies \(\mu = 1\).
\end{proof}
\chapter{Computation of hyperelliptic continued fractions}
\label{sec:orge1a60a6}
We now give formulas for computing the continued fraction expansion for quadratic Laurent series. Optimising these formulas is not only useful for computing and studying examples, but it also serves to illustrate the connection between the Pell equation and periodicity of the continued fraction. Of particular interest is that everything can be expressed as operations on polynomials.

We assume as usual that \(D\) is non-square of degree \(2d\) and that \(\LC(D)\) is a square in \(\K\), a field of characteristic not \(2\). Recall that we defined the polynomial part \(A = \gauss{\sqrt{D}}\).

It is well-known that the complete quotients of \(\sqrt{D}\) can be written as \(\alpha_n =  (r_n + \sqrt{D})/s_n\) with \(r_n, s_n \in \K[X]\) of bounded degree. We can slightly improve upon this representation by writing \(r_n = A + \text{terms of lower degree}\). This seems to be a new result:
\begin{thm}
\label{thm-optimised-sqrt-cq-representation}
Let \(\alpha = \sqrt{D}\). The complete quotients of \(\alpha\) can be written as
\begin{equation}
\label{quadratic-cq-representation}
\alpha_n = \frac{A + t_n + \sqrt{D}}{s_n} \quad \text{ for } n \geq 1
\end{equation}
where \(t_n, s_n \in \K[X]\) with
\begin{equation}
\label{sigma-reduced-degree-condition-in-prop}
\deg t_n < \deg s_n < \deg A
\end{equation}
for \(n \geq 1\). Moreover, there are the following recursion formulas for \(t_n\) and \(s_n\):
\begin{equation}
\label{quadratic-cq-recursion-formulas}
t_{n} + t_{n+1} = a_n \, s_n - 2 \, A, \quad s_{n} \, s_{n+1} = D - (A+t_{n+1})^2,
\end{equation}
initialised with \(t_0 = -A\) and \(s_0 = 1\). Finally \(\deg s_n = 0\) for \(n \geq 1\) \IFF \(\CF(\alpha)\) is periodic and the quasi-period length \(\QPL\) divides \(n\). 
\end{thm}
Note that \(\alpha_n\) being \(\sigma\)-reduced is equivalent to \eqref{sigma-reduced-degree-condition-in-prop} by Proposition \ref{sigma-reduced-degree-prop}.
\begin{cor}
\label{cor-pq-degree-periodicity}
The complete quotients satisfy \(\io{\alpha_n} \geq \io{\sqrt{D}}\), so for the partial quotients we have
\begin{equation*}
1 \leq \deg a_n \leq \deg A
\end{equation*}
with equality  \(\deg a_n = \deg A = d\) for \(n \geq 1\) \IFF the continued fraction \(\CF(\sqrt{D})\) is periodic, and the quasi-period length \(\QPL\) divides \(n\).
\end{cor}

\enlargethispage{1cm}
In fact, we show more generally:
\begin{thm}
\label{thm-quadratic-laurent-series-cf-representation}
Let \(\alpha \in \laurinx \K\) any Laurent series quadratic over \(\K(X)\). Then for a suitable \(D\) depending only on \(\alpha\), the complete quotients \(\alpha_n\) may also be written as in \eqref{quadratic-cq-representation}, where \(t_n\) and \(s_n\) follow the recursion formulas \eqref{quadratic-cq-recursion-formulas}.

Moreover, there exists \(N \geq 0\), such that \(t_n\) and \(s_n\) satisfy \eqref{sigma-reduced-degree-condition-in-prop} for all \(n \geq N\).
\end{thm}

The theorem also gives a more elementary proof of periodicity over finite fields:
\begin{cor}
\label{cor-finite-field-always-periodic}
If the base field \(\K\) is finite, any Laurent series quadratic over \(\K(X)\) has a periodic continued fraction expansion.
\end{cor}

Using this representation of the complete quotients of \(\sqrt{D}\), and our accumulated knowledge about the convergents, we also recover Abel's result from \cite{abel-1826-ueber-integ-differ}:
\begin{thm}[Abel 1826]
\label{thm-pellian-iff-cf-periodic}
\(D\) is Pellian \IFF \(\CF(\sqrt{D})\) is periodic.
\end{thm}

We shall prove these results in the first part of this chapter. The second part then explores some further consequences.

\section{Representing complete quotients with polynomials}
\label{sec:org95cb74d}
We begin by reiterating the formulas for hyperelliptic continued fraction expansions which can (with varying level of detail) be already found in \cite{abel-1826-ueber-integ-differ}, \cite{berry-1990-periodicity-continued-fractions} and \cite{poorten-tran-2000-quasi-elliptic-integrals}.

Let \(\alpha \in \laurinx \K\) be quadratic over \(\K(X)\), satisfying \(s \, \alpha^2 - 2\, r \, \alpha + w = 0\) where \(r, s, w \in \K[X]\). The discriminant \(4 \, D = 4 \, (r^2 - s\,w)\) yields \(D\), for which we choose a square root \(\sqrt{D}\). Then we write
\begin{equation}
\label{cf-quadratic-normalised-representation}
\alpha = \frac{r + \sqrt{D}}{s}
\end{equation}
after possibly multiplying \(r, s, w\) with \(-1\) to accommodate our choice of \(\sqrt{D}\). Note that here holds \(s \div D - r^2\) which is crucial for the following computations. This allows a common factor in \(r\) and \(s\) which then must divide \(D\) as well.

Clearly \(\alpha\) is determined by the polynomials \(r, s, D\) and our choice of \(\sqrt{D}\). For example for \(\alpha = \sqrt{D}\) we just put \(r = 0\), \(s = 1\) and \(w = - D\).

All complete quotients of a given \(\alpha\) can be written in this way; all of them with the same discriminant \(D\):
\begin{prop}
\label{prop-quadr-repr-rs}
The complete quotients of \(\alpha\) as in \eqref{cf-quadratic-normalised-representation} have for all \(n \geq 0\) the form
\begin{equation}
\label{quadr-repr-rs}
\alpha_n = \frac{r_n + \sqrt{D}}{s_n}, \quad \text{ where } s_n \div (D - r_n^2) \text{ and } r_n, s_n \in \K[X].
\end{equation}
\end{prop}
\begin{proof}
We prove this using complete induction. For \(n = 0\) we may take \(r_0 = r\) and \(s_0 = s\) which satisfy the desired conditions by hypothesis.

Suppose \eqref{quadr-repr-rs} holds for \(n\). Then write
\begin{equation*}
\frac{1}{\alpha_{n+1}} = \alpha_n - a_n
= \left( \frac{r_n + \sqrt{D}}{s_n} - a_n \right) \left( \fracsame{-r_n + \sqrt{D} + a_n \, s_n} \right)
= \frac{ D - (r_n - a_n \, s_n)^2}{s_n \, \left(a_n \, s_n - r_n + \sqrt{D}\right)}
\end{equation*}
and note that
\begin{equation*}
D - r_{n+1}^2 = D - (a_n \, s_n - r_n)^2 = D - a_n^2 \, s_n^2 + 2\, a_n \, s_n \, r_n - r_n^2
\end{equation*}
so by induction hypothesis \(s_n \div D - r_n^2\), this is divisible by \(s_n\) and we can set
\begin{align}
\label{cf-rs-recursion}
r_{n+1} &= a_n \, s_n - r_n, &
s_{n+1} &= \frac{D - r_{n+1}^2}{s_n}.
\end{align}
with \(r_{n+1}, s_{n+1} \in \K[X]\) and moreover \(s_{n+1} \div D - r_{n+1}^2\).
This concludes the induction step.
\end{proof}
\begin{rem}
It should be quite obvious that the discriminant does not change for the complete quotients. After all, the discriminant is invariant under the natural action of \(\GL{2}{\K(X)}\) by linear change of variables on bilinear forms in two variables over \(\K(X)\). Such a bilinear form gives of course a minimal polynomial for a quadratic \(\alpha\). But advancing in the continued fraction expansion can exactly be expressed in terms of this action, as seen in Section \ref{sec:org932a99d}.
\end{rem}

Berry (and Abel for \(\deg D = 4\)) give further simplifications of these formulas, see \cite{berry-1990-periodicity-continued-fractions} and \cite{abel-1826-ueber-integ-differ}. We prefer to perform simplifications of a different kind. And we still need to explain how to compute the \(a_n\) from our representation. 

We may rewrite \eqref{quadr-repr-rs} as
\begin{equation}
\label{quadr-repr-ats}
\alpha_n = \frac{A + t_n + \sqrt{D}}{s_n}
\end{equation}
by setting \(t_n = r_n - A\). The recursion formulas \eqref{cf-rs-recursion} then obviously change to
\begin{equation}
\label{cf-ats-recursion}
\begin{aligned}
t_0 &= r - A, & t_{n+1} &= a_n \, s_n - 2 \, A - t_n, \\
s_0 &= s,     & s_{n+1} &= \frac{D - A^2 - 2 \, A \, t_{n+1} - t_{n+1}^2}{s_n}.
\end{aligned}
\end{equation}

This already proves the first half of Theorem \ref{thm-quadratic-laurent-series-cf-representation}.

\begin{prop}
\label{cf-ats-bound-t1}
We can compute \(t_{n+1}\) and \(a_n\) with a single polynomial division, i.e.
\begin{equation*}
(2 \, A + t_n)= a_n \, s_n - t_{n+1} \text{ with } \deg t_{n+1} < \deg s_n.
\end{equation*}
\end{prop}
\begin{proof}
Recall from \eqref{sqrt-d-A-plus-eps} that \(\sqrt{D} = A + \varepsilon\) with \(\io{\varepsilon} > 0\). 
The equality follows directly from the formula for \(t_{n+1}\) in \eqref{cf-ats-recursion}, it remains to check \(\deg t_{n+1} < \deg s_n\). Using \(\gauss{\varepsilon} = 0\) and Remark \ref{truncation-of-sum} (\(\gauss{\cdot}\) is a homomorphism with respect to \(+\)) we find
\begin{equation*}
a_n = \gauss{\alpha_n} = \gauss{ \frac{A + t_n + \sqrt{D}}{s_n} } = \gauss{\frac{2 \, A +  t_n + \varepsilon}{s_n}} = \gauss{\frac{2 \, A + t_n}{s_n}}.
\end{equation*}
So by Remark \ref{truncation-of-rational} (taking \(\gauss{\cdot}\) of rational functions corresponds to polynomial division) \(-t_{n+1}\) must the remainder of the polynomial division of \(2 \, A + t_n\) by \(s_n\).
\end{proof}
\section{Complete quotients are eventually \(\sigma\)-reduced}
\label{sec:org780327e}
The representation \eqref{quadr-repr-ats} also gives a very simple way to check if some complete quotient is \(\sigma\)-reduced:
\begin{prop}
\label{sigma-reduced-degree-prop}
\(\alpha = \frac{A + t + \sqrt{D}}{s}\) is \(\sigma\)-reduced \IFF
\begin{equation}
\label{sigma-reduced-degree}
\deg t < \deg s < \deg A,
\end{equation}
and in this case \(\io{\alpha} = \deg s - \deg A\).
\end{prop}
\begin{proof}
With \(A - \sqrt{D} = - \varepsilon\) we note that
\begin{equation*}
\io{\sigma(\alpha)} = \io{\ifracBb{A+t - \sqrt{D}}{s}} = \io{t - \varepsilon} - \io{s} = \io{t - \varepsilon} + \deg s.
\end{equation*}
Hence \(0 < \io{\sigma(\alpha)}\) is equivalent to \(\deg t < \deg s\): In the case \(t = 0\), using \(\io{\varepsilon} > 0\) we have \(\io{\sigma(\alpha)} = \io{\varepsilon} + \deg s > 0\) \IFF we have \(s \neq 0\), i.e. \(\deg s > - \infty = \deg 0\).
If on the other hand \(t \neq 0\), then \(\io{t - \varepsilon} = \io{t} = - \deg t\), hence \(\io{\sigma(\alpha)} = \deg s - \deg t\).

So for the rest of the proof, we can assume \(\io{\sigma(\alpha)} > 0\).

We may write
\begin{multline*}
\io{\alpha} = \io{\ifracBb{A+t+\sqrt{D}}{s}} = \io{2 \sqrt{D} + t - \varepsilon} - \io{s} \\ \geq \min\left(\io{2 \sqrt{D}}, \io{t-\varepsilon}\right) + \deg s.
\end{multline*}
If \(\alpha\) is \(\sigma\)-reduced, then \(0 > \io{\alpha} = \io{2 \sqrt{D}} + \deg s\) because \(\io{t -\varepsilon} + \deg s > 0\). Hence \(\deg s < \deg A = -\io{\sqrt{D}}\).

Conversely, if \(\deg s < \deg A\), then \(\io{\alpha} = \io{2 \sqrt{D}} + \deg s < 0\) by the ultrametric ``equality''.

With \(\io{2 \sqrt{D}} = \io{A} = -\deg A\), we also showed \(\io{\alpha} = \deg s - \deg A\).
\end{proof}

An immediate and important consequence is that the degrees of the partial quotients of a \(\sigma\)-reduced \(\alpha\) are always bounded uniformly -- once we show that every continued fraction of a quadratic \(\alpha\) eventually becomes \(\sigma\)-reduced, this means all partial quotients have bounded degree.
\begin{cor}
\label{maximal-degree-implies-quasi-period}
Suppose \(\alpha\) as above is \(\sigma\)-reduced, and \(a = \gauss{\alpha}\). Then \(0 < \deg a \leq \deg A\).

Moreover, if \(\deg{a} = \deg A\), then there exists \(\mu \in \units \K\) such that \(\alpha = \mu \, (A + \sqrt{D})\).
\end{cor}
\begin{proof}
From \(\alpha\) being \(\sigma\)-reduced, the preceding proposition yields
\begin{equation*}
0 > \io{\alpha} = \deg s - \deg A \geq - \deg A.
\end{equation*}
But \(\io{\alpha} = \io{a} = - \deg a\), hence \(0 < \deg a \leq \deg A\).

Additionally, if \(\deg a = \deg A\) this means \(\deg s = 0\) and thus \(t = 0\). So we get \(\mu = \inv s \in \units \K\).
\end{proof}

The second half of Theorem \ref{thm-quadratic-laurent-series-cf-representation} follows from
\begin{prop}
\label{cf-compute-eventually-sigma-reduced}
Let \(\alpha \in \laurinx \K\) quadratic over \(\K(X)\). Then there exist \(N \in \N\) such that for all \(n \geq N\), the complete quotient \(\alpha_n\) is \(\sigma\)-reduced.
\end{prop}
\begin{proof}
Using Proposition \ref{sigma-reduced-degree-prop}, this boils down to an analysis of the degrees of \(t_n\) and \(s_n\).

From Proposition \ref{cf-ats-bound-t1} follows \(\deg t_{n+1} < \deg s_n\). Recall from Remark \ref{cf-absolute-values} that we have \(\deg a_n \geq 1\) for \(n \geq 1\), hence
\begin{multline*}
\deg t_{n+1} < \deg s_n < \degb{a_n \, s_n} = \degb{2 \, A + t_n + t_{n+1}} \\
\leq \max(\deg A, \deg t_n, \deg t_{n+1}) = \max(\deg A, \deg t_n).
\end{multline*}
So if \(\deg t_n \geq \deg A\) then \(\deg t_{n+1} + 2 \leq \deg t_n\). Then after a finite number of steps we must have \(\deg t_{n+j} < \deg A\) (actually \(\deg t_{n+j} + 2 \leq \deg A\)). And if \(\deg t_n < \deg A\), then clearly also \(\deg t_{n+1} < \deg A\) (actually \(\deg t_{n+1} + 2 \leq \deg A\)).

So we may now assume \(\deg t_n < \deg A\) for all \(n\) large enough.

Next, if \(t_{n+1} = 0\), then \(s_n \, s_{n+1} = D - A^2\) and \(\degb{D - A^2} < \deg A\) (see Proposition \ref{completion-of-square}). This implies \(\deg s_n + \deg s_{n+1} < \deg A\), so clearly \(\deg s_{n+1} < \deg A\), and trivially \(-\infty = \deg t_{n+1} < \deg s_{n+1}\), hence \(\alpha_{n+1}\) is \(\sigma\)-reduced.

If on the other hand \(t_{n+1} \neq 0\), then \(s_n \, s_{n+1} = D - A^2 - 2 \, A \, t_{n+1} - t_{n+1}^2\) and thus
\begin{multline*}
\deg s_n + \deg s_{n+1} = \max( \degb{D - A^2}, \deg A + \deg t_{n+1}, 2 \, \deg t_{n+1}) \\ = \deg A + \deg t_{n+1} < \deg A + \deg s_n
\end{multline*}
implies \(\deg s_{n+1} < \deg A\). If moreover \(\deg s_n < \deg A\) (if not, consider \(s_{n+2}\) and \(s_{n+1}\) instead), we also get \(\deg t_{n+1} < \deg s_{n+1}\) and so \(\alpha_{n+1}\) is \(\sigma\)-reduced.

All subsequent complete quotients then remain \(\sigma\)-reduced by Proposition \ref{sigma-reduced-all-complete-quotients}.
\end{proof}
\begin{rem}
From the proof, we easily deduce an effective bound for \(N\). The degree of \(t_n\) decreases by at least \(2\) in every step from \(t_1\), so at most \(\ifracBb{\deg t_1 - \deg A}{2}\) steps are required to arrive at \(\deg t_n < \deg A\). From there, we need only one or two additional steps to arrive at a \(\sigma\)-reduced complete quotient. So \(N \leq 3 + \ifracBb{\deg t_1 - \deg A}{2}\). This demonstrates the effectivity in Theorem \ref{thm-quadratic-laurent-series-cf-representation}.
\end{rem}

The \(\sigma\)-reduced case allows even simpler computation of the partial quotient:
\begin{rem}
If \(\alpha_n\) is \(\sigma\)-reduced, then we may use polynomial division of \(2 \, A\) by \(s_n\) to compute \(t_{n+1}\) (improving minimally upon \ref{cf-ats-bound-t1}):
\begin{equation*}
2 \, A = a_n \, s_n - (t_n + t_{n+1}),
\end{equation*}
as both \(\deg t_n, \deg t_{n+1} < \deg s_n\).
\end{rem}

\section{Periodicity and Pell equation}
\label{sec:org8d34125}
Let us now check the theorems given at the beginning of this chapter.

\begin{proof}[Proof of Theorem \ref{thm-optimised-sqrt-cq-representation}]
We expand upon Remark \ref{period-of-sqrt-d}, and work with \(A + \sqrt{D}\) instead of \(\sqrt{D}\). This changes only \(a_0\) and \(\alpha_0\). Of course \(A + \sqrt{D}\) has \(t_0 = 0\) and \(s_0 = 1\) which shows again (now using Proposition \ref{sigma-reduced-degree-prop}) that it is \(\sigma\)-reduced, hence also all complete quotients \(\alpha_n\) with \(n \geq 1\) are \(\sigma\)-reduced.

Then Theorem \ref{thm-optimised-sqrt-cq-representation} simply combines \eqref{quadr-repr-ats}, \eqref{cf-ats-recursion} (which follow from Proposition \ref{prop-quadr-repr-rs}) and Proposition \ref{sigma-reduced-degree-prop}.

Additionally, Theorem \ref{cf-berrys-thm} implies that \(\CF(\sqrt{D})\) is periodic \IFF \(\CF(A + \sqrt{D})\) is pure quasi-periodic, and both continued fraction have the same quasi-period length \(\QPL\). With Proposition \ref{cf-quasi-periodicity-complete-quotients} and Corollary \ref{maximal-degree-implies-quasi-period} it follows that \(\alpha_n = \frac{A + \sqrt{D}}{s_n}\) with \(s_n \in \units \K\) (i.e. \(\deg s_n = 0\)) holds \IFF \(\QPL \div n\) from minimality of the quasi-period length \(\QPL\).
\end{proof}

We give a few more details for
\begin{proof}[Proof of Corollary \ref{cor-pq-degree-periodicity}]
The degree inequalities were stated already in Corollary \ref{maximal-degree-implies-quasi-period} and follow from \(\deg a_n = \deg A - \deg s_n\). The corollary also says that \(\deg a_n = \deg A\) implies pure quasi-periodicity of \(\CF(A + \sqrt{D})\).
\end{proof}

\begin{proof}[Proof of Theorem \ref{thm-pellian-iff-cf-periodic}]
Set \(\alpha = \sqrt{D}\), and recall from Section \ref{sec:org932a99d} that (for \(n
\geq 1\))
\begin{equation*}
\sqrt{D} = \mfour{p_{n-1}}{p_{n-2}}{q_{n-1}}{q_{n-2}} \, \alpha_n \iff \alpha_n = (-1)^n \, \mfour{q_{n-2}}{-p_{n-2}}{-q_{n-1}}{p_{n-1}} \, \sqrt{D}
\end{equation*}
which we rewrite as
\begin{multline}
\label{cf-moebius-pell-denom}
\alpha_n
= \frac{q_{n-2} \, \sqrt{D} - p_{n-2}}{p_{n-1} - q_{n-1} \, \sqrt{D}}
= \frac{q_{n-2} \, \sqrt{D} - p_{n-2}}{p_{n-1} - q_{n-1} \, \sqrt{D}} \cdot \fracsame{p_{n-1} + q_{n-1} \, \sqrt{D}} \\
= \frac{D \, q_{n-1} \, q_{n-2} - p_{n-1} \, p_{n-2} + \sqrt{D} \left(p_{n-1} \, q_{n-2} - p_{n-2} \, q_{n-1}\right)}{p_{n-1}^2 - D \, q_{n-1}^2} \\
= \frac{(-1)^n \, (\dots) + \sqrt{D}}{(-1)^n \left(p_{n-1}^2 - D \, q_{n-1}^2\right)}
\end{multline}
so
\begin{equation}
\label{cf-sn-pell-eq}
s_n = {(-1)^n \left(p_{n-1}^2 - D \, q_{n-1}^2\right)}.
\end{equation}
Recall Theorem \ref{cf-berrys-thm} which states that periodicity and quasi-periodicity are equivalent in the current situation. So by Corollary \ref{maximal-degree-implies-quasi-period} (proved just above), it follows that \(\CF(\sqrt{D})\) is periodic \IFF for some \(n \geq 1\) we have \(\deg s_n = 0\) which means \((p_{n-1}, q_{n-1})\) solves the Pell equation \eqref{pellu}. 

On the other hand, we know that Pell solutions are  convergents (Proposition \ref{weak-pell-solutions-are-convergents}) and from the classification of convergents (Proposition \ref{cf-convergent-classification}) follows that every non-trivial solution of \eqref{pellu} has the shape \((p, q) = \mu \cdot (p_m, q_m)\) for some \(m \geq 0\) with \(\mu \in \units \K\) (because for a Pell solution \(p, q\) are coprime). This implies that \((p_m, q_m)\) likewise solves \eqref{pellu}, and then \(\deg s_{m+1} = 0\).
\end{proof}

\section{Torsion order and period length}
\label{sec:org61dfa9f}
Recall the notation from Chapter \ref{sec:org99a8e17}, and assume again that \(D\) is square-free. With \(2(g+1) = \deg D\), we get the following inequalities between the torsion order and the quasi-period length:
\begin{prop}
\label{prop-bounds-torsion-period-length}
Suppose \(\j{\OO} \in \Jac\) is torsion of order precisely \(m\), and let \(\QPL\) the quasi-period length of \(\CF(\sqrt{D})\). Then for \(g \geq 1\) we have the inequality\footnote{Note that the case \(g=0\) can easily be treated using Corollary \ref{deg-2-always-pellian}. See also Section \ref{sec:org46de0ba}.}
\begin{equation*}
g + \QPL \leq m \leq 1 + g \, \QPL
\end{equation*}
which for \(g = 1\) becomes the equality \(m = \QPL + 1\).
\end{prop}
\begin{proof}
Combining the knowledge from the proofs of Theorems \ref{thm-pellian-iff-torsion} and \ref{thm-pellian-iff-cf-periodic}, we know that the minimal \(n\) such that \eqref{convergent-divisor-equation} is satisfied with \(r = 0\) by \((p_{n-1}, q_{n-1})\) is exactly \(n = \QPL\), with \(m = \deg p_{n-1}\). So this \(m\) must be the torsion order of \(\j{\OO}\).

We then calculate, using \(1 \leq \deg a_i \leq g\) for \(i=1, \dots, l-1\) which holds by Corollary \ref{cor-pq-degree-periodicity},
\begin{align*}
m = \deg p_{l-1} = \deg a_0 + \deg q_{l-1} = g+1 + \deg q_{l-1} &\leq g+1 + (l-1) g = 1 + l \, g \\
& \geq g+1 + l-1 = l + g
\end{align*}
which yields the desired inequality. Clearly it collapses to an equality for \(g = 1\).
\end{proof}

So bounding the period length is as hard as bounding torsion.
\section{Period lengths over finite fields}
\label{sec:org4b81157}
We now give an (elementary) proof of Corollary \ref{cor-finite-field-always-periodic}, by showing that over a finite base field \(\K\) there are only finitely many possibilities for the \(\sigma\)-reduced complete quotients. As these form the tail of every continued fraction of a quadratic Laurent series, this means any repetition immediately implies periodicity. Of course we have to avoid characteristic \(2\) again.

\begin{prop}
\label{naive-period-length-bound}
Let \(\K = \F_q\) a finite field of odd characteristic, and recall that \(\deg D = 2d\). Then for a fixed \(D\), there are precisely
\begin{equation}
\label{sr-naive-period-bound}
\frac{q^{2 d }-1}{q+1}
\end{equation}
\(\sigma\)-reduced expressions of type \(\ifracBb{A + t + \sqrt{D}}{s}\).
\end{prop}

\begin{rem}
Note that the above counting does not yet take into account that we usually have the additional condition \(s \div D- r^2\). This further limits the number of possible complete quotients.
\end{rem}

\begin{proof}
For fixed \(e = \deg s\), there are \((q-1) \, q^e\) possibilities for \(s\), and as \(\deg t < \deg s\), there are \(q^e\) possibilities for \(t\). Summing over \(e\), we compute
\begin{equation*}
\sum_{e=0}^{d-1} (q-1) q^e \; q^e = (q-1) \, \frac{q^{2 d}- 1}{q^2 - 1} = \frac{q^{2 d }-1}{q+1}
\end{equation*}
using the formula for geometric sums.
\end{proof}

\begin{rem}
The above \eqref{sr-naive-period-bound} gives an elementary bound for the period length. Using our knowledge about quasi-periods, we could improve it further dividing by \(2/(q-1)\).
\end{rem}

But anyway we already have a far better bound for for the torsion order in the Jacobian (under the assumption that \(D\) is square-free), see Remark \ref{finite-field-torsion-bound}.

Then we can do much better:
\begin{cor}
If \(D\) is square-free, the quasi-period length is bounded by
\begin{equation*}
\QPL \leq m - g \leq (\sqrt{q} + 1)^{2g} - g.
\end{equation*}
\end{cor}

\section{Divisors of complete quotients}
\label{sec:org4e4485e}
We now wish to expand upon the results of Section \ref{sec:org5685823}, and make the connection between the convergent divisors and the continued fraction more explicit. This will be useful later to give an additional viewpoint on the reduction of continued fractions. See also \cite{berry-1990-periodicity-continued-fractions}, where it is shown that quasi-periodicity of arbitrary elements of \(\HEF \setminus \K(X)\) is equivalent to \(D\) being Pellian.

Recall the notation from Chapter \ref{sec:org99a8e17}, and the additional assumption that \(D\) is square-free.
Let \(\alpha = \frac{r + w \, Y}{s} \in \K(X,Y)\) an arbitrary element of the function field of the (hyper)elliptic curve \(\CC\) with \(r, s, w \in \K[X]\). Put \(\alpha_0 = \frac{r + w \, \sqrt{D}}{s} \in \laurinx \K\). We may assume \(\io{\alpha_0} \leq 0\), otherwise we simply pass to the inverse of \(\alpha\). 
We also require \(w, s \neq 0\) and  may of course assume \(\gcd(r, s, w) = 1\).

Then the finite poles of \(\alpha\) are zeroes of \(s\). So the divisor has the shape
\begin{equation*}
\Div \alpha = -\pd{Q_1} - \dots - \pd{Q_h} + \dots, \quad \text{} Q_i \in \CCa
\end{equation*}
with \(h \leq 2 \, \deg s\) and other poles only at infinity (the points \(O_\pm\)). 

We now generalise Lemma \ref{convergent-divisor-lemma} about the divisors of convergents of \(\sqrt{D}\) to rational functions on \(\CC\):

\begin{prop}
\label{general-convergent-divisor-lemma}
Let \((p, q) \in \Coset{\alpha_0}{\K}\) a convergent, then
\begin{equation}
\label{general-convergent-divisor-equation}
\Div(p - \alpha \, q) = -m \pd{O_-} - \pd{Q_1} - \dots - \pd{Q_h} + (m+h-e) \, \pd{O_+} + \pd{P_1} + \dots + \pd{P_e}
\end{equation}
where \(P_i \in \CCa\), \(m \geq 0\) and \(0 \leq e < h - \ord(\alpha_0) \leq h + m\).
\end{prop}
\begin{proof}
Set \(\phi = p - \alpha \, q\). Any finite poles (i.e. in \(\CCa\)) must be among the \(Q_i\) because \(\ord_P(p) \geq 0\) and  \(\ord_P(q) \geq 0\) imply
\begin{equation*}
\ord_P(\phi) \geq \min\left(\ord_P(p), \ord_P(\alpha) + \ord_P(q)\right) \geq \min(0, \ord_P(\alpha)).
\end{equation*}
From \((p, q)\) being a convergent, we know that \(\ord_{O_+}(\phi) = \io{\phi} > \deg q \geq 0\). As in \eqref{cv-plus-bigger-minus}, this implies with \(\io{\alpha_0 \, q} \leq 0\) that
\begin{equation*}
\ord_{O_-}(\phi) = \io{p + \alpha_0 \, q} = \io{p} = -\deg p = \io{\alpha_0} + \io{q}.
\end{equation*}
Hence \(m = - \ord_{O_-} \geq 0\).

With all possible poles determined, we can write \(\Div(\phi)\) as in \eqref{general-convergent-divisor-equation}, where possibly some of the \(P_i \in \CCa\) coincide with some \(Q_j\). The divisor must have degree \(0\), so  \(\ord_{O_+}(\phi) = m + h -e\), and 
\begin{equation*}
\deg q < m + h - e = \deg q - \io{\alpha_0} + h - e
\end{equation*}
implies \(e < h - \io{\alpha_0}\).
\end{proof}

We can make this even more precise for the canonical convergents \((p_n, q_n)\):
\begin{cor}
\label{cor-canonical-convergent-general-divisor}
Let \(\phi_n = p_n - \alpha \, q_n\), then
\begin{equation*}
\Div \phi_n = -(\deg p_n) \pd{O_-} - \pd{Q_1} - \dots - \pd{Q_h} + (\deg q_{n+1}) \pd{O_+} + \pd{P_1^n} + \dots + \pd{P_{e_n}^n}
\end{equation*}
where \(P^n_{i} \in \CCa\) (perhaps some coincide with a \(Q_j\)) and 
\begin{equation*}
e_n = \deg a_0 - \deg a_{n+1} + h \leq \deg a_0 + h - 1.
\end{equation*}
\end{cor}
\begin{proof}
We obtain the formula for \(e_n\) from 
\begin{equation*}
\deg q_{n+1} = \deg q_n + \deg a_{n+1} = \deg p_n + h - e_n = \deg q_n + \deg a_0 + h - e_n,
\end{equation*}
because the principal divisor \(\phi_n\) has degree \(0\).
\end{proof}

Via \eqref{cf-moebius-pell-denom}, we can now calculate the divisors of the complete quotients (thinking \(Y = \sqrt{D}\)):
\begin{cor}
Write \(\mathbf P^n = \pd{P^n_1} + \dots + \pd{P^n_{e_n}}\), then
\begin{multline*}
\Div \alpha_n = \Div\left( - \ifrac{\phi_{n-2}}{\phi_{n-1}} \right) \\
= (\deg p_{n-1}-\deg p_{n-2}) \, \pd{O_-} + (\deg q_{n-1} - \deg q_{n}) \, \pd{O_+} + \mathbf P^{n-2} - \mathbf P^{n-1} \\ 
= (\deg a_{n-1}) \, \pd{O_-} + (-\deg a_n) \, \pd{O_+} + \mathbf P^{n-2} - \mathbf P^{n-1} 
\end{multline*}
\end{cor}

Note how
\begin{align*}
\ord_{O_+}(\alpha_n) &= \io{\alpha_n} = \io{a_n} = - \deg a_n,  \\
\ord_{O_-}(\alpha_n) &= -\io{\cfsb{\alpha_n}} = -\io{a_{n-1}} = \deg a_{n-1}.
\end{align*}
This aligns with the observations in Section \ref{sec:org6896ed2}, in particular Remark \ref{cf-reduced-backwards-n} about the ``conjugate'' continued fraction expansion.

So the \(Q_i\) can no longer be seen directly in this divisor, but of course they could appear hidden among the \(P^{n-1}_i, P^{n-2}_i\).

\medskip

Let us now restrict to the case \(w = 1\) and \(s \div D - r^2\). This implies \(h \leq \deg s\) because now it is impossible for both a point and its conjugate to appear as a pole, and a self-conjugate point can appear at most as a pole of order \(1\) (assuming that \(D\) is square-free). 

If \(s \in \units\K\), then there are no finite poles, and we are essentially in the situation of Lemma \ref{convergent-divisor-lemma}.

\begin{rem}
\label{rem-general-convergent-divisor-translated-multiples}
Using \(\io{\alpha_0} \leq 0\), we may also assume that \(\deg r \leq d = \frac{1}{2} \deg D\) (otherwise we could subtract some multiple of \(s\) from \(r\) which does not change the subsequent complete quotients). This implies \(\io{\alpha_0} \geq \io{\sqrt{D}} - \io{s}\), so \(e < - \io{\alpha_0} \leq d + h - \deg s \leq d\) and hence \(e \leq g\), so we get
\begin{equation}
\label{jacobian-translate-point-multiples}
j(Q_1) + \dots + j(Q_h) + m \, j(O_-) = j(P_1) + \dots + j(P_e).
\end{equation}
We are thus representing a translate of the multiples of \(\OO\) as a sum of at most \(g\) points in the Jacobian.
\end{rem}

\begin{rem}
\label{convergent-divisor-rationality}
The divisor \(\pd{P_1} + \dots + \pd{P_e}\) is usually going to be a \(\K\)-rational divisor. Be aware that this does not mean that the \(P_i\) are defined over \(\K\). However they are defined over a field extension of degree at most \(e\) over \(\K\). So if \(e = 1\), the single point \(P_1\) is going to be defined over \(\K\). We will make use of this later in Sections \ref{sec:org9851fc8} and \ref{sec:org53fcc2b}.
\end{rem}

\chapter{Specialization of continued fractions}
\label{sec:orgd5f1900}
The first goal of this chapter is to explain and recover a theorem of van der Poorten (see Theorem 1 in \cite{poorten-1998-formal-power-series}, Theorem 2.1 in \cite{poorten-1999-reduction-continued-fractions} and Theorem 6 in \cite{poorten-2001-non-periodic-continued}) stating that the convergents of some \(\alpha\) modulo a prime number \(\pp\)  all arise by normalising and reducing the original convergents of \(\alpha\) (which is a Laurent series with rational coefficients).  

Here we actually prove this theorem (as Theorem \ref{convergent-reduction-surjective}) in the general setting of Laurent series defined over a discrete valuation ring (or its fraction field), once some natural conditions are satisfied.

Before we look at the convergents, we however need to understand what we mean by reducing convergents, and likewise continued fractions. For the latter, this immediately leads to a notion of good or bad reduction of polynomial continued fractions. In the case of good reduction of a continued fraction, van der Poorten's theorem becomes trivial, using the classification of convergents described in Chapter \ref{sec:org172eb73}. This suggests that the bad reduction case is more interesting. 

Understanding the reduction of the convergents also helps to understand reduction of the continued fraction better, and we will look at some simple cases at the end of the chapter. This goes already toward the calculation of the Gauss norms of the partial quotients and convergents. These will be further analysed for square roots of polynomials in the next chapter.

\section{Specialization of Laurent series}
\label{sec:org46a3df3}
\subsection{Discrete valuation rings}
\label{sec:orgb3fd29f}
We fix a discrete valuation ring \(\O\) with its unique (principal) maximal ideal \(\mm\). It produces two fields: the \emph{fraction field} \(K = \Fr(\O)\) and the \emph{residue field} \(k = \O/\mm\). In order to apply the theory from the preceding chapters, we require that \(\Char k \neq 2\), which implies \(\Char K \neq 2\) as well.

We denote the (non-archimedean) valuation of \(\O\) by \(\nu_0 : K \surject \Z \cup \{ \infty \}\). Recall that it satisfies
\begin{itemize}
\item \(\nu_0(x) = \infty \iff x = 0\),
\item \(\nu_0(x \, y) = \nu_0(x) + \nu_0(y)\) for all \(x, y \in \units K\),
\item \(\nu_0(x + y) \geq \min(\nu_0(x), \nu_0(y))\) for all \(x, y \in K\).
\end{itemize}

In the last point, we can replace ``\(\geq\)'' with ``\(=\)'' if \(\nu_0(x) \neq \nu_0(y)\).

Moreover we choose an uniformising parameter \(\uni\) (a generator of the maximal ideal \(\mm\) in \(\O\)), with satisfies \(\nu_0(\uni) = 1\). Recall
\begin{equation}
\label{dvr-valuation-defi}
\begin{aligned}
\O &= \{ x \in K \mid \nu_0(x) \geq 0 \},\\
\mm = \spann{\uni} &= \{x \in K \mid \nu_0(x) > 0 \},\\
\units \O &= \{ x \in K \mid \nu_0(x) = 0 \}.
\end{aligned}
\end{equation}

We get the reduction map \(\Redm : \O \to \O/\mm = k\); we usually write \(\Red{x} = \RedM{x}\) for more compact notation.

\begin{rem}
Note that by choosing a discrete non-archimedean valuation \(\nu_0\) on a given field \(K\), we get a discrete valuation ring \(\O\) through \eqref{dvr-valuation-defi}.
\end{rem}

For example, starting with \(K = \Q\) and some odd integer prime \(\pp\) with its corresponding \(\pp\)-adic valuation \(\nu_\pp\), one gets the localisation \(\O = \Z_{\spann{\pp}}\) of \(\Z\) at \(\pp\). In this case, the residue field \(k = \F_\pp\) is finite.

Another example would be \(K = \C(t)\) with a zero-order \(\ord_{t=t_0}\) (for some \(t_0 \in \C\)). Then \(\O = \C[t]_{\spann{t-t_0}}\) is a localisation of \(\C[t]\) at the prime ideal \(\spann{t-t_0}\), and \(t-t_0\) is a uniformising parameter. The residue field is now \(k = \C\), hence infinite. In this example we could actually replace \(\C\) by any field (of characteristic not \(2\)), even a finite field. The latter would make the residue field finite again.

\subsection{Gauss norms}
\label{sec:orge83be0b}
It is natural to extend such a valuation to polynomials; for absolute values this is called a \emph{Gauss norm}. In fact, we can extend the valuation even to a subset of Laurent series.

\begin{defi}
Define \(\nu : \laurinx K \to \Z \cup \{+\infty, -\infty\}\) by setting for \(u \in \laurinx K\), with \(u_n \in K\):
\begin{equation*}
\nub{u} = \nub{\lseries{N}{u}{n}} = \inf \{ \nu_0(u_n) \mid n \in \Z, n \leq N \}.
\end{equation*}
To avoid \(\nub{u} = -\infty\), we restrict to the subring
\begin{equation*}
\laurinx K_\nu = \{ u \in \laurinx K \mid \text{the } \nu_0(u_n) \text{ are bounded from below} \}.
\end{equation*}
\end{defi}
\begin{rem}
If \(x \in K\), note that because \(\nu_0\) is non-archimedean, \(u(x)\) converges \IFF \(\nu_0(u_n \, x^n) = \nu_0(u_n) + n \, \nu_0(x) \to +\infty\) as \(n \to \infty\). The boundedness condition ensures that \(u(x)\) converges for every \(x \in \mm\) (with \(\nu_0(x) > 0\)).
\end{rem}

\begin{prop}
\label{bounded-laurent-valuation}
\(\laurinx K_\nu\) is a ring, and the extended \(\nu\) is a discrete non-archimedean valuation on it.
\end{prop}
\begin{proof}
It suffices to check that \(\nu\) satisfies the usual properties of an ultrametric valuation on \(\laurinx K_\nu\). Then \(\laurinx K_\nu\) is automatically a ring (using the same arguments which show that \(\O\) defined as in \eqref{dvr-valuation-defi} is a ring).

It is also obvious that \(\nu\) is discrete because we take an infimum of a subset of \(\Z\) bounded from below.

Clearly, we have \(\nub{u} = \infty\) \IFF \(u = 0\).

Take \(u,v \in \laurinx K_\nu\) with
\begin{equation*}
u = \lseries{N}{u}{n}, \quad v = \lseries{M}{v}{m},
\end{equation*}

For the ultrametric inequality, let
\begin{equation*}
u + v = w = \lseries{\max(N,M)}{w}{l}.
\end{equation*}
Without loss of generality, one may assume \(N=M\), and then \(w_n = u_n + v_n\) for all \(n \leq N\):
\begin{multline*}
\nub{w} = \inf\{\nuOb{u_n + v_n} \mid n \leq N\}
\geq \inf\{\min(\nuOb{u_n}, \nuOb{v_n}) \mid n \leq N\} \\
= \min\left( \inf\{\nuOb{u_n} \mid n \leq N\}, \inf\{\nuOb{v_n} \mid n \leq N\} \right)
= \min(\nub{u}, \nub{v}).
\end{multline*}

For multiplicativity, let
\begin{equation*}
u\, v = w = \lseries{(N+M)}{w}{l}.
\end{equation*}
As \(\nu\) is invariant under multiplication with powers of \(X\), we may assume \(N=M=0\). From the definition of the Cauchy product
\begin{equation}
\label{cauchy-product-u-v-eq-w}
w_l = \sum_{n+m=l} u_n \, v_m
\end{equation}
it is obvious that \(\nub{w} \geq \nub{u} + \nub{v}\) must hold:
\begin{multline*}
\nu(w) = \inf\{\nu_0(w_l) \mid l \leq 0\} \geq \inf\left\{ \min(\nu_0(u_n) + \nu_0(v_m) \mid n + m = l ) \mid l \leq 0 \right\} \\
\geq \inf\left\{ \min(\nu(u) + \nu(v) \mid n + m = l ) \mid l \leq 0 \right\}
\geq \nu(u) + \nu(v).
\end{multline*}

Because \(\nu_0\) is discrete on \(K\), there exist \(n_0, m_0\) such that
\begin{equation*}
\nub{u} = \nuOb{u_{n_0}} \text{ and } \nub{v} = \nuOb{v_{m_0}}
\end{equation*}
and of course, we may choose \(n_0\) and \(m_0\) maximal. Then
\begin{equation*}
w_{n_0 + m_0} = \sum_{n+m = n_0 + m_0} u_n \, v_m = \sum_{n+m = n_0 + m_0, \atop n > n_0} u_n \, v_m  + u_{n_0} \, v_{m_0} + \sum_{n+m = n_0 + m_0, \atop m > m_0}  u_n \, v_m.
\end{equation*}
We have \(\nuOb{u_n} > \nub{u}\) for all terms with \(n > n_0\), hence the absolute value of the left sum is \(> \nub{u} + \nub{v}\). And we have \(\nuOb{v_m} > \nub{v}\) for all terms with \(m > m_0\), hence the absolute value of the right sum is \(> \nub{u} + \nub{v}\).

However, the middle term has absolute value \(\nuOb{u_{n_0}} + \nuOb{v_{m_0}} = \nub{u} + \nub{v}\), implying \(\nuOb{w_{n_0 + m_0}} = \nub{u} + \nub{v}\). It follows \(\nub{u\,v} \leq \nub{u} + \nub{v}\), and hence \(\nub{u\,v} = \nub{u} + \nub{v}\).
\end{proof}
\begin{rem}
Clearly, \(K \subset K[X] \subset \laurinx K_\nu\). For all \(x \in K\) we have \(\nu(x) = \nu_0(x)\), so henceforth we refer to \(\nu_0\) also as \(\nu\).
\end{rem}
\begin{rem}
Of course also \(\laurinx \O \subset \laurinx K_\nu\). Applying the reduction map on each coefficient, it extends naturally to
\begin{equation*}
\Redm : \O[X] \surject k[X], \qquad \Redm : \laurinx \O \surject \laurinx k.
\end{equation*}
For convenience, we use the same notation, including \(\Red{x} = \RedM{x}\) for \(x \in \laurinx \O\), and say that we \emph{reduce mod \(\nu\)} or \emph{specialize at \(\nu\)}.

In analogue to \eqref{dvr-valuation-defi}, we obviously get
\begin{align*}
\O[X] &= \{ u \in K[X] \mid \nub{u} \geq 0 \}, &
\mm[X] &= \{ u \in K[X] \mid \nub{u} > 0 \}, \\
\laurinx \O &= \{ u \in \laurinx K \mid \nub{u} \geq 0 \}, &
\laurinx \mm &= \{ u \in \laurinx K \mid \nub{u} > 0 \}.
\end{align*}
It is straightforward to check that \(\mm[X]\) respectively \(\laurinx \mm\) are the kernels of the (surjective) reduction map on \(\O[X]\) respectively \(\laurinx \O\) (consider the valuations of the coefficients of \(u\)). As \(k[X]\) is an integral domain, this implies that \(\mm[X]\) is a prime ideal of \(\O[X]\). And as \(\laurinx k\) is even a field, the ideal \(\laurinx \mm\) is a maximal ideal of \(\laurinx \O\).

Both \(\mm[X]\) and \(\laurinx \mm\) are obviously principal ideals in their respective ring, with generator \(\uni\) (the uniformising parameter of \(\nu_0\)).
\end{rem}
\begin{rem}
The fraction field of \(\laurinx \O\) is \(\laurinx K\). However \(\laurinx \O\) is \emph{not} a discrete valuation ring. It is not even a local ring, because \(\nu(u) = 0\) is not a sufficient condition for having \(u \in \units{\laurinx \O}\) (see Corollary \ref{laurent-inverse-bounded-corollary} below).

For example \(u = \pi + \inv X\) is not in \(\laurinx \mm\), but neither is it a unit of \(\laurinx \O\).
\end{rem}

\begin{defi}
We say for \(u \in \laurinx K\) that
\begin{itemize}
\item \(u\) is \emph{unbounded} if \(\nub{u} = -\infty\) i.e. \(u \not \in \laurinx K_\nu\),
\item \(u\) is \emph{bounded} if \(\nub{u} \neq -\infty\) i.e. \(u \in \laurinx K_\nu\),
\item \(u\) \emph{has \negval} if \(\nub{u} < 0\), in particular if it is unbounded,
\item \(u\) \emph{has \posval} if \(\nub{u} > 0\).
\end{itemize}
\end{defi}

For example, if \(u \in K[X]\) is a polynomial, it has \negval \IFF at least one of its coefficients has \negval; and it has \posval \IFF all its coefficients are either \(0\) or have \posval. Note the different logical operations: For \negval, we have \textbf{or}, for \posval we have \textbf{and}.

Let us now investigate how far away \(\laurinx K_\nu\) is from being a field (and \(\laurinx \O\) from being a discrete valuation ring). For example, the Laurent polynomial \(1 + u_{-1} \, \inv X\) with \(\nub{u_{-1}} < 0\) does not have a bounded inverse:
\begin{prop}
\label{laurent-inverse-bounded}
Let \(u \in \laurinx K_\nu\) with \(u_0 = \LC(u) \neq 0\) (so \(u \neq 0\)). Then \(\inv u \in \laurinx K_\nu\) if and only if \(\nub{u} = \nub{u_0}\).
\end{prop}
\begin{proof}
If \(u\) has a bounded inverse, we have \(\LC(\inv u) = \ifrac{1}{u_0}\) with \(\nub{\ifrac{1}{u_0}} \geq \nub{\inv u} = -\nub{u}\), hence \(\nub{u} \geq \nub{u_0}\). But by definition \(\nub{u_0} \geq \nub{u}\), so it follows \(\nub{u_0} = \nub{u}\).

Conversely, assume \(\nub{u_0} = \nub{u}\); dividing \(u\) by \(u_0\) and \(X^{-\ord u}\) (both are bounded), we may without loss of generality write \(u = 1 - v\) for \(v \in \laurinx \O\) with \(\nub{v} \geq 0\) and \(\io{v} > 0\) (so actually \(v \in \powerseriesinvx \O\) is a power series in \(\inv X\) without constant coefficient). Then
\begin{equation*}
\inv{u} = \frac{1}{1-v} = \sum_{j=0}^\infty {v}^j = \lseries{0}{w}{m}
\end{equation*}
converges in \(\laurinx K\). Only finitely many \({v}^j\) (always with \(\nub{v^j} \geq 0\)) contribute to each \(w_m\), so clearly \(\nub{w_m} \geq 0\) for all \(m\), and \(\inv u\) is bounded.
\end{proof}
\begin{cor}
\label{laurent-inverse-bounded-corollary}
Let \(u \in \laurinx \O \setminus \{0\}\). Then \(\inv u \in \laurinx \O\) (i.e. \(u \in \units{\laurinx \O}\)) \IFF \(\LC(u) \in \units \O\).
\end{cor}
\begin{proof}
If \(\inv u \in \laurinx \O\), then both \(\nu(u) \geq 0\) and \(-\nu(u) = \nub{\inv u} \geq 0\) hence \(\nu(u) = 0\). By Proposition \ref{laurent-inverse-bounded} follows \(\nu(\LC(u)) = 0\), i.e. \(\LC(u) \in \units \O\).

Conversely, if \(\LC(u) \in \units \O\), then \(\nu(\LC(u)) = 0\) and so we clearly have \(\nu(u) = 0\). Then Proposition \ref{laurent-inverse-bounded} implies \(\inv u \in \laurinx K_\nu\). With \(\nub{\inv u} = 0\) we obtain \(\inv u \in \laurinx \O\) as desired. 
\end{proof}

\subsection{Criterion for bounded square roots}
\label{sec:org22171ff}

In the next chapter, we will be particularly interested in the specialization of Laurent series which are square roots of polynomials. Proposition \ref{laurent-sqrt-d} already describes how to construct square roots that lie in \(\laurinx K\), we now give additional conditions which are sufficient to have the square root lie in \(\laurinx K_\nu\).

For a counterexample, take \(u = 1 + u_{-1} \, \inv X\) where \(u_{-1} \in K, \; \nub{u_{-1}} < 0\): then it is easy to see that \(\nub{u} = -\infty\).

\begin{prop}
\label{laurent-sqrt-bounded}
Let \(u \in \laurinx K_\nu\) such that \(\sqrt{u} \in \laurinx K\) and \(\nub{u} = \nub{\LC(u)}\). Then \(\sqrt{u} \in \laurinx K_\nu\), i.e. \(\sqrt{u}\) is bounded.
\end{prop}
\begin{proof}
Recall that \(u_0 = \LC(u)\) must be a square, and \(\io{u}\) must be even. We may thus divide \(u\) by \(u_0\) and an appropriate even power of \(X\) (because both are squares and bounded), and assume \(u = 1+v\) where \(v \in \powerseriesinvx \O\), i.e. \(\nub{v} \geq 0\), and \(\io{v} > 0\).

Hence
\begin{equation*}
\sqrt{u} = \sqrt{1+v} = \sum_{j=0}^{\infty} \binom{1/2}{j} \, {v}^j = \lseries{0}{w}{m}
\end{equation*}
converges in \(\laurinx K\). By the hypothesis \(\Char k \neq 2\), we have \(\nub{2} = 0\), so \(\nub{\binom{1/2}{j}} \geq 0\) (see also Lemma \ref{lemma-binomial-half}). As \(\lim_{j\to\infty} \io{{v}^j} = \lim_{j\to\infty} j \, \io{v} = -\infty\), only a finite number of \(\binom{1/2}{j} \, {v}^j\), each having \(\nub{\cdot} \geq 0\), influence each \(w_m\). Hence \(\nub{w_m} \geq 0\) for all \(m\), and \(\sqrt{u}\) is bounded.
\end{proof}

\section{Specialization of polynomial continued fractions}
\label{sec:org890b4d5}
Given \(\alpha \in \laurinx \O\), we can on the one hand see it as element of \(\laurinx K\), or reduce it to \(\Red{\alpha} \in \laurinx k\). For each, one gets a continued fraction over \(K[X]\) respectively \(k[X]\). If one is \emph{lucky}, then \(\CF(\alpha)\) has all ``data'' defined over \(\O\), so one can apply \(\Redm\), and ask: do \(\CF\) and \(\Redm\) commute?

The answer is yes, so the obstacle lies in \(\CF(\alpha)\) not having all data defined over \(\O\).
\medskip

Let us fix notations for the rest of the chapter:
Let \(\alpha \in \laurinx \O\) with \(\LC(\alpha) \in \units \O\) and \(\io{\alpha} \leq 0\), so that \(\alpha\) has a non-zero polynomial part. It has a continued fraction expansion \(\CF(\alpha)\) over \(K[X]\), with complete quotients \(\alpha_n \in \laurinx K\), partial quotients \(a_n \in K[X]\) and canonical convergents \((p_n, q_n) \in K[X]^2\), satisfying
\begin{align*}
\alpha &= [a_0, a_1, \dots],&
\alpha_n &= [a_n, a_{n+1}, \dots],&
p_n/q_n &= [a_0, \dots, a_n].
\end{align*}

For the \emph{specialization}, we set \(\gamma = \Red{\alpha} \in \laurinx k\). The condition \(\LC(\alpha) \in \units \O\) ensures \(\io{\gamma} = \io{\alpha} \leq 0\). Of course \(\gamma\) has a continued fraction expansion \(\CF(\gamma)\) with complete quotients denoted \(\gamma_n \in \laurinx k\) and partial quotients denoted \(c_n \in k[X]\). The canonical convergents of \(\gamma\) are written as \((u_n, v_n) \in k[X]^2\) to distinguish them easily, and they satisfy
\begin{align*}
\gamma &= [c_0, c_1, \dots],&
\gamma_m &= [c_m, c_{m+1}, \dots],&
u_m/v_m &= [c_0, \dots, c_m].
\end{align*}
\subsection{Good reduction}
\label{sec:org3baa196}
To answer the question about ``commuting'', we want to apply the reduction map on the complete quotients, motivating the following definition:
\begin{defi}
\label{def-cf-good-reduction}
We say that \(\CF(\alpha)\) has \emph{good reduction} at \(\nu\) if for all \(n \geq 0\)
\begin{equation*}
\alpha_n \in \laurinx \O \text{ and } \Red{\alpha_n} = \gamma_n.
\end{equation*}
\end{defi}

It turns out the second condition is a consequence of the first, and that it is also possible to describe good reduction in terms of the partial quotients:
\begin{thm}
\label{cf-good-red-partial-quotients}
\TFAE
\begin{enumerate}
\item \(\CF(\alpha)\) has good reduction.
\item \(\alpha_n \in \laurinx \O\) for all \(n \geq 0\).
\item \(a_n \in \O[X]\) and \(\LC(a_n) = \LC(\alpha_n) \in \units \O\) for all \(n \geq 0\).
\item \(\deg a_n = \deg c_n\) for all \(n \geq 0\).
\end{enumerate}
\end{thm}
\begin{rem}
For \(n=0\) we had \(\LC(\alpha_0) \in \units \O\) as a hypothesis.
\end{rem}
We begin to prove the theorem with the following observation:
\begin{rem}
If \(\alpha_n \in \laurinx \O\), then clearly \(a_n = \gauss{\alpha_n} \in \O[X]\).
\end{rem}

Next, let us show that \(a_n \in \O[X]\) cannot be a sufficient condition for good reduction:
\begin{prop}
\label{cf-good-red-leading-coeffs}
Let \(n \geq 0\). If \(\alpha_n \in \laurinx \O\), then \(\alpha_{n+1} \in \laurinx \O\) \IFF \(\LC(a_{n+1}) = \LC(\alpha_{n+1}) \in \units \O\).
\end{prop}
\begin{proof}
By Definition \ref{cf-complete-quotients}, we have \(\inv{\alpha_{n+1}} = \alpha_n - a_n \in \laurinx \O\), and clearly \(\LC(\inv{\alpha_{n+1}}) = \inv{ \LC(\alpha_{n+1})} \in \units \O\) \IFF \(\LC(\alpha_{n+1}) \in \units \O\).

Then the statement follows from Corollary \ref{laurent-inverse-bounded-corollary} applied to \(u = \inv{\alpha_{n+1}}\).
\end{proof}

This allows to show that the second condition in Definition \ref{def-cf-good-reduction} is an automatic consequence of the first condition:
\begin{prop}
\label{cf-good-red-second-condition-redundant}
Let \(n \geq 0\). If \(\alpha_n, \alpha_{n+1} \in \laurinx \O\) and \(\Red{\alpha_n} = \gamma_n\), then \(\Red{\alpha_{n+1}} = \gamma_{n+1}\).
\end{prop}
\begin{proof}
Clearly \(\Red{\alpha_n} = \gamma_n\) implies \(\Red{a_n} = c_n\), and by Propositions \ref{laurent-inverse-bounded} and \ref{cf-good-red-leading-coeffs} we have \(\alpha_{n+1} \in \units{\laurinx \O}\). Hence
\begin{equation*}
\inv{\gamma_{n+1}} = \gamma_n - c_n = \Red{\alpha_n} - \Red{a_n} = \Red{\inv{\alpha_{n+1}}} = \inv{\Red{\alpha_{n+1}}}
\end{equation*}
which implies \(\gamma_{n+1} = \Red{\alpha_{n+1}}\) as desired.
\end{proof}

\begin{rem}
\label{good-reduction-preserve-degrees}
If \(\LC(\alpha_n) \in \units \O\) and \(\Red{\alpha_n} = \gamma_n\), we have \(\io{\alpha_n} = \io{\gamma_n} \leq 0\) (\(< 0\) for \(n \geq 1\)), and hence \(\deg a_n = \deg c_n\).
\end{rem}

Let us now describe good reduction in terms of the partial quotients; for this we first have a look at the convergents:
\begin{prop}
\label{cf-good-red-convergents-bounded}
Let \(n \geq 0\) and suppose \(a_j \in \O[X]\) for \(j = 0, \dots, n\) and \(\LC(a_j) \in \units \O\) for \(j = 1,\dots, n\). Then \(p_n, q_n \in \O[X]\) and moreover \(\ifrac{p_n}{q_n} \in \laurinx \O\).
\end{prop}
\begin{proof}
The statement \(p_n, q_n \in \O[X]\) follows directly from the recursion formulas for the canonical convergents \eqref{canonical-convergent-recursion}. And the product formula for the leading coefficients \eqref{convergents-leading-coeff} implies
\begin{equation*}
\nub{\LC(q_n)} = \sum_{j=1}^n \nub{\LC(a_j)}.
\end{equation*}
But then \(\nub{\LC(a_j)} = 0\) for \(j = 1, \dots, n\) implies \(\nub{\LC(q_n)} = \nub{q_n} = 0\). So by Corollary \ref{laurent-inverse-bounded-corollary} we have \(q_n \in \units{\laurinx{\O}}\), hence \(\ifrac{p_n}{q_n} \in \laurinx{\O}\).
\end{proof}

We conclude this section by proving the equivalence of the alternative characterisations of good reduction.
\begin{proof}[Proof of Theorem \ref{cf-good-red-partial-quotients}]
Equivalence of 1. and 2. is a consequence of Proposition \ref{cf-good-red-second-condition-redundant} above. 

Next, 2. implies 3. by Proposition \ref{cf-good-red-leading-coeffs}.

Conversely, 3. implies 2.: Let \(m \geq 0\) and recall that \(\io{\ifrac{p_{m,n}}{q_{m,n}} - \alpha_m} > 2 \, \deg{q_{m,n}}\) from Proposition \ref{cf-expansion-yields-convergents}. Moreover, we have \(\ifrac{p_{m,n}}{q_{m,n}} \in \laurinx \O\) by Proposition \ref{cf-good-red-convergents-bounded}, so the \emph{first} coefficients of \(\alpha_m\) are also in \(\O\). As \(\limn \deg{q_{m,n}} = \infty\), we cover all coefficients, and thus \(\alpha_m \in \laurinx \O\).
\pagebreak

Next,  1. and 3. imply \(\Red{\alpha_n} = \gamma_n\), hence \(\Red{a_n} = c_n\) and \(\LC(a_n) \in \units \O\). The latter is equivalent to \(\deg a_n = \deg \Red{a_n}\), so we get \(\deg a_n = \deg c_n\). 

Finally 4. implies 2.: by Proposition \ref{bad-reduction-minimal-pole} (below, but independent of this theorem) there exists \(n\) with \(\deg a_n < \deg c_n\) if 2. is violated.
\end{proof}

\begin{rem}
So continued fraction expansion and specialization commute as soon as the partial quotients are defined over \(\O\) and do not ``drop degree'' on reduction, or even simpler, the degrees of the partial quotients match.
\end{rem}

Theorem \ref{thm-vdp-intro} of van der Poorten becomes almost trivial in this case:
\begin{cor}
\label{cf-good-red-convergents-reduction}
If \(\CF(\alpha)\) has good reduction, then for all \(n \geq 0\) we have \(u_n = \Red{p_n}\) and \(v_n = \Red{q_n}\) which by the classification of convergents (Proposition \ref{cf-convergent-classification}) implies that all convergents of \(\gamma\) are obtained by reducing convergents of \(\alpha\).
\end{cor}
\begin{proof}
We can think of \(p_n\) and \(q_n\) as polynomials in \(\Z[a_0, \dots, a_n]\) (see Proposition \ref{definition-convergents-matrix}). Of course \(u_n\) and \(v_n\) are obtained by replacing \(a_j\) with \(c_j\) in those polynomials. But \(c_j = \Red{a_j}\) for all \(j \geq 0\), so \((u_n, v_n) = (\Red{p_n}, \Red{q_n})\).

An arbitrary convergent of \(\gamma\) has perhaps an additional polynomial factor in \(k[X]\) which we can however lift to a polynomial of same degree in \(K[X]\). Because we have \(\deg a_{n+1} = \deg c_{n+1}\), multiplying \((p_n, q_n)\) with this polynomial still produces a convergent of \(\alpha\).
\end{proof}
\subsection{Bad reduction}
\label{sec:orgee082e8}
\begin{defi}
\label{def-cf-bad-reduction}
The opposite of good reduction of \(\CF(\alpha)\) is obviously \emph{bad reduction} of \(\CF(\alpha)\), by which we mean that there exists \(n \geq 1\) such that \(\alpha_n \not\in \laurinx \O\) (i.e. \(\nub{\alpha_n} < 0\), so there is a coefficient with \negval).
\end{defi}

The results for good reduction are still useful in this case, for example Propositions \ref{cf-good-red-leading-coeffs} and \ref{cf-good-red-second-condition-redundant} can be applied until we arrive at the complete quotient with bad reduction. They should also give an initial idea of what could go wrong in the case of bad reduction.

\begin{prop}
\label{bad-reduction-minimal-pole}
Suppose \(\CF(\alpha)\) has bad reduction and let \(n\) minimal with \(\alpha_n \not\in \laurinx \O\). Then in fact \(\nub{\LC(\alpha_n)} < 0\), i.e. \(\alpha_n\) has \negval in the leading coefficient.

If \(\gamma_n\) is defined, then \(\deg c_n > \deg a_n\) and \(\alpha_n\) is unbounded.
\end{prop}
\begin{proof}
The first statement is an immediate consequence of Proposition \ref{cf-good-red-leading-coeffs}: by minimality \(\alpha_{n-1} \in \laurinx \O\), so \(\nub{\LC(\alpha_n)} \neq 0\). But \(\nub{\LC(\alpha_n)} > 0\) is impossible because \(\inv{\LC(\alpha_n)} = \LC(\inv{\alpha_n}) = \LC(\alpha_{n-1} - a_{n-1}) \in \O\).

Now assume \(\gamma_n\) is defined:
As we have \(\alpha_0, \dots, \alpha_{n-1} \in \laurinx \O\) (we could say we have ``good reduction up to \(\alpha_{n-1}\)''), we certainly have \(\Red{\alpha_{n-1}} = \gamma_{n-1}\) using Proposition \ref{cf-good-red-second-condition-redundant} inductively. But by Proposition \ref{laurent-inverse-bounded} we have \(\LC(\alpha_{n-1} - a_{n-1}) \in \mm\), hence 
\begin{equation*}
\deg a_n = - \io{\alpha_{n}} = \io{\alpha_{n-1} - a_{n-1}} < \io{\gamma_{n-1} - c_{n-1}} = -\io{\gamma_n} = \deg c_n.
\end{equation*}

In particular \(\gamma_{n-1} - c_{n-1} \neq 0\) which implies \(\nub{\alpha_{n-1} - a_{n-1}} = 0\). But as the leading coefficient is in \(\mm\), Proposition \ref{laurent-inverse-bounded} implies that the inverse \(\alpha_n\) is unbounded.
\end{proof}

\begin{rem}
If we are using the computation scheme with \(t_n\) and \(s_n\) from Chapter \ref{sec:orge1a60a6} and we are already in the \(\sigma\)-reduced case, the \negval in the leading coefficient of \(\alpha_n\) corresponds to \posval in the leading coefficient of \(s_n\).

Unless \(\gamma\) is rational,\footnote{In the case \(\alpha = \sqrt{D}\) the reduction \(\gamma = \sqrt{\Red{D}}\) clearly is rational \IFF \(\Red{D}\) is a square.} \(\gamma_n\) is of course always defined.
\end{rem}

\subsection{Reduction and normalisation of continued fractions}
\label{sec:orgea2cbbe}
We can extend the reasoning of this section also to an arbitrary Laurent series \(\alpha \in \laurinx K_\nu\), as long as \(\alpha\) is bounded and satisfies \(\nu(\alpha) = \nu(\LC(\alpha))\) and \(\io{\alpha} \leq 0\). If these requirements are met, we can just divide \(\alpha\) by \(\LC(\alpha)\), or some \(g \in \units K\) with \(\nu(g) = \nu(\alpha)\). For the new series, we can apply the above results.

Of course reduction here must always be preceded by normalisation. But for example the existence of unbounded complete quotients is invariant under normalisation (see Proposition \ref{cf-scalar-multiplication} about multiplying a continued fraction with a constant), and is characteristic for bad reduction.

We will revisit these issues later, first we need to study the reduction of the convergents in more detail.

\section{Normalisation and reduction of convergents}
\label{sec:orgd1cf6c9}
In the case of good reduction of the continued fraction, we were able to simply reduce the canonical convergents. In the case of bad reduction of the continued fraction, we cannot expect the canonical convergents to be polynomials defined over \(\O\), so we need to normalise them first.

In other words, we wish to extend the reduction map in a useful way to all of \(K[X]\) (or even \(\laurinx K\)) by normalising to valuation \(0\) before reducing. Of course, extending the reduction map \(\O \to k\)  in this way from \(\O\) to \(K\) is not so useful. But for polynomials and Laurent series, there are usually several coefficients, so thinking projectively makes sense. For obvious reasons, this works only for bounded Laurent series.

\begin{defi}
Let \(u \in \mino{\laurinx{K}_\nu}\), and recall that \(\uni\) is a uniformising parameter of \(\O\) satisfying \(\nu(\uni) = 1\). Define the \emph{normalisation} \(\normal{u}\) for \(u\) as
\begin{equation*}
\normal{u} = \uni^{-\nu(u)} \, u \in \laurinx \O.
\end{equation*}
Clearly, \(\nu(\normal{u}) = 0\). For completeness, we also set \(\normal{0} = 0\).

If \(u \in K\), then \(\normal{u} \in \O\), and if \(u \in K[X]\), then \(\normal{u} \in \O[X]\).

We denote the composition of reduction and normalisation by
\begin{equation*}
\Redn{u} = \RedM{\normal{u}}.
\end{equation*}
\end{defi}

Before we start normalising convergents, we need to check that the normalisation factor is the same for the numerator and the denominator -- otherwise we are unable to normalise the convergent as a whole:
\begin{prop}
\label{normalise-convergents-preparation}
Suppose \(\io{\RedM{\alpha}} \leq 0\), and let \((p, q) \in \Batest{K}\) a rational approximation with \(\io{p - \alpha \, q} > 0\). Set \(g = \uni^{\nu(q)} \in K\). 

Then  \((p, q) = g \cdot(\normal{p}, \normal{q})\) and in particular \(\nub{p} = \nub{q} = \nub{g}\).
\end{prop}
\begin{proof}
By definition, we have \(q = g \, \normal{q}\), and \(\nub{q} = \nub{g}\).
The condition \(\io{p - \alpha \, q} > 0\) implies \(p = -\gauss{\alpha \, q} = - g \, \gauss{\alpha \, \normal{q}}\).
Let \(p' = - \gauss{\alpha \, \normal{q}} \in \O[X]\) with \(p = g \, p'\).

It remains to show \(p' = \normal{p}\):
Indeed \(\io{p' - \alpha \, \normal{q}} > 0\) implies \(\io{\RedM{p'} - \RedM{\alpha} \, \RedM{\normal{q}}} > 0\). But \(\io{\RedM{\alpha}\, \RedM{\normal{q}}} \leq 0\) by hypothesis, so also \(\io{\RedM{p'}} \leq 0\). This means \(\RedM{p'} \neq 0\), or \(\nub{p'} = 0\), hence \(p' = \normal{p}\) as desired.
\end{proof}

\begin{cor}
\label{normalise-convergents}
Every convergent and best-approximation \((p, q) \in \Baset{\alpha}{\K}\) (in particular the canonical convergents \((p_n, q_n)\)) satisfies \(\nub{p} = \nub{q}\).

Setting \(g_n = \uni^{\nu(q_n)}\) we get \((p_n, q_n) = g_n \cdot (\normal{p_n}, \normal{q_n})\).
\end{cor}
\begin{rem}
For \(n = -1\) we have \(q_{-1} = 0\) and \(p_{-1} = 1\). We just set \(g_{-1} = 1\), as no normalisation is required.
\end{rem}

We finally state and prove the generalised version of Theorem \ref{thm-vdp-intro} on the reduction of convergents by van der Poorten. First we check that convergents remain convergents after reduction.
\begin{prop}
\label{convergent-reduction}
Let \((p, q) \in \Coset{\alpha}{K}\) a convergent. Then \(\io{\Redn{p} - \gamma \, \Redn{q}} > \deg q \geq \deg \Redn{q}\), so \((\Redn{p}, \Redn{q}) \in \Coset{\gamma}{k}\) is also a convergent.
\end{prop}
\begin{proof}
The important observation is that for \(\beta \in \laurinx \O\) one has \(\io{\Red{\beta}} \geq \io{\beta}\), and for \(b \in \O[X]\) one has \(\deg \Red{b} \leq \deg b\), hence
\begin{equation*}
\io{\Redn{p} - \gamma \, \Redn{q}} \geq \io{p - \alpha \, q} > \deg{q} \geq \deg {\Redn{q}}.
\end{equation*}
\end{proof}

We restrict now to the conveniently enumerated canonical convergents. We find:
\begin{cor}
\label{definition-convergent-reduction-map-lambda}
The reduction of a (normalised) convergent remains a convergent. In particular, there exists a (unique) map \(\lambda : \N_0 \to \N_0\) defined by
\begin{equation*}
\ifrac{\Redn{p_n}}{\Redn{q_n}} = \ifrac{u_{\lambda(n)}}{v_{\lambda(n)}}.
\end{equation*}
More precisely, for each \(n\) there exists \(h_n \in \mino{k[X]}\) such that
\begin{equation*}
\Redn{p_n} = h_n \, u_{\lambda(n)}, \qquad \Redn{q_n} = h_n \, v_{\lambda(n)}.
\end{equation*}
\end{cor}
\begin{proof}
The map \(\lambda\) is well defined: every convergent of \(\gamma\) is a multiple of a unique canonical convergent of \(\gamma\) by Corollary \ref{cf-convergent-classification}.
\end{proof}
Here one has to be careful, though: the factor \(h_n\) need \emph{not be constant}! We will investigate this closer for some special cases later. See also Example \ref{ex-cfp2-zero-pattern-deg6} in Section \ref{sec:org0b837d5}, where non-constant \(h_n\) in fact occur. 

\pagebreak
This possibility of non-constant factors make the following less obvious because \(\deg \Redn{q_n}\) may not be non-decreasing:
\begin{prop}
\label{reduced-convergents-increasing}
The map \(\lambda\) is non-decreasing (it need not be increasing).
\end{prop}
\begin{proof}
Let \(n < n'\) and set \(m = \lambda(n), m' = \lambda(n')\), hence \(\deg{q_n} < \deg{q_{n'}}\).

If \(\deg \Redn{q_n} \leq \deg \Redn{q_{n'}}\), Proposition \ref{cf-best-approx-classification} (Classification of best-approximations) for \(\gamma\) implies directly \(m \leq m'\).

If however \(\deg{\Redn{q_n}} \geq \deg{\Redn{q_{n'}}}\), then
\begin{align*}
  \io{\Redn{p_n} - \gamma \, \Redn{q_n}} &> \deg{\Redn{q_n}} \geq \deg{\Redn{q_{n'}}}, \\
  \io{\Redn{p_{n'}} - \gamma \, \Redn{q_{n'}}} &> \deg{q_{n'}} > \deg{q_n}  \geq \deg{\Redn{q_n}}.
\end{align*}
Eliminating \(\gamma\), one obtains
\begin{multline*}
\io{\Redn{p_n} \, \Redn{q_{n'}} - \Redn{p_{n'}} \, \Redn{q_n}} =
\io{(\Redn{p_{n}} - \gamma \, \Redn{q_{n}}) \, \Redn{q_{n'}} - (\Redn{p_{n'}} - \gamma \, \Redn{q_{n'}}) \, \Redn{q_{n}}} \\
\geq \min\left(\io{\Redn{p_n} - \gamma \, \Redn{q_n}} + \io{\Redn{q_{n'}}},
           \io{\Redn{p_{n'}} - \gamma \, \Redn{q_{n'}}} + \io{\Redn{q_n}} \right) > 0
\end{multline*}
which implies \(\ifrac{\Redn{p_n}}{\Redn{q_n}} = \ifrac{\Redn{p_{n'}}}{\Redn{q_{n'}}}\), hence \(m = m'\).
\end{proof}
\begin{rem}
\label{best-approx-reduction-strictly-increasing-q}
If \(m < m'\), then Proposition \ref{cf-best-approx-classification} immediately implies \(\deg{\Redn{q_n}} < \deg{\Redn{q_{n'}}}\).
\end{rem}

Now we are ready to prove that the map \(\lambda\) is in fact surjective, a result which appeared first \cite{poorten-1999-reduction-continued-fractions}, and with a slightly different proof in \cite{poorten-1999-reduction-continued-fractions} and \cite{poorten-2001-non-periodic-continued}. \footnote{Note that van der Poorten speaks of good reduction only for the hyperelliptic curve, not for the continued fraction.} Unfortunately, both proofs are somewhat confusing, perhaps because van der Poorten does not include an argument why the map \(\lambda\) should be non-decreasing. He already seems to assume that property in his implicit definition of \(\lambda\), where he uses an elaborate enumeration scheme.\footnote{Van der Poorten does not explicitly define the map \(\lambda\) as we do it here.}
\begin{thm}
\label{convergent-reduction-surjective}
All the (coprime) convergents of \(\gamma\) arise as reductions of convergents of \(\alpha\). In other words, the map \(\lambda : \N_0 \to \N_0\) is surjective.
Moreover, if \(n = \min \inv\lambda(m)\), then \(\deg {v_m} = \deg{q_n}\).
\end{thm}
\begin{proof}
First, we show that \(\lambda\) has finite fibres. Indeed, for \(n \geq 0\) and \(m = \lambda(n)\) we have by definition of \(\lambda\)
\begin{equation*}
\Redn{p_n} = h_n \, u_m, \quad \Redn{q_n} = h_n \, v_m \text{ where } h_n \in \mino{k[X]},
\end{equation*}
hence \(\deg{q_{n+1}} \leq \deg{v_{m+1}}\):
\begin{multline}
\label{convergent-reduction-quality-improvement}
\deg{v_{m+1}} \geq \io{h_n} + \deg{v_{m+1}} = \io{h_n} + \io{u_{m} - \gamma \, v_{m}} \\
= \io{\Redn{p_{n}} - \gamma \, \Redn{q_{n}}} \geq \io{p_{n} - \alpha \, q_{n}} = \deg{q_{n+1}}
\end{multline}
Here we use Proposition \ref{cf-expansion-yields-convergents} about the approximation quality of the canonical convergents \((u_m, v_m)\) and \((p_n, q_n)\) (first and last equality).

Now we know that \(\limn \deg{q_{n+1}} = \infty\) so for fixed \(m\) there can only by finitely many \(n\) which satisfy the inequality.

Because we know that \(\lambda\) is monotonous, we can prove its surjectivity by checking that there are no gaps in the image.

There is no gap at the start because \(v_0 = 1\) and \(q_0 = 1\) imply \(\lambda(0) = 0\).

For \(n \geq 0\), we either have \(\lambda(n) = \lambda(n+1)\) in which case there is no gap.

Otherwise \(m = \lambda(n) < \lambda(n+1) = m'\), and we need to show \(m' = m+1\). Again, by definition of \(\lambda\)
\begin{equation*}
\Redn{p_{n+1}} = h_{n+1} \, u_{m'}, \quad \Redn{q_{n+1}} = h_{n+1} \, v_{m'} \text{ where } h_{n+1} \in \mino{k[X]}.
\end{equation*}
and in particular
\begin{equation*}
\deg{v_{m'}} \leq \deg{h_{n+1}} + \deg{v_{m'}} = \deg{\Redn{q_{n+1}}} \leq \deg{q_{n+1}}.
\end{equation*}
But from \(m+1 \leq m'\) and \eqref{convergent-reduction-quality-improvement} follows also
\begin{equation*}
\deg{q_{n+1}} \leq \deg{v_{m+1}} \leq \deg{v_{m'}},
\end{equation*}
so these are actually equalities, and as desired \(m' = \lambda(n+1) = m+1 = \lambda(n) + 1\), so there is no gap. Note that \(n+1\) is the minimal element of the fibre \(\inv \lambda(m')\), and we have shown \(\deg q_{n+1} = \deg{v_{\lambda(n+1)}}\).
\end{proof}
\begin{rem}
\label{convergent-reduction-minimal-maximal-coprime}
Observe that \(\deg q_{n+1} = \deg v_{m+1}\) implies \(\deg h_{n+1} = 0\), and from \eqref{convergent-reduction-quality-improvement} also \(\deg h_n = 0\). Hence both for the minimal and maximal fibre element, the reduced convergent remains coprime.
\end{rem}
\begin{cor}
\label{cor-lambda-degree-sum}
Suppose that \(\inv\lambda(m) = \{ n, \dots, n+l \}\). Then
\begin{equation}
\label{eq-lambda-degree-sum}
\deg c_{m+1} = \sum_{i=1}^{l+1} \deg a_{n+i} = \deg a_{n+1} + \dots + \deg a_{n+l+1}.
\end{equation}
\end{cor}
\begin{proof}
Both \(n\) and \(n+l+1\) are the minimal elements of their respective fibres, hence \(\deg q_n = \deg v_m\) and \(\deg q_{n+l+1} = \deg v_{m+1}\). The degree formula for the convergents \eqref{convergent-deg-growth} then gives the desired relation between the degrees of the partial quotients.
\end{proof}

If the reduction is not rational, we also get an additional criterion for good reduction:
\begin{prop}
\label{cf-good-reduction-lambda-bijective}
If \(\gamma \not\in k(X)\), the map \(\lambda\) is bijective \IFF \(\CF(\alpha)\) has good reduction.
\end{prop}
\begin{proof}
First observe that by Proposition \ref{reduced-convergents-increasing}, the map \(\lambda\) is bijective \IFF it is the identity.

If \(\CF(\alpha)\) has good reduction, Corollary \ref{cf-good-red-convergents-reduction} implies that \(\lambda\) is the identity.

Conversely, if \(\lambda\) is the identity, then from Theorem \ref{convergent-reduction-surjective} we obtain \(\deg q_n = \deg v_n\) for all \(n\), which in turn implies \(\deg a_n = \deg c_n\) for all \(n\). Then by Theorem \ref{cf-good-red-partial-quotients} \(\CF(\alpha)\) has good reduction.
\end{proof}

\medskip

We conclude this section by pointing out that while the canonical convergents are usually not normalised, the convergents we get as solutions of the linear system in Section \ref{sec:org542977d} are in fact optimally normalised (even independently of the valuation):

\begin{prop}
\label{prop-convergents-hankel-determinants-normalised}
Let \(\alpha \in \laurinx \O\) and suppose that \(\gamma = \Red{\alpha} \neq 0\). Let \(n\) such that \(\pqmatrix_n\) has full rank, and let \((p, q)\) correspond to an element of the kernel computed from the minors of \(\pqmatrix_n\) as in Remark \ref{rem-cramers-rule}.

Then \(p, q \in \O[X]\). Moreover, if \(\deg q = \deg \Redn{q}\), we have \(\nu(q) = 0\).
\end{prop}
\begin{proof}
By hypothesis, the coefficients of the Laurent series \(\alpha\) are in \(\O\). The minors of \(\pqmatrix_n\) are polynomials in these coefficients, so clearly the coefficients of \(p\) and \(q\) are in \(\O\) too (recall that we need full rank so they do not all vanish).

The coefficients of \(\gamma\) are obtained by reducing those of \(\alpha\), hence the kernel elements of \(\Red{\pqmatrix_n}\) correspond to convergents of \(\gamma\). For example there is \((\Redn{p}, \Redn{q})\), and then \(\deg q = \deg \Redn{q}\) implies that \(\Red{\pqmatrix_n}\) has full rank as well, so we may compute a convergent using the minors. But of course the reduction map \(\Redm\) is a ring homomorphism, so this convergent is exactly \((\Red{p}, \Red{q})\), with \(\Red{q} \neq 0\). Then clearly \(\nu(q) = 0\).
\end{proof}

\section{Calculating valuations}
\label{sec:org7a10a6e}
Once we understand the structure of \(\lambda\) and the reduction of convergents thanks to Theorem \ref{convergent-reduction-surjective}, we can go further and attempt to compute the valuations (Gauss norms) for the partial quotients \(a_n\), the canonical convergents \(q_n\) and often even for the complete quotients \(\alpha_n\). In the next chapter, we will see how there arise rather simple patterns in the case \(\alpha = \sqrt{D}\) with \(\deg D = 4\). For now, we remain in the general case which makes things a bit more complicated. However we will thus understand better the obstacles for generalising the degree \(4\) case.
\subsection{Relating complete quotients with convergents}
\label{sec:org826607c}
In the following, we always assume \(\gamma = \Red{\alpha} \not\in k(X)\).
\begin{prop}
\label{introduce-theta-n}
Define for \(n \geq -1\)
\begin{equation}
\label{define-theta-n}
\vartheta_n = \normal{p_n} - \alpha \, \normal{q_n}.
\end{equation}
Then \(\vartheta_n \in \laurinx \O\) with \(\nub{\vartheta_n} = 0\), and \(\io{\vartheta_n} = \deg q_{n+1}\).

With \(g_n = \uni^{-\nu(q_n)}\), we may then write
\begin{equation}
\label{cf-alphan-moebius-relation-normalised}
\alpha_{n} = - \frac{g_{n-2} \, \vartheta_{n-2}}{g_{n-1} \, \vartheta_{n-1}}
\end{equation}
as a quotient of elements of \(\laurinx \O\) up to a normalisation factor.
\end{prop}
\begin{rem}
Note that \(\vartheta_{-1} = 1\) and \(\vartheta_0 = a_0 - \alpha\).
\end{rem}
\begin{proof}
By definition of normalisation, we have \(\normal{p_n}, \normal{q_n} \in \O[X]\), and \(\Redn{q_n} \neq 0\). Of course \(p_n - \alpha \, q_n = g_n \, \vartheta_n\), so \(\io{\vartheta_n} = \deg q_{n+1}\) is an immediate consequence of Proposition \ref{cf-expansion-yields-convergents} and \(\uni \in K\).

As we assume \(\alpha \in \laurinx \O\), this implies \(\vartheta_n \in \laurinx \O\). Moreover, \(\gamma = \Red{\alpha} \not\in k(X)\) implies \(\Red{\vartheta_n} \neq 0\), hence \(\nub{\vartheta_n} = 0\).

Finally, from Proposition \ref{definition-convergents-matrix} we obtain (see also \eqref{complete-quotient-convergent-quotient})
\begin{equation*}
\alpha_n = \frac{q_{n-2} \, \alpha - p_{n-2}}{-q_{n-1} \, \alpha + p_{n-1}} = -\frac{g_{n-2} \, (\normal{p_{n-2}} - \alpha \, \normal{q_{n-2}})}{g_{n-1} \, (\normal{p_{n-1}} - \alpha \, \normal{q_{n-1}})} = - \frac{g_{n-2} \, \vartheta_{n-2}}{g_{n-1} \, \vartheta_{n-1}}.
\end{equation*}
\end{proof}

So in order to understand whether \(\alpha_n\) is bounded, we need a criterion for when the \(\vartheta_n\) have a bounded inverse:
\begin{prop}
\label{spec-theta-lc-valuation}
\TFAE
\begin{itemize}
\item \(\inv{\vartheta_n} \in \laurinx K_\nu\),
\item \(\vartheta_n \in \units{\laurinx \O}\),
\item \(\io{\vartheta_n} = \io{\Red{\vartheta_n}}\),
\item \(\nub{\LC(\vartheta_n)} = 0\),
\end{itemize}
\end{prop}
\begin{proof}
By the previous Proposition, we have \(\vartheta_n \in \laurinx \O\) and \(\nub{\vartheta_n} = 0\). So by Proposition \ref{laurent-inverse-bounded} the inverse is bounded \IFF
\begin{equation*}
\nub{\LC(\vartheta_n)} = 0 \iff \Red{\LC(\vartheta_n)} \neq 0 \iff \io{\vartheta_n} = \io{\Red{\vartheta_n}}.
\end{equation*}
Finally, it is clear that if the inverse is bounded, then \(\nu(\inv{\vartheta_n}) = 0\), so it is in \(\laurinx \O\).
\end{proof}
\begin{rem}
Of course \(\inv{\vartheta_n} \in \laurinx K_\nu\) implies via \eqref{cf-alphan-moebius-relation-normalised} that also \(\alpha_{n+1} \in \laurinx K_\nu\).
\end{rem}

We use this to show that there are always infinitely many bounded complete quotients:
\begin{prop}
\label{cf-infinitely-bounded-complete-quotients}
Let \(m \in \N\), and set \(n = \min \inv\lambda(m)\). Then \(\inv{\vartheta_{n-1}} \in \laurinx K_\nu\), hence \(\alpha_n \in \laurinx K_\nu\).
\end{prop}
\begin{proof}
With Theorem \ref{convergent-reduction-surjective} follows from \(n\) being minimal in the fibre \(\inv\lambda(m)\) that \(\deg q_n = \deg v_m\), and \(\lambda(n-1) = m-1\). By Remark \ref{convergent-reduction-minimal-maximal-coprime}, we moreover know \(\Redn{q_{n-1}} = h_{n-1} \, v_{m-1}\) with \(h_{n-1} \in k\), hence
\begin{equation*}
\io{\Red{\vartheta_{n-1}}} = \io{\Redn{p_{n-1}} - \gamma \, \Redn{q_{n-1}}} = \io{u_{m-1} - \gamma \, v_{m-1}} = \deg v_m = \deg q_n = \ios{\vartheta_{n-1}},
\end{equation*}
so Proposition \ref{spec-theta-lc-valuation} implies that \(\vartheta_{n-1}\) has bounded inverse. Then \eqref{cf-alphan-moebius-relation-normalised} implies that \(\alpha_n\) is bounded.
\end{proof}
Note that the condition for \(\alpha_n\) bounded we give here is only sufficient, but not necessary.

\subsection{Fibre analysis of \(\lambda\)}
\label{sec:org84e8497}
Using the Lemmata for estimating valuations in quotients of Laurent/power series from Section \ref{sec:org5f19ed2} in the appendix, we now attack the problem of computing valuations by doing case analysis for the different sizes of the fibres of \(\lambda\), and the degrees of the partial quotients. This is successful mostly when we can read off the valuations (Gauss norms) from the leading coefficients.

The simplest case is the following, we get information on everything (recall that \(\nu(g_n) = \nu(q_n)\) for all \(n \geq 0\)):
\begin{prop}[Single element fibre]
\label{prop-single-element-fibre-analysis}
Let \(m \in \N\) such that \(\inv \lambda(m) = \{n\}\) has a single element. Then \(\alpha_{n+1}\) is bounded and
\begin{equation}
\label{eq-single-fibre-alpha-val}
\nu(\alpha_{n+1}) = \nu(\LC(\alpha_{n+1})) = \nu(a_{n+1}) =  \nu(g_{n-1}) - \nu(g_{n}).
\end{equation}
The normalised complete quotient reduces to
\begin{equation}
\label{eq-1elem-red-cq}
\Redn{\alpha_{n+1}} = \frac{h_{n-1}}{h_n} \, \gamma_{m+1} \quad \text{ with } h_{n-1}, h_n \in \units k,
\end{equation}
hence \(\deg a_{n+1} = \deg c_{m+1}\).

For the corresponding convergent we have
\begin{equation*}
\nu(g_{n+1}) = \nu(q_{n+1}) = \nu(\LC(q_{n+1})) = \nu(g_{n-1}).
\end{equation*}
\end{prop}
\begin{proof}
Both \(n\) and \(n+1\) are the minimal elements of their fibres, so Proposition \ref{cf-infinitely-bounded-complete-quotients} implies that both \(\vartheta_{n-1}, \vartheta_{n} \in \units{\laurinx \O}\). Hence \(\alpha_{n+1}\) is bounded, and \eqref{eq-single-fibre-alpha-val} follows from \eqref{cf-alphan-moebius-relation-normalised} and \(\nu(\vartheta_{n-1}) = \nu(\vartheta_{n}) = 0\).

Normalising and reducing \(\alpha_{n+1}\), we get
\begin{equation*}
\Redn{\alpha_{n+1}} = \RedM{\frac{g_{n}}{g_{n-1}} \, \alpha_{n+1}} = -\frac{\Red{\vartheta_{n-1}}}{\Red{\vartheta_{n}}} = - \frac{h_{n-1} \, (u_{m-1} - \gamma \, v_{m-1})}{h_{n} \, (u_{m} - \gamma \, v_{m})} = \frac{h_{n-1}}{h_n} \, \gamma_{m+1}.
\end{equation*}
Here \(h_{n-1}, h_n \in \units k\) by Remark \ref{convergent-reduction-minimal-maximal-coprime}.

Again using that \(n\) and \(n+1\) are minimal in their fibres, Theorem \ref{convergent-reduction-surjective} implies \(\deg \Redn{q_n} = \deg q_n\) and \(\deg \Redn{q_{n+1}} = \deg q_{n+1}\). This means \(\nu(g_n) = \nu(q_n) = \nu(\LC(q_n))\) and
\begin{equation*}
\nu(g_{n+1}) = \nu(q_{n+1}) = \nu(\LC(q_{n+1})) = \nu(\LC(a_{n+1})) + \nu(\LC(q_n)) = \nu(g_{n-1}).
\end{equation*}
For \(\deg a_{n+1} = \deg c_{m+1}\) see also Corollary \ref{cor-lambda-degree-sum}.
\end{proof}

If there is more than one element in the fibre, we can say a few things in general. However boundedness of the complete quotients cannot be determined a priori, except for the first and last complete quotient. But even if the complete quotients are bounded, the reduction of the normalisation is \emph{never} a complete quotient of \(\gamma\) as in the single element case of Proposition \ref{prop-single-element-fibre-analysis} above.
\begin{prop}[Multiple element fibre]
\label{prop-multiple-element-fibre-analysis}
Let \(m \in \N\) such that \(\inv \lambda(m) = \{n, n+1, \dots, n+l \}\) has \(l \geq 2\) elements. Then \(\alpha_{n+1}\) is unbounded and \(\alpha_{n+l+1}\) is bounded. The \(\alpha_{n+i+1}\) for \(1 \leq i < l\) can be bounded \emph{or} unbounded.

If some \(\alpha_{n+i+1}\) (for \(1 \leq i \leq l\)) is bounded, the reduction of the normalised complete quotient is a rational function (and a polynomial for \(i = l\), as \(h_{n+l} \in \units k\)):
\begin{equation*}
\Redn{\alpha_{n+i+1}} = - \frac{h_{n+i-1}}{h_{n+i}}.
\end{equation*}
In particular \(\io{\Redn{\alpha_{n+l+1}}} = - \deg h_{n+l-1}\). In this case, we also get
\begin{equation}
\label{eq-multiple-fibre-alpha-val}
\nu(\LC(\alpha_{n+i+1})) \geq \nu(a_{n+i+1}) \geq \nu(\alpha_{n+i+1}) = \nu(g_{n+i-1}) - \nu(g_{n+i}),
\end{equation}
and thus
\begin{equation}
\label{eq-multiple-fibre-qn-val}
\nu(\LC(q_{n+i+1})) \geq \nu(q_{n+i+1}) \geq \nu(g_{n+i-1}).
\end{equation}
\end{prop}

\begin{proof}
Here \(n\) and \(n+l+1\) are minimal in their fibre, so \(\vartheta_{n-1}, \vartheta_{n+l} \in \units{\laurinx \O}\) by Proposition \ref{cf-infinitely-bounded-complete-quotients}; and \(h_{n-1}, h_n, h_{n+l}\) are constant by Remark \ref{convergent-reduction-minimal-maximal-coprime}. Moreover, Theorem \ref{convergent-reduction-surjective} tells us that \(\deg q_n = \deg v_m\) and \(\deg q_{n+l+1} = \deg v_{m+1}\), from which we deduce
\begin{equation*}
\deg a_{n+1} + \dots + \deg a_{n+l+1} = \deg c_{m+1}
\end{equation*}
as in Corollary \ref{cor-lambda-degree-sum}.

Observe that \(\vartheta_n\) has an unbounded inverse because
\begin{equation*}
\io{\vartheta_n} = \deg q_{n+1} = \deg q_n + \deg a_{n+1} < \deg q_{n} + \deg c_{m+1} = \deg v_{m+1} = \io{\Red{\vartheta_n}}.
\end{equation*}
Hence \(\alpha_{n+1}\) is unbounded.
But \(\alpha_{n+l+1}\) is of course bounded by Proposition \ref{cf-infinitely-bounded-complete-quotients}, even if it need not have a bounded inverse. For the complete quotients in between, we cannot a priori say anything.

But assume that \(\alpha_{n+i+1}\) (where \(1 \leq i \leq l\)) \emph{is bounded}. Then it follows
\begin{equation*}
\Redn{\alpha_{n+i+1}} = \RedM{\frac{g_{n+i}}{g_{n+i-1}} \, \alpha_{n+i+1}} = -\frac{\Red{\vartheta_{n+i-1}}}{\Red{\vartheta_{n+i}}} = - \frac{h_{n+i-1} \, (u_{m} - \gamma \, v_{m})}{h_{n+i} \, (u_{m} - \gamma \, v_{m})} = -\frac{h_{n+i-1}}{h_{n+i}}.
\end{equation*}
As always \(\nu(\vartheta_i) = 0\), we may deduce \eqref{eq-multiple-fibre-alpha-val} directly from \eqref{cf-alphan-moebius-relation-normalised}, with the inequalities obvious from the definition of \(\nu\) on polynomials and Laurent series as infimum over the coefficients. With the recurrence relation \eqref{canonical-convergent-recursion}, we then get (again only in the bounded case)
\begin{equation*}
\nu(\LC(q_{n+i+1})) \geq \nu(q_{n+i+1}) \geq \min\left(\nu(a_{n+i+1}) + \nu(q_{n+i}), \nu(q_{n+i-1})\right) \geq \nu(g_{n+i-1}).
\end{equation*}
\end{proof}

Observe that \(\io{\Redn{\alpha_{n+l+1}}} = - \deg h_{n+l-1}\), while \(\io{\alpha_{n+l+1}} = -\deg a_{n+l+1} \neq 0\). Its inverse, and hence \(\alpha_{n+l}\), can be bounded only if \(h_{n+l-1}\) is non-constant.

\pagebreak

For a fibre with just two elements, we can under the simplest conditions precisely calculate the valuations:
\begin{prop}[Two element fibre]
\label{prop-two-element-fibre-analysis}
Let \(m \in \N\) such that \(\inv \lambda(m) = \{ n, n+1 \}\) has two elements. Then \(\alpha_{n+1}\) is unbounded, but \(\alpha_{n+2}\) is bounded, with
\begin{equation*}
\Redn{\alpha_{n+2}} = - \frac{h_{n+1}}{h_{n}}, \quad \text{ where } h_{n}, h_{n+1} \in \units k.
\end{equation*}
If moreover \(\deg a_{n+2} = 1\), then for the partial quotients we have
\begin{align}
\label{eq-2elem-pq1}
\nu(a_{n+1}) &= \nu(g_{n-1}) - \nu(g_n) - (1+\deg a_{n+1}) \, \nu(\LC(\vartheta_n)), \\
\label{eq-2elem-pq1-lc}
\nu(\LC(a_{n+1})) &= \nu(g_{n-1}) - \nu(g_n) - \nu(\LC(\vartheta_n)), \\
\label{eq-2elem-pq2}
\nu(\alpha_{n+2}) = \nu(a_{n+2}) &= \nu(g_n) - \nu(g_{n+1}), \\
\label{eq-2elem-pq2-lc}
\nu(\LC(a_{n+2})) &= \nu(g_n) - \nu(g_{n+1})  + \nu(\LC(\vartheta_n)), 
\end{align}
and for the convergents we have
\begin{align}
\label{eq-2elem-conv1}
\nu(g_{n+1}) = \nu(q_{n+1}) &= \nu(g_{n-1}) - (1+\deg a_{n+1}) \, \nu(\LC(\vartheta_n)), \\
\nu(\LC(q_{n+1})) &= \nu(g_{n-1}) - \nu(\LC(\vartheta_n)), \\
\label{eq-2elem-conv2}
\nu(g_{n+2}) = \nu(q_{n+2}) = \nu(\LC(q_{n+2})) &= \nu(g_n) + (1+\deg a_{n+1}) \, \nu(\LC(\vartheta_n)).
\end{align}
\end{prop}

\begin{proof}
The first part follows from Proposition \ref{prop-multiple-element-fibre-analysis}.
Here \(n\) and \(n+2\) are minimal in their fibre, so \(\vartheta_{n-1}, \vartheta_{n+1} \in \units{\laurinx \O}\), and \(h_{n-1}, h_n, h_{n+1}\) are all constant.

Now \(\io{\vartheta_n} = \deg q_{n+1}\), but \(\io{\Red{\vartheta_n}} = \deg v_{m+1} = \io{\vartheta_n} + \deg a_{n+2}\) because \(\deg c_{m+1} = \deg a_{n+1} + \deg a_{n+2}\). So the first \(\deg a_{n+2}\) coefficients of \(\vartheta_n\) vanish after reduction, and when assuming \(\deg a_{n+2} = 1\) we can apply the results of section \ref{sec:org5f19ed2} to compute the valuations. In particular note that \(\nu(\LC(\vartheta_n)) > 0\), while the next coefficient of \(\vartheta_n\) is in \(\units \O\).

With Proposition \ref{inverse-drop1-valuation-lemma} on the valuations of a quotient of Laurent series, we easily compute \eqref{eq-2elem-pq1} and \eqref{eq-2elem-pq1-lc} from the quotient presentation \eqref{cf-alphan-moebius-relation-normalised} of \(\alpha_{n+1}\). Of course \(a_{n+1}\) contains precisely the first \(1 + \deg a_{n+1}\) coefficients of \(\alpha_{n+1}\).

Then \eqref{eq-2elem-pq1} allows to compute
\begin{equation*}
\nu(a_{n+1} \, q_{n}) = \nu(g_{n-1}) - (1 + \deg a_{n+1}) \, \nu(\LC(\vartheta_{n})) < \nu(q_{n-1}).
\end{equation*}
This implies \eqref{eq-2elem-conv1} via \(q_{n+1} = a_{n+1} \, q_n + q_{n-1}\) and the ultrametric ``equality''. As \(n\) is minimal in the fibre, we have \(\deg q_n = \deg \Redn{q_n}\) and hence \(\nu(g_n) = \nu(q_n) = \nu(\LC(q_n))\), so
\begin{equation*}
\nu(\LC(q_{n+1})) = \nu(\LC(q_n)) + \nu(\LC(a_{n+1})) = \nu(g_{n-1}) - \nu(\LC(\vartheta_n)).
\end{equation*}

On the other hand, the first part of Lemma \ref{cauchy-valuation-lemma} applied to \eqref{cf-alphan-moebius-relation-normalised} gives \eqref{eq-2elem-pq2} and \eqref{eq-2elem-pq2-lc} -- there are just two coefficients in \(a_{n+2}\). By Theorem \ref{convergent-reduction-surjective}, we also know that \(\deg \Redn{q_{n+2}} = \deg q_{n+2}\), so we can compute the valuation of the convergent via the leading coefficient:
\begin{multline*}
\nu(g_{n+2}) = \nu(q_{n+2}) = \nu(\LC(q_{n+2})) = \nu(\LC(q_{n+1})) + \nu(\LC(a_{n+2})) \\
= \nu(g_{n-1})  - \nu(\LC(\vartheta_n)) + \nu(g_n) - \nu(g_{n+1}) + \nu(\LC(\vartheta_n))  \\
= \nu(g_n) + (1 + \deg a_{n+1}) \, \nu(\LC(\vartheta_{n})).
\end{multline*}
\end{proof}

\begin{rem}
The \(h_n\) and also the quotients \(h_{n-1}/h_n\) do not seem to follow any larger (obvious) patterns. If they are all constants, we locally -- in ``areas'' with only single element fibres -- observe patterns as in Proposition \ref{cf-scalar-multiplication}. But that is an unsurprising consequence of \eqref{eq-1elem-red-cq}.
\end{rem}

\begin{rem}
For \(\deg a_{n+2} > 1\), there is more than one coefficient of \(\vartheta_n\) that vanishes, and our reasoning which essentially boils down to geometric series arguments, breaks down. If we wanted to treat for example fibres \(\inv\lambda(m) = \{n, n+1, n+2\}\) with three elements, we get additional complications, as \(h_{n+1}\) can now be non-constant.
\end{rem}

\medskip

We have seen that the reduction of the normalisation of a bounded complete quotient of \(\alpha\) yields a complete quotient of \(\gamma\) \IFF we are at a single element fibre of \(\lambda\). Otherwise, it becomes a rational (or even polynomial) function.

We have also seen that the \(g_n\) do not change at the single element fibres. We will later investigate this closer for \(\CF(\sqrt{D})\) with \(\deg D = 4\) (see Theorem \ref{thm-genus1-zero-patterns}).

\chapter{Specialization of hyperelliptic continued fractions}
\label{sec:org5f9d2ce}
We now apply and extend the reduction theory for continued fractions from the previous chapter to square roots. After briefly treating reduction of periodic continued fractions, we finally prove Theorem \ref{thm-intro-infinite-poles-rationals} from the introduction, after rephrasing it to include number fields.
We go on to study the valuations more closely for \(\deg D = 4\) which leads to Theorem \ref{thm-intro-genus1-unbounded-gauss-norm} about unbounded valuations, also from the introduction.

We also explain how reduction of abelian varieties is related with the reduction of continued fractions via the reduction of the divisors of the convergents. This leads to a well-known effective method for testing if \(D\) is Pellian by reducing modulo two primes.
We conclude with a discussion of specialization of continued fractions, i.e. when the base field is \(\C(t)\).

We continue using the notation from the previous chapter. From now on, let \(D \in \O[X]\) non-square with even degree \(2d\) and \(\LC(D) \in \units \O\) a square. Then of course \(\deg \Red{D} = 2d\) and Proposition \ref{laurent-sqrt-bounded} implies \(\alpha = \sqrt{D} \in \laurinx{\O}\) and \(\gamma = \Red{\alpha} = \sqrt{\Red{D}}\). Recall we also defined \(A = \gauss{\sqrt{D}}\) and note that \(\Red{A} = \gauss{\sqrt{\Red{D}}}\) under the preceding hypotheses.

For example for \(D \in \Z[X]\) we can ask that \(D\) is monic to ensure that \(\sqrt{D} \in \laurinx \Q_{\nu_\pp}\) for every prime number \(\pp \neq 2\).

\section{Reduction of periodic quadratic continued fractions}
\label{sec:orgf6381b8}
In this section, we discuss reduction of periodic \(\CF(\sqrt{D})\). Then all necessary information is contained in finitely many partial quotients, and we can study reduction by looking at this finite data.

First, we check that nothing strange can happen -- we should not be able to reduce to a non-periodic continued fraction. Recall that periodicity of \(\CF(\sqrt{D})\) is equivalent to \(D\) being Pellian (see Theorem \ref{thm-pellian-iff-torsion}).

\begin{prop}
If \(D\) is Pellian, then either \(\Red{D}\) is a square, or it is also Pellian.
\end{prop}
\begin{proof}
Let \((p, q) \in \solu{D}\). By normalising it, we have also \((\normal{p}, \normal{q}) \in \solu{D}\), with reduction \(\Redn{q} \neq 0\) in \(k[X]\). Of course \({\normal{p}}^2 - D \, {\normal{q}}^2 = \omega \in \O\). If \(\omega \in \mm\), then \({\Redn{p}}^2 - \Red{D} \, {\Redn{q}}^2 = 0\) which implies \(\Red{D}\) is a square.

Otherwise we have \(\omega \in \units \O\), hence \(\Red{\omega} \in \units k\). Then clearly \((\Redn{p}, \Redn{q}) \in \solu{\Red{D}}\) and \(\Red{D}\) is Pellian.
\end{proof}

This proof shows that the degree \(\deg q\) of the minimal solution can only decrease under reduction. This has been exploited by Platonov \cite{platonov-2014-number-theoretic-properties} to produce Jacobians of hyperelliptic curves over \(\Q\) with torsion points of various order. In a previous article together with Petrunin \cite{platonov-petrunin-2012-the-torsion-problem}, he gives \(\Q\)-rational torsion points of orders \(36\) and \(48\). It seems they employ a refined brute force approach for searching Pellian polynomials by checking that \(D\) is Pellian only modulo several primes which speeds up the necessary calculations sufficiently (see also Example \ref{ex-periodic-good-red-2} in Section \ref{sec:orgc648887}).

The following does not even require that \(D\) is Pellian:
\begin{prop}
\label{prop-red-square-bad-reduction}
If \(\Red{D}\) is a square, then \(\CF(\sqrt{D})\) has bad reduction, with \(\alpha_1 \not\in \laurinx \O\).
\end{prop}
\begin{proof}
This is rather obvious because now \(\gamma = \sqrt{\Red{D}} \in k[X]\), so \(c_0 = \gamma_0\) and already \(\gamma_1\) does not exist. So we must have bad reduction of \(\CF(\sqrt{D})\) by Proposition \ref{bad-reduction-minimal-pole}.
\end{proof}

\begin{rem}
\label{rem-reduction-to-square-no-val-info}
If \(\Red{D}\) is square, the map \(\lambda\) has image \(\{0\}\), so there is a single infinite fibre. We neglected to treat this case in Section \ref{sec:org84e8497}. As already \(\Red{a_0 - \alpha_0} = 0\), we do not get much information about the valuations. So we do not know whether \(\alpha_1\) should be bounded or not.
\end{rem}

\begin{rem}
\label{bad-reduction-quasi-period-factor}
Suppose that \(\CF(\sqrt{D})\) is quasi-periodic, with \(\mu \in \units K\) such that \(\alpha_\QPL = \mu \, (A + \sqrt{D})\). Then \(\deg a_{\QPL} = d\) being maximal implies by Proposition \ref{bad-reduction-minimal-pole} that bad reduction of the continued fraction cannot start at \(\QPL\). As \(\nu(A) = \nu(\sqrt{D}) = 0\), we have \(\nu(\mu) = 0\) unless bad reduction of \(\CF(\sqrt{D})\) occurred already before \(\alpha_\QPL\), i.e. somewhere inside the quasi-period.

In particular, this means that \(\mu \in K\) cannot have too many different factors.
\end{rem}

\begin{rem}
\label{rem- reduction-to-square-find-bad-reduction}
If \(k\) has positive characteristic, it is possible that the (quasi-)period length shortens. This is best understood using the geometric viewpoint from Chapter \ref{sec:org99a8e17} and will be analysed later in Section \ref{sec:orgaebd15f}.

Anyway, we can easily determine whether we have good or bad reduction of periodic \(\CF(\alpha)\) by checking whether any of \(\nu(\LC(a_1)), \dots, \nu(\LC(a_{\QPL}))\) is negative (with \(\QPL\) the quasi-period length).

The quasi-period being palindromic (see Proposition \ref{palindromic-period}) also implies that the bad reduction of the continued fraction must start at the latest at \(\frac{\QPL}{2}\) for \(\QPL\) even, or \(\frac{\QPL-1}{2}+2\) for \(\QPL\) odd (in the latter case, we have to account for \(\nu(\mu) \neq 0\)).
\end{rem}

\subsection{Reduction in the \(\deg D = 2\) case}
\label{sec:org46de0ba}
Let us briefly describe  what happens in the case \(\deg D = 2\).

Suppose for simplicity that \(D\) is monic, then we can write \(D = (X+b)^2 + \omega\) with \(b, \omega \in \O\) and \(\omega \neq 0\) so that \(D\) is not a square. Of course \(A = X + b\), and one easily computes
\begin{equation*}
\alpha_0 = \sqrt{D}, \quad \alpha_{2i+1} = \frac{A+\sqrt{D}}{\omega}, \quad \alpha_{2i} = A + \sqrt{D}.
\end{equation*}
So bad reduction occurs \IFF \(\Red{\omega} = 0\), in which case \(\Red{D}\) is a square. Obviously we can reduce the \(\alpha_{2i}\) directly, but the \(\alpha_{2i+1}\) only after normalising. If \(\Red{\omega} = 0\), then the map \(\lambda\) has a single infinite fibre and clearly the \(\Redn{\alpha_n}\) are all polynomials.

The partial quotients are
\begin{equation*}
a_0 = A, \quad a_{2i+1} = \frac{2 \, A}{\omega}, \quad a_{2i} = 2 \, A
\end{equation*}
with Gauss norms
\begin{equation*}
\nu(a_0) = 0, \quad \nu(a_{2i+1}) = - \nu(\omega), \quad \nu(a_{2i}) = 0
\end{equation*}
which of course remain bounded.

As to the convergents, it is easy to see that (\(\ceil{\cdot}_\Z\) is the ceiling function)
\begin{equation*}
\nu(p_n) = \nu(q_n) \geq - \ceil{\frac{n}{2}}_\Z \, \nu(\omega) \text{ and } \nu(\LC(q_n)) = - \ceil{\frac{n}{2}}_\Z \, \nu(\omega)
\end{equation*}
hence \(\nu(q_n) = - \ceil{\frac{n}{2}}_\Z \, \nu(\omega)\) and \(\deg q_n = \deg \Redn{q_n}\). Indeed we expect this in the case of good reduction of \(\CF(\sqrt{D})\).

Otherwise we have bad reduction of \(\CF(\sqrt{D})\), hence \(\Red{A}^2 = \Red{D}\) and \(\nu(\omega) > 0\). Then we also know that \((\Redn{p_n}, \Redn{q_n}) = h_n \, (\Red{A}, 1)\) for all \(n \geq 0\). We can even calculate this: set \(\eta = \pi^{-\nu(\omega)}\) (so that \(\eta/\omega \in \units \O\)). Then for even \(n\) we get \(\inv{g_n} = \eta^{n/2}\) and for odd \(n\) we get \(\inv{g_n} = \eta^{(n+1)/2}\). We calculate for even \(n\):
\begin{align*}
\normal{p_n} &= \eta^{n/2} p_n = 2A \, \eta^{n/2} \, p_{n-1} + \eta^{n/2} \, p_{n-2} = 2A \, \normal{p_{n-1}} + \eta \, \normal{p_{n-2}}, \text{ and similarly } \\
\normal{q_n} &= 2A \, \normal{q_{n-1}} + \eta \, \normal{q_{n-2}}
\end{align*}
which yields
\begin{equation*}
\Redn{p_n} = 2 \Red{A} \, \Redn{p_{n-1}}, \quad \Redn{q_n} = 2 \Red{A} \, \Redn{q_{n-1}}.
\end{equation*}
On the other hand, we get for \(n\) odd
\begin{align*}
\normal{p_n} &= \eta^{(n+1)/2} p_n = \frac{2A}{\omega} \, \eta^{(n+1)/2} p_{n-1}  + \eta^{(n+1)/2} p_{n-2} = 2A \, \frac{\eta}{\omega} \, \normal{p_{n-1}} + \eta \, \normal{p_{n-2}}, \text{ and } \\
\normal{q_n} &= 2A \, \frac{\eta}{\omega} \, \normal{q_{n-1}} + \eta \, \normal{q_{n-2}}
\end{align*}
so
\begin{equation*}
\Redn{p_n} = 2 \Red{A}\, \Red{\eta/\omega} \, \Redn{p_{n-1}}, \quad \Redn{q_n} = 2 \Red{A} \, \Red{\eta/\omega} \, \Redn{q_{n-1}}.
\end{equation*}
It follows that \(h_n\) is \(\Red{A}^n\) times some constant factor depending on \(n\).

\pagebreak

\section{Reduction of non-periodic quadratic continued fractions}
\label{sec:org326a39b}
As before, let \(D \in \O[X]\) non-square with even degree, and \(\LC(D) \in \units \O\) a square, so that \(\alpha = \sqrt{D} \in \laurinx{\O}\), and \(\gamma = \Red{\alpha} = \sqrt{\Red{D}}\). But now, we assume that \(\CF(\sqrt{D})\) is non-periodic. Recall that this requires \(\deg D \geq 4\) (Corollary \ref{deg-2-always-pellian} and Theorem \ref{thm-pellian-iff-cf-periodic}).
\subsection{Reduction to square}
\label{sec:org1ce08e7}
If \(\Red{D}\) is a square, this implies bad reduction of \(\CF(\sqrt{D})\) by Proposition \ref{prop-red-square-bad-reduction}. Then \(\lambda : \N_0 \to \N_0\) has image \(\{0\}\), so we do not get a lot of information from it.

Anyway, for a fixed \(D\), this can happen only for finitely many valuations \(\nu\). From Proposition \ref{completion-of-square} about completion of the square and Proposition \ref{laurent-sqrt-bounded} about boundedness of the square root, it follows that \(\Red{D}\) is a square \IFF \(\nu(D-A^2) > 0\). So this can be checked easily, and concerns only finitely many valuations.

See Example \ref{ex-nonperiodic-to-square-2} in Section \ref{sec:orge6e4f78} for a non-periodic \(\CF(\sqrt{D})\) where \(\Red{D}\) is a square.

\subsection{Reduction to periodic and denominators}
\label{sec:orge7a9aff}
We now study the case where \(\CF(\gamma)\) becomes periodic. This happens automatically if \(k\) is finite, for example with \(D \in \Z[X]\) and reduction modulo some odd prime (see e.g. Corollary \ref{cor-finite-field-always-periodic}). Instead of talking about denominators which is rather vague, we are looking for \negval. 

\begin{lemma}
\label{bad-reduction-to-periodic}
If \(\CF(\gamma)\) is periodic, then infinitely many fibres of \(\lambda\) have at least \(2\) elements. Hence \(\CF(\alpha)\) has bad reduction at \(\nu\), and there exists \(n > 0\) where \(\alpha_n\) has \negval in the leading coefficient.
\end{lemma}
\begin{proof}
Corollary \ref{cor-pq-degree-periodicity} implies that for all \(n \geq 1\) holds \(\deg a_n < \frac{1}{2} \deg D\), and that there are infinitely many (because of pure periodicity of \(\CF(\Red A + \sqrt{\Red D})\)) \(m \geq 1\) such that \(\deg c_m = \frac{1}{2} \deg \Red{D}\).

However, \(\deg D = \deg \Red{D} = 2 d\), and good reduction of \(\CF(\alpha)\) would by Remark \ref{good-reduction-preserve-degrees} imply that \(\deg a_n = \deg c_n\) for all \(n\). In fact, by Corollary \ref{cor-lambda-degree-sum}, for every \(m \geq 1\) with \(\deg c_m = d\), the fibre \(\inv\lambda(m-1)\)  has more than a single element, so \(\lambda\) is certainly not bijective.

So \(\CF(\alpha)\) must have bad reduction. The statement about \negval then follows from Proposition \ref{bad-reduction-minimal-pole}.
\end{proof}
\begin{rem}
\label{bad-reduction-to-periodic-infinite-poles}
Proposition \ref{prop-multiple-element-fibre-analysis} implies that in the case of bad reduction of \(\CF(\alpha)\) each fibre with more than one element yields an unbounded complete quotient. It follows that there are infinitely many unbounded complete quotients. Compare also Proposition \ref{bad-reduction-minimal-pole}.
\end{rem}

If \(\deg D = 4\), the statement of the Lemma becomes an equivalence:
\begin{prop}
\label{bad-reduction-is-periodic-deg4}
Suppose \(\deg D = 4\) and \(\Red{D}\) non-square. Then \(\CF(\alpha)\) has bad reduction \IFF \(\CF(\gamma)\) is periodic.
\end{prop}
\begin{proof}
This extends Lemma \ref{bad-reduction-to-periodic}, it only remains to prove that bad reduction of \(\CF(\alpha)\) implies periodicity. As \(\deg D = 4\), we have \(\deg a_n = 1\) for all \(n \geq 1\). By Proposition \ref{bad-reduction-minimal-pole}, there is a minimal complete quotient \(\alpha_n\) with \(\nu(\alpha_n) < 0\). Because \(\Red{D}\) is non-square, \(\CF(\gamma)\) is infinite and hence the proposition also implies \(\deg c_n > 1\).

Then from Corollary \ref{cor-pq-degree-periodicity} follows \(\deg c_n \leq \frac{1}{2} \deg \Red{D} = 2\), so \(\deg c_n = 2\) and thus \(\CF(\gamma)\) must be periodic.
\end{proof}

\begin{rem}
If the residue field \(k\) is finite (and \(K\) is obviously infinite), then unless \(\Red{D}\) is square, one always has periodic \(\CF(\gamma)\) (see Corollary \ref{cor-finite-field-always-periodic}). So it is impossible to avoid bad reduction of \(\CF(\sqrt{D})\) in that case, for example if the base field \(K\) is a number field.
\end{rem}

\begin{rem}
Lemma \ref{bad-reduction-to-periodic} works only if \(\alpha\) is a square root of a polynomial (or shares a complete quotient with some \(\sqrt{D}\)). As we are interested in periodicity, we may assume that \(\alpha\) is \(\sigma\)-reduced (the complete quotients eventually have this property, see Proposition \ref{cf-compute-eventually-sigma-reduced}). But then \(\deg a_n = \deg A\) is equivalent to \(\alpha_n = \mu \,(A + \sqrt{D})\), so we would necessarily end up in the continued fraction expansion of \(\sqrt{\mu^2 \, D}\).
\end{rem}

\subsection{Primes occurring in infinitely many denominators}
\label{sec:org1fe25a8}
We are now ready to attack the proof of Theorem \ref{thm-intro-infinite-poles-rationals} from the introduction, about a prime occurring in infinitely many denominators of the \(a_n\). We first give a more technical version for a fixed prime, which holds in full generality.

\begin{prop}
\label{fibre-conditions-infinite-poles}
Suppose that there are infinitely many fibres of \(\lambda\) with one element, and infinitely many fibres with at least two elements. Then there exist infinitely many \(n\) with \(\nu(\LC(\alpha_n)) < 0\), i.e. infinitely many complete (and partial) quotients have the ``prime'' as a factor in the denominator of the leading coefficient.
\end{prop}
\begin{proof}
Let \(N \geq 0\). By assumption, there exists \(n \geq N\) such that \(\{n \} = \inv\lambda(m)\) for some \(m\). Then Proposition \ref{prop-single-element-fibre-analysis} implies
\begin{equation}
\label{eq-fibre-conditions-valuation}
\nu(\LC(\alpha_{n+1})) = \nu(\alpha_{n+1}) = \nu(g_{n-1}) - \nu(g_{n}) = \nu(g), \quad \nu(g_{n+1}) = \nu(g_{n-1}),
\end{equation}
where we set \(g = \ifrac{g_{n-1}}{g_{n}}\). Hence \(\inv g \, \alpha_{n+1} \in \laurinx \O\) with leading coefficient in \(\units \O\). As then also \(\inv g \, \alpha_{n+1} - \inv g \, a_{n+1} \in \laurinx \O\), its leading coefficient is in \(\O\), i.e. does not have \negval. We deduce that the first complete quotient \(g \, \alpha_{n+2}\) cannot have \posval in the leading coefficient (compare the proof of Proposition \ref{bad-reduction-minimal-pole}):
\begin{equation*}
\nu(\LC(g \, \alpha_{n+2})) \leq 0 \implies \nu(\LC(\alpha_{n+2})) \leq - \nu(g) = - \nu(\LC(\alpha_{n+1})).
\end{equation*}
So if \(\nu(\LC(\alpha_{n+1})) \neq 0\), either \(\alpha_{n+1}\) or \(\alpha_{n+2}\) has the desired \negval in the leading coefficient.

Otherwise \(\nu(\LC(\alpha_{n+1})) = 0\), so \(\alpha_{n+1} \in \laurinx \O\) and we can reproduce the argument from Proposition \ref{bad-reduction-minimal-pole}: Let \(n' > n\) minimal such that \(\inv\lambda(\lambda(n'))\) has multiple elements. In this case we know \(\nu(\LC(\alpha_{n'+1})) < 0\) by the minimality of \(n'\). So the desired pole is in the leading coefficient of \(\alpha_{n'+1}\).
\end{proof}

\begin{rem}
\label{infinite-poles-after-renormalising}
The proof also illustrates that there are infinitely many partial quotients with \negval, even if we multiply \(\alpha\) with \(\pi^e\) for some \(e \in \Z\) (or some other constant). This merely changes \(g\) in in \eqref{eq-fibre-conditions-valuation}; and recall Proposition \ref{cf-scalar-multiplication} about multiplying a continued fraction with a constant factor. We do not even need to assume \(\alpha \in \laurinx \O\) here, if we define \(\lambda\) appropriately -- it is determined by the sequences \(\deg a_n\) and \(\deg c_m\) which do not change under this multiplication.
\end{rem}

Our results from Section \ref{sec:org84e8497} tell us that multiple element fibres correspond to unbounded complete quotients, and hence bad reduction of the continued fraction. If we want this to occur repeatedly, it is very natural to ask for infinitely many such fibres.

On the other hand, asking also for infinitely many fibres with just a single element is a more technical condition. Right now we cannot avoid this because we do not understand how the valuations behave for multiple element fibres (except for \(\deg D = 4\), to be treated in Theorem \ref{thm-genus1-zero-patterns}). Recall that the complete quotients belonging to multiple element fibres, if at all, reduce to rational functions, about which we have hardly any information (see Proposition \ref{prop-multiple-element-fibre-analysis}).

At least we have a simple criterion that guarantees the existence of infinitely many fibres of \(\lambda\) with just a single element:
\begin{prop}
\label{prop-criterion-infinite-single-element-fibres}
With \(\alpha = \sqrt{D}\), suppose that \(\CF(\alpha)\) is non-periodic, but \(\CF(\gamma)\) is periodic. Let \(\delta = \min\{\deg a_n \mid n \geq 0 \}\). If there exists \(m\) such that \(\deg c_m = \delta\), then \(\lambda : \N_0 \to \N_0\) has infinitely many fibres with a single element.
\end{prop}
\begin{proof}
Recall that \(\deg D = 2d\). From Corollary \ref{cor-pq-degree-periodicity} we know that \(\deg a_n < d = \deg a_0 = \deg c_0\) for \(n \geq 1\), so certainly \(m \geq 1\). But the quasi-period of \(\CF(\gamma)\) begins at \(c_1\); this means there are actually infinitely many \(m\) with \(\deg c_m = \delta\).

Then Corollary \ref{cor-lambda-degree-sum} implies for \(\inv\lambda(m-1) = \{n-1,\dots,n-1+l\}\) that \(\deg a_n + \dots + \deg a_{n+l} = \deg c_m = \delta\). Minimality of \(\delta\) forces \(l = 0\), hence the fibre has a single element. It follows that there are infinitely many fibres of \(\lambda\) with a single element.
\end{proof}

\begin{rem}
Note that along the way, we have also proved that there are infinitely many partial quotients with \(\deg a_n = \delta\), so the minimal degree must be assumed infinitely often (however, we used periodicity of \(\CF(\gamma)\) which is in general a rather strong hypothesis).
\end{rem}

\medskip

Now we restrict ourselves to \(K\) being a number field. The ring of integers \(\O_K\) of a number field, while it need not be a unique factorisation domain, has unique factorisations of ideals into prime ideals. Every \(x \in K\) can be written as \(x = \frac{a}{b}\) with \(a \in \O_K\) and \(b \in \N\). In the theorem below, we refer to \(b\) as the denominator.

Each prime ideal \(\PP\) of \(\O_K\) corresponds to a non-archimedean valuation \(\nu_\PP\) on \(K\). By localising \(\O_K\) at \(\PP\), one obtains a discrete valuation ring with a finite residue field. The latter is a finite extension of \(\F_\pp\), where \(\pp\) is the unique prime number (of \(\Z\)) contained in \(\PP\).\footnote{See for example \cite{neukirch-1999-algebraic-number-theory}, Chapter I \S 8. Or any other decent textbook on algebraic number theory.}

\pagebreak
The following generalises Theorem \ref{thm-intro-infinite-poles-rationals} from the introduction:
\begin{thm}
\label{thm-infinite-poles-number-field}
Let \(K\) a number field, suppose that \(D \in K[X]\) is monic, non-square and has even degree, but is not Pellian.

Then for all but finitely prime numbers (in \(\Z\)), the prime \(\pp\) appears in infinitely many \(a_n\) (actually \(\LC(a_n)\)) in a (the) denominator.
\end{thm}

The primes excluded are \(2\) (because the residue field would have characteristic \(2\)), those which already appear in a denominator in \(D\), and those which make \(D \mod\PP\) a square polynomial. Additionally, we may need to exclude a finite number of primes, depending on where the first partial quotient of minimal degree \(\delta\) occurs. This can be made effective, as discussed  below in Remark \ref{minimal-an-degree-effective}.

The Theorem relies on \(\alpha\) being a square root. For other elements of \(K(X, \sqrt{D})\), we may in fact have good reduction of the continued fraction at infinitely many primes, see Section \ref{sec:org53fcc2b} for an example.

\begin{proof}
Removing the finitely many primes with \(\nu_\PP(D) < 0\) (i.e. \(\pp\) is in the denominator of \(D\)), and ignoring the primes \(\PP\) above \(2\), the conditions on \(D\) ensure that \(\sqrt{D} \in \laurinx \O\) for \(\nu = \nu_\PP\). Let us also ignore the \(\PP\) for which \(D - A^2 \in \PP[X]\), i.e. where the reduction \(\Red{D}\) is a square (there are only finitely many, as \(D - A^2 \in K[X]\) is a polynomial).

Of course \(D\) not Pellian means that \(\CF(\sqrt{D})\) is non-periodic. However, the residue field \(k\) is finite for every prime, so \(\CF(\sqrt{\Red{D}})\) must necessarily be periodic. Then Lemma \ref{bad-reduction-to-periodic} implies that we are always in the case of bad reduction (of the continued fraction), and there are infinitely many fibres of \(\lambda\) which have at least two elements. 

In order to apply Proposition \ref{fibre-conditions-infinite-poles}, we use Proposition \ref{prop-criterion-infinite-single-element-fibres}, so we need to check that there exists \(m\) with \(\deg c_m = \delta = \min\{\deg a_n \mid n \geq 0 \}\).

Let \(n_0\) the minimal \(n\) with \(\deg a_n = \delta\) (obviously \(n_0 \geq 1\)). We restrict to primes \(\PP\) for which we have good reduction of \(\CF(\sqrt{D})\) up to \(n_0\), i.e. where \(\alpha_0, \alpha_1, \dots, \alpha_{n_0} \in \laurinx \O\). This excludes only finitely many \(\PP\): we can factor each \(\LC(a_n) \in K\) for \(n=0,\dots,n_0\) into a product (with possibly negative exponents) of prime ideals of \(\O_K\). Of course \(\nu_\PP(\LC(a_n)) < 0\) happens \IFF \(\PP\) appears with a negative exponent in the factorisation. Of these there are obviously just finitely many, and by Proposition \ref{bad-reduction-minimal-pole} the bad reduction of \(\CF(\sqrt{D})\) starts only later for all other primes.

For the remaining primes \(\PP\), the complete quotients up to \(\alpha_{n_0}\) are thus contained in \(\laurinx \O\). By Proposition \ref{cf-good-red-leading-coeffs} and Remark \ref{good-reduction-preserve-degrees} this implies \(\deg c_{n_0} = \deg a_{n_0} = \delta\).

Hence \(\nu_\PP(\LC(a_n)) < 0\) for infinitely many \(n\). If we write \(\LC(a_n) = \ifrac{a}{b}\) with \(a \in \O_K\) and \(b \in \N\), then naturally \(\nu_\PP(a) \geq 0\), hence \(\nu_\PP(b) > 0\). Applying the Norm, \(b \in \N\) must have \(\pp \div b\) as desired.
\end{proof}

Let us briefly discuss effectivity of \(\delta = \min\{ \deg a_n \mid n \geq 0 \}\).
In \cite{zannier-2016-hyper-contin-fract}, it is shown how a Skolem-Mahler-Lech theorem for algebraic groups implies that the sequence of the \(\deg a_n\) is eventually periodic (even if \(\CF(\sqrt{D})\) is not periodic!). While in certain cases it seems possible to obtain a bound for the period length (of the degrees) from this result, there is unfortunately no information on the pre-period. Summing upper bounds for the pre-period and the period would of course produce an upper bound for \(\delta\).

However, the issues with effectivity are actually related to finding \(\max\{ \deg a_n \mid n \geq 1 \}\). For finding the minimal degree, we do not need to know the entire period (of degrees):
\begin{rem}
\label{minimal-an-degree-effective}
Let \(\OOmultszar\) the Zariski closure of \(\{ n \, \OO \mid n \in \Z \}\) in the Jacobian of \(\CC\); we need to find the maximal \(r\) such that \(\OOmultszar \subset W_r\) but \(\OOmultszar \not\subset W_{r-1}\). Using the divisor relations coming from the convergents, explained in Sections \ref{sec:org5685823} and \ref{sec:org4e4485e}, this implies \(\delta = g - r + 1\) for said maximal \(r\) (recall from Section \ref{sec:orgdaae3ac} that \(W_r\) is the \(r\) fold symmetric sum of \(\CC\) embedded in its Jacobian variety).

We can effectively compute \(\OOmultszar\) from a factorisation of the Jacobian as in Theorem 1.2 of \cite{gaudron-remond-2014-polarisations-isogenies} (which extends a deep result of Masser and Wüstholz, \cite{masser-wuestholz-2014-polarization-estimates-abelian}).
As the \(W_r\) can also be effectively represented, we can determine in which of the \(W_r, \; (r=1, \dots, g-1)\) our subvariety \(\OOmultszar\) is not contained. Certainly \(\OOmultszar\) is contained in \(W_g = \Jac(\CC)\).
\end{rem}

In practice finding \(\delta\) is not a big issue because we usually immediately find a partial quotient with \(\deg a_n = 1\) (which in fact \emph{must} occur for \(\deg D = 4\) or \(6\) if \(D\) is non-Pellian and non-square; see Theorem 1.3 of \cite{zannier-2016-hyper-contin-fract}, stated below as Theorem \ref{thm-zannier-an-deg-bound}), and then we know that \(\delta = 1\).

Theorem \ref{thm-genus1-zero-patterns} below gives another (similar) proof for the occurrence of a prime in the denominators of infinitely many \(a_n\) in the case \(\deg D = 4\). This relies on being able to control cancellation issues sufficiently, so we do not need the single element fibres to estimate the Gauss norms.

\section{Genus 1 valuation patterns}
\label{sec:orgbb63dc1}
We analyse the case \(\deg D = 4\) much closer now, and will describe how the valuations of the complete quotients, partial quotients and convergents behave in the case of bad reduction at \(\nu\). When studying examples (see Tables \ref{cfr-mod5-valuations-table} and \ref{cfr-mod19-valuations-table} in Section \ref{sec:orgec32865}), one notes that the valuations (Gauss norms) of the partial quotients \(a_n\) are often divisible by \(4\), with alternating signs, while the valuations of the convergents \(q_n\) are always divisible by \(2\), again with alternating signs. Both also exhibit an almost pseudo-periodic behaviour. The theorem below aims to explain these patterns:

\enlargethispage{1cm}
\begin{thm}
\label{thm-genus1-zero-patterns}
Suppose \(\deg D = 4\), and that \(\CF(\sqrt{D})\) is non-periodic, while \(\CF(\sqrt{\Red{D}})\) is periodic with quasi-period \(\QPL\), so that we have bad reduction as shown in Proposition \ref{bad-reduction-is-periodic-deg4}. Then we observe the following:
\begin{itemize}
\item The unbounded complete quotients \(\alpha_n\) are exactly those with
\begin{equation}
\label{deg4-unboundedset-eq}
n \in \unboundedset = \{ j \, (\QPL+1) - 1 \mid j \geq 1 \}.
\end{equation}
\item Defining
\begin{equation*}
f_n = \nu(\LC(\vartheta_{n-1})) \geq 0
\end{equation*}
we have \(f_n > 0\) \IFF \(n \in \unboundedset\) (so \(f_n = 0\) otherwise).
\item Recursively defining \(F_0 = 0\) and \(F_n = -(F_{n-1} + f_n)\), we get formulas for the valuations
\begin{align*}
\nu(a_n) &= 2 \, (F_{n-2} + F_n), & \nu(\LC(a_n)) &= \nu(a_n) + f_{n-1} + f_n, \\
\nu(q_n) &= 2 \, F_n, & \nu(\LC(q_n)) &= \nu(q_n) + f_n.
\end{align*}
\end{itemize}
\end{thm}
\begin{rem}
Note that if \(n-1, n \not \in \unboundedset\), we have \(F_{n-2} = F_n\) and thus \(\nu(a_n) = 4 \, F_n\), which explains the divisibility by \(4\).
\end{rem}
\begin{rem}
For higher genus, one probably has to consider other coefficients besides \(f_n\). But it is not at all clear how this generalises.
\end{rem}

The following is the general version of Theorem \ref{thm-intro-genus1-unbounded-gauss-norm} (recall that \(\QPL + 1\) is the torsion order of \(\j{\OOred}\), see Proposition \ref{prop-bounds-torsion-period-length}):
\begin{cor}
\label{cor-genus1-unbounded-gauss-norm}
Under the same hypotheses as Theorem \ref{thm-genus1-zero-patterns}, and additionally assuming the quasi-period \(\QPL\) of \(\CF(\sqrt{\Red{D}})\) is odd, the Gauss norms grow at least linearly (in particular they are unbounded):
\begin{equation*}
(-1)^n \nu(a_{n}) \geq 2 \left( \floor{\frac{n-1}{\QPL+1}}_\Z + \floor{\frac{n+1}{\QPL+1}}_\Z \right), \qquad
(-1)^n \nu(q_{n}) \geq 2 \floor{\frac{n+1}{\QPL+1}}_\Z.
\end{equation*}
\end{cor}
\begin{proof}
We can easily write \(F_n\) as an alternating sum of the \(f_n\):
\begin{equation*}
F_n = \sum_{j=0}^n (-1)^{n-j+1} \, f_j = \sum_{j \in \unboundedset, \atop j \leq n} (-1)^{n-j+1} \, f_j = (-1)^n \sum_{j \in \unboundedset, \atop j \leq n} (-1)^{j+1} \, f_j.
\end{equation*}
In case \(\QPL\) is odd, for every \(j \in \unboundedset\) we have \(j+1 = i \, (\QPL + 1)\) even. For \(j + 1 \leq n + 1\), we have \(1 \leq i \leq \frac{n+1}{\QPL+1}\). As every \(f_j \geq 1\), this implies \((-1)^n \, F_n \geq \floor{\frac{n+1}{\QPL+1}}_\Z\). With the formulas from the theorem, we get the desired estimates for the Gauss norm.
\end{proof}
\begin{rem}
For even \(\QPL\), it is completely unclear if the \(F_n\) could be bounded. In example calculations, we sometimes observe cancellation, but not always (see Table \ref{cfr-mod19-valuations-table}). As we currently have almost no control over the \(f_n\) for \(n \in \unboundedset\), any result in this direction would be quite surprising.

In some examples, we see almost periodic patterns in the values of the \(f_n\). Usually though, there comes a disturbance in these patterns at some point. We will revisit this issue briefly in Section \ref{sec:org9851fc8}.
\end{rem}

However, we can check that the valuations are negative for infinitely many \(n\) (giving another proof of Theorem \ref{thm-infinite-poles-number-field} for \(\deg D = 4\)):
\begin{cor}
\label{cor-infinite-poles-deg4}
Under the hypotheses of the Theorem \ref{thm-genus1-zero-patterns}, there are infinitely many \(n\) with \(\nu(a_n) < 0\), and infinitely many \(n\) with \(\nu(q_n) < 0\).
\end{cor}
\begin{proof}
The previous Corollary \ref{cor-genus1-unbounded-gauss-norm} gives a stronger statement when the quasi-period length \(\QPL\) of \(\CF(\gamma)\) is odd, so we only need to check the case where \(\QPL\) is even. In particular, this means \(\QPL \geq 2\).

So if we take \(n \in \unboundedset\), this implies (from the structure of \(\unboundedset\)) that \(f_n > 0\) but \(f_{n+1} = f_{n+2} = 0\), hence
\begin{equation*}
F_n = -(F_{n-1} + f_n), \quad F_{n+1} = - F_n = F_{n-1} - f_n, \quad F_{n+2} = - F_{n+1} = F_n.
\end{equation*}
If both \(F_{n-1} \geq 0\) and \(F_{n} \geq 0\), then also \(F_{n-1} + F_{n} = - f_n \geq 0\); but that contradicts our choice of \(n\). So one of \(\nu(q_{n-1}) < 0\) or \(\nu(q_{n}) < 0\) must be satisfied.

Similarly, if both \(F_{n-1} + F_{n+1} = 2 \, F_{n-1} - f_n \geq 0\) and \(F_{n} + F_{n+2} = 2 \, F_n = -2 \, F_{n-1} - 2 \, f_n\geq 0\), then also \(F_{n-1} + F_{n+1} + F_{n} + F_{n+2} = -3 \, f_n \geq 0\); this is again a contradiction. Hence at least one of \(\nu(a_{n+1}) < 0\) or \(\nu(a_{n+2}) < 0\) is satisfied.

As \(\unboundedset\) is infinite, we find infinitely many of these partial quotients and convergents.
\end{proof}

\begin{rem}
\label{red-deg4-coprime-convergents}
For all \(n\), we have \(\Redn{p_n}\) and \(\Redn{q_n}\) coprime because for single and two element fibres we observed that all the \(h_n\) must be constant (see Propositions \ref{prop-single-element-fibre-analysis} and \ref{prop-two-element-fibre-analysis}).
\end{rem}

We begin the proof of Theorem \ref{thm-genus1-zero-patterns} by analysing the fibres of \(\lambda\). For the rest of this section, assume the hypotheses on \(D\) from the Theorem are satisfied.
\begin{prop}
\label{bad-reduction-periodic-deg4-fibres}
The fibres of \(\lambda\) have at most \(2\) elements. The fibres with \(2\) elements are given by
\begin{equation*}
\inv\lambda(j \, \QPL - 1) = \{ j (\QPL + 1) - 2, j (\QPL + 1) - 1 \}, \quad j \geq 1,
\end{equation*}
all other fibres have just one element.
\end{prop}
\begin{proof}
The degrees of the \(a_n\) are given by the sequence \(2,1,1,1,\dots\), while the degrees of the \(c_n\) are given by the sequence \(2, 1, \dots, 1, 2, 1, \dots, 1,2,1, \dots\) with precisely \(\QPL-1\)  ``1'' between the ``2'' (the quasi-period is determined by the degrees of the partial quotients, see Corollary \ref{maximal-degree-implies-quasi-period}). Recall that for a fibre \(\{n, n+1, \dots, n+l \} = \inv\lambda(m)\) we always have \(\deg a_{n+1} + \dots + \deg a_{n+l+1} = \deg c_{m+1}\) (see Corollary \ref{cor-lambda-degree-sum}). So clearly, we can have at most two elements in a fibre.

The first fibre with two elements, due to \(\deg c_{\QPL} = 2\) by the properties of the quasi-period, is
\begin{equation*}
\inv\lambda(\QPL - 1) = \{ \QPL - 1, \QPL \}.
\end{equation*}
In fact, we generally have \(\deg c_{j \, \QPL} = 2\). In between, there are always \(\QPL - 1\) fibres with a single element, so the minimal element increases by \(\QPL + 1\) each time:
\begin{equation*}
\inv\lambda(j \, \QPL - 1) = \{ \QPL - 1 + (j-1) \, (\QPL + 1), \QPL + (j-1) \, (\QPL + 1) \} = \{ j \, (\QPL +1) - 2, j \, (\QPL + 1) -1 \}.
\end{equation*}
\end{proof}

\begin{proof}[Proof of Theorem \ref{thm-genus1-zero-patterns}]
With our analysis of fibres with one or two elements (Proposition \ref{prop-single-element-fibre-analysis} and \ref{prop-two-element-fibre-analysis}), Proposition \ref{bad-reduction-periodic-deg4-fibres} above implies directly \eqref{deg4-unboundedset-eq}: the unbounded complete quotients come only from the minimal element of the two element fibres (index of course shifted by \(1\)). These results also show that \(\io{\Red{\vartheta_n}} > \io{\vartheta_n}\), i.e. \(f_{n+1} = \nu(\LC(\vartheta_n)) > 0\),  happens just for the minimal element of the two element fibres. As the definition of \(f_n\) corrects for the index shift, it is clear that \(f_n > 0\) happens precisely if \(\alpha_n\) is unbounded.

\medskip

It remains to check the valuation formulas, for which we use a complete induction. Recall that \(\nu(g_n) = \nu(q_n)\).

For \(n = 0\), by our assumption on \(D\) we have \(\nu(\LC(a_0)) = \nu(a_0) = 0\). As \(q_0 = 1\), the valuation formulas are clearly satisfied.

Actually, we should also check \(n = 1\). But the careful reader will find that we use the induction hypothesis for ``\(n-2\)'' only for \(\nu(q_{n-2})\). So we can check \(n = -1\) instead of \(n = 1\).

By convention, we have \(p_{-1} =1, q_{-1} = 0\), so \(\vartheta_{-1} = 1\). So \(\nu(q_{-1}) = \infty\) looks like a problem, but in fact we only need \(\nu(g_{-1}) = 0\). Recall that \(g_{-1} = 1\) is the normalisation factor of the ``canonical convergent'' \((p_{-1}, q_{-1}) = (1, 0)\).

\smallskip

For the induction step, we first check the single element fibre case:

Suppose \(\{ n \} = \inv\lambda(m)\), so we refer to Proposition \ref{prop-single-element-fibre-analysis}. In this case, \(f_n = f_{n+1} = 0\). Hence
\begin{equation*}
\nu(a_{n+1}) = \nu(\LC(a_{n+1})) = \nu(\alpha_{n+1}) \\= \nu(g_{n-1}) - \nu(g_{n}) = 2 (F_{n-1} - F_n) = 2 (F_{n-1} + F_{n+1})
\end{equation*}
which covers \(a_n\) and its leading coefficient.

We also get
\begin{equation*}
\nu(\LC(q_{n+1})) = \nu(q_{n+1}) = \nu(g_{n-1}) = 2 \, F_{n-1} = 2 \, F_{n+1}.
\end{equation*}

\smallskip

Next we verify the valuation formulas for the two element fibre: for \(\{ n, n+1 \} = \inv \lambda(m)\), we use Proposition \ref{prop-two-element-fibre-analysis}. As observed above, \(f_n = f_{n+2} = 0\), but \(f_{n+1} > 0\). We already computed in \eqref{eq-2elem-pq1}
\begin{equation*}
\nu(a_{n+1}) = \nu(g_{n-1}) - \nu(g_n) -2 \, f_{n+1} = 2 (F_{n-1} - F_n - f_{n+1}) = 2 \, (F_{n-1} + F_{n+1}),
\end{equation*}
and also \(\nu(\LC(a_{n+1})) = \nu(a_{n+1}) + f_{n+1}\) as desired. For the convergent, we had
\begin{equation*}
\nu(q_{n+1}) = \nu(g_{n-1}) - 2 f_{n+1} = 2 (F_{n-1} - f_{n+1}) = 2 \, F_{n+1}.
\end{equation*}
Moreover, using \(F_n = - F_{n-1}\), we find
\begin{multline*}
\nu(\LC(q_{n+1})) = \nu(\LC(a_{n+1})) + \nu(\LC(q_{n})) \\ = 2 \, (F_{n-1} + F_{n+1}) + f_{n+1} + 2 \, F_n + f_n 
= 2 \, F_{n+1} + f_{n+1}.
\end{multline*}

For the second element of the fibre, we have (analogous to the calculation for the single element fibre, but using only \(f_{n+2} = 0\))
\begin{equation*}
\nu(a_{n+2}) = \nu(g_n) - \nu(g_{n+1}) = 2 (F_{n} + F_{n+2})
\end{equation*}
and \(\nu(\LC(a_{n+2})) = \nu(a_{n+2}) + f_{n+1}\), as desired. For the convergent, we have
\begin{equation*}
\nu(\LC(q_{n+2})) = \nu(q_{n+2}) = \nu(q_n) + 2 \, f_{n+1} = 2 \, (F_n + f_{n+1}) = -2 \, F_{n+1} = 2 \, F_{n+2}.
\end{equation*}

This concludes the proof of Theorem \ref{thm-genus1-zero-patterns}.
\end{proof}

\enlargethispage{1cm}

\pagebreak

To visualise these formulas better, have a look at the three following tables, with horizontal lines directly before the unbounded complete quotients:

\begin{center}
\begin{tabular}{lrlllrll}
\(n\) & \(\deg a_n\) & \(\deg q_n\) & \(\io{\vartheta_n}\) & \(m\) & \(\deg c_m\) & \(\deg v_m\) & \(\io{\Red{\vartheta_n}}\)\\
\hline
0 & 2 & 0 & 1 & 0 & 2 & 0 & 1\\
1 & 1 & 1 & 2 & 1 & 1 & 1 & 2\\
\(\vdots\) &  &  &  &  &  &  & \\
\(l-1\) & 1 & \(l-1\) & \(l\) & \(l-1\) & 1 & \(l-1\) & \(l+1\)\\
\hline
\(l\) & 1 & \(l\) & \(l+1\) & \(l-1\) &  & \(l-1\) & \(l+1\)\\
\(l+1\) & 1 & \(l+1\) & \(l+2\) & \(l\) & 2 & \(l+1\) & \(l+2\)\\
\(l+2\) & 1 & \(l+2\) & \(l+3\) & \(l+1\) & 1 & \(l+2\) & \(l+3\)\\
\(\vdots\) &  &  &  &  &  &  & \\
\(2l\) & 1 & \(2l\) & \(2l+1\) & \(2l-1\) & 1 & \(2l\) & \(2l+2\)\\
\hline
\(2l+1\) & 1 & \(2l+1\) & \(2l+2\) & \(2l-1\) &  & \(2l\) & \(2l+2\)\\
\(2l+2\) & 1 & \(2l+2\) & \(2l+3\) & \(2l\) & 2 & \(2l+2\) & \(2l+3\)\\
\(2l+3\) & 1 & \(2l+3\) & \(2l+4\) & \(2l+1\) & 1 & \(2l+3\) & \(2l+4\)\\
\(\vdots\) &  &  &  &  &  &  & \\
\end{tabular}
\end{center}

For simplicity, we assume here that all \(f_n \leq 1\). For \(l\) odd, we get
\begin{center}
\begin{tabular}{lrrrrrr}
\(n\) & \(f_n\) & \(\nu(a_n)\) & \(\nu(\alpha_n)\) & \(\nu(\LC(a_n))\) & \(\nu(q_n)\) & \(\nu(\LC(q_n))\)\\
\hline
0 & 0 & 0 & 0 & 0 & 0 & 0\\
1 & 0 & 0 & 0 & 0 & 0 & 0\\
\(\vdots\) &  &  &  &  & 0 & 0\\
\(l-1\) & 0 & 0 & 0 & 0 & 0 & 0\\
\hline
\(l\) & 1 & -2 & \(-\infty\) & -1 & -2 & -1\\
\(l+1\) & 0 & 2 & 3 & 2 & 2 & 2\\
\(l+2\) & 0 & -4 & -4 & -4 & -2 & -2\\
\(l+3\) & 0 & 4 & 4 & 4 & 2 & 2\\
\(\vdots\) &  &  &  &  &  & \\
\(2l-1\) & 0 & -4 & -4 & -4 & -2 & -2\\
\(2l\) & 0 & 4 & 4 & 4 & 2 & 2\\
\hline
\(2l+1\) & 1 & -6 & \(-\infty\) & -5 & -4 & -3\\
\(2l+2\) & 0 & 6 & 6 & 7 & 4 & 5\\
\(2l+3\) & 0 & -8 & -8 & -8 & -4 & -4\\
\(2l+4\) & 0 & 8 & 8 & 8 & 4 & 4\\
\(\vdots\) &  &  &  &  &  & \\
\end{tabular}
\end{center}

\pagebreak
But for \(l\) even, we get
\begin{center}
\begin{tabular}{lrrrrrr}
\(n\) & \(f_n\) & \(\nu(a_n)\) & \(\nu(\alpha_n)\) & \(\nu(\LC(a_n))\) & \(\nu(q_n)\) & \(\nu(\LC(q_n))\)\\
\hline
0 & 0 & 0 & 0 & 0 & 0 & 0\\
1 & 0 & 0 & 0 & 0 & 0 & 0\\
\(\vdots\) &  &  &  &  & 0 & 0\\
\(l-1\) & 0 & 0 & 0 & 0 & 0 & 0\\
\hline
\(l\) & 1 & -2 & \(-\infty\) & -1 & -2 & -1\\
\(l+1\) & 0 & 2 & 3 & 2 & 2 & 2\\
\(l+2\) & 0 & -4 & -4 & -4 & -2 & -2\\
\(l+3\) & 0 & 4 & 4 & 4 & 2 & 2\\
\(\vdots\) &  &  &  &  &  & \\
\(2l-1\) & 0 & 4 & 4 & 4 & 2 & 2\\
\(2l\) & 0 & -4 & -4 & -4 & -2 & -2\\
\hline
\(2l+1\) & 1 & 2 & \(-\infty\) & 3 & 0 & 1\\
\(2l+2\) & 0 & -2 & -2 & -1 & 0 & 0\\
\(2l+3\) & 0 & -8 & -8 & -8 & 0 & 0\\
\(2l+4\) & 0 & 8 & 8 & 8 & 0 & 0\\
\(\vdots\) &  &  &  &  &  & \\
\end{tabular}
\end{center}

See also the tables for Example \ref{ex-cfp1-zero-pattern-deg4} in Section \ref{sec:orgec32865}

\section{Reduction of abelian varieties}
\label{sec:org376166a}
To understand how the quasi-period length may change under reduction, and hence to understand bad reduction to a periodic continued fraction, it serves to study reduction of torsion points on the Jacobian of the (hyper)elliptic curve.
\subsection{Reduction of curve and its Jacobian}
\label{sec:orgd3ba0b6}
Our first step is to define a model of \(\CC\) over \(\O\). Here we can mostly retrace the steps from Section \ref{sec:orgf4393ef}: instead of over the field \(\K\), we are working over the discrete valuation ring \(\O\). Note that \(\Spec \O\) is a local affine Dedekind scheme of dimension \(1\) with just two points: the generic point and a closed point corresponding to \(\mm\).\footnote{Instead of \(\Spec \O\), we could also work with any Dedekind scheme, for example \(\Spec \Z\) if \(D \in \Z[X]\). We stick to the local case for simplicity and consistency of notation.}

For this section, we assume that both \(D\) and \(\Red{D}\) are square-free to ensure that the curve \(\CC\) (and its Jacobian) has good reduction at \(\nu\).

Gluing together
\begin{equation*}
\Spec \O[X,Y]/\spann{Y^2 - D(X)} \text{ and } \Spec \O[U,V]/\spann{V^2 - D^\flat(U)}
\end{equation*}
via the morphisms given by \(X \, U = 1\) and \(U^{g+1} \, Y = V\) (respectively \(X^{g+1} \, V = Y\)) we get a scheme \(\XX\) of dimension \(2\) which is our model of \(\CC\) over \(\O\). Note that the coefficients of \(D^\flat\) are those of \(D\) in reverse order, so \(D^\flat \in \O[U]\).

Think of the surface \(\XX\) as containing two curves: the fibre \(\XX_0\) over the generic point of \(\Spec \O\) which is essentially our curve \(\CC\), and the fibre \(\XX_\mm\) over the closed point of \(\Spec \O\) which is the curve \(\CCred\) defined over \(k\), with \(D\) replaced by \(\Red{D}\).

\begin{prop}
The fibered surface \(\XX \to \Spec \O\) is normal, regular, projective and flat, in other words it is a normal arithmetic surface.
\end{prop}
\begin{proof}
Normal and regular are local conditions and may be checked at the stalks. Hence this follows from \(\CC\) and \(\CCred\) being normal and smooth (and thus regular), see Propositions \ref{affine-curve-is-smooth-and-normal} and \ref{projective-curve-is-smooth-and-normal}.

Flatness follows from surjectivity via Proposition 4.3.9 of \cite{liu-2002-algebraic-geometry-arithmetic}: clearly the generic point of \(\XX\) maps to the generic point of \(\Spec \O\).

As the fibre \(\XX_0\) is proper, and both fibres are geometrically connected, Remark 3.3.28 of \cite{liu-2002-algebraic-geometry-arithmetic} with surjectivity implies that \(\XX \to \Spec \O\) is proper.

Then we can apply the second part of Remark 9.3.5 in \cite{liu-2002-algebraic-geometry-arithmetic} to obtain that \(\XX \to \Spec \O\) is projective.
\end{proof}

In order to properly define the reduction map, we need our field \(K\) to be Henselian, i.e. complete with respect to the valuation \(\nu\). See also Section 10.1.3 in \cite{liu-2002-algebraic-geometry-arithmetic} for further details.

\begin{defi}
Let \(\hat K\) the completion of \(K\), and \(\hat \O = \{ x \in \hat K \mid \nu(x) \geq 0 \}\). This remains a discrete valuation ring with residue field still \(k\). We now consider \(\XX\) as a scheme over \(\hat \O\).

For a closed point \(P \in \XX_0 = \CC\), the Zariski closure \(\closure{\{P\}}\) in \(\XX\) is irreducible and has a unique closed point, the point of \(\closure{\{P\}} \cap \XX_\mm\). This defines a reduction map \(\Redm : \CC(\hat K) \to \CC(k)\) which extends linearly to Weil divisors.
\end{defi}

\begin{rem}
For a point \(P = (x, y) \in \CCa\) we easily see that \(\nu(x) \geq 0\) implies \(\nu(y) \geq 0\), so in this case we set \(\Red{P} = (\Red{x}, \Red{y})\). Otherwise \(\nu(x) < 0\), but then write \(P = (u, v) \in \CCinf\) where now \(\nu(u) \geq 0\), so we may set \(\Red{P} = (\Red{u}, \Red{v})\). This also covers \(O_\pm\), clearly \(\Red{O_+} = O_+\) and \(\Red{O_-} = O_-\).
\end{rem}

\begin{rem}
Actually, for rational points \(P \in \CC(K)\) we do not have to worry about \(K\) being Henselian because the minimal polynomial of \(x\) remains irreducible after reduction.
\end{rem}

\begin{rem}
The reduction map extends to the algebraic closure \(\Kbb\) of \(\hat K\). Namely, for \(x\) algebraic over \(\hat K\), we define the valuation
\begin{equation*}
\nu(x) = \ifrac{\nu\left(\Nm_{\hat K(x)/\hat K}(x)\right)}{[\hat K(x):\hat K]}.
\end{equation*}
In fact the integral closure \(\Obb\) of \(\hat \O\) in \(\K\) is still a valuation ring (but no longer discrete).
\end{rem}

\pagebreak
We also get a reduction map for the Jacobian: it is defined as a quotient of divisors of degree \(0\) modulo principal divisors, and the latter are preserved by the reduction map:
\begin{prop}
\label{reduction-principal-divisor}
Let \(\di{D}\) a principal divisor over \(\CC(\Kbb)\). Then also the divisor \(\Red{\di{D}}\) over \(\CCred(\closure k)\) is principal.
\end{prop}
\begin{proof}
Let 
\begin{equation*}
\di{D} = \Div f = \divsum n_P \cdot \pd{P},
\end{equation*}
a principal divisor over \(\CC(\Kbb)\) with \(f \in \Kbb(X, Y)\). As only finitely many \(n_P \neq 0\), we may assume all \(P \in \CC(\hat K)\), and \(f \in \hat K(X, Y)\) by passing to a finite extension of \(\hat K\).

Of course we may write \(f = g/h\) with \(g, h \in \hat \O[X, Y]\), and multiply with a suitable power of the uniformiser \(\pi\) such that \(\nu(g) = \nu(h) = 0\) (recall that \(\hat \O[X, Y] \subset \laurinx{\hat \O}\)). Thus \(f\) is also a rational function on \(\XX\) which does not vanish nor has a pole on all of \(\XX_\mm\) (so there is no vertical component), because both \(\Red{g} \neq 0\) and \(\Red{h} \neq 0\). 

Moreover, note that for every \(P \in \CC(\hat K)\) we have that \(\closure{\{ P \}}\) is a zero or pole (with multiplicity \(n_P\)) of \(f\) on \(\XX\). This follows from zeroes and poles being Zariski-closed, or the stalks being isomorphic: \(\O_{\XX_0,P} \iso \O_{\XX,\closure{\{P\}}}\) (see the proof of Lemma 8.3.3 and Definition 7.1.27 of multiplicities in \cite{liu-2002-algebraic-geometry-arithmetic}). So
\begin{equation*}
\di{D}_\XX = \Div f_\XX = \divsum n_P \cdot \pd{\closure{\{P\}}},
\end{equation*}
which we intersect with \(\XX_\mm\) to get the divisor
\begin{equation*}
\di{D}_{\XX_\mm} = \divsum n_P \cdot \pd{\Red{P}}.
\end{equation*}
We wish to show that this is the divisor of the function \(f_\mm = \Red{g} / \Red{h} \in k(X, Y)\). Let \(P \in \XX\) a closed point (i.e. a closed point in \(\XX_\mm\)), and consider the intersection number \(i_P(\cdot, \cdot)\). Without loss of generality, we may assume that \(n_P \geq 0\) (otherwise pass to \(-\di{D}\) and \(1/f\)). Now by Corollary 9.1.32 in \cite{liu-2002-algebraic-geometry-arithmetic}, we have \(i_P(\pd{\closure{\{P\}}}, \XX_\mm) = 1\). This implies (using Definitions 7.1.27 and 9.1.1 in \cite{liu-2002-algebraic-geometry-arithmetic}) that
\begin{equation*}
n_P = i_P(\Div f_\XX, \XX_\mm) = \mathrm{length} \, \O_{\XX,P} /\! \left( \spann{f_\XX} + \mm \, \O_{\XX,P} \right) = \mathrm{length} \, \O_{\XX_\mm,P} /\! \spann{f_\mm} = \ord_P \left( f_\mm \right)
\end{equation*}
and hence \(\di{D}_{\XX_\mm} = \Div f_\mm\) is principal as desired.
\end{proof}

\subsection{Reduction of torsion points and periodicity test}
\label{sec:org36b2a71}
While it is not so clear how the period length changes when reducing a periodic continued fraction, it is quite well understood how the torsion order of \(\j{\OO}\) can change:
\begin{thm}[Serre-Tate]
\label{thm-serre-tate-torsion-reduction}
Suppose \(D\) and \(\Red{D}\) are square-free. Let \(P \in \Jac\) be torsion of order \(n\), and suppose that \(\RedM{P} \in \Jacred\) has order \(m\). If \(\Char k = 0\), then \(n = m\), otherwise there exists \(e \in \N\) such that \(n = \pp^e \, m\) with \(\pp = \Char k\).
\end{thm}
\begin{proof}
As \(0 = \RedM{n \, P} = n \, \RedM{P}\), we see that \(\Redm : \Jac \to \Jacred\) restricts to a homomorphism of groups \(\Redm : \Jac[n] \to \Jacred[n]\) (the subgroups of points with torsion order dividing \(n\)). But the conditions we pose on \(D\) and \(\nu\) ensure that \(\Jac\) has good reduction at \(\nu\). So by Theorem 1 and Lemma 2 of \cite{serre-tate-1968-good-reduction-abelian}, for \(\Char k \notdiv n\), this map \(\Jac[n] \iso \Jacred[n]\) is actually an isomorphism of groups; this is always the case in zero characteristic.

For positive characteristic \(\Char k = \pp\), we may write \(n = \pp^e \, n'\) with \(\pp \notdiv n'\), and assume that \(\pp \notdiv m\): because \(m \div n\), we can remove any common power of \(\pp\) and go to a multiple of \(P\).

Now \(\pp^e \, P\) has order precisely \(n'\) not divisible by \(\pp\), so \(\RedM{\pp^e \, P}\) has the same order \(n'\). However, \(\pp^e\) is coprime with the order of \(\RedM{P}\), implying that \(\pp^e \, \RedM{P} = \RedM{\pp^e \, P}\) has likewise order \(m\). So \(n' = m\), and we are done.
\end{proof}

The above theorem enables an old trick to effectively test if a point is torsion, mentioned already in \cite{davenport-1981-the-integration-algebraic}, and described in \cite{yu-1999-arith} for hyperelliptic continued fractions.

\begin{rem}[Reduction modulo two primes]
\label{rem-periodicity-test-reduction-two-primes}
Given a square-free \(D \in K[X]\) with \(K\) some number field, \(\deg D\) even and \(\LC(D)\) a square as usual, we can always find two prime ideals \(\PP_1\) and \(\PP_2\) such that \(D \in \O_{\PP_i}[X]\), \(\nu_{\PP_i}(\LC(D)) = 0\) and \(D\) is square-free modulo \(\PP_i\), for \(i = 1,2\). Of course the residue fields are finite, and we may assume they are of different characteristics \(\pp_1\) and \(\pp_2\). Then \(\j{\OO_{\PP_i}}\) is torsion of order \(m_i\). Assuming that also \(\j{\OO}\) is torsion of order \(m\), we can write
\begin{equation*}
m = {\pp_1}^{e_1} \, m_1 = {\pp_2}^{e_2} \, m_2, \quad e_i \geq 0.
\end{equation*}
This implies \(e_1 \leq e'_1 = \nu_{\pp_1}(m_2)\) and \(e_2 \leq e'_2 = \nu_{\pp_2}(m_1)\), and moreover
\begin{equation*}
m \div \gcd({\pp_1}^{e'_1} \, m_1,  {\pp_2}^{e'_2} \, m_2) \div \lcm(m_1, m_2).
\end{equation*}
This already gives a bound for the torsion order \(m\) which translates into a bound for the period length via Proposition \ref{prop-bounds-torsion-period-length}. So we can test for periodicity effectively.

Indeed as \(m_1, m_2 \div m\), often it is even possible to immediately find a contradiction if \(m_1\) and \(m_2\) have too many different prime factors.
\end{rem}

\smallskip

Also if \(K\) is finitely generated over a number field, we can specialize \(D\) to be defined over a number field. If \(\j{\OO}\) is already torsion, this should not alter the torsion order, so we can lift the torsion bound as obtained above, and still determine effectively if \(\CF(\sqrt{D})\) is periodic.

However, we need to be careful to avoid bad reduction of the continued fraction: it might happen that specializing a non-Pellian \(D\), we end up with a Pellian \(D\). Usually, it should however not be a problem to find a specialization where this does not happen. See for example Proposition \ref{specialization-only-countably-many-bad-reduction} below.

\subsection{Shortening of quasi-period}
\label{sec:orgaebd15f}
Theorem \ref{thm-serre-tate-torsion-reduction} also gives a little bit of information on how the quasi-period may change in the case of bad reduction of the continued fraction.

Suppose that \(\CF(\sqrt{D})\) has quasi-period length \(\QPL\). Set \(d_i = \deg a_i < d = \deg D, \; i = 1, \dots, \QPL-1\). The torsion order of \(\j{\OO}\) is
\begin{equation}
\label{shortening-degrees-sum-above}
m = \deg p_{\QPL - 1} = d + d_1 + \dots + d_{\QPL-1}.
\end{equation}
Because the quasi-period is palindromic (see Proposition \ref{palindromic-period}) we have \(d_i = d_{\QPL - i}\).

Assuming that \(\Red{D}\) is not a square, with \(\CF(\sqrt{\Red{D}})\) having quasi-period length \(\QPL'\), we set \(d'_i = \deg c_i < d, \; i = 1, \dots, \QPL'-1\), with \(d'_i = d'_{\QPL-i}\). The torsion order of \(\j{\OOred}\) is then
\begin{equation}
\label{shortening-degrees-sum-below}
m' = \deg u_{\QPL' - 1} = d + d'_1 + \dots + d'_{\QPL'-1},
\end{equation}
where \(m = \pp^e \, m'\) for some non-negative integer \(e\) if \(\Char k = \pp\), and \(m = m'\) if  \(\Char k = 0\). 

If \(m' \neq m\), then \eqref{shortening-degrees-sum-below} has to be repeated \(\pp^e\) times to make up \eqref{shortening-degrees-sum-above}. Recall from Corollary \ref{cor-lambda-degree-sum} that each \(d'_i = d_{i_1} + \dots + d_{i_j}\).

For example
\begin{align*}
d'_1 &= d_1 + \dots + d_{j_1}, &  d'_{\QPL-1} &= d_{\QPL-1} + \dots + d_{\QPL - j_1} \\
d'_2 &= d_{j_1 + 1} + \dots + d_{j_2}, \\
d'_3 &= d_{j_2 + 1} + \dots + d_{j_3},
\end{align*}
and so on. Of course \(m' = m\) does not prevent bad reduction of \(\CF(\sqrt{D})\), as the \(d'_i\) might just be larger than the \(d_i\) -- a better criterion is to check if \(\QPL' = \QPL\).

Unfortunately, we do not get a lot more information about these degrees in general. But in some special cases, we can at a glance exclude the possibility of bad reduction of the continued fraction:

\begin{itemize}
\item If the sequence of \(\deg a_n\) starts with \(2, 1, 1, 1, 2, 1, 1, 1, 2, 1, \dots\), then bad reduction of \(\CF(\sqrt{D})\) is impossible because \(\deg c_n\) cannot follow the sequences  \(2, 2, 1, 2, 2, 1,  \dots\) or \(2, 1, 2, 2, 1, 2, \dots\). Then some complete quotients \(\gamma_n\) would have quasi-period length \(1\), but others would have quasi-period length \(2\), which is impossible.
\item Similarly, if the \(\deg a_n\) start with \(3, 1, 2, 1, 3, 1, 2, 1, 3, 1, \dots\), then bad reduction of \(\CF(\sqrt{D})\) would make \(\deg c_n\) start with \(3, 3, 1, 3, 3, 1, \dots\) or \(3, 1, 3, 3, 1, 3, \dots\). As above, this is not possible.
\end{itemize}

\subsection{Reduction of convergent divisors}
\label{sec:org9851fc8}
We now attempt to give a geometric description for the reduction of a hyperelliptic continued fraction, in terms of the divisors associated to convergents.

Recall from Section \ref{sec:org4e4485e} that we can write the divisors of the canonical convergents of \(\alpha \in K(X, Y)\) as
\begin{multline*}
\Div (p_n - \alpha \, q_n) = \Div (\vartheta_n) = \\ -(\deg p_n) \pd{O_-} - \pd{Q_1} - \dots - \pd{Q_h} + (\deg q_{n+1}) \pd{O_+} + \pd{P_1^n} + \dots + \pd{P_{e_n}^n}
\end{multline*}
where \(e_n = \deg a_0 - \deg a_{n+1} + h\). If \(\alpha = Y\) (\(= \sqrt{D}\)), then this divisor satisfies \(P_i^n \neq \sigma(P_j^n)\) if \(i \neq j\) because \(p_n\) and \(q_n\) are coprime.

What happens when we reduce this divisor, and pass to
\begin{multline*}
\Div (\Redn{p_n} - \gamma \, \Redn{q_n}) = \Div (\Red{\vartheta_n}) = \\ -(\deg p_n) \pd{\Red{O_-}} - \pd{\Red{Q_1}} - \dots - \pd{\Red{Q_h}} + (\deg q_{n+1}) \pd{\Red{O_+}} + \pd{\Red{P_1^n}} + \dots + \pd{\Red{P_{e_n}^n}}
\end{multline*}
as in the proof of Proposition \ref{reduction-principal-divisor}?

\begin{enumerate}
\item Of course \(\Red{O_\pm} = O_\pm\).
\item The \(\Red{Q_i}\) are always the same, and we can control them from the start.
\item It is possible that \(\Red{P^n_i} = O_+\) which means \(\io{\Red{\vartheta_n}} > \io{\vartheta_n}\).
\item Or \(\Red{P^n_i} = O_-\) which means \(\degb{\Redn{q_n}} < \deg q_n\).
\item Or if \(\alpha = Y\), then possibly \(\Red{P^n_i} = \sigma(\Red{P^n_j})\) for some \(i \neq j\). This corresponds to \(\Redn{p_n}\) and \(\Redn{q_n}\) sharing a common factor.
\item Otherwise, \(\Red{P^n_i}\) is just a finite point.
\end{enumerate}

In the case \(\alpha = Y\) with \(g = 1\), we have \(e_n \leq 1\), so there is at most \(P^n_1\) and we do not have to worry about case 5. We also know that \(P^n_1\) must be \(K\)-rational. But for higher genus, we may need to work over an algebraic extension of \(K\) (not necessarily the algebraic closure because the degree of the equations defining the \(P^n_i\) is uniformly bounded in terms of \(\deg D\)).

Let us have a closer look at the genus \(1\) case, and study how it is related to the valuation analysis from Theorem \ref{thm-genus1-zero-patterns}.

\begin{prop}
Under the same hypotheses as for Theorem \ref{thm-genus1-zero-patterns} and additionally \(D, \Red{D}\) square-free, we have for \(n \geq 0\)
\begin{equation*}
-(n+2) \, \j{\OO} = j(P_n) \quad \text{ where } P_n = (x_n, y_n) \in \CCa.
\end{equation*}
Let \(\j{\OOred}\) have torsion order \(m = \QPL +1\), then
\begin{enumerate}
\item If \(m \div n+2\), then \(\Red{P_n} = O_+\) and \(\nu(x_n) = - f_{n+1} < 0\).
\item If \(m \div n+1\), then \(\Red{P_n} = O_-\) and \(\nu(x_n) = - f_n < 0\).
\item Otherwise \(\Red{P_n}\) is a finite point and \(\nu(x_n) \geq 0\).
\end{enumerate}
\end{prop}

\begin{rem}
Notice that \(P_n\) reduces to infinity precisely when \(n\) is in a two element fibre (compare Proposition \ref{bad-reduction-periodic-deg4-fibres}). 
\end{rem}

\begin{proof}
We have \(\alpha = Y\), \(g = 1\), and \(\CF(\sqrt{D})\) non-periodic. This implies \(\deg p_n = n+2\) and
\begin{equation*}
\Div(\vartheta_n) = -(n+2) \, \pd{O_-} + (n + 1) \, \pd{O_+} + \pd{P_n}
\end{equation*}
where \(P_n  = (x_n, y_n) \in \CCa(K)\) for all \(n \geq 0\) because \(\deg a_n = 1\) for \(n \geq 1\).
\pagebreak

The reduction of this divisor is
\begin{equation*}
\Div(u_{\lambda(n)} - Y \, v_{\lambda(n)}) = -(n+2) \, \pd{O_-} + (n + 1) \, \pd{O_+} + \pd{\Red{P_n}}.
\end{equation*}
With \(\CF(\gamma)\) periodic, the point \(\OOred\) over \(k\) has torsion order \(m = \QPL + 1\).

\begin{enumerate}
\item If \(m \div n+2\), this forces \(\Red{P_n} = O_+\).
\item If \(m \div n+1\), this forces \(\Red{P_n} = O_-\).
\item Otherwise \(\Red{P_n} \in \CCa(k)\).
\end{enumerate}

Recall from Proposition \ref{introduce-theta-n} and \eqref{cf-sn-pell-eq} that
\begin{equation*}
\vartheta_n \, \sigma(\vartheta_n) = {g_n}^{-2} \, (-1)^{n+1} \, s_{n+1} = b_n (X - x_n)
\end{equation*}
for some \(b_n \in \units K\). From the normalisation of \(\vartheta_n\) with \(\nub{\vartheta_n} = 0\) (and analogously \(\nub{\sigma(\vartheta_n)} = 0\)), we get \(\nub{b_n (X - x_n)} = 0\).

\begin{enumerate}
\item In the case \(m \div n+2\), we have
\begin{equation*}
\io{\Red{\vartheta_n}} = 1 + \io{\vartheta_n}, \quad \io{\Red{\sigma(\vartheta_n)}} = \io{\sigma(\vartheta_n)}
\end{equation*}
which means \(f_{n+1} > 0\). In fact,
\begin{equation*}
f_{n+1} = \nub{\LC(\vartheta_n)} + \nub{\LC(\sigma(\vartheta_n))} = \nub{b_n}.
\end{equation*}
This forces \(\nub{x_n} = - \nub{b_n} = - f_{n+1} < 0\) because \(\nub{b_n (X - x_n)} = 0\), so \(x_n\) has \negval as expected.

\item In the case \(m \div n+1\), we have
\begin{equation*}
\io{\Red{\vartheta_n}} = \io{\vartheta_n}, \quad \io{\Red{\sigma(\vartheta_n)}} = 1 + \io{\sigma(\vartheta_n)}.
\end{equation*}
so \(f_{n+1} = 0\). But \(\nub{\LC(\sigma(\vartheta_n))} = \nu(\LC(\normal{q_n})) = f_n > 0\), and similarly as above \(\nub{x_n} = - \nub{\LC(\sigma(\vartheta_n))} = - f_n < 0\). We find that \(x_n\) has again \negval.

\item Otherwise, we have
\begin{equation*}
\io{\Red{\vartheta_n}} = \io{\vartheta_n}, \quad \io{\Red{\sigma(\vartheta_n)}} = \io{\sigma(\vartheta_n)}
\end{equation*}
and hence \(f_{n+1} = 0\). This implies \(\nub{b_n} = 0\) and thus \(\nub{x_n} \geq 0\), i.e. \(x_n \in \O\).
\end{enumerate}
\end{proof}
Observe how this matches Theorem \ref{thm-genus1-zero-patterns} and that \(-f_{n+1}\) is the valuation of both \(x_n\) and \(x_{n+1}\) at the two element fibre.

So we have found a second description of the \(f_n\) from Theorem \ref{thm-genus1-zero-patterns}. Unfortunately, this still does not suggest what type of patterns they might follow, or whether they might be bounded. Generally, we should not expect the \(f_n\) to be bounded, see for example the proposition on page 55 of \cite{silverman-tate-2015-rational-points-on}. It suggests that for an elliptic curve defined over \(\Q\), there are rational points \(P = (x,y)\) with arbitrarily low \(\nu_\pp(x)\) for any prime \(\pp\).
\section{Good reduction at infinitely many primes}
\label{sec:org53fcc2b}
We mentioned before that Theorem \ref{thm-infinite-poles-number-field} holds only for \(\CF(\sqrt{D})\), but not for other elements of \(K(X, \sqrt{D})\).

\begin{thm}
\label{thm-good-reduction-infinite-primes}
Let \(D = X^4 + 16 \, X^2 + 24 \, X + 9\) which is not Pellian (the torsion orders of \(\OO_3\) and \(\OO_{17}\) differ just by \(1\)). Set \(\alpha = \frac{\sqrt{D} - 3}{X}\).

There are infinitely many primes \(\pp\) for which \(\CF(\alpha)\) has good reduction (hence \(\pp\) never divides a denominator of a partial quotient \(a_n\)).
\end{thm}

This is related to questions treated in \cite{corrales-schoof-1997-support-problem-its}, and earlier in \cite{schinzel-1960-the-congruence-a}. Here we present an explicit proof for our particular example, to illustrate these arguments more concretely. In fact, the given problem boils down to an analogue for elliptic curves (more generally abelian varieties) of
\begin{prop}
There exist infinitely many prime numbers \(\pp\) such that for all \(n \in \Z\) we have \(2^n \not\equiv 5 \mod \pp\).\footnote{The reader might find it enjoyable to try and prove this exercise for himself.}
\end{prop}

\begin{proof}[Proof of Theorem \ref{thm-good-reduction-infinite-primes}]
Note that \(D\) non Pellian implies that \(\CF(\alpha)\) is not quasi-periodic by Theorem A in \cite{berry-1990-periodicity-continued-fractions}.

Recall from Proposition \ref{general-convergent-divisor-lemma} that the divisors induced by the convergents have the shape
\begin{equation*}
\Div(p_n - \alpha \, q_n) = (\deg p_n) \, \OO - \pd{Q} + \pd{P_n}
\end{equation*}
because \(\alpha\) has a single pole at \(Q = (0, ?)\). Note that \(X \div D - 3^2\), so we are in the situation of Theorem \ref{thm-quadratic-laurent-series-cf-representation}. In principle, we could have \(P_n = O_+\) for a single \(n\), but the reduction arguments below imply that \(P_n\) must always be a finite point.

Recall from Section \ref{sec:org9851fc8} that here bad reduction of \(\CF(\alpha)\) is equivalent to \(P_n\) reducing to a point at infinity \(O_\pm\) which means
\begin{equation*}
j(\Red{Q}) + (\deg p_n) \, j(O_-) = j(O_\pm) = 0 \text{ or } j(O_-) \quad \mod \pp.
\end{equation*}
So if we ensure that for all \(m \in \Z\)
\begin{equation}
\label{eq-good-red-inf-reduced-translate}
j(\Red{Q}) + m \, j(O_-) \neq 0 \quad \mod \pp
\end{equation}
then we know that we must have good reduction of \(\CF(\alpha)\) at this prime \(\pp\) (and additionally we confirm that \(P_n \neq O_\pm\) for all \(n\)).

We deduce from the \v{C}ebotarev density theorem (see Theorem 13.4 and Lemma 13.5 in \cite{neukirch-1999-algebraic-number-theory}) that this holds for infinitely many primes \(\pp\). For reasons of space, we will assume the reader is already familiar with this famous theorem, and also ramification of prime ideals, Galois theory and the Frobenius automorphism.

Our curve \(\CC\) is an elliptic curve, isomorphic to its Jacobian. We write it in Weierstrass form
\begin{equation*}
\EC : V^2 = U^3 + 16 \, U^2 - 36 \, U = U (U - 2) (U + 18)
\end{equation*}
using the transformation
\begin{equation*}
U = 2 X^2 + 2 Y, \qquad V = 2 X (U + 16) + 24
\end{equation*}
which sends \(j(O_-)\) to \(R_1 = (-16, -24)\) and \(j(Q)\) to \(R_2 = (6, 24)\) (over \(\Q\)). These are non-torsion rational points. Note that the \(2\) torsion points of \(\EC\) are by design rational. This implies that for any of the four choices for a point \(R_i' \in \CC\) with \(2 \, R_i' = R_i\), the point \(R_i'\) is defined over the same number field \(K_i\) (for \(i=1,2\)). Here the fields are
\begin{equation*}
K_1 = \Q(\zeta), \text{ where } \zeta^4 + 1 = 0, \qquad K_2 = \Q(\sqrt{6}).
\end{equation*}
The composite field \(K_1 K_2\) has degree \(8\) and is Galois, with abelian Galois group \(G = \Gal(K_1 K_2 / \Q)\). We denote by \(H_i\) the subgroup of \(G\) whose fixed field is \(K_i\). 

Ignoring the finitely many primes where the curve \(\EC\) has bad reduction (just \(2, 3, 5\), the discriminant of \(\EC\) is \(2^{12} \cdot 3^{4} \cdot 5^{2}\)), we now wish to find infinitely many primes \(\pp\) such that \(\Red{R_1'}\) is a \(\F_\pp\)-rational point, but \(\Red{R_2'}\) is not. 
In that case, the point \(m \, \Red{R_1'} + \Red{R_2'}\) cannot be \(\F_\pp\)-rational for any \(m \in \Z\). In particular it is not \(0\) or torsion of order \(2\). This implies that for all \(m \in \Z\)
\begin{equation*}
m \, R_1 + R_2 \neq 0 \mod\pp
\end{equation*}
which is equivalent to \eqref{eq-good-red-inf-reduced-translate}.

If we restrict to the (infinitely many) primes \(\pp\) which are unramified over \(K_1 K_2\) (and hence over \(K_1\) and \(K_2\)), the condition on rationality of \(\Red{R_1'}\) and \(\Red{R_2'}\) amounts to saying that \(\pp\) has a prime divisor \(\PP_1\) of degree \(1\) over \(K_1\) (i.e. the residue field \(k(\PP_1)\) has degree \(1\) over \(\F_\pp\)), but over \(K_2\) all the prime divisors of \(\pp\) have degree \(> 1\) (i.e. \([k(\PP_2):\F_\pp] > 1\) for any prime divisor \(\PP_2\)).

As described in Lemma 13.5 of \cite{neukirch-1999-algebraic-number-theory}, this happens \IFF the conjugacy class of the Frobenius automorphisms of prime ideals of \(K_1 K_2\) lying over \(\pp\) intersects \(H_1\), but not \(H_2\). Here \(H_1 = \{\id, \sigma_1\}\) where \(\sigma_1\) is defined by \(\sigma_1(\zeta) = \zeta\) and \(\sigma_1(\sqrt{6}) = - \sqrt{6}\). The conjugacy classes are trivial because \(G\) is abelian, so we are looking precisely for the primes \(\pp\) for which \(\sigma_1\) is a Frobenius automorphism of some prime \(\PP\) of \(K_1 K_2\) over \(\pp\).

But the set of these primes has positive density \(\geq \frac{1}{8}\) by the \v{C}ebotarev density theorem (Theorem 13.4 in \cite{neukirch-1999-algebraic-number-theory}), so in particular there are infinitely many of them.
\end{proof}

\begin{rem}
In our computations, we observed that among the first 100 odd prime numbers, there are 62 prime numbers \(\pp\) for which \(\CF(\alpha)\) has good reduction at \(\pp\). This is a much higher density than predicted by our \v{C}ebotarev density estimate, but of course the latter gives only a sufficient condition.

Moreover, we argued with \(R_i'\) satisfying \(2 \, R_i' = R_i\). Instead we could argue with \(m \, R_i' = R_i\) for any \(m\); then we probably get additional primes with the desired property.
\end{rem}

\section{Complex functions case}
\label{sec:org6b10a31}
Let us now discuss the situation of specialization, i.e. where \(K = \C(t)\) and the reduction map works by assigning a special value \(t_0 \in \C\) to \(t\). The corresponding valuation is \(\nu = \ord_{t_0}\) measuring the zero-order at \(t_0\), with uniformising parameter \(t - t_0\), and the discrete valuation ring \(\O = \C[t]_{\spann{t - t_0}}\), the localisation of the maximal ideal \(\spann{t - t_0}\) in \(\C[t]\). We also write \(\Red{\alpha} = \alpha_{t=t_0}\) to distinguish different specializations.

\subsection{Results of Masser and Zannier}
\label{sec:org087aad4}
This is precisely the situation found in the article of Masser and Zannier on the connection between the Pell equation and Unlikely intersections \cite{masser-zannier-2015-torsion-points-on}. Let me restate their results in our language of specialization of continued fractions.

For genus \(1\), they mention the following result:

\begin{prop}[Masser-Zannier]
\label{mz-thm-spec-periodic-deg4}
Let \(D = X^4 + X + t\), then \(\CF(\sqrt{D})\) is non-periodic. The set of \(t_0\) such that \(\CF(\sqrt{D_{t=t_0}})\) is periodic (i.e. \(\CF(\sqrt{D})\) has bad reduction at \(t - t_0\) by Proposition \ref{bad-reduction-is-periodic-deg4}), is infinite and denumerable.
\end{prop}

For genus \(2\) however, their Theorem P1 says:
\begin{thm}[Masser-Zannier]
\label{mz-thm-spec-periodic-deg6}
Let \(D = X^6 + X + t\), then \(\CF(\sqrt{D})\) is non-periodic. The set of \(t_0\) such that \(\CF(\sqrt{D_{t=t_0}})\) is periodic is \emph{finite}.
\end{thm}
For example \(\CF(\sqrt{D_{t=0}})\) is periodic. But because \(\deg D = 6\), it is now possible that \(\CF(\alpha)\) has  bad reduction at \(t-t_0\), even if \(\CF(\sqrt{D_{t=t_0}})\) is non-periodic.

\begin{table}[h!]
\centering
\begin{tabular}{|l|l|}
\hline
\(D = X^{6} + X + t\) & basefield \(\C(t)\)\\ \hline
Discriminant of $D$: \((-46656) \cdot (t^{5} - \frac{3125}{46656})\) & $D$ never reduces to a square.\\ \hline
\(D\) is not Pellian & \\ \hline
\multicolumn{2}{|l|}{Partial quotients}\\ \hline
\multicolumn{2}{|l|}{\(a_{0} = X^{3}\)}\\ \hline
\multicolumn{2}{|l|}{\(a_{1} = 2 X^{2} - 2 t X + 2 t^{2}\)}\\ \hline
\multicolumn{2}{|l|}{\(a_{2} = \frac{-\frac{1}{2} X - \frac{1}{2} t}{t^{3}}\)}\\ \hline
\end{tabular}
\end{table}

\enlargethispage{1cm}
To describe the \(t_0\) in the theorem, we need to search for an increase in the degree of the partial quotients when specialising, as seen in Lemma \ref{bad-reduction-minimal-pole}. Clearly \(\deg c_0 = 3, \deg c_1 = 2\) for the partial quotients of any specialization. Their Theorem P2 says:
\begin{thm}[Masser-Zannier]
\label{mz-thm-spec-weak-pell-deg6}
Let \(D = X^6 + X + t\), with \(\CF(\sqrt{D})\) non-periodic. The set of \(t_0\) such that for \(\gamma = \sqrt{D_{t=t_0}}\) there exists \(n \geq 2\) with \(\deg c_n = 2\), is an infinite and denumerable subset of \(\closure \Q\).
\end{thm}
This set also includes the \(t_0\) with \(\CF(\sqrt{D_{t=t_0}})\) periodic: because the period always begins at \(c_1\), there are then infinitely many \(n\) with \(\deg c_n = 2\) for each of these \(t_0\). By Theorem \ref{cf-good-red-partial-quotients} the increase of degrees is necessary for bad reduction of \(\CF(\sqrt{D})\), so this infinite set is actually the set of all \(t_0\) with bad reduction of \(\CF(\alpha)\) at \(t-t_0\).

The hard part in the proof of this theorem is showing that this set of \(t_0\) is \emph{infinite} which is done in Section 11 of \cite{masser-zannier-2015-torsion-points-on}.

With the theory of Chapter \ref{sec:orgd5f1900}, it is however not so hard to show:

\begin{prop}
\label{specialization-only-countably-many-bad-reduction}
Let \(\alpha \in \laurinx{\C(t)}\) with \(\LC(\alpha) = 1\). Then there exist at most countably many \(t_0 \in \C\) such that \(\CF(\alpha)\) has bad reduction at \(t-t_0\) (the valuation being \(\nu = \ord_{t=t_0}\)).
\end{prop}
\begin{proof}
For every \(n \geq 0\), let \(d_n\) the denominator of \(\LC(\alpha_{n}) \in \C(t)\) which is a monic polynomial in \(\C[t]\). Clearly \(\ord_{t=t_0}(\LC(\alpha_n)) < 0\) holds \IFF \(d_n(t_0) = 0\).

Then if \(\CF(\alpha)\) has bad reduction at \(t-t_0\), there exists by Proposition \ref{bad-reduction-minimal-pole} at least one \(n\) such that \(d_n(t_0)\).

Of course there are only countably many polynomials \(d_n\), each with finitely many zeroes (even though \(\deg d_n\) might increase with \(n\)). Hence there are only countably many possibilities for bad reduction of \(\CF(\alpha)\).
\end{proof}

Masser and Zannier also give another example in degree \(6\), with different behaviour (see Section 3.4.5 in \cite{zannier-2012-some-problems-unlikely}):
\begin{prop}[Masser-Zannier]
\label{mz-spec-infinite-periodic-deg6}
Let \(D = X^6 + X^2 + t\), then \(\CF(\sqrt{D})\) is non-periodic. The set of \(t_0\) such that \(\CF(\sqrt{D_{t=t_0}})\) is periodic, is infinite and denumerable.
\end{prop}

\subsection{Repeated occurrences of \(t-t_0\)}
\label{sec:orgf9f9759}
Also for the specialisation case we can say something about the occurrence of ``primes'' in infinitely many denominators of partial quotients \(a_n\). In degree \(4\), we can simply use Corollary \ref{cor-infinite-poles-deg4} to deduce this from Proposition \ref{mz-thm-spec-periodic-deg4} above:

\begin{cor}
Let \(D = X^4 + X + t\) as in Proposition \ref{mz-thm-spec-periodic-deg4}. For each of the infinitely many \(t_0\) where \(\CF(\sqrt{D})\) has bad reduction, there exist infinitely many \(n\) such that \(t-t_0\) appears in a denominator of \(a_n\).
\end{cor}

In degree \(6\), the situation becomes more subtle, and we need to use Proposition \ref{fibre-conditions-infinite-poles}. If \(D = X^6 + X + t\), then for \(t_0 = 0\), there is the problem that the only fibre of \(\lambda : \N_0 \to \N_0\) (the map describing the reduction of convergents, introduced in Section \ref{sec:orgea2cbbe}) with a single element is \(\inv\lambda(0) = \{0 \}\), so we cannot apply the proposition. This happens because for this specialization we have \(\deg c_n \neq 1\) for all \(n\).

\begin{table}[h!]
\centering
\begin{tabular}{|l|l|}
\hline
\(D = X^{6} + X\) & basefield \(\Q\)\\ \hline
Discriminant of $D$: \(5^{5}\) & $D$ never reduces to a square.\\ \hline
\multicolumn{2}{|l|}{period length \(2\) for \(\CF(\sqrt{D})\)}\\ \hline
\multicolumn{2}{|l|}{Minimal Pell solution}\\ \hline
\(p_{1} = 2 X^{5} + 1\) & \(q_{1} = 2 X^{2}\)\\ \hline
\multicolumn{2}{|l|}{Partial quotients of \(\sqrt{D}\)}\\ \hline
\multicolumn{2}{|l|}{\(a_{0} = X^{3}\)}\\ \hline
\multicolumn{2}{|l|}{\(a_{1} = 2 X^{2}\)}\\ \hline
\multicolumn{2}{|l|}{\(a_{2} = 2 X^{3}\)}\\ \hline
\end{tabular}
\end{table}

For all other \(t_0\), we can use Proposition \ref{prop-criterion-infinite-single-element-fibres} because clearly \(\deg c_2 = 1\) for any specialization to \(t_0 \neq 0\), and the minimal degree \(\delta\) is of course \(1\). This implies infinitely many fibres with a single element, and from Lemma \ref{bad-reduction-to-periodic} and Proposition \ref{fibre-conditions-infinite-poles} then follows:

\begin{cor}
Let \(D = X^6 + X + t\) as in Theorem \ref{mz-thm-spec-periodic-deg6}. For each of the finitely many \(t_0 \neq 0\) where \(\CF(\sqrt{D_{t=t_0}})\) is periodic, there exist infinitely many \(n\) such that \(t-t_0\) appears in a denominator of \(a_n\).
\end{cor}

It is likely this property also holds for \(t_0 = 0\), but this would require a different argument, for example an analogue to Theorem \ref{thm-genus1-zero-patterns} for degree \(6\) which is unfortunately not in sight.

For \(D = X^6 + X^2 + t\), we have however \(a_1 = 2 \, X\), so Proposition \ref{prop-criterion-infinite-single-element-fibres}  implies infinitely many fibres of \(\lambda\) with a single element for every \(t_0\), so again by Lemma \ref{bad-reduction-to-periodic} and Proposition \ref{fibre-conditions-infinite-poles} follows:
\begin{cor}
Let \(D = X^6 + X^2 + t\) as in Proposition \ref{mz-spec-infinite-periodic-deg6}. For each of the finitely many \(t_0 \neq 0\) where \(\CF(\sqrt{D_{t=t_0}})\) is periodic, there exist infinitely many \(n\) such that \(t-t_0\) appears in a denominator of \(a_n\).
\end{cor}

Now let us return to \(D = X^6 + X + t\) and consider the remaining \(t_0\) with non-periodic bad reduction of the continued fraction. To understand this better, we need Theorem 1.3 from \cite{zannier-2016-hyper-contin-fract} (which we state for arbitrary base field \(\K\)):
\begin{thm}[Zannier]
\label{thm-zannier-an-deg-bound}
Let \(D \in \K[X]\) of degree \(2d\), non-square, but \(\LC(D)\) square. Suppose further that if \(D = E^2 \, D'\) with \(D'\) Pellian, then \(\deg D' \leq \frac{3}{2} d\) (for example assume \(D\) is square-free). Then there are only finitely many \(n\) with \(\deg a_n > \frac{d}{2}\).
\end{thm}
This is a consequence of a Skolem-Mahler-Lech theorem for algebraic groups (for example the Jacobian of \(\CC\)) explained in the same article.

For \(\deg D = 6\), in particular for \(D = X^6 + X + t\) and its specializations, this means that if \(\CF(\sqrt{D})\) is not periodic, only finitely many partial quotients have degree \(\deg a_n > 1\); and the same property holds for \(\CF(\sqrt{D_{t=t_0}})\). This means that for the \(t_0\) with non-periodic \(\CF(\sqrt{D_{t=t_0}})\), the corresponding map \(\lambda\) has only finitely many fibres with more than one element. Again, we cannot apply Proposition \ref{fibre-conditions-infinite-poles}. This does not mean there might not be infinitely many \(a_n\) with \(t-t_0\) in the denominator, but for every \(n\) large enough, we can normalise the complete quotient \(\alpha_n\) to
\begin{multline*}
\normal{\alpha_n} = (t-t_0)^{-\ord_{t_0}(\alpha_n)} \, \alpha_n = \mu \, \alpha_n \\ = [\mu \, a_n, \inv{\mu} \, a_{n+1}, \mu \, a_{n+2}, \inv \mu \, a_{n+3}, \dots] = [b_0, b_1, b_2, b_3, b_4, \dots]
\end{multline*}
such that none of the \(b_i\) has \(t-t_0\) in a denominator. So \(\CF(\normal{\alpha_n})\) has good reduction at \(t-t_0\).

In fact, we cannot even exclude the possibility that \(\mu = 1\), in which case only finitely many \(a_n\) have \(t- t_0\) in the denominator, despite there being bad reduction of \(\CF(\alpha)\) at \(t -t_0\) (compare Remark \ref{infinite-poles-after-renormalising}).
\chapter{Heights}
\label{sec:org56e7cb7}
While the valuations used in the previous chapters give a local estimate for the complexity of the partial quotients, affine and projective heights provide a global measure of complexity. For the convergents, and more generally Padé approximations, the projective logarithmic height of the convergents has known lower bounds (in the non-Pellian case), see \cite{bombieri-cohen-1997-siegels-lemma-pade}: they should increase at least quadratically. A lower bound for the height of the partial quotients is more delicate, and has been found only recently: \cite{zannier-2016-hyper-contin-fract} gives a lower bound for affine logarithmic height, with at least quadratic growth for a frame of fixed length of partial quotients.

Upper bounds for the projective heights of the partial quotients follow from those of the convergents.
\section{Heights}
\label{sec:org743079c}
For the convenience of the reader, I will list some definitions and properties of heights to be used in this chapter, following mainly \cite{bombieri-gubler-2006-heights-diophantine-geometry} in notation and normalisation. For the Weil height machine, I follow \cite{hindry-silverman-2000-diophantine-geometry}.
\subsection{Places and Product formula}
\label{sec:orgfdc4da2}
For a number field \(K\), one defines the set of \emph{places} \(M_K\) as the equivalence classes of non-trivial absolute values on \(K\), where two absolute values are equivalent if they induce the same topology. A place contains either only archimedean absolute values, and then is called \emph{infinite}, or only non-archimedean absolute values, in which case we call it \emph{finite}. The infinite places correspond to embeddings of \(K\) into \(\C\) up to complex conjugation, hence there is only a finite number. The finite places correspond to prime ideals of the ring of integers of \(K\), so there are infinitely many.

In order to define heights, one carefully chooses and fixes an absolute value to represent a place. For \(\Q\), there is one infinite place, the restriction of the standard complex absolute value, and countably many finite places corresponding to the prime numbers. We represent these places by
\begin{align*}
M_\Q &= \{ \pp \in \N \text{ prime } \} \cup \{ \infty \}, \\
\abs{x}_\infty &= \max(x,-x), \\
\abs{x}_\pp &= \pp^{-n} \text{ for } x \neq 0  \text{ where } x = \pp^n \, \frac{a}{b} \text{ with } n \in \Z, a, b \in \Z \setminus \pp \Z
\end{align*}

This ensures \(K = \Q\) satisfies the \emph{product formula}
\begin{equation*}
\prod_{\nu \in M_\Q} \abs{x}_\nu = 1 \text{ for } x \in \units \Q.
\end{equation*}
The product on the left is well defined because for a fixed \(x\), only finitely many \(\nu \in M_\Q\) have \(\abs{x}_\nu \neq 1\).

On any number field \(K\), a place \(\omega\) on \(K\) restricts to a unique place \(\nu\) on \(\Q\), written \(\omega \div \nu\), and we can choose a cleverly normalised representative \(\omega\) such that
\begin{equation*}
\prod_{\omega \div \nu} \abs{x}_\omega = \abs{x}_\nu \text{ for all } x \in \units \Q.
\end{equation*}
In consequence, on any number field \(K\), there is a \emph{product formula}
\begin{equation*}
\prod_{\omega \in M_K} \abs{x}_\omega = 1 \text{ for } x \in \units K,
\end{equation*}
where on the left side only finitely \(\omega\) have \(\abs{x}_\omega \neq 1\).

We also remark that \(\omega\) normalised in this way satisfies an \emph{improved triangle equality}: Let \(x_1, \dots, x_r \in K\), then
\begin{equation}
\label{improved-triangle-equality}
\abs{x_1 + \dots + x_r}_\omega \leq \max(1, \abs{r}_\omega) \, \max\left(\abs{x_1}_\omega, \dots, \abs{x_r}_\omega\right).
\end{equation}

\subsection{Height on projective and affine space}
\label{sec:org5c9a4ee}
The product formula allows defining an \emph{exponential absolute projective height} on \(\P^n(K)\) for a number field \(K\) by setting
\begin{equation*}
\Hproj(x_0 : \dots : x_n) = \prod_{\nu \in M_K} \max\left( \abs{x_0}_\nu, \dots, \abs{x_n}_\nu \right)
\end{equation*}
By considering \(\A^n(K) \subset \P^n(K)\), this also defines an \emph{affine height} on \(\A^n(K)\),
\begin{equation*}
\Haff(x_1, \dots, x_n) = \Hproj(1 : x_1 : \dots : x_n)
\end{equation*}
and in particular we get a height \(H\) on \(K = \A^1(K)\). Note that the affine height is always larger than the projective height, i.e.
\begin{equation*}
\Hproj(x_1, \dots, x_n) \leq \Haff(x_1, \dots, x_n).
\end{equation*}

It can be shown that this definition does not depend on the number field \(K\), and thus extends uniquely to \(\closure \Q\).

Often it is more convenient to work with the \emph{logarithmic heights}, \(\hproj = \log \circ \Hproj\), \(\haff = \log \circ \Haff\) and \(h = \log \circ H\).
\subsection{Height of polynomials}
\label{sec:org3ad38dd}
A non-zero polynomial in \(K[X]\) of degree \(\leq n\) can be considered both as a point in \(\P^n(K)\), or in \(\A^{n+1}(K)\). So we define
\begin{align*}
\Hproj(a_n \, X^n + \dots + a_0) &= \Hproj(a_n : \dots : a_0) \\
\Haff(a_n \, X^n + \dots + a_0) &= \Haff(a_n, \dots, a_0)
\end{align*}
and likewise the logarithmic heights \(\hproj\) and \(\haff\). Note that this means the projective height of a polynomial depends only on its zeroes, while the affine height coincides with the height on \(K\) for constant polynomials.

Similarly as in Section \ref{sec:orge83be0b}, we define the \emph{Gauss norm} for \(\nu\) on \(K[X]\) by
\begin{equation*}
\abs{a_n \, X^n + \dots + a_0}_\nu = \max\left( \abs{a_n}_\nu, \dots, \abs{a_0}_\nu\right).
\end{equation*}
If \(\nu\) is non-archimedean, this is even an non-archimedean absolute value (see Proposition \ref{bounded-laurent-valuation}). Nevertheless, the notation is useful also for archimedean absolute values.

In the following, \(K\) is always a number field which has at most \([K:\Q]\) archimedean places.

\begin{prop}
\label{proj-height-poly-mult}
Let \(f_1, \dots, f_r \in K[X]\) and \(f = f_1 \cdots f_r\). Then
\begin{equation*}
-\deg f \, \log 2 + \sum_{i=1}^r \hproj(f_i) \leq \hproj(f) \leq \deg f \, \log 2 + \sum_{i=1}^r \hproj(f_i)
\end{equation*}
\end{prop}
For a proof see \cite{bombieri-gubler-2006-heights-diophantine-geometry}, Theorem 1.6.13.

\begin{prop}
\label{aff-height-poly-add}
Let \(f_1, \dots, f_r \in K[X]\) and \(f = f_1 + \cdots + f_r\). Then
\begin{equation*}
\haff(f_1 + \dots + f_r) \leq \haff(f_1) + \dots + \haff(f_r) + \log r.
\end{equation*}
\end{prop}
This follows from Proposition 1.5.15 in \cite{bombieri-gubler-2006-heights-diophantine-geometry}.

\begin{prop}
\label{proj-height-poly-division}
Let \(a, b, q, r \in \mino{K[X]}\) with \(a = q \, b + r\) and \(\deg r < \deg b\). Set \(N = \deg q = \deg a - \deg b\). Then
\begin{align}
\label{height-poly-div-q}
\hproj(q) &\leq \hproj(a) + N \left(\log 2 +  \hproj(b)\right), \\
\label{height-poly-div-r}
\hproj(r) &\leq \hproj(a) + (N+1) \left( \log 2 + \hproj(b) \right).
\end{align}
\end{prop}
\begin{rem}
The bound for \(\hproj(q)\) holds also if \(r = 0\).
\end{rem}
\begin{proof}
We can assume \(\LC(b) = 1\), as the projective height for polynomials is invariant under multiplication with a constant factor. This conveniently implies \(\abs{b}_\nu = \max(1, \abs{b}_\nu)\) for every \(\nu \in M_K\).

Using the standard algorithm for division, we define a sequence of polynomials, beginning with \(a_0 = a\) and continuing for \(i \geq 0\) via
\begin{equation*}
a_{i} = \LC(a_i) \, X^{N_i} \, b + a_{i+1} \text{ where } \deg a_{i+1} < \deg a_i \text{ and } N_i = \deg a_i - \deg b.
\end{equation*}
Using \eqref{improved-triangle-equality} on the coefficients, we estimate
\begin{multline*}
\abs{a_{i+1}}_\nu = \abs{ \LC(a_i) b - a_i }_\nu \\ \leq \max(1,\abs{2}_\nu) \, \max\left(\abs{\LC(a_i)}_\nu \abs{b}_\nu, \abs{a_i}_\nu\right) \leq \max(1,\abs{2}_\nu) \, \abs{a_i}_\nu \, \abs{b}_\nu
\end{multline*}
and obtain
\begin{equation*}
\abs{a_i}_\nu \leq \abs{a}_\nu \, \left(\max(1, \abs{2}_\nu) \abs{b}_\nu\right)^i.
\end{equation*}

There are at most \(N+1\) steps necessary in the algorithm to reach \(\deg a_i < \deg b\) (because in every step, the degree decreases by at least \(1\), hence \(\deg a_i < \deg a - i\), so \(i \leq N+1\)), at which point \(a_i = r\). Consequently
\begin{equation*}
\abs{r}_\nu \leq \abs{a}_\nu \, \left(\max(1, \abs{2}_\nu) \, \abs{b}_\nu\right)^i \leq \abs{a}_\nu \, \left(\max(1, \abs{2}_\nu) \, \abs{b}_\nu\right)^{N+1}
\end{equation*}
which implies \eqref{height-poly-div-r}.
The coefficients of \(q\) are precisely \(\LC(a_0), \dots, \LC(a_{i-1})\), so
\begin{equation*}
\abs{q}_\nu \leq \max\left( \abs{a_0}_\nu, \dots, \abs{a_{i-1}}_\nu\right) \leq \abs{a}_\nu \, \left(\max(1, \abs{2}_\nu) \, \abs{b}_\nu\right)^{i-1} \leq \abs{a}_\nu \, \left(\max(1, \abs{2}_\nu) \, \abs{b}_\nu\right)^{N}
\end{equation*}
whence \eqref{height-poly-div-q}.
\end{proof}

\begin{prop}
\label{proj-height-poly-zeroes}
Let \(f \in K[X]\) a polynomial of degree \(r\) with roots \(\alpha_1, \dots, \alpha_r \in \closure \Q\) (accounted for multiplicities). Then
\begin{equation*}
-r \, \log 2 + \hproj(f) \leq \haff(\alpha_1) + \dots + \haff(\alpha_r) \leq r \, \log 2 + \hproj(f).
\end{equation*}
\end{prop}
\begin{proof}
This follows directly from Proposition \ref{proj-height-poly-mult}, noting that \(f\) factors as
\begin{equation*}
f = \mu \, (X - \alpha_1) \cdots (X - \alpha_r)
\end{equation*}
with \(\mu \in K\) having \(\hproj(\mu) = 0\), while \(\hproj(X - \alpha_i) = \hproj(1: \alpha_i) = \haff(\alpha_i)\).
\end{proof}

\subsection{Weil's Height Machine and Néron-Tate height}
\label{sec:org721d5ba}
On varieties defined over a number field \(K\), there is a plethora of different height functions. However, many of them differ only by a bounded function, so they produce essentially the same height. We capture this notion of \emph{quasi-equivalence of heights} in the following notation:

\begin{defi}
Let \(V\) a variety defined over a number field \(K\), and let \(f_1, f_2 : V(\closure{\Q}) \to \R\) two functions. We write \(f_1 \qeq f_2\) if there exists \(C \in \R\) such that
\begin{equation*}
\abs{f_1(P) - f_2(P)} \leq C \text{ for all } P \in V(\closure{\Q}).
\end{equation*}
\end{defi}

\pagebreak
We reproduce the following from Theorem B.3.2 in \cite{hindry-silverman-2000-diophantine-geometry}
\begin{thm}[Weil's Height Machine]
\label{weil-height-machine}
Let \(K\) a number field. For every smooth projective variety \(V/ K\) there exists a map
\begin{equation*}
h_V : \DIV(V) \longto \{ \text{functions } V(\closure{\Q}) \to \R \}
\end{equation*}
satisfying the following:
\begin{enumerate}
\item (Normalisation) For \(\di{H} \subset \P^n\) a hyperplane holds \(h_{\P^n, \di{H}} \qeq \hproj\).
\item (Functoriality) For \(\phi : V \to W\) a morphism and \(\di{D} \in \DIV(W)\) a divisor holds \(h_{V, \phi^*(\di{D})} \qeq h_{W,\di{D}} \circ \phi\).
\item (Additivity) For \(\di{D}, \di{E} \in \DIV(V)\) holds \(h_{V,\di{D}+\di{E}} \qeq h_{V,\di{D}} + h_{V,\di{E}}\).
\item (Linear Equivalence) For \(\di{D}, \di{E} \in \DIV(V)\) with \(\di{D} \sim \di{E}\) holds \(h_{V,\di{D}} \qeq h_{V,\di{E}}\).
\end{enumerate}
\end{thm}
There are further properties which we omit because we will not use them directly.

\begin{thm}[Néron-Tate height]
\label{neron-tate-height}
Let \(K\) a number field, and \(\mathcal A/K\) an abelian variety. Let \(\di{D} \in \DIV(\mathcal A)\) have symmetric divisor class (i.e. \([-1]^*\di{D} \sim \di{D}\)). Then there exists the (unique) \emph{canonical height on \(\mathcal A\) relative to \(\di{D}\)}, a height function \(\Th_{\mathcal A,\di{D}} : \mathcal A(\closure{\Q}) \longto \R\) satisfying the following:
\begin{enumerate}
\item It is equivalent to the height from the height machine: \(\Th_{\mathcal A,\di{D}} \qeq h_{\mathcal A,\di{D}}\).
\item For all integers \(m\) and \(P \in \mathcal A(\closure{\Q})\) we have \(\Th_{\mathcal A,\di{D}}(m \, P) = m^2 \, \Th_{\mathcal A,\di{D}}(P)\).
\item It is a quadratic form.
\end{enumerate}
\end{thm}

\begin{prop}
\label{height-zero-iff-torsion}
Take \(\mathcal A\) and \(\di{D}\) as in Theorem \ref{neron-tate-height}, and \(\di{D}\) moreover ample, then for all \(P \in \mathcal A(\closure{\Q})\) we have \(\Th_{\mathcal A,\di{D}}(P) \geq 0\), and \(\Th_{\mathcal A,\di{D}}(P) = 0\) \IFF \(P\) is torsion on \(\mathcal A\).
\end{prop}

The preceding Theorem and Proposition are adapted from Theorem B.5.1 and Proposition B.5.3 in \cite{hindry-silverman-2000-diophantine-geometry}, respectively.

\subsection{Heights on the Jacobian}
\label{sec:orgbf62da0}
With these tools, we are finally able to setup our heights on our (hyper)elliptic curve \(\CC\) and its Jacobian \(\Jac\). On the Jacobian, we use the height corresponding to the Theta divisor. The Theta divisor induced by the map \(j: \CC \to \Jac\) defined in Section \ref{sec:orgdaae3ac} is unfortunately not symmetric for \(g > 1\). So we use a different embedding, which differs only by a translation on the Jacobian.

Indeed let \(P_0 \in \CC\) one of the Weierstrass points, i.e. \(P_0 = (\xi, 0)\) where \(\xi\) is one of the roots of \(D\). This implies \(2 \pd{P_0} \sim \pd{O_+} + \pd{O_-}\) (via the function \(X - \xi\)), hence the canonical divisor is a multiple of \(P_0\): \(\canondiv \sim 2(g-1) \, \pd{P_0}\).

Now embed the curve into the Jacobian via
\begin{equation*}
j_0 : \CC \longto \Jac, \quad P \mapsto \j{P} - \j{P_0}
\end{equation*}
and note that \(j_0(\canondiv) = 0\) (recall that \(j_0\) extends naturally to divisors).

Now Theorem A.8.2.1. in \cite{hindry-silverman-2000-diophantine-geometry} implies that
\begin{equation*}
\Theta_0 = j_0(\CC) + \dots + j_0(\CC) \quad (g-1 \text{ copies})
\end{equation*}
is symmetric, i.e. \([-1]^* \Theta_0 = \Theta_0\). Then Theorem \ref{neron-tate-height} implies that the Néron-Tate height \(\Th = \Th_{\Jac, \Theta_0}\) associated to the height \(h_{\Jac, \Theta_0}\) is a quadratic form.

Theorem A.8.2.1. also says that \(j_0^*\Theta_0 \sim g \, \pd{P_0}\), so \(\Th \circ j_0 \qeq g\, h_{\CC,P_0}\). 

Recall that the hyperelliptic curve comes with a degree 2 map \(\pi : \CC \to \P^1\), with the hyperplane \(H = \{(\xi : 1) \}\) in \(\P^1\) having \(\pi^*(H) = 2 \, \pd{P_0}\), hence \(\hproj(\pi(P)) \qeq 2 \, h_{\CC,P_0}\). So we get \(2 \, \Th \circ j_0 \approx g \, \hproj(\pi(P))\), or more precisely for \(P = (x, y) \in \CCa\):
\begin{equation*}
\Th(j_0(P)) \qeq \frac{g}{2} \, \haff(x).
\end{equation*}

\section{Height of convergents}
\label{sec:orgcf28bd0}
We are now ready to study the height of the convergents. We begin by analysing the coefficients of the Laurent series \(\sqrt{D}\), and comparing the heights of the numerator and denominator of a convergent.
\subsection{Height bounds for series coefficients of \(\sqrt{D}\)}
\label{sec:org584c9f9}

We need some bounds for the absolute values (and height) of the coefficients of the power series \(\sqrt{D}\). 

\begin{prop}
\label{sqrt-d-coefficient-height}
Recall that \(\deg D = 2\,d\) and write
\begin{equation}
\label{sqrt-d-coefficients-for-height}
\sqrt{D} = X^d \sum_{n=0}^\infty w_n \, X^{-n}.
\end{equation}
For any place \(\nu\) on a number field, we have
\begin{equation*}
\abs{w_n}_\nu \leq \abs{\sqrt{\LC(D)}}_\nu \cdot \left(  \max(1, \abs{1/4}_\nu) \cdot \max(1, \abs{(2d)^2}_\nu) \cdot  \abs{D}_\nu / \abs{\LC(D)}_\nu \right)^n,
\end{equation*}
which implies
\begin{equation*}
h(w_n) \leq \frac{1}{2} h(\LC(D)) + n \, \left( \log 4 + 2 \, \log(2d) + \hproj(D) \right).
\end{equation*}
\end{prop}

Before we can prove this, we need an estimate for the growth of the binomial coefficient:

\begin{lemma}
\label{lemma-binomial-half}
For \(n \geq 1\), there exists an integer \(b_n \in \Z\) with \(\abs{b_n}_\R \leq 2^{2n-3}\)  such that the binomial coefficient \(\binom{1/2}{n} = \ifrac{b_n}{2^{2n-1}}\). 

For \(\nu\) a place on a number field, not over \(2\), this implies \(\abs{\binom{1/2}{n}}_\nu \leq 1\) for all \(n \geq 0\).

But if \(\nu\) represents a place over \(2\) (with the normalisations introduced at the beginning of the chapter), we find \(\abs{\binom{1/2}{n}}_\nu \leq 2^{2n}\) for all \(n \geq 0\).
\end{lemma}
This is an easy exercise. Note that the \(b_n\) are closely related to the Catalan numbers (see for example  \cite{aigner-2007-course-enumeration}, pages 101, 102). See also Theorem 5 in \cite{siegel-2014-some-appl-diophantine} for a generalisation of this lemma.

\begin{proof}[Proof of Proposition \ref{sqrt-d-coefficient-height}]
We write
\begin{equation*}
D = d_{2d} \, X^{2d} + d_{2d-1} \, X^{2d-1} + \dots + d_0 = d_{2d} \, X^{2d} (1 + f(X)) \text{ with } f(X) \in K[\inv X].
\end{equation*}
We may then compute 
\begin{equation*}
\sqrt{D} = \sqrt{d_{2d}} \, X^d \; \sum_{n=0}^\infty \binom{1/2}{n} \, f(X)^n
\end{equation*}
which converges in \(\laurinx K\) because \(\io{f(X)} > 0\). Now let \(\nu\) any place on \(K\), and write \(f(X) = f_1 X^{-1} + \dots + f_{2d} X^{-2d}\), to define
\begin{equation*}
C_\nu = \max(1, \abs{f_1}_\nu, \dots, \abs{f_{2d}}_\nu) = \abs{1 + f}_\nu = \abs{D}_\nu / \abs{\LC(D)}_\nu.
\end{equation*}

Studying for \(i \geq 0\) the power
\begin{equation*}
\binom{1/2}{i} \, f(X)^i = \sum_{j_1 + \dots + j_{2d}=i} \binom{1/2}{i} \, \binom{i}{j_1, \dots, j_{2d}} \left( \prod_{l=1}^{2d} {f_l}^{j_l} \right) \; X^{-( j_1 + 2 \, j_2 + \dots + (2d) \, j_{2d})},
\end{equation*}
we note that \(\binom{i}{j_1, \dots, j_{2d}} \leq {2d}^i\) is an integer, so the coefficient of every summand is bounded in \(\abs{\cdot}_\nu\) by \((\max(1, \abs{1/4}_\nu \, \max(1, \abs{m}_\nu) \, C_\nu)^i\).

Now observe that every \(w_i/w_0\) (clearly \(w_0 = \sqrt{\LC(D)}\)) is a sum of at most \((2d)^i\) of these, so with the improved triangle inequality \eqref{improved-triangle-equality} we obtain the desired result
\begin{equation*}
\abs{w_n}_\nu \leq \abs{w_0}_\nu \, \max(1, \abs{(2d)}_\nu)^n \, (\max(1, \abs{1/4}_\nu \, \max(1, \abs{m}_\nu) \, C_\nu)^n
\end{equation*}

With \(\hproj(1 + f) = \hproj(D)\) and \(C_\nu \geq 1\) the inequality for the height follows after we replace \(\abs{w_i}_\nu\) with \(\max(1, \abs{w_i}_\nu)\) (also for \(i = 0\)).
\end{proof}

We can now bound the projective height of the numerator of a convergent in terms of the height of the denominator.\footnote{It is not quite clear if there is a similar bound in the other direction.}
\begin{prop}
\label{convergent-numerator-denominator-height}
For every convergent \((p, q) \in \Coset{\sqrt{D}}{K}\) holds
\begin{equation*}
\hproj(p) \leq \hproj(q) + (\deg p) \left(\log 2 + \log 4 + \log(2d) + \hproj(D) \right).
\end{equation*}
\end{prop}
\begin{proof}
Recall from Proposition \ref{convergent-q-determines-p} that \(p = \gauss{\sqrt{D} \, q}\). However we did not define a height for \(\sqrt{D}\). To workaround this problem, let \(m = \deg q\) and write \(D = A_m + \varepsilon_m\) with \(A_m\) a Laurent polynomial and \(\io{\varepsilon} > m\). This ensures \(p = \gauss{A_m \, q}\), so \(\hproj(p) \leq \hproj(A_m \, q)\). With Proposition \ref{sqrt-d-coefficient-height}, we bound the projective height
\begin{equation*}
\hproj(A_m) \leq (d + m) \, \left( \log 4 + 2 \, \log(2d) + \hproj(D)\right).
\end{equation*}
Note that \(A_m\) has precisely \(d + m = d + \deg q = \deg p\) coefficients. The overall bound then follows from the bound for the product (Proposition \ref{proj-height-poly-mult}).
\end{proof}

\subsection{Lower bound}
\label{sec:org1835e85}
In \cite{bombieri-cohen-1997-siegels-lemma-pade}, a general result about the height of Padé approximations predicts that the projective height of the convergents of a square root should grow quadratically in the degree of the convergent.

If we just want to prove this for square root, a shorter and simpler proof by Zannier suffices. It is explained briefly in \cite{zannier-2016-hyper-contin-fract}, here we give a bit more detailed version.

\begin{thm}
\label{convergent-height-lower-bound}
If \(D\) is not Pellian, there exists a constant \(C = C(D) > 0\) such that for every convergent \((p, q) \in \Coset{\sqrt{D}}{\closure Q}\) we have for \(\deg q\) large enough:
\begin{equation*}
C \cdot (\deg q)^2 \leq \hproj(q).
\end{equation*}
\end{thm}

\begin{proof}
Recall that by Lemma \ref{convergent-divisor-lemma} and subsequent remarks, the convergents produce an equality on the Jacobian (recall \(\j{\OO} = -j(O_-) = \j{O_+} - \j{O_-}\))
\begin{equation*}
-m \, \j{\OO} = j(P_1) + \dots + j(P_r)
\end{equation*}
with \(r \leq g\) and \(P_i = (x_i, y_i) \in \CCa\). The \(x_i\) are precisely the zeroes of \(\Omega = p^2 - D \, q^2\) (accounted for multiplicities). And we have \(m = \deg p = \deg q + g + 1 \geq \deg q\).

Applying the Néron-Tate height, and using Lemma \ref{quadratic-form-sum-bound} from the Appendix, we obtain
\begin{equation*}
m^2 \Th(\j{\OO}) = \Th(-m \, \j{\OO}) = \Th(j(P_1) + \dots + j(P_r)) \leq g \left( \Th(j(P_1)) + \dots + \Th(j(P_r)) \right).
\end{equation*}
Next, there is a constant \(C_1 \geq 0\) depending only on \(D\) such that \(\Th(j(P_i)) \leq C_1 + \haff(x_i)\) (see Section \ref{sec:orgbf62da0}). Moreover, we know that
\begin{equation*}
h(x_1) + \dots + h(x_r) \leq g \, \log 2 + \hproj(\Omega).
\end{equation*}
We combine these estimates to
\begin{equation*}
m^2 \Th(\j{\OO}) \leq g \, C_1 + g^2 \, \log 2 + g \, \hproj(\Omega).
\end{equation*}
Because \(D\) is not Pellian, the point \(\j{\OO}\) is not torsion in the Jacobian, so \(\Th(\j{\OO}) > 0\).
We get a constant \(C_2 > 0\) such that
\begin{equation*}
C_2 \, (\deg q)^2 \leq \hproj(\Omega).
\end{equation*}

As we are only interested in the projective height of \(q\), we may without restriction normalise the convergent \((p, q)\) such that \(p\) is monic, so that both \(p^2\) and \(D \, q^2\) have to be monic. Of course, for a monic polynomial, affine and projective height coincide, and using Proposition \ref{aff-height-poly-add} we can estimate
\begin{equation*}
\hproj(\Omega) \leq \haff(\Omega) \leq \log 2 + \haff(p^2) + \haff(D\,q^2) = \log 2 + \hproj(p^2) + \hproj(D\,q^2).
\end{equation*}

By Proposition \ref{convergent-numerator-denominator-height}, we have \(\hproj(p) \leq \hproj(q) + \deg p \, C_3\) for some \(C_3 \geq 0\) depending only on \(D\). With Proposition \ref{proj-height-poly-mult} we get
\begin{align*}
\hproj(D \, q^2) &\leq 2 \, \deg p \, \log 2 + \hproj(D) + 2 \, \hproj(q), \\
\hproj(p^2) &\leq 2 \, \deg p \, \log 2 + 2 \, \hproj(p) \leq 2 \, \deg p \, \log 2 + 2 \deg p \, C_3 + 2 \, \hproj(q).
\end{align*}
Combining these, and noting \(\deg p = \deg q + d\), we find
\begin{equation*}
C_2 \, (\deg q)^2 \leq C_4 + C_5 \, \deg q + 4 \, \hproj(q)
\end{equation*}
with \(C_4, C_5 \geq 0\), so for example \(C = C_2 / 8\) yields the desired constant.
\end{proof}

This does not yet give a lower bound for the height of partial quotients. This seems more challenging, especially for the projective height -- see also Example \ref{ex-dbh1-bounded-height} in Section \ref{sec:orge27ae2f}. But there are new results for the affine height if we take some type of average, see Theorem 1.4 in \cite{zannier-2016-hyper-contin-fract}:
\begin{thm}[Zannier]
There exist \(M \in \N\) and \(C > 0\) such that for \(n\) large enough
\begin{equation*}
C \, n^2 \leq \max( \haff(a_{n-i}) \mid i=0,\dots,M ).
\end{equation*}
\end{thm}
\subsection{Upper bound}
\label{sec:orgdaa4c63}
An upper bound for the height of the convergents can be deduced with more elementary tools, using the Toeplitz determinants from Section \ref{sec:org542977d}. It is then straightforward to deduce an upper bound also for the height of the partial quotients.

\begin{thm}
\label{convergents-upper-proj-height-bound}
For the canonical convergents \((p_m, q_m)\) of \(\sqrt{D}\) we obtain the height bounds
\begin{align}
\label{eq-height-pm-upper-bound}
\hproj(p_m) &\leq ((n+1) \,d + \tfrac{3}{2} (n^2 + n)) \, (\log 4 + 2 \, \log(2d) + \hproj(D)), \\
\label{eq-height-qm-upper-bound}
\hproj(q_m) &\leq (n \,d + \tfrac{1}{2} (3 \, n^2 + n)) \, (\log 4 + 2 \, \log(2d) + \hproj(D)),
\end{align}
where \(\deg D = 2d\) and \(n = \deg q_m\).
\end{thm}

\begin{proof}
We wish to apply the results of Section \ref{sec:org542977d} for \(\alpha = \sqrt{D}\).  Connecting the notations of \eqref{sqrt-d-coefficients-for-height} and \eqref{convergents-condition-matrix}, we have \(N = d\) and \(A_{j} = w_{d-j}\). As we chose \(n = \deg q_m\) and the canonical convergents are coprime, Proposition \ref{convergent-linear-matrix-full-rank} tells us that the matrix \(\pqmatrix_n\) has full rank. So the kernel has dimension \(1\), and we can compute a solution \((p, q)\) using \eqref{toeplitz-matrix-convergents} which differs from \((p_m, q_m)\) only by a constant factor, and thus has the same projective height. The coefficients of \(p\) and \(q\) are (up to signs) the minors of
\begin{equation*}
\pqmatrix_n = \left( \begin{array}{ccccc|ccc}
-1 &       &        &        &    & w_0 \\
   &\ddots &        &        &    & w_{1}  & \ddots \\
   &       & \ddots &        &    & \vdots   & \ddots & w_{0} \\
   &       &        & \ddots &    & \vdots   & \ddots & \vdots \\
   &       &        &        & -1 & w_{d+n}   & \dots  & w_{d} \\ \hline
   &       &        &        &    & w_{d+n+1} & \dots  & w_{d+1} \\
   &       & 0      &        &    & \vdots   & \ddots & \vdots \\
   &       &        &        &    & w_{d+2n}  & \dots  & w_{d+n}
\end{array} \right).
\end{equation*}
If we strike any column (to get the minor), we obtain a \((d+2n+1)\times(d+2n+1)\) matrix. Now set
\begin{equation*}
C_\nu = \max(1, \abs{1/4}_\nu) \cdot \max(1, \abs{(2d)^2}_\nu) \cdot \abs{D}_\nu / \abs{\LC(D)}_\nu
\end{equation*}
so that \(\abs{w_j}_\nu \leq \abs{w_0}_\nu \, {C_\nu}^j\).

If we strike a column in the right block (to compute the coefficients of \(q\)), by using Laplace development we get (up to sign) a minor \(\mathcal M'\) of the lower right block of dimensions \(n\times n\), with determinant
\begin{equation*}
\det \mathcal M' = \sum_{\sigma \in S_n} \sign(\sigma) \, \mathcal M'_{1 \;\sigma(1)} \dots \mathcal M'_{n \; \sigma(n)}
\end{equation*}
with \(\abs{\mathcal M'_{ij}}_\nu \leq \abs{w_0}_\nu \, {C_\nu}^{d+n+i}\), hence
\begin{equation*}
\abs{q}_\nu \leq \prod_{i=1}^n \abs{w_0}_\nu \, {C_\nu}^{d+n+i} \leq {\abs{w_0}_\nu}^n {C_\nu}^{n\,d + n^2 + n(n+1)/2}.
\end{equation*}
When taking the product over all \(\nu\), the first term with \(w_0\) vanishes by the product formula, and likewise the term with \(\LC(D)\). We arrive at \eqref{eq-height-qm-upper-bound} by a straightforward calculation.

If we strike a column in the left block (to compute the coefficients of \(p\)), we can use similar arguments,  with Laplace development we only get a \((n+1) \times (n+1)\) matrix \(\mathcal M''\), with \(\abs{\mathcal M''_{ij}}_\nu \leq \abs{w_0}_\nu \, {C_\nu}^{d+n+i-1}\), so
\begin{equation*}
\abs{p}_\nu \leq \prod_{i=1}^{n+1} \abs{w_0}_\nu \, {C_\nu}^{d+n+i-1} \leq {\abs{w_0}_\nu}^{n+1} {C_\nu}^{(n+1) \, d + (n+1) n  + n(n+1)/2}.
\end{equation*}
Again, \eqref{eq-height-pm-upper-bound} follows by a straightforward calculation.
\end{proof}

\begin{cor}
The projective height of the convergents grows at most quadratically:
\begin{equation*}
\hproj(p_m) = O(m^2), \qquad \hproj(q_m) = O(m^2).
\end{equation*}
\end{cor}
\begin{proof}
The partial quotients have bounded degree \(1\leq \deg a_i \leq d\), hence \(m \leq \deg q_m = n \leq d \, m\), and the above theorem gives \(\hproj(p_m) = O(n^2)\) and \(\hproj(q_m) = O(n^2)\).
\end{proof}

\begin{cor}
\label{partial-quotients-upper-proj-height-bound}
The projective height of the partial quotients also grows at most quadratically:
\begin{equation*}
\hproj(a_m) = O(m^2)
\end{equation*}
\end{cor}
\begin{proof}
We can compute the partial quotients from subsequent convergents as in
\begin{equation*}
a_m = \gauss{\frac{p_m}{p_{m-1}}} \qquad a_m = \gauss{\frac{q_m}{q_{m-1}}}
\end{equation*}
and then Proposition \ref{proj-height-poly-division} yields
\begin{equation*}
\hproj(a_m) \leq \hproj(q_m) + (\deg a_m) ( \log 2 + \hproj(q_{m-1}) ) = O(m^2).
\end{equation*}
\end{proof}
\begin{rem}
An explicit bound is
\begin{multline*}
\hproj(a_m) \leq \left( ((a+1) n + a) \, d + a \, \log 2 + \tfrac{1}{2} \left( 3 (a+1) n^2 + (7 a+ 1) n + 3 a^2 +a \right) \right)
\\ \cdot \left(\log 4 + 2 \, \log(2d) + \hproj(D)\right)
\end{multline*}
where \(a = \deg a_m\), \(n = \deg q_{m-1}\).
\end{rem}
\section{Connecting heights and valuations}
\label{sec:orgba963e8}
From the definitions of the height of polynomials in Section \ref{sec:org3ad38dd} and the Gauss norms in Chapter \ref{sec:orgd5f1900} it is clear that there is a direct connection between the height and the valuations computations, as for example in Theorem \ref{thm-genus1-zero-patterns}. Note that for a polynomial \(f \in K[X]\), we have
\begin{align}
\hproj(f) & = \sum_{\nu \in M_K} \, \log \abs{f}_\nu, \\
\haff(f) & = \sum_{\nu \in M_K} \, \max\left(0, \log \abs{f}_\nu\right).
\end{align}
For \(\nu\) non-archimedean, \(\log \abs{f}_\nu\) is essentially \(c \cdot \nu(f)\) for some \(c < 0\).

However, in Chapters \ref{sec:orgd5f1900} and \ref{sec:org5f9d2ce} we do not treat the places over \(2\), and certainly not the archimedean places. So this connection must remain incomplete. Still, we try to point out some phenomena relating the global picture of heights and the local picture of Gauss norms.

We restrict to the genus \(1\) case with \(\deg D = 4\) and \(D\) non-Pellian, so that we may use Theorem \ref{thm-genus1-zero-patterns}. 

Corollary \ref{cor-genus1-unbounded-gauss-norm} says that for places with \(\j{\OOred}\) having even torsion order, the Gauss norms of \(q_n\) grow at least linearly in \(n\). But \(\hproj(q_n)\) should grow quadratically, which suggests that either the Gauss norms grow faster than linearly, or the number of places with bad reduction of the continued fraction before \(n\) also grows linearly.

Computational evidence suggests that the latter is the case. Also, if we work over \(\Q\), the Hasse-Weil interval (see Remark \ref{finite-field-torsion-bound}) predicts that \(\Jacred(\F_\pp)\) grows about linearly in \(\pp\), so the quasi-period length of \(\CF(D_\pp)\) and hence the first occurrence of \(\pp\) in a denominator of \(a_n\) grows linearly in \(\pp\). However, the number of primes \(\pp \leq n\) grows only as \(n/\log n\).

\smallskip

The valuations \(\nu(q_n)\) mostly alternate between positive and negative signs, so for the projective height they might cancel each other out. But for the affine height, there is no such cancellation, and in fact computations for examples suggest that \(\haff(q_n)\) grows more or less cubically. This is in line with Remark 4.8 (ii) in \cite{zannier-2016-hyper-contin-fract}, which says that \(\haff(q_n)\) should at most grow cubically in \(n\).

\smallskip

Similar observations can be made for the partial quotients.

\chapter{Examples}
\label{sec:orgda48525}
We now apply the theory developed in this thesis to some examples. Hopefully, this illustrates our theorems and their limitations. To this end, we include examples also for the corner cases that have been somewhat neglected in the theoretical part. 

For \(D\) defined over the rationals (or perhaps a number field), and \(\pp\) some prime (in the ring of integers), we use the notations \(D_\pp\) for \(\Red{D}\) in \(\F_\pp[X]\), \(\nu_\pp\) for the \(\pp\)-adic valuation, we denote by \(\OO_\pp\) the torsion divisor \(\pd{O_+} - \pd{O_-}\) on \(\CCred\) over \(\F_\pp\).

For \(D\) defined over \(\C(t)\), we use analogous notation, with \(t - t_0\) or \(t = t_0\) instead of \(\pp\).

\section{Reduction to a square}
\label{sec:orge6e4f78}
We begin with some examples where \(D\) reduces (or specializes) to a square.

\begin{table}[h!]
\centering
\caption*{Example \cfexample{ex-periodic-to-square-1}}
\begin{tabular}{|l|l|}
\hline
\(D = (X - 1) \cdot X \cdot (X - t) \cdot (X + t - 1)\) & basefield \(\C(t)\)\\ \hline
Discriminant of $D$: \((4) \cdot (t - \frac{1}{2})^{2} \cdot (t - 1)^{4} \cdot t^{4}\) & Primes with $\Red{D}$ square: \(t - 1, t\)\\ \hline
period length \(2\) for \(\CF(\sqrt{D})\) & quasi-period length \(1\) for \(\CF(\sqrt{D})\)\\ \hline
\multicolumn{2}{|l|}{Minimal Pell solution}\\ \hline
\(p_{0} = X^{2} - X - \frac{1}{2} t^{2} + \frac{1}{2} t\) & \(q_{0} = 1\)\\ \hline
\multicolumn{2}{|l|}{Partial quotients of \(\sqrt{D}\)}\\ \hline
\multicolumn{2}{|l|}{\(a_{0} = X^{2} - X - \frac{1}{2} t^{2} + \frac{1}{2} t\)}\\ \hline
\multicolumn{2}{|l|}{\(a_{1} = \frac{-8 X^{2} + 8 X + 4 t^{2} - 4 t}{(t - 1)^{2} \cdot t^{2}}\)}\\ \hline
\multicolumn{2}{|l|}{\(a_{2} = 2 X^{2} - 2 X - t^{2} + t\)}\\ \hline
\end{tabular}
\end{table}

Example \ref{ex-periodic-to-square-1} is very simple, but it illustrates already that bad reduction of the continued fraction is not the same as bad reduction of the elliptic curve. Only \(D_{t=0}\) and \(D_{t=1}\) are square, and \(t, t-1\) are the only irreducible/prime factors appearing in the coefficient denominators (of \(a_1\)).

By the way, this means we can specialise \(t\) to say an integer \(t_0 \in \Z \setminus\{0, 1\}\) to get a periodic continued fraction over \(\Q\). This continued fraction has bad reduction precisely at the prime numbers dividing \(t_0 \, (t_0 - 1)\). In this way we obtain an example also for the reduction modulo \(\pp\) case.

\smallskip

\begin{table}[h!]
\centering
\caption*{Example \cfexample{ex-nonperiodic-to-square-2}}
\begin{tabular}{|l|l|}
\hline
\(D = X^{4} + 2 X^{2} + t X + 1\) & basefield \(\C(t)\)\\ \hline
Discriminant of $D$: \((-27) \cdot t^{2} \cdot (t^{2} - \frac{256}{27})\) & Primes with $\Red{D}$ square: \(t\)\\ \hline
\(D\) is not Pellian & \\ \hline
\multicolumn{2}{|l|}{Partial quotients of \(\sqrt{D}\)}\\ \hline
\multicolumn{2}{|l|}{\(a_{0} = X^{2} + 1\)}\\ \hline
\multicolumn{2}{|l|}{\(a_{1} = \frac{2 X}{t}\)}\\ \hline
\multicolumn{2}{|l|}{\(a_{2} = \frac{1}{2} t X - \frac{1}{8} t^{2}\)}\\ \hline
\multicolumn{2}{|l|}{\(\deg a_n = 2, 1, 1, 1, \dots\)}\\ \hline
\end{tabular}
\end{table}

Example \ref{ex-nonperiodic-to-square-2} clearly reduces to a square at \(t = 0\). We check that \(\CF(\sqrt{D})\) is non-periodic by specializing to \(t = 3\). Then reduction of the continued fraction \(\CF(D_{t=3})\) modulo \(5\) and \(7\) yields torsion orders \(5\) respectively \(10\) which implies non-periodicity for both \(\CF(\sqrt{D_{t=3}})\) and \(\CF(\sqrt{D})\). As we are in the degree \(4\) case, this implies good reduction of \(\CF(\sqrt{D})\) at \(t-3\) (by Proposition \ref{bad-reduction-is-periodic-deg4}). Of course \(D_{t=3}\) reduces then to a square modulo \(3\), so the example works for the reduction modulo \(\pp\) case too.

\medskip

So both in the periodic and non-periodic case, it is possible that \(D\) reduces to a square.
\section{Reduction periodic to periodic}
\label{sec:orgc648887}
For \(\deg D = 4\), we have seen that torsion order \(m\) and quasi-period length \(\QPL\) satisfy \(m = \QPL +1\) (see Proposition \ref{prop-bounds-torsion-period-length}). Together with the discussion of Section \ref{sec:orgaebd15f} on how the quasi-period may shorten, and rational torsion on elliptic curves being bounded by \(12\), there cannot be many examples of bad reduction of a periodic continued fraction for \(D \in \Q[X]\). 

\begin{table}[h!]
\centering
\caption*{Example \cfexample{ex-periodic-good-red-1}}
\begin{tabular}{|l|l|}
\hline
\(D = X^{4} - 8 X^{3} - 42 X^{2} + 424 X - 119\) & basefield \(\Q\)\\ \hline
Discriminant of $D$: \(-1 \cdot 2^{29} \cdot 3^{5}\) & Primes with $\Red{D}$ square: \(3\)\\ \hline
\multicolumn{2}{|l|}{period length \(8\) for \(\CF(\sqrt{D})\)}\\ \hline
\multicolumn{2}{|l|}{Minimal Pell solution}\\ \hline
\(\deg p_{7} = 9\) & \(\deg q_{7} = 7\)\\ \hline
\multicolumn{2}{|l|}{Partial quotients of \(\sqrt{D}\)}\\ \hline
\(a_{0} = X^{2} - 4 X - 29\) & \(a_{4} = \frac{4}{3} X - \frac{44}{3}\)\\ \hline
\(a_{1} = \frac{1}{96} X + \frac{1}{96}\) & \(a_{5} = \frac{1}{32} X + \frac{5}{32}\)\\ \hline
\(a_{2} = -4 X + 12\) & \(a_{6} = -4 X + 12\)\\ \hline
\(a_{3} = \frac{1}{32} X + \frac{5}{32}\) & \(a_{7} = \frac{1}{96} X + \frac{1}{96}\)\\ \hline
\end{tabular}
\end{table}

Indeed this is the case in Example \ref{ex-periodic-good-red-1}, where we see only \(2\) and \(3\) in the denominators. As \(\OO\) has order \(9\), the only candidate for bad reduction of \(\CF(D)\) is \(3\), but \(D_3\) is already a square. Everywhere else we have good reduction of \(\CF(\sqrt{D})\).

But if we are working over number fields, and can increase the torsion orders, then in principle one should be able to construct example with bad reduction of the continued fraction (by applying Theorem \ref{thm-serre-tate-torsion-reduction}).

\medskip

We also analysed an example with \(\deg D = 6\) given in \cite{platonov-2014-number-theoretic-properties} (\(f_{33}\) in Section 6). In Example \ref{ex-periodic-good-red-2}, we have torsion order \(33\), so both \(3\) and \(11\) are good (but a priori not the only) candidates for bad reduction of \(\CF(\sqrt{D})\). But again \(D_3\) is already square. And we do not see \(11\) in the denominators (it suffices to check the first half of the palindromic quasi-period as mentioned in Section \ref{sec:orgaebd15f}). We have good reduction of \(\sqrt{D}\) at all other primes which is not surprising given that the example seems to have been constructed by testing for this.
Also observe that the factor for the quasi-period is \(\mu = 3\), as predicted by Remark \ref{bad-reduction-quasi-period-factor}.

\begin{table}[h!]
\centering
\caption*{Example \cfexample{ex-periodic-good-red-2}}
\begin{tabular}{|p{9cm}|p{5cm}|}
\hline
\(D = 4 X^{6} + 28 X^{5} + 37 X^{4} - 30 X^{3} + 87 X^{2} - 54 X + 9\) & basefield \(\Q\)\\ \hline
Discriminant of $D$: \(-1 \cdot 2^{22} \cdot 3^{14} \cdot 127\) & Primes with $\Red{D}$ square: \(3\)\\ \hline
period length \(54\) for \(\CF(\sqrt{D})\) & quasi-period length \(27\) for \(\CF(\sqrt{D})\)\\ \hline
\multicolumn{2}{|p{14cm}|}{Minimal Pell solution}\\ \hline
\(\deg p_{26} = 33\) & \(\deg q_{26} = 30\)\\ \hline
\multicolumn{2}{|p{14cm}|}{Partial quotients of \(\sqrt{D}\)}\\ \hline
\(a_{0} = 2 X^{3} + 7 X^{2} - 3 X + 3\) & \(a_{9} = 6 X\)\\ \hline
\(a_{1} = \frac{1}{9} X + \frac{1}{2}\) & \(a_{10} = \frac{1}{9} X + \frac{7}{18}\)\\ \hline
\(a_{2} = 3 X - \frac{9}{2}\) & \(a_{11} = -3 X - \frac{3}{2}\)\\ \hline
\(a_{3} = \frac{2}{27} X^{2} + \frac{1}{3} X + \frac{2}{9}\) & \(a_{12} = -\frac{1}{3} X - \frac{5}{6}\)\\ \hline
\(a_{4} = 6 X^{2} + 21 X - 9\) & \(a_{13} = \frac{2}{3} X + 1\)\\ \hline
\(a_{5} = \frac{1}{9} X - \frac{1}{6}\) & \(a_{14} = 2 X + 3\)\\ \hline
\(a_{6} = X + \frac{11}{2}\) & \(a_{15} = -\frac{1}{9} X - \frac{5}{18}\)\\ \hline
\(a_{7} = -2 X + 2\) & \(a_{16} = -9 X - \frac{9}{2}\)\\ \hline
\(a_{8} = -\frac{1}{9} X - \frac{1}{2}\) & \(a_{17} = \frac{1}{27} X + \frac{7}{54}\)\\ \hline
\end{tabular}
\end{table}

\clearpage

\section{Reduction non-periodic to periodic}
\label{sec:org3de9d88}
\subsection{Genus 1}
\label{sec:orgec32865}
For a polynomial of degree \(4\), our Theorem \ref{thm-genus1-zero-patterns} describes the behaviour of the valuations. We now give an example to illustrate this, both for odd and even quasi-period length of \(\CF(\sqrt{D_\pp})\).

\begin{table}[h!]
\centering
\caption*{Example \cfexample{ex-cfp1-zero-pattern-deg4}}
\begin{tabular}{|l|l|}
\hline
\(D = X^{4} + 5 X^{2} - 3 X + 19\) & basefield \(\Q\)\\ \hline
Discriminant of $D$: \(3^{2} \cdot 7^{2} \cdot 11^{2} \cdot 17\) & Primes with $\Red{D}$ square: \(3\)\\ \hline
\(D\) is not Pellian because of incompatible torsion orders & torsion order \(8\) modulo \(5\)\\ \hline
 & torsion order \(3\) modulo \(7\)\\ \hline
\multicolumn{2}{|l|}{Partial quotients of \(\sqrt{D}\)}\\ \hline
\multicolumn{2}{|l|}{\(a_{0} = X^{2} + \frac{5}{2}\)}\\ \hline
\multicolumn{2}{|l|}{\(a_{1} = -\frac{2}{3} X - \frac{17}{6}\)}\\ \hline
\multicolumn{2}{|l|}{\(a_{2} = -\frac{24}{329} X + \frac{33270}{108241}\)}\\ \hline
\multicolumn{2}{|l|}{\(\deg a_n = 2, 1, 1, 1, 1, \dots\)}\\ \hline
\end{tabular}
\end{table}

Example \ref{ex-cfp1-zero-pattern-deg4} is chosen randomly. Note again that while the discriminant has a finite number of prime divisors, only \(D_3\) is a square polynomial, and of course \(\CF(\sqrt{D})\) has bad reduction for every odd prime number \(\pp\) by Lemma \ref{bad-reduction-to-periodic} and Corollary \ref{cor-finite-field-always-periodic}.

\subsubsection{modulo 5}
\label{sec:org421d2d5}
\begin{table}[hp]
\caption{\label{cfr-mod5-valuations-table}
5-adic valuations for Example \ref{ex-cfp1-zero-pattern-deg4}}
\centering
\begin{tabular}{rrrrrrr}
\(n\) & \(\lambda(n)\) & \(\nu(\alpha_n)\) & \(\nu(a_n)\) & \(\nu(\LC(a_n))\) & \(\nu(q_n)\) & \(\nu(\LC(q_n))\)\\
\hline
0 & 0 & 0 & 0 & 0 & 0 & 0\\
1 & 1 & 0 & 0 & 0 & 0 & 0\\
2 & 2 & 0 & 0 & 0 & 0 & 0\\
3 & 3 & 0 & 0 & 0 & 0 & 0\\
4 & 4 & 0 & 0 & 0 & 0 & 0\\
5 & 5 & 0 & 0 & 0 & 0 & 0\\
6 & 6 & 0 & 0 & 0 & 0 & 0\\
7 & 6 & \(-\infty\) & -2 & -1 & -2 & -1\\
8 & 7 & 2 & 2 & 3 & 2 & 2\\
9 & 8 & -4 & -4 & -4 & -2 & -2\\
10 & 9 & 4 & 4 & 4 & 2 & 2\\
11 & 10 & -4 & -4 & -4 & -2 & -2\\
12 & 11 & 4 & 4 & 4 & 2 & 2\\
13 & 12 & -4 & -4 & -4 & -2 & -2\\
14 & 13 & 4 & 4 & 4 & 2 & 2\\
15 & 13 & \(-\infty\) & -6 & -5 & -4 & -3\\
16 & 14 & 6 & 6 & 7 & 4 & 4\\
17 & 15 & -8 & -8 & -8 & -4 & -4\\
18 & 16 & 8 & 8 & 8 & 4 & 4\\
19 & 17 & -8 & -8 & -8 & -4 & -4\\
20 & 18 & 8 & 8 & 8 & 4 & 4\\
21 & 19 & -8 & -8 & -8 & -4 & -4\\
22 & 20 & 8 & 8 & 8 & 4 & 4\\
23 & 20 & \(-\infty\) & -10 & -9 & -6 & -5\\
24 & 21 & 10 & 10 & 11 & 6 & 6\\
25 & 22 & -12 & -12 & -12 & -6 & -6\\
26 & 23 & 12 & 12 & 12 & 6 & 6\\
27 & 24 & -12 & -12 & -12 & -6 & -6\\
28 & 25 & 12 & 12 & 12 & 6 & 6\\
29 & 26 & -12 & -12 & -12 & -6 & -6\\
30 & 27 & 12 & 12 & 12 & 6 & 6\\
31 & 27 & \(-\infty\) & -14 & -13 & -8 & -7\\
32 & 28 & 14 & 14 & 15 & 8 & 8\\
33 & 29 & -16 & -16 & -16 & -8 & -8\\
34 & 30 & 16 & 16 & 16 & 8 & 8\\
\end{tabular}
\end{table}

Table \ref{cfr-mod5-valuations-table} lists the 5-adic valuations (Gauss norms). Note how the changes in the patterns, and the unbounded \(\alpha_n\) are aligned with the 2-element fibres of \(\lambda\). We can also read off the quasi-period length of \(\CF(\sqrt{D_{5}})\) from the first occurrence of non-zero valuations, and determine it to be \(7\) (this works only in degree \(4\)).

As the quasi-period length is odd, we can observe (as predicted by Corollary \ref{cor-genus1-unbounded-gauss-norm}) that the valuations increase in absolute value. We also see that the sign of the exponent is alternating, and that almost all the \(\nu(a_n)\) are divisible by \(4\) (as predicted by Theorem \ref{thm-genus1-zero-patterns}). Pay attention to the valuations of the leading coefficients being larger. For \(q_n\) this indicates that \(\Redn{q_n}\) has a lower degree.

\subsubsection{modulo 19}
\label{sec:org46eabc7}

\begin{table}[hp]
\caption{\label{cfr-mod19-valuations-table}
19-adic valuations for Example \ref{ex-cfp1-zero-pattern-deg4}}
\centering
\begin{tabular}{rrrrrrr}
\(n\) & \(\lambda(n)\) & \(\nu(\alpha_n)\) & \(\nu(a_n)\) & \(\nu(\LC(a_n))\) & \(\nu(q_n)\) & \(\nu(\LC(q_n))\)\\
\hline
0 & 0 & 0 & 0 & 0 & 0 & 0\\
1 & 1 & 0 & 0 & 0 & 0 & 0\\
2 & 2 & 0 & 0 & 0 & 0 & 0\\
3 & 3 & 0 & 0 & 0 & 0 & 0\\
4 & 4 & 0 & 0 & 0 & 0 & 0\\
5 & 5 & 0 & 0 & 0 & 0 & 0\\
6 & 5 & \(-\infty\) & -2 & -1 & -2 & -1\\
7 & 6 & 2 & 2 & 3 & 2 & 2\\
8 & 7 & -4 & -4 & -4 & -2 & -2\\
9 & 8 & 4 & 4 & 4 & 2 & 2\\
10 & 9 & -4 & -4 & -4 & -2 & -2\\
11 & 10 & 4 & 4 & 4 & 2 & 2\\
12 & 11 & -4 & -4 & -4 & -2 & -2\\
13 & 11 & \(-\infty\) & 2 & 3 & 0 & 1\\
14 & 12 & -2 & -2 & -1 & 0 & 0\\
15 & 13 & 0 & 0 & 0 & 0 & 0\\
16 & 14 & 0 & 0 & 0 & 0 & 0\\
17 & 15 & 0 & 0 & 0 & 0 & 0\\
18 & 16 & 0 & 0 & 0 & 0 & 0\\
19 & 17 & 0 & 0 & 0 & 0 & 0\\
20 & 17 & \(-\infty\) & -2 & -1 & -2 & -1\\
21 & 18 & 2 & 2 & 3 & 2 & 2\\
22 & 19 & -4 & -4 & -4 & -2 & -2\\
23 & 20 & 4 & 4 & 4 & 2 & 2\\
24 & 21 & -4 & -4 & -4 & -2 & -2\\
25 & 22 & 4 & 4 & 4 & 2 & 2\\
26 & 23 & -4 & -4 & -4 & -2 & -2\\
27 & 23 & \(-\infty\) & 2 & 3 & 0 & 1\\
28 & 24 & -2 & -2 & -1 & 0 & 0\\
29 & 25 & 0 & 0 & 0 & 0 & 0\\
30 & 26 & 0 & 0 & 0 & 0 & 0\\
31 & 27 & 0 & 0 & 0 & 0 & 0\\
32 & 28 & 0 & 0 & 0 & 0 & 0\\
33 & 29 & 0 & 0 & 0 & 0 & 0\\
34 & 29 & \(-\infty\) & -2 & -1 & -2 & -1\\
35 & 30 & 2 & 2 & 3 & 2 & 2\\
36 & 31 & -4 & -4 & -4 & -2 & -2\\
37 & 32 & 4 & 4 & 4 & 2 & 2\\
\end{tabular}
\end{table}

Table \ref{cfr-mod19-valuations-table} is for the 19-adic valuations. The patterns are very similar to the table for \(5\). We can also read off the quasi-period of \(\CF(\sqrt{D_{19}})\): it is \(6\), hence even. The alternating signs of the valuations  then lead to cancellation of exponents at the 2-element fibres. However it remains an open question whether these valuations are eventually periodic. If so, our computations suggest that their period length must be significantly larger than the quasi-period length of \(\CF(\sqrt{D_\pp})\).

Note that the torsion order of \(\OO_{19}\) is 7, while the torsion order of \(\OO_5\) is 8 (from Proposition \ref{prop-bounds-torsion-period-length}). This implies that \(\CF(\sqrt{D})\) cannot be periodic, using the arguments from Remark \ref{rem-periodicity-test-reduction-two-primes} in Section \ref{sec:org36b2a71} with reduction modulo two primes.

\subsection{Genus 2}
\label{sec:org0b837d5}
We also give an example of degree \(6\), to illustrate the difficulties arising in higher genus, and the more complicated patterns of the valuations in this case. Moreover, we will find convergents where \(\Redn{p_n}\) and \(\Redn{q_n}\) share a common linear factor.

\begin{table}[h!]
\centering
\caption*{Example \cfexample{ex-cfp2-zero-pattern-deg6}}
\begin{tabular}{|l|l|}
\hline
\(D = X^{6} + 7 X^{4} + 8 X^{3} + 9 X^{2} + 5\) & basefield \(\Q\)\\ \hline
Discriminant of $D$: \(-1 \cdot 2^{10} \cdot 5 \cdot 7^{2} \cdot 353^{2}\) & $D$ never reduces to a square.\\ \hline
\(D\) is not Pellian & \\ \hline
\multicolumn{2}{|l|}{Partial quotients of \(\sqrt{D}\)}\\ \hline
\multicolumn{2}{|l|}{\(a_{0} = X^{3} + \frac{7}{2} X + 4\)}\\ \hline
\multicolumn{2}{|l|}{\(a_{1} = -\frac{8}{13} X + \frac{896}{169}\)}\\ \hline
\multicolumn{2}{|l|}{\(a_{2} = -\frac{2197}{100508} X - \frac{112744970}{631366129}\)}\\ \hline
\multicolumn{2}{|l|}{\(\deg a_n = 3, 1, 1, 1, 1, \dots\)}\\ \hline
\end{tabular}
\end{table}

Example \ref{ex-cfp2-zero-pattern-deg6} is again a random non-periodic example. See how the coefficient size explodes worse than in the genus 1 example. And observe that the prime \(13\) appears already in \(a_1\). However, it turns out that \(\CF(\sqrt{D_{13}})\) has a quite long quasi-period length: it is \(126\).

So unlike in genus \(1\), the first occurrence of a prime \(\pp\) in a denominator of the \(a_n\) does not give so much information on the quasi-period length of \(\CF(\sqrt{D_\pp})\).

\subsubsection{modulo 3}
\label{sec:org2ff323f}

\begin{table}[hp]
\caption{\label{cf2-mod3-degrees-table}
Degrees for reduction mod 3 in Example \ref{ex-cfp2-zero-pattern-deg6}}
\centering
\begin{tabular}{rrrrrrr}
\(n\) & \(m\) & \(\deg a_n\) & \(\deg c_m\) & \(\deg q_n\) & \(\deg \Redn{q_n}\) & \(\deg v_m\)\\
\hline
0 & 0 & 3 & 3 & 0 & 0 & 0\\
1 & 1 & 1 & 1 & 1 & 1 & 1\\
2 & 2 & 1 & 1 & 2 & 2 & 2\\
3 & 2 & 1 & 1 & 3 & 2 & 2\\
4 & 3 & 1 & 2 & 4 & 4 & 4\\
5 & 4 & 1 & 1 & 5 & 5 & 5\\
6 & 5 & 1 & 1 & 6 & 6 & 6\\
7 & 5 & 1 & 1 & 7 & 7 & 6\\
8 & 5 & 1 & 1 & 8 & 6 & 6\\
9 & 6 & 1 & 3 & 9 & 9 & 9\\
10 & 7 & 1 & 1 & 10 & 10 & 10\\
11 & 8 & 1 & 1 & 11 & 11 & 11\\
12 & 8 & 1 & 1 & 12 & 11 & 11\\
13 & 9 & 1 & 2 & 13 & 13 & 13\\
14 & 10 & 1 & 1 & 14 & 14 & 14\\
15 & 11 & 1 & 1 & 15 & 15 & 15\\
16 & 11 & 1 & 1 & 16 & 16 & 15\\
17 & 11 & 1 & 1 & 17 & 15 & 15\\
18 & 12 & 1 & 3 & 18 & 18 & 18\\
19 & 13 & 1 & 1 & 19 & 19 & 19\\
20 & 14 & 1 & 1 & 20 & 20 & 20\\
21 & 14 & 1 & 1 & 21 & 20 & 20\\
22 & 15 & 1 & 2 & 22 & 22 & 22\\
23 & 16 & 1 & 1 & 23 & 23 & 23\\
24 & 17 & 1 & 1 & 24 & 24 & 24\\
25 & 17 & 1 & 1 & 25 & 25 & 24\\
26 & 17 & 1 & 1 & 26 & 24 & 24\\
27 & 18 & 1 & 3 & 27 & 27 & 27\\
28 & 19 & 1 & 1 & 28 & 28 & 28\\
\end{tabular}
\end{table}

In Table \ref{cf2-mod3-degrees-table}, we compare the degrees of partial quotients between \(\CF(\sqrt{D})\) and its reduction \(\CF(\sqrt{D_3})\). We put \(m = \lambda(n)\), and be aware that the columns depending on \(m\) contain \emph{repeated entries}. 

Note particularly that the sequence of the \(\deg \Redn{q_n}\) is also decreasing, and sometimes is larger than the corresponding \(\deg v_m\). This means that \(\Redn{p_n}, \Redn{q_n}\) have a common linear factor. This of course happens here only in the 3-element fibres of \(\lambda\) (in the table \(m = \lambda(n)\)). For example
\begin{align*}
\Redn{p_7} &= (X + 1)\cdot(2 X^{9} + 2 X^{8} + X^{7} + 2 X^{6} + 2 X^{5} + 2 X^{4} + 2),\\ \Redn{q_7} &= (X + 1)\cdot(2 X^{6} + 2 X^{5} + 2 X^{3} + X^{2}).
\end{align*}

\begin{table}[hp]
\caption{\label{cf2-mod3-valuations-table}
3-adic valuations for Example \ref{ex-cfp2-zero-pattern-deg6}}
\centering
\begin{tabular}{rrrrrrr}
\(n\) & \(\lambda(n)\) & \(\nu(\alpha_n)\) & \(\nu(a_n)\) & \(\nu(\LC(a_n))\) & \(\nu(q_n)\) & \(\nu(\LC(q_n))\)\\
\hline
0 & 0 & 0 & 0 & 0 & 0 & 0\\
1 & 1 & 0 & 0 & 0 & 0 & 0\\
2 & 2 & 0 & 0 & 0 & 0 & 0\\
3 & 2 & \(-\infty\) & -2 & -1 & -2 & -1\\
4 & 3 & 2 & 2 & 3 & 2 & 2\\
5 & 4 & -4 & -4 & -4 & -2 & -2\\
6 & 5 & 4 & 4 & 4 & 2 & 2\\
7 & 5 & \(-\infty\) & -5 & -5 & -3 & -3\\
8 & 5 & 5 & 6 & 6 & 2 & 3\\
9 & 6 & -5 & -5 & -5 & -2 & -2\\
10 & 7 & 4 & 4 & 4 & 2 & 2\\
11 & 8 & -4 & -4 & -4 & -2 & -2\\
12 & 8 & \(-\infty\) & 0 & 2 & -2 & 0\\
13 & 9 & 0 & 0 & 2 & 2 & 2\\
14 & 10 & -4 & -4 & -4 & -2 & -2\\
15 & 11 & 4 & 4 & 4 & 2 & 2\\
16 & 11 & \(-\infty\) & -5 & -5 & -3 & -3\\
17 & 11 & 5 & 6 & 6 & 2 & 3\\
18 & 12 & -5 & -5 & -5 & -2 & -2\\
19 & 13 & 4 & 4 & 4 & 2 & 2\\
20 & 14 & -4 & -4 & -4 & -2 & -2\\
21 & 14 & \(-\infty\) & 2 & 3 & 0 & 1\\
22 & 15 & -2 & -2 & -1 & 0 & 0\\
23 & 16 & 0 & 0 & 0 & 0 & 0\\
24 & 17 & 0 & 0 & 0 & 0 & 0\\
25 & 17 & \(-\infty\) & -2 & -2 & -2 & -2\\
26 & 17 & 2 & 4 & 4 & 0 & 2\\
27 & 18 & -2 & -2 & -2 & 0 & 0\\
28 & 19 & 0 & 0 & 0 & 0 & 0\\
\end{tabular}
\end{table}

The patterns for the valuations in Table \ref{cf2-mod3-valuations-table} are now more interesting, as there are fibres of \(\lambda\) with \(2\) or \(3\) elements. But at least these are still isolated. Observe the differences between the valuation of the entire polynomial (the Gauss norm) and of the leading coefficient between 2-element fibres and 3-element fibres. Note that now odd valuations are occurring.

Also, we cannot read off the quasi-period length of \(\CF(\sqrt{D_3})\) just by counting the rows with only zero valuations. From Table \ref{cf2-mod3-degrees-table}, we know that it is actually \(6\), not \(3\) (by looking for \(c_m\) of degree \(3\) which first occurs for \(m = 6\)). This corresponds to torsion order \(9 = \deg p_{5}\) of \(\OO_{3}\) (via Theorem \ref{thm-pellian-iff-torsion} and Remark \ref{convergent-omega-degree}).

\subsubsection{modulo 19}
\label{sec:org5ccc0d4}
Another interesting prime would be \(19\). There \(\CF(\sqrt{D_{19}})\) has (quasi-)period length \(6\), with degrees of the \(a_n\) having the periodic pattern
\(\deg a_n = \overline{3, 1, 1, 1, 1, 1}.\)

This degree pattern implies that \(\lambda\) has only fibres with \(1\) or \(3\) elements.

\begin{table}[hp]
\caption{\label{cf2-mod19-valuations-table}
19-adic valuations for Example \ref{ex-cfp2-zero-pattern-deg6}}
\centering
\begin{tabular}{rrrrrrr}
\(n\) & \(\lambda(n)\) & \(\nu(\alpha_n)\) & \(\nu(a_n)\) & \(\nu(\LC(a_n))\) & \(\nu(q_n)\) & \(\nu(\LC(q_n))\)\\
\hline
0 & 0 & 0 & 0 & 0 & 0 & 0\\
1 & 1 & 0 & 0 & 0 & 0 & 0\\
2 & 2 & 0 & 0 & 0 & 0 & 0\\
3 & 3 & 0 & 0 & 0 & 0 & 0\\
4 & 4 & 0 & 0 & 0 & 0 & 0\\
5 & 5 & 0 & 0 & 0 & 0 & 0\\
6 & 5 & \(-\infty\) & -1 & -1 & -1 & -1\\
7 & 5 & 1 & 2 & 2 & 0 & 1\\
8 & 6 & -1 & -1 & -1 & 0 & 0\\
9 & 7 & 0 & 0 & 0 & 0 & 0\\
10 & 8 & 0 & 0 & 0 & 0 & 0\\
11 & 9 & 0 & 0 & 0 & 0 & 0\\
12 & 10 & 0 & 0 & 0 & 0 & 0\\
13 & 11 & 0 & 0 & 0 & 0 & 0\\
14 & 11 & \(-\infty\) & -1 & -1 & -1 & -1\\
15 & 11 & 1 & 2 & 2 & 0 & 1\\
16 & 12 & -1 & -1 & -1 & 0 & 0\\
17 & 13 & 0 & 0 & 0 & 0 & 0\\
18 & 14 & 0 & 0 & 0 & 0 & 0\\
19 & 15 & 0 & 0 & 0 & 0 & 0\\
20 & 16 & 0 & 0 & 0 & 0 & 0\\
21 & 17 & 0 & 0 & 0 & 0 & 0\\
22 & 17 & \(-\infty\) & -1 & -1 & -1 & -1\\
23 & 17 & 1 & 2 & 2 & 0 & 1\\
24 & 18 & -1 & -1 & -1 & 0 & 0\\
25 & 19 & 0 & 0 & 0 & 0 & 0\\
26 & 20 & 0 & 0 & 0 & 0 & 0\\
27 & 21 & 0 & 0 & 0 & 0 & 0\\
28 & 22 & 0 & 0 & 0 & 0 & 0\\
29 & 23 & 0 & 0 & 0 & 0 & 0\\
30 & 23 & \(-\infty\) & -1 & -1 & -1 & -1\\
31 & 23 & 1 & 2 & 2 & 0 & 1\\
32 & 24 & -1 & -1 & -1 & 0 & 0\\
33 & 25 & 0 & 0 & 0 & 0 & 0\\
34 & 26 & 0 & 0 & 0 & 0 & 0\\
35 & 27 & 0 & 0 & 0 & 0 & 0\\
\end{tabular}
\end{table}

In Table \ref{cf2-mod19-valuations-table}, note how the valuations are reset to \(0\) after the 3-element fibres. This illustrates nicely why we require infinitely many fibres of \(\lambda\) with multiple elements in Proposition \ref{fibre-conditions-infinite-poles}.

\subsubsection{modulo 5}
\label{sec:org0144a9b}
So far, the regularity of these valuation patterns for \(\deg D = 6\) has been deceiving, so let us look at the 5-adic valuations too. The quasi-period length of \(\CF(\sqrt{D_5})\) is just \(6\). The partial quotients period is
\(\deg a_n = \overline{3, 1, 1, 2, 1, 1}.\)

So compared to \(\pp = 3\), there are also 2-element fibres. This makes the patterns much more complicated, as seen in Table \ref{cf2-mod5-valuations-table} (and in other examples, this might be even worse).

\begin{table}[hp]
\caption{\label{cf2-mod5-valuations-table}
5-adic valuations for Example \ref{ex-cfp2-zero-pattern-deg6}}
\centering
\begin{tabular}{rrrrrrr}
\(n\) & \(\lambda(n)\) & \(\nu(\alpha_n)\) & \(\nu(a_n)\) & \(\nu(\LC(a_n))\) & \(\nu(q_n)\) & \(\nu(\LC(q_n))\)\\
\hline
0 & 0 & 0 & 0 & 0 & 0 & 0\\
1 & 1 & 0 & 0 & 0 & 0 & 0\\
2 & 2 & 0 & 0 & 0 & 0 & 0\\
3 & 2 & \(-\infty\) & -6 & -3 & -6 & -3\\
4 & 3 & 6 & 6 & 9 & 6 & 6\\
5 & 4 & -12 & -12 & -12 & -6 & -6\\
6 & 5 & 12 & 12 & 12 & 6 & 6\\
7 & 5 & \(-\infty\) & -13 & -13 & -7 & -7\\
8 & 5 & 13 & 14 & 14 & 6 & 7\\
9 & 6 & -13 & -13 & -13 & -6 & -6\\
10 & 7 & 12 & 12 & 12 & 6 & 6\\
11 & 8 & -12 & -12 & -12 & -6 & -6\\
12 & 8 & \(-\infty\) & 6 & 9 & 0 & 3\\
13 & 9 & -6 & -6 & -3 & 0 & 0\\
14 & 10 & 0 & 0 & 0 & 0 & 0\\
15 & 11 & 0 & 0 & 0 & 0 & 0\\
16 & 11 & \(-\infty\) & -1 & -1 & -1 & -1\\
17 & 11 & 1 & 2 & 2 & 0 & 1\\
18 & 12 & -1 & -1 & -1 & 0 & 0\\
19 & 13 & 0 & 0 & 0 & 0 & 0\\
20 & 14 & 0 & 0 & 0 & 0 & 0\\
21 & 14 & \(-\infty\) & -8 & -4 & -8 & -4\\
22 & 15 & 8 & 8 & 12 & 8 & 8\\
23 & 16 & -16 & -16 & -16 & -8 & -8\\
24 & 17 & 16 & 16 & 16 & 8 & 8\\
25 & 17 & \(-\infty\) & -17 & -17 & -9 & -9\\
26 & 17 & 17 & 18 & 18 & 8 & 9\\
27 & 18 & -17 & -17 & -17 & -8 & -8\\
28 & 19 & 16 & 16 & 16 & 8 & 8\\
29 & 20 & -16 & -16 & -16 & -8 & -8\\
30 & 20 & \(-\infty\) & 10 & 13 & 2 & 5\\
31 & 21 & -10 & -10 & -7 & -2 & -2\\
32 & 22 & 4 & 4 & 4 & 2 & 2\\
33 & 23 & -4 & -4 & -4 & -2 & -2\\
34 & 23 & \(-\infty\) & 3 & 3 & 1 & 1\\
35 & 23 & -3 & -2 & -2 & -2 & -1\\
36 & 24 & 3 & 3 & 3 & 2 & 2\\
37 & 25 & -4 & -4 & -4 & -2 & -2\\
\end{tabular}
\end{table}

\clearpage
\section{Non-constant degrees of partial quotients}
\label{sec:orge7e06d0}

The following example was constructed in collaboration with Prof. Zannier and Francesca Malagoli, to answer a question raised during preparation of \cite{zannier-2016-hyper-contin-fract}: In the article, it is a consequence of the Skolem-Mahler-Lech Theorem for algebraic groups (mentioned before) that the sequence of the \(\deg a_n\) (for \(\alpha = \sqrt{D}\)) becomes eventually periodic. However, in any non-periodic examples known previously, these degrees stabilised on a single value. Of course, in that case periodicity of the degrees is not very interesting. 

So we searched for an non-periodic example where the degrees assume multiple values infinitely often.

We found Example \ref{ex-dnc1-non-constant-deg} which has infinitely many partial quotients \(a_n\) both of degree \(1\) and of degree \(2\) (we remark that this would be impossible for \(\deg D = 4\) or \(6\), so we cannot do better than \(\deg D = 8\)).

\begin{table}[h!]
\centering
\caption*{Example \cfexample{ex-dnc1-non-constant-deg}}
\begin{tabular}{|p{9cm}|p{5cm}|}
\hline
\(D = X^{8} - X^{7} - \frac{3}{4} X^{6} + \frac{7}{2} X^{5} - \frac{21}{4} X^{4} + \frac{7}{2} X^{3} - \frac{3}{4} X^{2} - X + 1\) & basefield \(\Q\)\\ \hline
Discriminant of $D$: \(-1 \cdot 2^{2} \cdot 3 \cdot 13 \cdot 173^{2}\) & Primes in denominators of $D$: \(2\)\\ \hline
$D$ never reduces to a square. & \\ \hline
\(D\) is not Pellian because of incompatible torsion orders & torsion order \(10\) modulo \(3\)\\ \hline
 & torsion order \(40\) modulo \(11\)\\ \hline
\multicolumn{2}{|p{14cm}|}{Partial quotients of \(\sqrt{D}\)}\\ \hline
\multicolumn{2}{|p{14cm}|}{\(a_{0} = X^{4} - \frac{1}{2} X^{3} - \frac{1}{2} X^{2} + \frac{3}{2} X - 2\)}\\ \hline
\multicolumn{2}{|p{14cm}|}{\(a_{1} = \frac{2}{3} X + \frac{7}{9}\)}\\ \hline
\multicolumn{2}{|p{14cm}|}{\(a_{2} = -\frac{27}{4} X - \frac{45}{8}\)}\\ \hline
\multicolumn{2}{|p{14cm}|}{\(a_{3} = \frac{16}{81} X^{2} - \frac{64}{81} X + \frac{200}{81}\)}\\ \hline
\multicolumn{2}{|p{14cm}|}{\(a_{4} = -\frac{27}{200} X - \frac{171}{400}\)}\\ \hline
\multicolumn{2}{|p{14cm}|}{\(\deg a_n = 4, 1, 1, 2, 1, 1, 1, 1, 1, 1, 1, 1, 2, 1, 1, 1, 1, 1, 1, 1, 1, 2, 1, 1, 1, 1, 1, 1, 1, 1, 2, 1, \dots\)}\\ \hline
\end{tabular}
\end{table}

This is related to the fact that the Jacobian in Example \ref{ex-dnc1-non-constant-deg} is not simple. It contains an elliptic curve, and infinitely many multiples of the point \(\OO\) lie on a certain translate of it. This causes the degrees of the \(a_n\) to follow the pattern \(4, 1, 1, \overline{2, 1, 1, 1, 1, 1, 1, 1, 1}\). For details, we refer to an article in preparation together with Malagoli and Zannier.

\medskip

Here we remark only that if we reduce modulo 3, we actually get a square-free polynomial in \((X+1)^2\) (and divisible by \((X+1)^2\) too, hence the \(3\) in the discriminant). So all the partial quotients have at least degree \(2\) (see also Table \ref{cfnc1-mod3-degrees-table}).
\begin{table}[hp]
\caption{\label{cfnc1-mod3-degrees-table}
Degrees modulo 3 for Example \ref{ex-dnc1-non-constant-deg}}
\centering
\begin{tabular}{rrrrrrr}
\(n\) & \(m\) & \(\deg a_n\) & \(\deg c_m\) & \(\deg q_n\) & \(\deg \Redn{q_n}\) & \(\deg v_m\)\\
\hline
0 & 0 & 4 & 4 & 0 & 0 & 0\\
1 & 0 & 1 & 4 & 1 & 0 & 0\\
2 & 1 & 1 & 2 & 2 & 2 & 2\\
3 & 2 & 2 & 2 & 4 & 4 & 4\\
4 & 2 & 1 & 2 & 5 & 4 & 4\\
5 & 3 & 1 & 2 & 6 & 6 & 6\\
6 & 3 & 1 & 2 & 7 & 7 & 6\\
7 & 3 & 1 & 2 & 8 & 7 & 6\\
8 & 3 & 1 & 2 & 9 & 6 & 6\\
9 & 4 & 1 & 4 & 10 & 10 & 10\\
10 & 4 & 1 & 4 & 11 & 10 & 10\\
11 & 5 & 1 & 2 & 12 & 12 & 12\\
12 & 6 & 2 & 2 & 14 & 14 & 14\\
13 & 6 & 1 & 2 & 15 & 14 & 14\\
14 & 7 & 1 & 2 & 16 & 16 & 16\\
15 & 7 & 1 & 2 & 17 & 17 & 16\\
16 & 7 & 1 & 2 & 18 & 17 & 16\\
17 & 7 & 1 & 2 & 19 & 16 & 16\\
18 & 8 & 1 & 4 & 20 & 20 & 20\\
19 & 8 & 1 & 4 & 21 & 20 & 20\\
20 & 9 & 1 & 2 & 22 & 22 & 22\\
21 & 10 & 2 & 2 & 24 & 24 & 24\\
22 & 10 & 1 & 2 & 25 & 24 & 24\\
23 & 11 & 1 & 2 & 26 & 26 & 26\\
24 & 11 & 1 & 2 & 27 & 27 & 26\\
25 & 11 & 1 & 2 & 28 & 27 & 26\\
26 & 11 & 1 & 2 & 29 & 26 & 26\\
\end{tabular}
\end{table}

Note that \(\lambda\) has still infinitely many fibres with a single element. Observe the sequence \(\deg \Redn{q_n}\) is sometimes decreasing, so there are again convergents with a common factor between \(\Redn{p_n}\) and \(\Redn{q_n}\).

\section{Recurring partial quotients}
\label{sec:orge27ae2f}
Recall that \cite{zannier-2016-hyper-contin-fract} gave a lower bound for an average of the affine heights of partial quotients (see the end of Section \ref{sec:org1835e85}). A strengthening of this would be a bound like
\begin{equation*}
C \, n^2 \leq \hproj(a_n)
\end{equation*}
for the non-Pellian case.

However, together with Prof. Zannier and Francesca Malagoli, and some assistance from Solomon Vishkautsan for the computations, we have found Example \ref{ex-dbh1-bounded-height} below. There for \(n = 7 + 17 \, j \pm 1, \; j \in \N_0\) the partial quotients have the shape
\begin{equation*}
a_n = C_n \, (X-2), \quad C_n \in \closure{\Q}
\end{equation*}
so in particular \(\hproj(a_n)\) remains constant on this subsequence and the above lower bound is \emph{impossible} in general.

\begin{table}[h!]
\centering
\caption*{Example \cfexample{ex-dbh1-bounded-height}}
\begin{tabular}{|p{9cm}|p{5cm}|}
\hline
\(D = X^{12} + (-8 \tau^{4} + 6 \tau^{3} - 28 \tau^{2} + 22 \tau + 22) X^{10} + (-8 \tau^{4} + 6 \tau^{3} - 28 \tau^{2} + 22 \tau + 22) X^{9} + (83 \tau^{4} - 62 \tau^{3} + 291 \tau^{2} - 225 \tau - 309) X^{8} + (166 \tau^{4} - 124 \tau^{3} + 582 \tau^{2} - 450 \tau - 618) X^{7} + (-127 \tau^{4} + 92 \tau^{3} - 447 \tau^{2} + 327 \tau + 529) X^{6} + (-630 \tau^{4} + 462 \tau^{3} - 2214 \tau^{2} + 1656 \tau + 2514) X^{5} + (-538 \tau^{4} + 398 \tau^{3} - 1893 \tau^{2} + 1434 \tau + 2115) X^{4} + (158 \tau^{4} - 102 \tau^{3} + 546 \tau^{2} - 336 \tau - 758) X^{3} + (552 \tau^{4} - 384 \tau^{3} + 1926 \tau^{2} - 1332 \tau - 2394) X^{2} + (368 \tau^{4} - 256 \tau^{3} + 1284 \tau^{2} - 888 \tau - 1596) X + 92 \tau^{4} - 64 \tau^{3} + 321 \tau^{2} - 222 \tau - 399\) & basefield \(K = \Q(\tau)\), where \(\tau\) has minimal polynomial \(t^{5} + 3 t^{3} - 6 t - 3\)\\ \hline
 & $D$ never reduces to a square.\\ \hline
\(D\) is not Pellian because of incompatible torsion orders & torsion order \(42\) modulo \(\tau\)\\ \hline
 & torsion order \(861\) modulo \(3 \tau^{4} - 2 \tau^{3} + 11 \tau^{2} - 8 \tau - 11\)\\ \hline
\multicolumn{2}{|p{14cm}|}{Partial quotients of \(\sqrt{D}\)}\\ \hline
\multicolumn{2}{|p{14cm}|}{\(a_{0} = X^{6} + (-4 \tau^{4} + 3 \tau^{3} - 14 \tau^{2} + 11 \tau + 11) X^{4} + (-4 \tau^{4} + 3 \tau^{3} - 14 \tau^{2} + 11 \tau + 11) X^{3} + (16 \tau^{4} - 12 \tau^{3} + 56 \tau^{2} - 43 \tau - 65) X^{2} + (32 \tau^{4} - 24 \tau^{3} + 112 \tau^{2} - 86 \tau - 130) X + 28 \tau^{4} - 22 \tau^{3} + 98 \tau^{2} - 79 \tau - 119\)}\\ \hline
\multicolumn{2}{|p{14cm}|}{\(a_{1} = (\frac{7}{18} \tau^{4} - \frac{5}{18} \tau^{3} + \frac{25}{18} \tau^{2} - \tau - \frac{5}{3}) X - \frac{118}{81} \tau^{4} + \frac{55}{54} \tau^{3} - \frac{140}{27} \tau^{2} + \frac{98}{27} \tau + \frac{115}{18}\)}\\ \hline
\multicolumn{2}{|p{14cm}|}{\(a_{2} = (\frac{54453438756}{1411680971} \tau^{4} + \frac{15324049962}{1411680971} \tau^{3} + \frac{225993869532}{1411680971} \tau^{2} + \frac{96839726742}{1411680971} \tau - \frac{49873817664}{1411680971}) X - \frac{233866366289476431120}{1992843163883502841} \tau^{4} - \frac{33996626533462319838}{1992843163883502841} \tau^{3} - \frac{951199485540229734912}{1992843163883502841} \tau^{2} - \frac{278019658364688097056}{1992843163883502841} \tau + \frac{305591634283859417718}{1992843163883502841}\)}\\ \hline
\multicolumn{2}{|p{14cm}|}{\(\deg a_n = 6, 1, 1, 1, 1, 1, 1, 3, 1, 1, 1, 1, 1, 1, 1, 2, 2, 1, 1, 1, 1, 1, 1, 1, 3, 1, 1, 1, 1, 1, \dots\)}\\ \hline
\end{tabular}
\end{table}

This also gives an example where \(\deg a_n\) assumes three different values infinitely often, and again this is related to the Jacobian containing an elliptic curve. We hope to describe this example in much more detail in the article in preparation together with Malagoli and Zannier mentioned above.

\appendix
\chapter{Appendix}
\label{sec:orgbd3afba}
\section{Polynomial Pell equation in characteristic 2}
\label{sec:orgfe2a066}
Let us quickly have a look at the polynomial Pell equation in characteristic \(2\) and give a criterion which allows to easily test for and construct solutions in this case.
\begin{thm}
Let \(\K\) a field of characteristic \(2\) and \(D \in \K[X]\). There exists a non-trivial solution (with \(q \neq 0\)) of
\begin{equation*}
p^2 - D \, q^2 = \eta, \qquad p, q \in \K[X], \eta \in \units \K
\end{equation*}
if and only if there exist \(E \in \K[X], r \in \K\) such that \(D = E^2 + r\).

Moreover, \(r = 0\) is possible if and only if there exists a non-trivial solution with \(\eta\) a \emph{square}.
\end{thm}
\begin{proof}
Let us first treat the second case \(r = 0\). Suppose \(D = E^2\), and choose \(\mu \in \units\K\), \(p = E-\mu, \; q = 1\). This yields
\begin{equation*}
p^2 - D \, q^2 = (E-\mu)^2 - E^2 = \mu^2 = \eta,
\end{equation*}
hence \(\eta\) can be chosen a square.

On the other hand, suppose \((p,q) \in \K[X]^2\) with \(q \neq 0\) is a solution with \(\eta = \mu^2\) a square, then
\begin{equation*}
D \, q^2 = p^2 - \mu^2 = (p - \mu)^2
\end{equation*}
implies \(D\) is a square because \(\K[X]\) is a unique factorisation domain.

For the general case, note that if \(D = E^2 + r\) with \(r \neq 0\), then
\begin{equation*}
p = E, q = 1, \eta = r \implies p^2 - D \, q^2 = -r = \eta
\end{equation*}
gives the desired non-trivial solution.

Conversely, if there exists with a solution \((p,q) \in \K[X]^2\) with \(q \neq 0\), set \(K = \K(\sqrt{\eta})\) and reduce to the case with \(r = 0\) -- we now write \(D = E^2\) with \(E \in K[X]\), or rather \(E = E_0 + \mu \, E_1\) with \(E_0, E_1 \in \K[X]\) (here again \(\mu = \sqrt{\eta}\)). We obtain
\begin{equation*}
D = E^2 = E_0^2 + \mu^2 \, E_1^2 = E_0^2 + \eta \, E_1^2
\end{equation*}
and plugging it into the Pell equation we have
\begin{equation*}
0 = p^2 - q^2 \, \left(E_0^2 - \eta \, E_1^2\right) + \eta = (p - q \, E_0)^2 - \eta \, (q \, E_1 +1)^2
\end{equation*}
If \((q \, E_1 + 1) \neq 0\), then \(\mu = \ifracBB{p-q\,E_0}{q \, E_1 + 1}\in K \cap \K(X) = \K\) and we are actually in the first case. Otherwise, \(q \, E_1 = 1\), so \(E_1 \in \units \K\) (because \(q \in \K[X]\)), hence \(D = E_0^2 + \eta \, E_1^2 = E_0^2 + r\) with \(r = \eta \, E_1^2 \in \K\).
\end{proof}
\begin{rem}
So if we require \(\eta = 1\), we see that in characteristic \(2\) non-trivial solutions to the Pell equations only exist if \(D\) is actually a square.
\end{rem}

\begin{rem}
The proof also yields a classification of the Pell solutions:

For the first case with \(D = E^2\), the solutions always have the shape \(p = q \, E - \mu\). And obviously, we are free to choose \(q\) here, so there are a lot of non-trivial solutions in this case.

In the second case with \(D = E^2 + r\), we need to expand this observation. But note that we actually showed \(q \in \units \K\) in the above proof, so essentially \(q = 1\) after multiplying \(\eta\) with a square factor. Hence there is \emph{only one} non-trivial solution up to a constant factor.
\end{rem}

\section{Valuations in Laurent series quotients}
\label{sec:org5f19ed2}
The problem that arises with bad reduction is that we can no longer read off \(\nu(\alpha_n)\) from the leading coefficient. This also means that \(\nu(a_n)\) could be different, so we need to compute the valuations of the coefficients of \(\alpha_n\). As the latter can be written as a quotient of \(\vartheta_i\)'s, we naturally need to study quotients of Laurent series.

Indeed we may work with quotients of power series, as multiplying with powers of \(X\) only shifts coefficient indices. For convenience, we work in \(\powerseries{K}{Z}\) (think \(Z = \inv X\)) to avoid negative indices.

As in Chapters \ref{sec:orgd5f1900} and \ref{sec:org5f9d2ce}, \(K\) is the fraction field of a discrete valuation ring \(\O\) with maximal ideal \(\mm\) and valuation \(\nu\).

\medskip

Let \(a_n, c_n \in \O, b_n \in K\), and consider the Cauchy product
\begin{equation*}
\left(\sum_{n=0}^\infty a_n \, Z^n\right) \left( \sum_{n=0}^\infty b_n \, Z^n\right) = \sum_{n=0}^\infty c_n \, Z^n.
\end{equation*}

For the coefficients, we get the relations
\begin{equation*}
c_n = \sum_{i+j=n} a_i \, b_j
\end{equation*}
which we can recursively solve to \(b_n\) as
\begin{equation}
\label{series-inverse-rec-formula}
b_n = \frac{1}{a_0} \left( c_n - \sum_{i+j=n,\atop i\neq 0} a_i \, b_j \right).
\end{equation}
For the first couple of indices, we compute
\begin{align*}
b_0 &= \frac{c_0}{a_0} \\
b_1 &= \frac{1}{a_0^2} \, \left(a_0 \, c_1 - a_1 \, c_0 \right) \\
b_2 &= \frac{1}{a_0^3} \, \left(a_{1}^{2} c_{0} - a_{0} a_{2} c_{0} - a_{0} a_{1} c_{1} + a_{0}^{2} c_{2} \right) 
\end{align*}
So we can try to calculate or estimate the valuations of the coefficient with these formulas.

The following Lemma addresses the simplest case (sufficient to treat \(\deg D = 4\)).
\begin{lemma}
\label{cauchy-valuation-lemma}
\begin{itemize}
\item Suppose \(\nu(c_0) > 0\), but \(\nu(c_1) = \nu(a_0) = 0\). Then \(\nu(b_0) = \nu(c_0) > 0\) and \(\nu(b_1) = 0\).
\item Suppose \(\nu(a_0) > 0\) and \(\nu(c_0) = \nu(a_1) = 0\). Then \(\nu(b_0) = - \nu(a_0) < 0\) and \(\nu(b_1) = -2 \, \nu(a_0) < 0\).
\end{itemize}
\end{lemma}
\begin{proof}
The valuation of \(b_0\) is obvious. In the first situation, we deduce from \(\nu(c_0) > 0\) and \(\nu(a_1) \geq 0\)
\begin{equation*}
\nu(b_1) = \nu(c_1 \, a_0 - c_0 \, a_1) = \min(0, \nu(c_0) + \nu(a_1)) = 0.
\end{equation*}

In the second situation, \(\nu(c_1) \geq 0, \nu(a_0) > 0\) implies
\begin{equation*}
\nu(b_1) = - 2\, \nu(a_0) + \nu(c_1 \, a_0 - c_0 \, a_1) = - 2 \, \nu(a_0) + \min(\nu(c_1) + \nu(a_0), 0) = -2 \, \nu(a_0).
\end{equation*}
\end{proof}

We can actually generalise this somewhat, but first we need a better description of the formulas for the \(b_n\):
\begin{prop}
Define \(B_n = - (-a_0)^{n+1} \, b_n\). Then we find
\begin{equation*}
B_n = \sum_{i_0 + \dots + i_l = n\atop 0\leq i_0 \leq n,  1 \leq i_1, \dots, i_l \leq n} c_{i_0} \, a_{i_1} \cdots a_{i_l} \, (-a_0)^{n-l}.
\end{equation*}
\end{prop}

Essentially, we are summing over integer partitions of \(n\) with (at most) \(n+1\) parts. However, except for the parts which are \(0\), the ordering of the parts matters.

\begin{proof}
We prove this by induction. Clearly \(B_0 = c_0\), precisely what the formula produces as no \(a_i\) appears in the sum.

For the induction step, we use the recursion formula \eqref{series-inverse-rec-formula}
\begin{multline*}
B_n = (-a_0)^n \, c_n + \sum_{i+j=n,\atop i \neq 0} a_i \, (-a_0)^{i-1} \, B_j \\
    = (-a_0)^n \, c_n + \sum_{i+j=n,\atop i \neq 0} a_i \, (-a_0)^{i-1} \, \sum_{i_0 + \dots + i_l = l\atop 0\leq i_0 \leq j, 1 \leq i_1, \dots, i_l \leq j} c_{i_0} \, a_{i_1} \cdots a_{i_l} \, (-a_0)^{j-l} \\
    = \sum_{i_0 + \dots + i_l + i = n\atop 0\leq i_0 \leq n,  1 \leq i_1, \dots, i_l, i\leq n} c_{i_0} \, a_{i_1} \cdots a_{i_l} \, a_{i} \, (-a_0)^{n-l-1}.
\end{multline*}
Essentially, we are recursing by fixing the last (or first) \(a_i\).
\end{proof}

\pagebreak
We can now generalise the second part of Lemma \ref{cauchy-valuation-lemma}:
\begin{prop}
\label{inverse-drop1-valuation-lemma}
If \(c_0, a_1 \in \units \O\) and \(a_0 \in \mm\), then for all \(n \geq 0\) we have \(B_n \in \units \O\). This implies \(\nu(b_n) = -(n+1) \, \nu(a_0)\).
\end{prop}
\begin{proof}
It is clear that \(B_n \in \O\), as all the summands are in \(\O\) (recall that \(a_i, c_i \in \O\)). We show that precisely one summand lies in \(\units \O\), while all others are in \(\mm\).

Of course, with \(i_0 = 0\) and \(i_j = 1\) for the rest, we get \(c_0 \, a_1^n \in \units \O\).

For all other summands, we show that \(l < n\) which implies that \(a_0\) appears in \(c_{i_0} \, a_{i_1} \cdots a_{i_l} \, (-a_0)^{n-l}\), so the product is in \(\mm\).

If still \(i_0 = 0\), but one of the \(i_j \neq 1\), i.e. \(i_j \geq 2\), then clearly \(l < i_1 + \dots + i_l = n\).

If on the other hand \(i_0 > 0\), then immediately \(l \leq i_1 + \dots + i_l < n\).
\end{proof}

\section{A lemma for a quadratic form}
\label{sec:orgab72b2f}
Let \(G\) a \(\Z\)-module (an abelian group) and \(q: G \to \R\) a quadratic form. By abuse of notation, we also denote the corresponding \(\Z\)-bilinear form by \(q : G \times G \to \R\).
\begin{lemma}
\label{quadratic-form-sum-bound}
Suppose that \(q\) is positive (i.e. \(q(g) \geq 0\) for all \(g \in G\)).
Let \(g_1, \dots, g_r \in G\). Then
\begin{equation}
\label{eq-quadratic-form-sum-bound}
q(g_1 + \dots + g_r) \leq r \cdot \left( q(g_1) + \dots + q(g_r) \right) \leq r^2 \, \max\{q(g_i) \mid i=1,\dots,r \}.
\end{equation}
\end{lemma}
\begin{proof}
Because \(q\) is a quadratic form, we have
\begin{equation*}
q(g_1 + \dots + g_r) = \sum_{i=1}^r q(g_i) + 2 \, \sum_{1 \leq i < j \leq r} q(g_i, g_j).
\end{equation*}
Moreover \(q\) positive implies that
\begin{equation*}
0 \leq q(g_i - g_j) = q(g_i) + q(g_j) - 2 \, q(g_i, g_j)
\end{equation*}
so we deduce
\begin{equation*}
q(g_1 + \dots + g_r) \leq \sum_{i=1}^r q(g_i) + 2 \, \sum_{1 \leq i < j \leq r} q(g_i) + q(g_j) = \sum_{1\leq i, j \leq r} q(g_i) = r \, \sum_{i=1}^r q(g_i).
\end{equation*}
The second inequality in \eqref{eq-quadratic-form-sum-bound} is then obvious.
\end{proof}

\bibliography{phd-thesis.en}
\bibliographystyle{amsalpha}
\end{document}